\documentclass[a4paper]{amsart}

\usepackage{amsmath,amstext,amssymb,mathrsfs,amscd,amsthm,indentfirst}
\usepackage{amsfonts}

\usepackage{booktabs}
\usepackage{verbatim}
\usepackage{enumerate}

\usepackage{graphicx}
\numberwithin{equation}{section}

\newcommand{\bC}{\mathbb{C}}

\newcommand{\bF}{\mathbb{F}}

\newcommand{\bH}{\mathbb{H}}

\newcommand{\bN}{\mathbb{N}}
\newcommand{\bO}{\mathbb{O}}

\newcommand{\bQ}{\mathbb{Q}}
\newcommand{\bR}{\mathbb{R}}

\newcommand{\bZ}{\mathbb{Z}}

\renewcommand{\phi}{\varphi}

\newcommand\lra{\longrightarrow}
\newcommand\lla{\longleftarrow}

\newcommand\Diff{\mathrm{Diff}}
\newcommand\Emb{\mathrm{Emb}}

\newcommand\Bun{\mathrm{Bun}}

\newcommand\colim{\operatorname*{colim}}
\newcommand\hocolim{\operatorname*{hocolim}}

\newcommand\Ker{\operatorname*{Ker}}

\newcommand{\Ob}{\mathrm{Ob}}
\newcommand{\Mor}{\mathrm{Mor}}
\newcommand{\map}{\mathrm{map}}

\let\emptyset\varnothing

\newcommand{\surg}{\mathscr{S}}
\newcommand{\R}{\bR}

\newcommand{\Int}{\mathrm{int}}

\newcommand{\Hom}{\mathrm{Hom}}

\newcommand{\IM}{\mathrm{Im}}

\renewcommand{\epsilon}{\varepsilon}

\newcommand{\Spin}{\mathrm{Spin}}
\newcommand{\Map}{\mathrm{Map}}
\newcommand{\Gr}{\mathrm{Gr}}

\newcommand{\nin}{\not\in}
\newcommand{\nat}{\,\natural\,}

\mathchardef\ordinarycolon\mathcode`\:
\mathcode`\:=\string"8000
\begingroup \catcode`\:=\active
  \gdef:{\mathrel{\mathop\ordinarycolon}}
\endgroup

\usepackage[all]{xy}
\SelectTips{cm}{10}
\CompileMatrices
\usepackage{amsthm}
\theoremstyle{plain}

\newtheorem{theorem}{Theorem}[section]
\newtheorem{addendum}[theorem]{Addendum}
\newtheorem{proposition}[theorem]{Proposition}
\newtheorem{lemma}[theorem]{Lemma}

\newtheorem{corollary}[theorem]{Corollary}

\theoremstyle{definition}
\newtheorem{definition}[theorem]{Definition}

\newtheorem{claim}[theorem]{Claim}

\theoremstyle{remark}

\newtheorem{remark}[theorem]{Remark}
\newtheorem*{remark*}{Remark}

\title[Stable moduli spaces]{Stable moduli spaces\\ of high dimensional manifolds}

\author{S{\o}ren Galatius}
\email{galatius@stanford.edu}
\address{Department of Mathematics\\
	Stanford University\\
	Stanford CA, 94305}
\author{Oscar Randal-Williams}
\thanks{S. Galatius was partially supported by NSF grants DMS-0805843
  and DMS-1105058 and the Clay Mathematics Institute.  Both authors
  were supported by ERC Advanced Grant No.\ 228082, and the Danish
  National Research Foundation through the Centre for Symmetry and
  Deformation.}
\email{o.randal-williams@math.ku.dk}
\address{Institut for Matematiske Fag\\
Universitetsparken 5\\
DK-2100 K{\o}benhavn {\O}\\
Denmark}

\dedicatory{Dedicated to Ib Madsen on the occasion of his 70th birthday}

\date{\today}
  %Infinite loop spaces

\begin{document}
\begin{abstract}
  We prove an analogue of the Madsen--Weiss theorem for
  high-dimensional manifolds.  For example, we explicitly describe the
  ring of characteristic classes of smooth fibre bundles whose fibres
  are connected sums of $g$ copies of $S^n \times S^n$, in the limit
  $g \to \infty$.  Rationally it is a polynomial ring in certain
  explicit generators, giving a high-dimensional analogue of Mumford's
  conjecture.

  More generally, we study a moduli space $\mathcal{N}(P)$ of those
  nullbordisms of a fixed $(2n-1)$-dimensional manifold $P$ which are
  $(n-1)$-connected relative to $P$.  We determine the homology of
  $\mathcal{N}(P)$ after stabilisation using certain self-bordisms of
  $P$.  The stable homology is identified with that of a certain
  infinite loop space.
\end{abstract}

\maketitle

\section{Introduction and statement of results}
\label{sec:intr-stat-results}

For any smooth compact manifold $W$, the diffeomorphism group
$\Diff(W)$ has a classifying space $B\Diff(W)$. This classifies smooth
fibre bundles with fibre $W$, in the sense that for a smooth manifold
$X$, there is a natural bijection between the set of isomorphism
classes of smooth fibre bundles $E \to X$ with fibre $W$ and the set
$[X, B\Diff(W)]$ of homotopy classes of maps.  The cohomology groups
$H^k(B\Diff(W))$ therefore give characteristic classes of such
bundles, and it is desirable to understand as much as possible about
these cohomology groups.  The difficulty of this question depends
highly on $W$: it is essentially completely understood when the
dimension of $W$ is 0 or 1, and much effort has been devoted to
understanding the case where the dimension of $W$ is 2. Mumford
(\cite{Mumford}) formulated a conjecture about the case where $W =
\Sigma_g$ is an oriented surface of genus $g$, in the limit $g \to
\infty$.  If we let $\Diff(\Sigma_g, D^2)$ denote the diffeomorphism
group which fixes some chosen disc $D^2 \subset \Sigma_g$, Mumford's
conjecture predicted an isomorphism
\begin{equation*}
  \varprojlim H^*(B\Diff(\Sigma_g,
  D^{2});\bQ) \cong \bQ[\kappa_1, \kappa_2, \kappa_3, \dots] 
\end{equation*}
for certain classes $\kappa_i \in H^{2i}(B\Diff(\Sigma_g,D^{2}))$.
Mumford's conjecture was finally proved by Madsen and Weiss
(\cite{MW}) in a strengthened form.

The goal of the present paper is to prove analogues of the
Madsen--Weiss theorem and Mumford's conjecture for manifolds of higher
dimension.  We have results for manifolds of any even dimension
greater than 4.  As an interesting special case of our results, we
completely determine the stable rational cohomology ring
\begin{equation*}
  \varprojlim H^*(B\Diff(W_g,D^{2n});\bQ),
\end{equation*}
where $W_g = \#^g S^n \times S^n$ denotes the connected sum of $g$
copies of $S^n \times S^n$.  To state our result, we recall that for
each characteristic class of oriented $2n$-dimensional vector bundles $c \in H^{2n+k}(BSO(2n))$, we can define the
associated \emph{generalised Mumford--Morita--Miller} class of a
smooth fibre bundle $\pi: E \to B$ with oriented $2n$-dimensional
fibres as
\begin{equation*}
  \kappa_c(E) = \pi_!(c(T_v E)) \in H^{k}(B),
\end{equation*}
where $T_v(E) = \Ker(D\pi)$ is the fibrewise tangent bundle of $\pi$.
When the fibre is taken to be $W_g$, there is a corresponding
universal class $\kappa_c \in H^k(B\Diff(W_g,D^{2n}))$ which for $k >
0$ is compatible with increasing $g$.
\begin{theorem}\label{thmcor:rational-coho}
  Let $2n > 4$ and let $\mathcal{B} \subset H^*(BSO(2n);\bQ)$ be
  the set of monomials in the classes $e, p_{n-1}, p_{n-2}, \dots,
  p_{\lceil\frac{n+1}{4}\rceil}$ of total degree greater than $2n$.
  Then the natural map
  \begin{equation*}
    \bQ[\kappa_c \,|\, c \in \mathcal{B}] \lra \varprojlim
    H^*(B\Diff(W_g,D^{2n});\bQ)
  \end{equation*}
  is an isomorphism.
\end{theorem}

The strengthened form of Mumford's conjecture proved by Madsen and
Weiss states that a certain map
\begin{equation*}
  \hocolim_{g \to \infty} B\Diff(\Sigma_g, D^2) \lra
  \Omega^\infty_\bullet MTSO(2)
\end{equation*}
induces an isomorphism in integral homology.  We will prove a similar
homotopy theoretic strengthening of
Theorem~\ref{thmcor:rational-coho}, which also applies to more general
manifolds.  

\subsection{Definitions and recollections}
\label{sec:definitions}
To state the main results in their general form, we recall the
following definitions.

\subsubsection{Classifying spaces}
\label{sec:classifying-spaces}
We shall use the following model for the classifying space
$B\Diff(W,\partial W)$ of the topological group of diffeomorphisms of
a compact manifold $W$, restricting to the identity on a neighbourhood
of $\partial W$.  We first pick an embedding $\partial W \hookrightarrow \{0\}
\times \R^\infty$ and let $\Emb^\partial(W,(-\infty, 0] \times
\R^\infty)$ denote the space of all extensions to an embedding of $W$
(required to be standard on a collar neighbourhood of $\partial W$).
We then let $B\Diff(W,\partial W) = \Emb^\partial(W, (-\infty, 0] \times
\R^\infty)/ \Diff(W,\partial W)$ be the orbit space.  If $W$ is closed
and $A \subset W$ is a compact codimension 0 submanifold, we write
$B\Diff(W,A) = B\Diff(W - \Int (A), \partial A)$.  The construction of
$B\Diff(W,\partial W)$ has the following naturality property: any
inclusion $W \subset W'$ of a codimension 0 submanifold induces a
continuous map $B\Diff(W, \partial W) \to B\Diff(W',\partial W')$,
well defined up to homotopy.  (On the point-set level it depends on a
choice of embedding of the cobordism $W' - \Int(W)$ into $[0,1] \times
\R^\infty$.)  For example, a choice of inclusion $W_g - \Int (D^{2n})
\to W_{g+1}$ induces a map $B\Diff(W_g,D^{2n}) \to
B\Diff(W_{g+1},D^{2n})$; these define the inverse system in
Theorem~\ref{thmcor:rational-coho}.

\subsubsection{Thom spectra}
\label{sec:thom-spectra}

For any space $B$ and any map $\theta: B \to
BO(d)$, where $BO(d) = \mathrm{Gr}_d(\R^\infty)$, there is a Thom
spectrum $MT\theta = B^{-\theta}$ constructed in the following way.
First, we let $B(\R^n) = \theta^{-1}(\mathrm{Gr}_d(\R^n))$.  The
Grassmannian $\mathrm{Gr}_d(\R^n)$ carries a $(n-d)$-dimensional vector
bundle $\gamma_n^\perp$, the orthogonal complement of the tautological bundle. Then the $n$th space of the spectrum $MT\theta$ is the Thom space
$B(\R^n)^{\theta^*\gamma_n^\perp}$.  The associated infinite loop
space is the direct limit
\begin{equation*}
  \Omega^\infty MT\theta = \colim_{n \to \infty} \Omega^n
  \left(B(\R^n)^{\theta^* \gamma_n^\perp}\right),
\end{equation*}
and we shall write $\Omega^\infty_\bullet MT\theta$ for the basepoint
component.  The rational cohomology of this space is easy to describe;
in the case where the bundle classified by $\theta$ is oriented, it is
as follows: for each $c \in H^{d+k}(B)$, there is a corresponding
``generalised Mumford--Morita--Miller class'' $\kappa_c \in
H^k(\Omega^\infty MT\theta)$, and $H^*(\Omega^\infty_\bullet MT\theta;
\bQ)$ is the free graded-commutative algebra on the classes
$\kappa_c$, where $c$ runs through a basis for the vector space $H^{>
  d}(B;\bQ)$. We describe the general case in Section \ref{sec:GMTW}.

\subsubsection{Moore--Postnikov towers}
\label{sec:moore-postn-towers}

For any map $A \to X$ of spaces and any $n \geq 0$, there is a
factorisation $A \to B \to X$ with the property that $\pi_i (A) \to
\pi_i(B)$ is surjective for $i = n$ and bijective for $i < n$, and
$\pi_i(B) \to \pi_i(X)$ is injective for $i =n$ and bijective for $i >
n$ (the requirements are imposed for all basepoints).  This is the
$n$th stage of the Moore--Postnikov tower for the map $A \to X$, and
can be constructed for example by attaching cells of dimension greater
than $n$ to $A$.  It is well known that a factorisation $A \to B \to
X$ with these properties is unique up to weak homotopy equivalence.
In the case where $A$ is a point, $B = X\langle n\rangle \to X$ is the
$n$-connective cover of the based space $X$, characterised by the
property that $\pi_i(X\langle n\rangle) = 0$ for $0 \leq i \leq n$,
and that $\pi_i(X\langle n\rangle) \to \pi_i(X)$ is an isomorphism for
$i > n$.  (Some authors write $X\langle n+1\rangle$ for what we denote
$X\langle n\rangle$.)

\subsection{Connected sums of copies of $S^n \times S^n$}

We can now state our homotopy theoretic version of
Theorem~\ref{thmcor:rational-coho}, generalising Madsen--Weiss'
theorem to dimension $2n$ (recall that we assume $2n > 4$
throughout).  As before, we write $W_g = \#^g S^n \times S^n$ for the
connected sum of $g$ copies of $S^n \times S^n$.

If we pick a disc $D^{2n} \subset W_g$, there is a classifying space
$B\Diff(W_g, D^{2n})$ and there are maps $B\Diff(W_g, D^{2n}) \to
B\Diff(W_{g+1},D^{2n})$ induced by taking connected sum with one more
copy of $S^n \times S^n$.  Let $\theta^n: BO(2n)\langle n\rangle \to
BO(2n)$ be the $n$-connective cover, and $MT\theta^n$ the associated
Thom spectrum.  Let us also say that a continuous map is a
\emph{homology equivalence} if it induces an isomorphism in integral
homology (and hence in any homology or cohomology theory).
\begin{theorem}\label{thm:main-A}
  Let $2n > 4$.  There is a homology equivalence
  \begin{equation*}
    \hocolim_{g \to \infty} B\Diff(W_g, D^{2n}) \lra
    \Omega^\infty_\bullet MT\theta^n.
  \end{equation*}
  More generally, if $W$ is any $(n-1)$-connected closed $2n$-manifold
  which is parallelisable in the complement of a point, there is a
  homology equivalence
  \begin{equation*}
    \hocolim_{g \to \infty} B\Diff(W \# W_g, D^{2n}) \lra
    \Omega^\infty_\bullet MT\theta^n.
  \end{equation*}
\end{theorem}

It is easy to deduce Theorem~\ref{thmcor:rational-coho} from
Theorem~\ref{thm:main-A}.  In \cite{Stability} we prove that the maps
$B\Diff(W_g,D^{2n}) \to B\Diff(W_{g+1},D^{2n})$ induce isomorphisms in
integral homology up to degree $\lfloor (g-4)/2\rfloor$ (cf.\ also
\cite{BerglundMadsen}).  Thus, Theorem~\ref{thm:main-A} also
determines the homology and cohomology of $B\Diff(W_g,D^{2n})$ in this
range.

\subsection{The moduli space of highly connected null-bordisms}
\label{sec:moduli-space-highly}

The determination, in Theorems~\ref{thmcor:rational-coho} and
\ref{thm:main-A}, of the stable homology and cohomology of the space $B\Diff(W
\# g S^n \times S^n, D^{2n})$ is a special case of
Theorem~\ref{thm:main-C-new} below, in which we determine the stable
homology of $B\Diff(W)$ for more general manifolds $W$.  We also
consider manifolds equipped with an additional \emph{tangential
  structure}, defined as follows.
  
\begin{definition}\label{defn:tangential-structures}
  Let $\theta: B \to BO(2n)$ be a map.  A $\theta$-structure on a
  $2n$-dimensional manifold $W$ is a bundle map $\ell: TW \to \theta^*
  \gamma$, i.e.\ a fibrewise linear isomorphism. Such a pair $(W, \ell)$ will be called a $\theta$-manifold. A $\theta$-structure
  on a $(2n-1)$-dimensional manifold $M$ is a bundle map $\epsilon^1
  \oplus TM \to \theta^* \gamma$.  If $\ell$ is a $\theta$-structure
  on $W$, the \emph{induced} structure on $\partial W$ is obtained by
  composing with a certain isomorphism $\epsilon^1 \oplus T(\partial
  W) \to TW\vert_{\partial W}$.  In fact, there are two such
  isomorphisms: One comes from a collar $[0,1) \times \partial W \to
  W$ of $\partial W$.  Differentiating this gives an isomorphism
  $\epsilon^1 \oplus T(\partial W) \to TW\vert_{\partial W}$, and the
  resulting $\theta$-structure on $\partial W$ will be called the
  \emph{incoming} restriction.  Another comes from a collar $(-1,0]
  \times \partial W \to W$; this is the \emph{outgoing} restriction.
  When $W$ is a cobordism, we will generally use the incoming
  restriction to induce a $\theta$-structure on the source of $W$ and
  the outgoing restriction on the target.
  
  Let $\ell_0: TW\vert_{\partial W} \to \theta^* \gamma$ be a
  $\theta$-structure on $\partial W$, and $\Bun^\partial (TW,\theta^*
  \gamma;\ell_0)$ denote the space of all bundle maps $\ell: TW \to
  \theta^* \gamma$ which restrict to $\ell_0$ over $\partial W$,
  equipped with the compact-open topology.  The group
  $\Diff(W, \partial W)$ of diffeomorphisms of $W$ which restrict to
  the identity near $\partial W$ acts on $\Bun^\partial(TW,\theta^*
  \gamma;\ell_0)$ by precomposing a bundle map with the differential
  of a diffeomorphism.
\end{definition}

The most general case of our theorem concerns the \emph{moduli space
  of highly connected null-bordisms}, defined as follows.

\begin{definition}\label{defn:N}
  Let $P\subset \R^\infty$ be a closed $(2n-1)$-dimensional manifold
  with $\theta$-structure $\ell_P: \epsilon^1 \oplus TP \to \theta^*
  \gamma$.  A null-bordism is a pair $(W,\ell_W)$, where $W \subset
  (-\infty,0] \times \R^\infty$ is a compact manifold with $\partial W
  = \{0\} \times P$ and $(-\epsilon,0] \times P \subset W$ for some
  $\epsilon > 0$, and $\ell_W: TW \to \theta^*\gamma$ is a
  $\theta$-structure satisfying $\ell_W \vert_{\partial W} = \ell_P$.
  A null-bordism $(W,\ell_W)$ is highly connected if $(W,P)$ is
  $(n-1)$-connected, and the \emph{moduli space of highly connected
    null-bordisms} is the set $\mathcal{N}^\theta(P,\ell_P)$ of all
  highly connected null-bordisms of $(P,\ell_P)$.  It is topologised
  as the disjoint union
  \begin{equation}\label{eq:24}
    \coprod_W
    (\Emb^\partial(W,(-\infty,0] \times \R^\infty) \times
    \Bun^\partial(TW, \theta^* \gamma; \ell_P))/\Diff(W,\partial W)
  \end{equation}
  where the disjoint union is over compact manifolds $W$ with
  $\partial W = P$ for which $(W,P)$ is $(n-1)$-connected, one of each
  diffeomorphism class.

  If $K \subset [0,1] \times \R^\infty$ is a cobordism with collared
  boundary $\partial K = (\{0\} \times P_0) \cup (\{1\} \times P_1)$
  we say that $K$ is \emph{highly connected} if each pair $(K,\{i\} \times
  P_i)$ is $(n-1)$-connected.  If $K$ is equipped with a $\theta$-structure
  $\ell_K$ restricting to $\ell_0$ and $\ell_1$ on the boundaries,
  then there is an induced map $\mathcal{N}^\theta(P_0,\ell_0) \to
  \mathcal{N}^\theta(P_1,\ell_1)$ defined by taking union with $K$ and
  subtracting 1 from the first coordinate.
\end{definition}

This moduli space classifies smooth families of null-bordisms of $P$,
in the sense that if $X$ is a smooth manifold without boundary, there
is a natural bijection between the set of homotopy classes $[X, \mathcal{N}^\theta(P,\ell_P)]$,
and the set of equivalence classes of triples $(\pi,\phi, \ell)$,
where $\pi: E \to X$ is a proper submersion (i.e.\ smooth fibre
bundle), $\phi$ is a diffeomorphism $\partial E \cong X \times P$ over
$X$, such that
$(E, \partial E)$ is $(n-1)$-connected, and $\ell$ is a
$\theta$-structure on the fibrewise tangent bundle $T_\pi E =
\Ker(D\pi)$.

Let us also introduce notation for each of the disjoint summands
in~\eqref{eq:24}.

\begin{definition}\label{defn:BDifftheta}
  Let $W$ be a compact $2n$-dimensional manifold, and $\ell_0:
  TW\vert_{\partial W} \to \theta^* \gamma$ be a $\theta$-structure on
  $\partial W$. We shall write
  \begin{equation*}
    B\Diff^\theta(W; \ell_0) = (E\Diff(W,\partial W) \times
    \Bun^\partial(TW,\theta^* \gamma;\ell_0))/\Diff(W,\partial W)
  \end{equation*}
  for the homotopy orbit space of the action of $\Diff(W,\partial W)$ on $\Bun^\partial(TW,\theta^* \gamma;\ell_0)$.  If $\ell: TW \to
  \theta^* \gamma$ is a particular extension, we shall write
  $B\Diff^\theta(W;\ell_0)_\ell \subset B\Diff^\theta(W; \ell_0)$
  for the path component containing $\ell$.
\end{definition}

Using the model $E\Diff(W,\partial W) = \Emb^\partial(W,(-\infty,0]
\times \R^\infty)$, we have the homeomorphism
\begin{equation*}
  \mathcal{N}^\theta(P,\ell_P) =\coprod_{W} B\Diff^\theta(W;\ell_P).
\end{equation*}

\begin{definition}
  \label{defn:spherical}
  A tangential structure $\theta: B \to BO(2n)$ is called
  \emph{spherical} if any $\theta$-structure on the lower hemisphere
  $\partial_- D^{2n+1} \subset \partial D^{2n+1}$ extends to some
  $\theta$-structure on the whole sphere.  (If $B$ is path connected, this is
  equivalent to the sphere $S^{2n}$ admitting a $\theta$-structure.)
\end{definition}
Most of the usual structures, for example $SO$, $\Spin$, $\Spin^c$,
etc.\ are spherical, but some are not, e.g.\
framings. Theorem~\ref{thm:main-C-new} below determines the homology
of $\mathcal{N}^\theta(P,\ell_P)$ after stabilising with cobordisms in
the $(P,\ell_P)$-variable.  The following definition makes the
stabilisation procedure precise.

\begin{definition}\label{defn:universalthetaend}
  Let $\theta: B \to BO(2n)$ be spherical, and $K \subset [0,\infty)
  \times \R^\infty$ be a submanifold with $\theta$-structure $\ell_K$,
  such that $x_1 : K \to [0,\infty)$ has the natural numbers as
  regular values. For $A \subset [0, \infty)$, we let $(K\vert_A,
  \ell_K \vert_A)$ denote the $\theta$-manifold $K \cap x_1^{-1}(A)$.
\begin{enumerate}[(i)]
\item Let $W \subset [0,1] \times \R^\infty$ be a cobordism with
  $\theta$-structure $\ell_W$, and suppose that $(W\vert_0,
  \ell_W\vert_0) = (K\vert_0,\ell_K\vert_0)$.  We say that
  $(K,\ell_K)$ \emph{absorbs} $(W,\ell_W)$ if there exists an
  embedding $j: W \to K$ which is the identity on $W\vert_0 =
  K\vert_0$, such that $\ell_K \circ Dj: TW \to \theta^*\gamma$ is
  homotopic to $\ell_W$ relative to $W\vert_0$.  That
  $K\vert_{[i,\infty)}$ absorbs a $\theta$-bordism $W \subset [i,i+1]
  \times \R^\infty$ is defined similarly.

\item\label{it:univ:3} We say that $(K,\ell_K)$ is a \emph{universal
    $\theta$-end} if for
  each integer $i \geq 0$, $K\vert_{[i,i+1]}$ is a highly connected
  cobordism and $K\vert_{[i,\infty)}$ absorbs $W$ for any highly
  connected cobordism $W \subset [i,i+1] \times \R^\infty$ with
  $\theta$-structure $\ell_W$ such that $(W\vert_i, \ell_W\vert_i) =
  (K\vert_i,\ell_K\vert_i)$.
\end{enumerate}
\end{definition}

For example, in dimension 2 with $\theta = \mathrm{id}: BO(2) \to
BO(2)$, we can construct a universal $\theta$-end by letting each
$K_{[i,i+1]}$ be diffeomorphic to $\R P^2$ with two discs removed.
For $\theta = \theta^n: BO(2n)\langle n\rangle \to BO(2n)$, a
universal $\theta$-end can be constructed by letting each
$K_{[i,i+1]}$ be diffeomorphic to $S^n \times S^n$ with two discs
removed.  In many other cases, a universal $\theta$-end $K$ can be
constructed as the infinite iteration of a single self-bordism
$K\vert_{[0,1]}$.  In particular, this will be the case in the
examples in Section~\ref{sec:exampl-appl} below.

As we shall see, universal $\theta$-ends are
unique up to isomorphism in the following sense.  If $(K,\ell_K)$ and
$(K', \ell_K')$ are two universal $\theta$-ends with $K\vert_0 =
K'\vert_0$, then there exists a diffeomorphism $K \to K'$ preserving
$\theta$-structure up to homotopy, relative to $K\vert_0$.  More
generally, given a highly connected cobordism $(W,\ell_W)$ from
$K\vert_0$ to $K'\vert_0$, there exists a similar diffeomorphism from
$K$ to $W \circ K'$.

\begin{theorem}\label{thm:main-C-new}
  Let $2n > 4$ and let $\theta : B \to BO(2n)$ be spherical. Let $(K,\ell_K)$ be a universal $\theta$-end with
  $\mathcal{N}^\theta(K\vert_0, \ell_K\vert_0) \neq \emptyset$. Then
  there is a homology
  equivalence
  \begin{equation*}
    \hocolim_{i \to \infty} \mathcal{N}^\theta(K\vert_{i},
    \ell_{K}\vert_{i}) \lra \Omega^\infty MT\theta',
  \end{equation*}
  where $\theta' : B' \to B \overset{\theta}\to BO(2n)$ is the
  $n$th stage of the Moore--Postnikov tower for $\ell_K : K \to B$.
\end{theorem}

The property of being a universal $\theta$-end can often be checked in
practice, using the following addendum, as it is essentially a
homotopical property.

\begin{addendum}\label{add:1}
  Let $\theta: B \to BO(2n)$ be spherical, let $K \subset [0,\infty)
  \times \R^\infty$ be a submanifold such that $K \vert_{[i,i+1]}$ is
  a highly connected cobordism for each integer $i$, and let $\ell_K$
  be a $\theta$-structure on $K$.  Then $(K,\ell_K)$ is a universal
  $\theta$-end if and only if the following conditions
  hold.
  \begin{enumerate}[(i)]
  \item\label{it:add:1} For each integer $i$, the map
    $\pi_n(K\vert_{[i, \infty)}) \to \pi_n(B)$ is surjective, for all
    basepoints in $K$.
    
  \item\label{it:add:2} For each integer $i$, the map
    $\pi_{n-1}(K\vert_{[i, \infty)}) \to \pi_{n-1}(B)$ is injective,
    for all basepoints in $K$.
    
  \item\label{it:add:3} For each integer $i$, each path component of
    $K\vert_{[i, \infty)}$ contains a submanifold diffeomorphic to
    $S^n \times S^n - \Int(D^{2n})$, which in addition has
    null-homotopic structure map to $B$.
  \end{enumerate}
\end{addendum}

\begin{remark}\label{rem:limit-of-B-diff-theta}
  It is often useful to
  consider the homology equivalence in Theorem~\ref{thm:main-C-new}
  one path component at a time, so we spell out the resulting
  statement using the notation of Definition~\ref{defn:BDifftheta}.
  Any path component of the infinite loop space $\Omega^\infty
  MT\theta'$ is homotopy equivalent to the basepoint component
  $\Omega^\infty_\bullet MT\theta'$.  On the left hand side of the
  homology equivalence, the path component of an element $(W,\ell_W)
  \in \mathcal{N}^\theta(K\vert_0,\ell_K\vert_0)$ is the homotopy
  colimit of the spaces
  \begin{equation*}
    B\Diff^\theta(W \cup K\vert_{[0,i]}; \ell_i)_{\ell_W \cup
      \ell_K\vert_{[0,i]}}.
  \end{equation*}
  
  Conversely, given a triple $(W,K,\ell)$ where $K \subset
  [0,\infty) \times \R^\infty$ is a non-compact manifold such that
  $K\vert_{[i,i+1]}$ is a highly connected cobordism for each integer $i
  \geq 0$, $W \subset (-\infty, 0] \times \R^\infty$ is a compact
  manifold with collared boundary $\partial W = K\vert_0$ such that
  $(W, \partial W)$ is $(n-1)$-connected, and $\ell: T(W \cup K) \to
  \theta^* \gamma$ is a bundle map, the Pontryagin--Thom construction
  described below provides a map
  \begin{equation}
    \label{eq:26}
    \hocolim_{i \to \infty} B\Diff^\theta(W \cup K\vert_{[0,i]}; \ell_i)_\ell \lra
    \Omega^\infty_\bullet MT\theta',
  \end{equation}
  where $\theta': B' \to B \to BO(2n)$ is obtained from the $n$th
  Moore--Postnikov stage of the underlying map $W \cup K \to B$.  By
  Theorem~\ref{thm:main-C-new}, the map~\eqref{eq:26} is a homology
  isomorphism, provided that $K$ is a universal $\theta'$-end.  In
  particular, Theorem~\ref{thm:main-A} can be deduced this way: If we
  let each $K\vert_{[i,i+1]}$ be diffeomorphic to $S^n \times S^n -
  \Int(D^{2n})$ and let $\theta = \mathrm{id}: BO(2n) \to BO(2n)$,
  then $\theta' = \theta^n: BO(2n)\langle n\rangle \to BO(2n)$, and
  $K$ is a universal $\theta^n$-end.  Similarly, all examples in
  Section~\ref{sec:exampl-appl} below arise in this way.

  Let us also remark that the homotopy colimit~\eqref{eq:26} may be
  replaced by the strict colimit $B\Diff^\theta_c(W \cup K;\ell)$,
  defined by
  \begin{equation*}
    B\Diff^\theta_c(W \cup K;\ell) = \big(E\Diff_c(W \cup K) \times \Bun_c(T(W
    \cup K), \theta^* \gamma;\ell)\big)/\Diff_c(W \cup K),
  \end{equation*}
  where $\Diff_c(W \cup K)$ is the topological group of compactly
  supported diffeomorphisms of the non-compact manifold $W \cup K$,
  and $\Bun_c(T(W \cup K), \theta^*\gamma;\ell)$ is the space of bundle
  maps which agree outside of a compact subset of $W \cup K$ with
  $\ell$.
\end{remark}

\begin{remark}\label{rem:what-is-the-map}
  The maps in all
  the theorems above are induced by the Pontryagin--Thom construction.
  We shall briefly explain this in the setting of
  Theorem~\ref{thm:main-C-new}, after replacing
  $\mathcal{N}^\theta(P,\ell_P)$ by a weakly equivalent space, and
  refer the reader to \cite[\S 2.3]{MT} for further details. First we say that a submanifold $W \subset (-\infty,0] \times
  \R^{q-1}$ with collared boundary is \emph{fatly embedded} if the
  canonical map from the normal bundle $\nu W$ to $\R^q$ restricts to
  an embedding of the disc bundle into $(-\infty,0] \times \R^{q-1}$.
  In that case the Pontryagin--Thom collapse construction gives a
  continuous map from $[-\infty,0] \wedge S^{q-1}$ to the Thom space
  of $\nu W$.  Secondly we replace $\theta': B' \to BO(2n)$ by a
  fibration, and redefine $\mathcal{N}^\theta(P,\ell_P)$ as a space of
  pairs $(W,\ell_W)$ where $W \subset (-\infty,0] \times \R^\infty$
  is a fatly embedded submanifold, collared near $\partial W = \{0\}
  \times P$, and $\ell_W: W \to B'$ is a continuous map such that
  $\theta' \circ \ell_W: W \to BO(2n) = \Gr_{2n}(\R^\infty)$ is
  \emph{equal} to the Gauss map and whose
  restriction to $\partial W$ is equal to a specified map $\ell_P: P
  \to B'$.  There is a forgetful map from the space of such pairs to
  the space in Definition~\ref{defn:N}, and standard homotopy
  theoretic methods imply that it is a weak equivalence.  If $P \subset
  \R^{q-1} \subset \R^\infty$, the Pontryagin--Thom construction
  (composed with $\ell_P$) gives a point
  \begin{equation*}
    \alpha(P,\ell_P) \in \Omega^{q-1}\big(B'(\R^q)^{(\theta')^*
      \gamma^\perp_q} \big) \subset \Omega^{\infty-1} MT\theta',
  \end{equation*}
  and if $(W,\ell_W) \in \mathcal{N}^\theta(P,\ell_P)$ has $W
  \subset (-\infty,0] \times \R^{q-1}$, it gives a path
  \begin{equation*}
    \alpha(W,\ell_W): [-\infty,0] \lra
    \Omega^{q-1}\big(B'(\R^q)^{(\theta')^* \gamma^\perp_q} \big)
    \subset \Omega^{\infty-1} MT\theta',
  \end{equation*}
  starting at the basepoint and ending at $\alpha(P,\ell_P)$.  The
  space of such paths is homotopy equivalent to the based loop space,
  which is $\Omega^\infty MT\theta'$.  Finally, the non-compact
  manifold $K\subset [0,\infty) \times \R^\infty$ admits a
  homotopically unique $\theta'$-structure lifting its
  $\theta$-structure and extending the canonical $\theta'$-structure
  on $P = K\vert_0$.  The Pontryagin--Thom construction applied to each
  cobordism $K\vert_{[i,i+1]}$ then gives a path
  $\alpha(K\vert_{[i,i+1]}, \ell_K): [i,i+1] \to \Omega^{\infty-1}
  MT\theta'$ and the entire process now commutes (strictly) with the
  stabilisation maps.
\end{remark}

\subsection{Algebraic localisation}

There is one final algebraic version of our main theorem. Fix $P$, a
closed $(2n-1)$-manifold with $\theta$-structure $\ell_P : \epsilon^1
\oplus TP \to \theta^*\gamma$.  As explained in
Definition~\ref{defn:N}, a cobordism $(K,\ell_K)$ from $(P, \ell_P)$ to
itself with $K \subset [0,1] \times \R^\infty$, which is
$(n-1)$-connected with respect to both boundaries, gives a self-map of $\mathcal{N}^\theta(P,\ell_P)$ defined by $W \mapsto
W \cup_P K - e_1$.  We shall write $\mathcal{K}_0$ for the set of
isomorphism classes of such $(K,\ell_K)$, where we identify
$(K,\ell_K)$ with $(K', \ell_{K'})$ if there is a diffeomorphism $\phi:
K \to K'$ which is the identity near $\partial K$ such that $\phi^*
\ell_{K'}$ is homotopic to $\ell_K$ relative to $\partial K$.  It is
clear that the homotopy class of the self-map of
$\mathcal{N}^\theta(P,\ell_P)$ induced by $(K,\ell_K)$ depends only on
the isomorphism class of $(K,\ell_K)$, and we get an action of the
non-commutative monoid $\mathcal{K}_0$ on $H_*(\mathcal{N}^\theta(P, \ell_P))$.  Our theorem determines the algebraic localisation
\begin{equation*}
  H_*(\mathcal{N}^\theta(P,\ell_P))[\mathcal{K}^{-1}]
\end{equation*}
at a certain commutative submonoid $\mathcal{K}\subset
\mathcal{K}_0$ which we now describe.

We say that a $\theta$-cobordism $K : P \leadsto P$ has \emph{support}
in a closed subset $A \subset P$ if it contains $[0,1] \times (P - A)
: (P - A) \leadsto (P - A)$ as a sub-cobordism with the product
$\theta$-structure. We let $\mathcal{K} \subset \mathcal{K}_0$ consist of those
elements which admit a representative with support in a regular
neighbourhood of a simplicial complex of dimension at most $n-1$
inside $P$, and prove the following lemma.

\begin{lemma}\label{lem:Kcomm}
  The subset $\mathcal{K}\subset \mathcal{K}_0$ is a commutative
  submonoid.
\end{lemma}

We may localise the $\bZ[\mathcal{K}]$-module
$H_*(\mathcal{N}^\theta(P,\ell_P))$ at any submonoid $\mathcal{L}
\subset \mathcal{K}$.  The content of Theorem \ref{thm:main-D} below
is an isomorphism
$$H_*(\mathcal{N}^\theta(P, \ell_P))[\mathcal{L}^{-1}] \cong
H_*(\Omega^\infty MT\theta')$$ under certain conditions, where
$\theta' : B' \to B \overset{\theta}\to BO(2n)$ is the $(n-1)$st stage
of the Moore--Postnikov tower for $\ell_P : P \to B$.  To describe the
isomorphism explicitly, recall that in
Remark~\ref{rem:what-is-the-map} we described a map
\begin{equation*}
  \mathcal{N}^\theta(P,\ell_P) \lra \Omega^\infty MT\theta',
\end{equation*}
compatible with gluing highly connected cobordisms of $(P,\ell_P)$
equipped with $\theta'$-structures, and hence the induced map 
\begin{equation}
  \label{eq:21}
  H_*(\mathcal{N}^\theta(P,\ell_P)) \lra H_*(\Omega^\infty MT\theta')
\end{equation}
is a map of $\bZ[\mathcal{K}']$-modules, where the monoid
$\mathcal{K}'$ is defined like $\mathcal{K}$ but using $\theta'$
instead of $\theta$.  An obstruction theoretic argument, which we
explain in more detail in Section~\ref{sec:proof-lemma-refl}, shows
that the natural map $\mathcal{K}' \to \mathcal{K}$ is a bijection,
so~\eqref{eq:21} is naturally a homomorphism of
$\bZ[\mathcal{K}]$-modules.

\begin{theorem}\label{thm:main-D}
  Let $2n > 4$ and let $\theta : B \to BO(2n)$ be spherical. Let $P$
  be a closed $(2n-1)$-manifold with $\theta$-structure $\ell_P :
  \epsilon^1 \oplus TP \to \theta^*\gamma$, such that
  $\mathcal{N}^\theta(P, \ell_P)$ is non-empty.  Then the
  morphism~\eqref{eq:21} induces an isomorphism
  \begin{equation*}
    H_*(\mathcal{N}^{\theta}(P, \ell_P))[\mathcal{K}^{-1}] \lra
  H_*(\Omega^\infty MT\theta').
  \end{equation*}

  Furthermore, localisation at a submonoid $\mathcal{L} \subset
  \mathcal{K}$ agrees with localisation at $\mathcal{K}$, provided
  $\mathcal{L}$ satisfies the following conditions.
  \begin{enumerate}[(i)]
  \item\label{item:16} The group $\pi_n(B)$ is generated by the subgroups
    $\IM(\pi_n(K) \to \pi_n(B))$, $K \in \mathcal{L}$.
    
  \item\label{item:20} The subgroup of $\pi_{n-1}(P)$ generated by
    $\Ker(\pi_{n-1}(P) \to \pi_{n-1}(K))$, $K \in \mathcal{L}$,
    contains $\Ker(\pi_{n-1}(P) \to \pi_{n-1}(B))$.
    
  \item\label{item:21} There is an element of $\mathcal{L}$ containing
    a submanifold diffeomorphic to $S^n \times S^n -
    \Int(D^{2n})$.
  \end{enumerate}
  (There is a bijection $\pi_0 (P) = \pi_0 (K)$, and if $P$ is not
  connected, conditions (\ref{item:16}), (\ref{item:20}) and (\ref{item:21}) are
  required to hold for each path component of $P$.)
\end{theorem}

Applying the functor $\Hom_{\bZ[\mathcal{K}]}(-, \bQ)$ to both sides
of the isomorphism in the theorem identifies the subring of
$H^*(\mathcal{N}^\theta(P,\ell_P);\bQ)$ consisting of
$\mathcal{K}$-invariants with $H^*(\Omega^\infty_\bullet
MT\theta';\bQ)$.  Observing that these classes are also invariant under the larger monoid
$\mathcal{K}_0$, we deduce the isomorphism
\begin{equation*}
  H^*(\mathcal{N}^\theta(P,\ell_P);\bQ)^{\mathcal{K}_0} \cong
  H^*(\Omega^\infty_\bullet MT\theta';\bQ).
\end{equation*}
The left hand side can be interpreted as characteristic classes of
certain bundles, invariant under fibrewise gluing of trivial bundles.

\subsection{Examples and applications}
\label{sec:exampl-appl}

Recall that the connective cover $BO(d)\langle k\rangle$ is $BSO(d)$
if $k = 1$, $B\mathrm{Spin}(d)$ if $k = 2$ or $3$, and is often called
$B\mathrm{String}(d)$ if $k = 4$, $5$, $6$ or $7$.  We write
$MTSO(d)$, $MT\mathrm{Spin}(d)$ and $MT\mathrm{String}(d)$ for the
corresponding Thom spectra.  As special cases of
Theorem~\ref{thm:main-C-new} we have the following maps, which become
homology equivalences in the limit $g \to \infty$.  All are deduced
from Theorem~\ref{thm:main-C-new} as in
Remark~\ref{rem:limit-of-B-diff-theta}, with $\theta = \mathrm{id}:
BO(2n) \to BO(2n)$. 
\begin{align*}
  B\Diff(g S^3 \times S^3, D^6) &\lra
  \Omega^\infty_\bullet  MT\mathrm{Spin}(6)\\
  B\Diff(g(\bH P^2\# \overline{\bH P}^2), D^8) &\lra
  \Omega^\infty_\bullet MT\mathrm{Spin}(8)\\
  B\Diff(g S^4 \times S^4, D^8) &\lra
  \Omega^\infty_\bullet MT\mathrm{String}(8)\\
  B\Diff(g S^5 \times S^5, D^{10}) &\lra
  \Omega^\infty_\bullet MT\mathrm{String}(10)\\
  B\Diff(g S^6 \times S^6, D^{12}) &\lra
  \Omega^\infty_\bullet MT\mathrm{String}(12)\\
  B\Diff(g S^7 \times S^7, D^{14}) &\lra
  \Omega^\infty_\bullet MT\mathrm{String}(14)\\
  B\Diff(g(\bO P^2 \# \overline{\bO P}^2), D^{16}) &\lra \Omega^\infty_\bullet
  MT\mathrm{String}(16)
\end{align*}

A slightly different type of example is given by $B\Diff(\bC P^3 \# g
S^3 \times S^3, U)$, where $U \subset \bC P^3$ is a tubular
neighbourhood of $\bC P^1$.  In this case the stable homology is that
of $\Omega^\infty_\bullet MT\mathrm{Spin}^c(6)$, where $B\Spin^c(6)$
is the homotopy fibre of the map $\beta w_2: BSO(6) \to K(\bZ,3)$.

An example where we need a more complicated stabilisation (not induced
by connected sum) comes from $\bR P^6$.  The map $\bR P^6 \to BO(6)$
lifts canonically to a 3-connected map $\bR P^6 \to
B\mathrm{Pin}^-(6)$, where $\theta: B\mathrm{Pin}^-(6)\to BO(6)$ is the
homotopy fibre of $w_2 + w_1^2: BO(6) \to K(\bZ/2,2)$.  The standard
self-indexing Morse function $f: \bR P^6 \to [0,6]$ given by
\begin{equation*}
  f(x_0;\cdots;x_6) = \sum_{i=0}^6 i \cdot x_i^2
\end{equation*}
has one critical point of each index, and we let $W = f^{-1}([0,2.5])
\cong \R P^2 \times D^4$.  Cutting out a parallel copy of $W$ gives
a $\theta$-bordism $\widetilde K \cong f^{-1}([2.5,3.5])$ from $\partial
W = \R P^2 \times S^3$ to $-\partial W$ (i.e.\ $\R P^2 \times S^3$
equipped with the opposite $\theta$-structure).  Hence $K_0 =
\widetilde K \circ (-\widetilde K)$ is a cobordism from $\partial W$ to
itself, and we let $K$ be the infinite iteration.  In this situation
we get a stable homology equivalence
\begin{equation*}
  B\Diff((\R P^2 \times D^4) \cup_\partial g K_0, \partial) \lra
  \Omega^\infty_\bullet MT\mathrm{Pin}^-(6).
\end{equation*}

Another interesting special case concerning the manifolds $W_g = \#^g
S^n \times S^n$ is the following.  Let $(Y,y)$ be a pointed space, and
consider the homotopy orbit space
\begin{equation*}
  \mathcal{S}_g^n(Y,y) = (E\Diff(W_g,D^{2n}) \times
  \Map((W_g,D^{2n}), (Y,y)))/\Diff(W_g,D^{2n}).
\end{equation*}
We can determine the stable homology of these spaces using a
Pontryagin--Thom map
\begin{equation}\label{eq:27}
  \coprod_{g \geq 0} \mathcal{S}_g^n(Y,y) \lra \Omega^\infty(Y\langle
  n -1 \rangle_+ \wedge MT\theta^n),
\end{equation}
defined as in Remark~\ref{rem:what-is-the-map}.  Any map $f: (S^n,D^n)
\to (Y,y)$ may be composed with the projection $S^n \times S^n \to
S^n$ to give a map $(W_1,D^{2n}) \to (Y,y)$, which induces a map
$\mathcal{S}_g^n(Y,y) \to \mathcal{S}_{g+1}^n(Y,y)$.  Thus each such
$f$ gives a self-map of the left hand side of~\eqref{eq:27} and a
compatible self-map of the right hand side which is a weak
equivalence.  Up to homotopy, the self-maps depend only on $[f] \in
\pi_n(Y,y)$ and different elements of $\pi_n(Y,y)$ give homotopy
commuting self-maps.  Therefore,~\eqref{eq:27} induces a map from the
stabilised homology
\begin{equation}\label{eq:28}
  \bigg(\bigoplus_{g \geq 0}
  H_*(\mathcal{S}_g^n(Y,y))\bigg)[\pi_n(Y,y)^{-1}] \lra
  H_*(\Omega^\infty(Y\langle n -1 \rangle_+ \wedge MT\theta^n)).
\end{equation}
Applying Theorem~\ref{thm:main-D} to the projection $\theta: BO(2n)
\times Y \to BO(2n)$ implies that~\eqref{eq:28} becomes an
isomorphism, after restricting to appropriate path components.  This
result is a generalisation of the result of Cohen and Madsen
\cite{CM}, who proved the special case where $2n=2$ and $Y$ is simply
connected.  (The case $2n=2$ was generalised to non-simply connected
$Y$ in \cite{GR-W}.)

As a final application, in \cite{Detection} we deduce a generalisation
of the detection result of Ebert (\cite{Ebert09Independence}).  We
will prove that for any abelian group $k$ and any non-zero cohomology
class $c \in H^*(\Omega^\infty_\bullet MTSO(2n);k)$, there exists a
bundle $p: E \to B$ of smooth oriented manifolds, such that the
characteristic class associated to $c$ is non-vanishing in $H^*(B;k)$.
(The case $k=\bQ$ was proved by Ebert.)

\subsection{Cobordism categories and outline of proof}
\label{sec:cobordism-categories}

Finally, let us say a few words about our method of proof, which
follows the strategy in \cite{GR-W} and \cite{GMTW}.  A central object
is the \emph{cobordism category} $\mathcal{C}_{\theta}(\R^N)$, whose
objects are closed $(d-1)$-dimensional manifolds $M \subset \R^{N}$
and whose morphisms are $d$-dimensional cobordisms $W \subset [0,t]
\times \R^N$, both equipped with $\theta$-structures.  

\begin{remark}
  The applications described above use only the case where $d$ is
  even.  Our results about cobordism categories are valid for odd $d$
  as well, but we do not know an interpretation in terms of stable
  homology in that case. In fact, Ebert (\cite{Ebert09}) has shown
  that there are non-trivial classes in $H^*(\Omega^\infty_\bullet
  MTSO(2n+1);\bQ)$ which are trivial when restricted to any
  $B\Diff^+(M^{2n+1})$.  Thus there can be no analogue of e.g.\
  Theorem~\ref{thm:main-C-new}, expressing $H_*(\Omega^\infty_\bullet
  MTSO(2n+1))$ as a direct limit of $H_*(B\Diff(W\cup
  K\vert_{[0,i]}, K\vert_i))$'s. It is
  an interesting question to find an odd-dimensional analogue of our
  results.
\end{remark}

In the limit $N \to \infty$, the main result of \cite{GMTW} gives a
weak equivalence
\begin{equation}\label{eq:2}
  \Omega B\mathcal{C}_\theta \simeq \Omega^\infty MT\theta.
\end{equation}
As in \cite{GR-W}, our strategy will be to find subcategories
$\mathcal{C} \subset \mathcal{C}_\theta$, as small as possible, such
that the inclusion induces a weak equivalence $\Omega B \mathcal{C}
\to \Omega B\mathcal{C}_\theta$.  The proof of
Theorem~\ref{thm:main-C-new} will consist of applying a version of the
``group completion'' theorem to a very small subcategory of
$\mathcal{C}_\theta$.

Let $P$ be a $(2n-1)$-dimensional manifold with $\theta$-structure
$\ell_P : \epsilon^1 \oplus TP \to \theta^*\gamma$, and suppose the
underlying map $P \to B$ is $(n-1)$-connected. We pick a
self-indexing Morse function $f: P \to [0,2n-1]$ and set $L =
f^{-1}([0,n-\tfrac{1}{2}])$. The restriction $L \to B$ is then still
$(n-1)$-connected.  Then we pick a (collared) embedding $L \to
(-\infty,0] \times \R^\infty$, and consider the subcategory
$\mathcal{C}_{\theta,L} \subset \mathcal{C}_{\theta}$ where objects $M
\subset \R \times \R^\infty$ satisfy $M \cap \big((-\infty, 0] \times
\R^{\infty}\big) = L$ and morphisms $W \subset [0,t] \times \R \times
\R^\infty$ satisfy $W \cap \big([0,t] \times (-\infty,0] \times
\R^\infty\big) = [0,t] \times L$.  For both objects and morphisms,
these identities are required to hold as $\theta$-manifolds.  (For
later purposes, we note that the category only depends on $\partial
L$: If $\partial L_1 = \partial L_2$, then there is an isomorphism of
categories which cuts out $\Int(L_1)$ and replaces it with $L_2$.  In
fact, it is convenient to mentally cut out $\Int(L)$ and think of
objects as manifolds with boundary, and morphisms as manifolds with
corners.) In Section~\ref{sec:defin-recoll} we prove that the
inclusion map induces a weak equivalence
\begin{equation}\label{eq:3}
  B\mathcal{C}_{\theta,L} \lra B \mathcal{C}_\theta.
\end{equation}

Secondly, we filter $\mathcal{C}_{\theta,L}$ by \emph{connectivity of morphisms}: for $\kappa \geq -1$, the subcategory $\mathcal{C}_{\theta,L}^{\kappa}$ has
the same objects, but a morphism $W$ from $M_0$ to $M_1$ is required
to satisfy that the inclusion $M_1 \to W$ is $\kappa$-connected, i.e.\
that any map $(D^i,\partial D^i) \to (W,M_1)$ is homotopic to one with
image in $M_1$, for $i \leq \kappa$.  In
Section~\ref{sec:surgery-morphisms} we prove that the inclusion map
induces a weak equivalence
\begin{equation}\label{eq:4}
  B\mathcal{C}_{\theta,L}^\kappa \lra B\mathcal{C}_{\theta,L},
\end{equation}
as long as $\kappa \leq (d-2)/2$.  (In the case where $\kappa=0$, this
is the ``positive boundary subcategory'', and this case was proved in
\cite{GMTW}.)

Thirdly, we filter $\mathcal{C}_{\theta, L}^\kappa$ by \emph{connectivity of objects}: for $l \geq -1$, the subcategory $\mathcal{C}_{\theta,L}^{\kappa,l}
\subset \mathcal{C}_{\theta,L}^{\kappa}$ is the full subcategory on
those objects where the structure map $M \to B$ induces an injection
$\pi_i(M) \to \pi_i(B)$ for all $i \leq l$ and all basepoints, or
equivalently the inclusion $L \to M$ is $l$-connected.  In
Section~\ref{sec:surg-objects-below} we prove that the inclusion map
induces a weak equivalence
\begin{equation}\label{eq:5}
  B\mathcal{C}_{\theta,L}^{\kappa,l} \lra
  B\mathcal{C}_{\theta,L}^\kappa,
\end{equation}
provided $l \leq (d-3)/2$ and $l \leq \kappa$.  (In the case where $l
= 0$ and $B$ is connected, this is the full subcategory on objects
which are path connected, and this case was proved in \cite{GR-W}.)

Fourthly, we focus on the case where $d = 2n > 4$, where we have
now reduced to $\mathcal{C}_{\theta,L}^{n-1,n-2}$, the full
subcategory on those objects for which the inclusion $L \to M$ is
$(n-2)$-connected.  In the final step we let $\mathcal{C}$ denote the
full subcategory on those objects $M$ which can be obtained from $L$ by attaching handles of index at least $n$. (This is equivalent to the condition that $M - \Int(L)$ is
diffeomorphic to a handlebody with handles of index at most $(n-1)$, which if $n > 3$ is in turn equivalent to the inclusion $L \to M$ being $(n-1)$-connected.)  In Section~\ref{sec:surg-objects-middle} we prove that the inclusion map induces a weak equivalence
\begin{equation}\label{eq:6}
  \Omega B \mathcal{C} \lra \Omega B \mathcal{C}_{\theta,L}^{n-1,n-2},
\end{equation}
provided that $\theta$ is spherical.

In the setup and notation of Theorem \ref{thm:main-C-new}, let us
suppose for simplicity that the map $\ell_K: K \to B$ is $n$-connected
(so that $B'=B$), and apply the above discussion with
$P=K\vert_0$. This gives a $\theta$-manifold $L$, and we will show how
to construct a canonical ``double" $\theta$-manifold $D(L)$ having the
following special property: for any object $P \in \mathcal{C}$ there
is a homotopy equivalence
\begin{equation*}
  \mathcal{C}(D(L),P) \simeq \mathcal{N}^\theta(P,\ell_P).
\end{equation*}
The weak equivalences~\eqref{eq:2}, \eqref{eq:3}, \eqref{eq:4},
\eqref{eq:5} and \eqref{eq:6} establish the homotopy equivalence
\begin{equation*}
  \Omega B\mathcal{C} \simeq \Omega^\infty MT\theta,
\end{equation*}
and the proof of Theorem~\ref{thm:main-C-new} in this case will be
completed by applying a suitable version of the ``group completion''
theorem to the canonical map $\mathcal{C}(D(L),P) \to \Omega
B\mathcal{C}$.

The weak equivalences~\eqref{eq:4}, \eqref{eq:5} and \eqref{eq:6} are
established using a parametrised surgery procedure, and the proof
depends on the contractibility of certain spaces of surgery data.
Contractibility is proved in a similar way in all three cases, and we
defer this to Section~\ref{sec:Connectivity}.  Finally, in
Section~\ref{sec:ApplyingGC} we explain how to use a version of the
group completion theorem to prove Theorem~\ref{thm:main-C-new} and tie
things together.

Sections~\ref{sec:surgery-morphisms}--\ref{sec:Connectivity} contain
the main technical steps, but on a first reading it is possible to
skip to Section~\ref{sec:ApplyingGC} after reading
Section~\ref{sec:defin-recoll}, to see the overall structure of the
argument.  The
%\mnote{sg: added these guiding words}
reader mainly interested in Theorems~\ref{thmcor:rational-coho} and
\ref{thm:main-A} can take $\theta = \theta^n: BO(2n)\langle n\rangle
\to BO(2n)$ and $L \cong D^{2n-1}$ in the above outline and throughout
the paper.  Considering only this special case would not significantly
simplify the main technical steps in
Sections~\ref{sec:surgery-morphisms}--\ref{sec:Connectivity}, but the
group completion arguments in Section~\ref{sec:ApplyingGC} do
simplify, and we incorporate a separate discussion of this case in
Section~\ref{sec:outl-proof-theor}.

%%% Local Variables: 
%%% mode: latex
%%% TeX-master: "Moduli"
%%% End: 

\section{Definitions and recollections}
\label{sec:defin-recoll}

\subsection{Tangential structures}
Throughout this paper, an important role will be played by the notion
of a tangential structure on manifolds.  This will be important even
for the proof of theorems which do not explicitly mention tangential
structures on manifolds.  However, for the proofs of
Theorems~\ref{thmcor:rational-coho} and \ref{thm:main-A}, the structure
$\theta = \theta^n: BO(2n)\langle n\rangle \to BO(2n)$
suffices.

\begin{definition}
  A \emph{tangential structure} is a map $\theta : B \to BO(d)$. A
  \emph{$\theta$-structure} on a $d$-manifold $W$ is a bundle map
  (i.e.\ fibrewise linear isomorphism) $\ell : TW \to
  \theta^*\gamma$. A $\theta$-manifold is a pair $(W, \ell)$. More
  generally, a $\theta$-structure on a $k$-manifold $M$ (with $k \leq
  d$) is a bundle map $\ell : \epsilon^{d-k} \oplus TM \to
  \theta^*\gamma$.
\end{definition}

Given vector bundles $U$ and $V$ of the same dimension, but not
necessarily over the same space, we write $\Bun(U, V)$ for the
subspace of $\map(U,V)$ (with the compact-open topology) consisting of
the bundle maps. Thus, $\Bun(TW, \theta^*\gamma)$ is the space of
$\theta$-structures on $W$.

\subsection{Spaces of manifolds}
We recall the definition and main properties of spaces of
submanifolds, from \cite{GR-W}. Fix a tangential structure $\theta : B
\to BO(d)$.

\begin{definition}
  For an open subset $U \subset \bR^n$, we denote by $\Psi_\theta(U)$
  the set of pairs $(M^d, \ell)$ where $M^d \subset U$ is a smooth
  $d$-dimensional submanifold that is a closed as a topological
  subspace, and $\ell$ is a $\theta$-structure on $M$.

  We denote by $\Psi_{\theta_{d-m}}(U)$ the set of pairs $(M, \ell)$
  where $M \subset U$ is a smooth $(d-m)$-dimensional submanifold that
  is closed as a topological subspace, and $\ell$ is a
  $\theta$-structure on $M$, i.e.\ a bundle map $\epsilon^m \oplus TM
  \to \theta^*\gamma$.
\end{definition}

In \cite[\S 2]{GR-W} we have defined a topology on these sets so that
$U \mapsto \Psi_{\theta_{d-m}}(U)$ defines a continuous sheaf of
topological spaces on the site of open subsets of $\bR^n$. We will not
give full details of the topology again here, but remind the reader
that the topology is ``compact-open" in flavour: disregarding
tangential structures, points nearby to $M$ are those which near some
large compact subset $K \subset U$ look like small normal deformations
of $M$. In particular, a typical neighbourhood of the empty manifold
$\emptyset \in \Psi_\theta(U)$ consists of all those manifolds in $U$
disjoint from some compact $K$.

\begin{definition}
  We define $\psi_\theta(n,k) \subset \Psi_\theta(\bR^n)$ to be the
  subspace consisting of those $\theta$-manifolds $(M, \ell)$ such
  that $M \subset \bR^k \times (-1,1)^{n-k}$. We make the analogous
  definition of $\psi_{\theta_{d-m}}(n,k)$.
\end{definition}

\subsection{Semi-simplicial spaces and non-unital categories}

Let $\Delta$ denote the category of finite non-empty totally ordered
sets and monotone maps, the simplicial indexing category. Let
$\Delta_{\mathrm{inj}} \subset \Delta$ denote the subcategory with the
same objects but only injective monotone maps as morphisms.  For a
category $\mathcal{C}$, a \emph{simplicial object in $\mathcal{C}$} is
a contravariant functor $X : \Delta \to \mathcal{C}$, and a
\emph{semi-simplicial object in $\mathcal{C}$} is a contravariant
functor $X : \Delta_{\mathrm{inj}} \to \mathcal{C}$. A map of
(semi-)simplicial objects is a natural transformation of functors.

We call a semi-simplicial object in the category of topological spaces
a \emph{semi-simplicial space}. More concretely, it consists of a
space $X_n = X(0 < 1 < \cdots < n)$ for each $n \geq 0$ and face maps
$d_i : X_n \to X_{n-1}$ defined for $i=0,\ldots,n$ satisfying the
simplicial identities $d_i d_j = d_{j-1} d_i$ for $i < j$. We often
denote a semi-simplicial space by $X_\bullet$, where we treat
$\bullet$ as a place-holder for the simplicial degree.

The \emph{geometric realisation} of a semi-simplicial space
$X_\bullet$ is defined to be
\begin{equation*}
  \vert X_\bullet \vert = \coprod_{n \geq 0} X_n \times \Delta^n / \sim
\end{equation*}
where $\Delta^n$ denotes the standard topological $n$-simplex and the
equivalence relation is generated by $(d_i(x), y) \sim (x, d^i(y))$
where $d^i : \Delta^{n} \to \Delta^{n+1}$ the inclusion of the $i$th
face. This space is given the quotient topology.

The \emph{$k$-skeleton} of $\vert X_\bullet \vert$ is
\begin{equation*}
  \vert X_\bullet \vert^{(k)} = \coprod_{n = 0}^k X_n \times \Delta^n / \sim
\end{equation*}
with the quotient topology, and one easily checks that $\vert
X_\bullet \vert = \cup_{k \geq 0} \vert X_\bullet \vert^{(k)}$ with
the direct limit topology. A useful consequence of this is the
following: a map from a compact space to $\vert X_\bullet \vert$ lands
in a finite skeleton.  We recall the following result.
\begin{lemma}
  If $X_\bullet \to Y_\bullet$ is a map of semi-simplicial spaces such
  that each $X_n \to Y_n$ is a weak homotopy equivalence, then $\vert
  X_\bullet \vert \to \vert Y_\bullet \vert$ is too.\qed
\end{lemma}

\begin{remark}
The term \emph{semi-simplicial object} we have defined above is not quite standard  (though is gaining popularity) and deserves some justification. Our justification is that it agrees with Eilenberg and Zilber's original usage of ``semi-simplicial complex" \cite{EZ}. Another is that the alternative used in the literature is \emph{$\Delta$-space}, but as $\Delta$ is the indexing category for full simplicial objects this seems counterintuitive.
\end{remark}

A \emph{non-unital topological category} $\mathcal{C}$ consists of a pair of spaces $(\mathcal{O}, \mathcal{M})$ of objects and morphisms, equipped with source and target maps $s, t : \mathcal{M} \to \mathcal{O}$. We let $\mathcal{M} \times_{t\mathcal{O}s} \mathcal{M}$ denote the fibre product made with the maps $t$ and $s$, and require in addition a composition map $\mu : \mathcal{M} \times_{t\mathcal{O}s} \mathcal{M} \to \mathcal{M}$
which satisfies the evident associativity requirement.

A non-unital topological category $\mathcal{C}$ has a semi-simplicial nerve, generalising the simplicial nerve of a topological category \cite{SegalNerve}. Define $N_\bullet \mathcal{C}$ by $N_0\mathcal{C} = \mathcal{O}$ and
$$N_k \mathcal{C} = \mathcal{M} \times_{t\mathcal{O}s} \mathcal{M} \times_{t\mathcal{O}s} \cdots \times_{t\mathcal{O}s} \mathcal{M} \quad k > 0$$
being the space of $k$-tuples of composable morphisms, and let the face maps be given by composing and forgetting morphisms, as in the simplicial nerve of a topological category. We define the \emph{classifying space} of a non-unital topological category by
$$B\mathcal{C} = \vert N_\bullet \mathcal{C} \vert.$$

\subsection{Definition of the cobordism categories}

For convenience in the rest of the paper, we introduce the following notation. All of our constructions will take place inside $\bR \times \bR^N$, and we write $x_1 : \bR \times \bR^N \to \bR$ for the projection to the first coordinate. Given a manifold $W \subset \bR \times \bR^N$ and a set $A \subset \bR$, we write
$$W \vert_A = W \cap x_1^{-1}(A),$$
and we also write $\ell\vert_A$ for the restriction of a $\theta$-structure $\ell$ on $W$ to this manifold.

Our definition of the cobordism category of $\theta$-manifolds is similar to that of \cite{GR-W} (the only difference is that here will we only define a non-unital category); it follows that of \cite{GMTW} in spirit, but is different in some technical points. We use the spaces of manifolds of the last section in order to describe the point-set topology of these categories.

\begin{definition}
  For each $\epsilon > 0$ we let the non-unital topological category
  $\mathcal{C}_\theta(\bR^N)_\epsilon$ have space of objects
  $\psi_{\theta_{d-1}}(N, 0)$. The space of morphisms from $(M_0,
  \ell_0)$ to $(M_1, \ell_1)$ is the subspace of those $(t, (W, \ell))
  \in \bR \times \psi_\theta(N+1, 1)$ such that $t>0$
  and
  \begin{equation*}
    W \vert_{(-\infty, \epsilon)} = (\bR \times M_0)
    \vert_{(-\infty, \epsilon)} \in \Psi_\theta((-\infty,
    \epsilon)\times \bR^N)
  \end{equation*}
  and
  \begin{equation*}
    W \vert_{(t- \epsilon, \infty)} = (\bR \times M_1)
    \vert_{(t- \epsilon, \infty)} \in \Psi_\theta((t-
    \epsilon, \infty)\times \bR^N).
  \end{equation*}
  Here $\bR \times M_i$ denotes the $\theta$-manifold with underlying
  manifold $\bR \times M_i \subset \bR \times \bR^{N}$ and
  $\theta$-structure
  $$T(\bR \times M_i) \lra \epsilon^1 \oplus TM_i \overset{\ell_i}\lra
  \theta^*\gamma.$$ Composition in this category is defined by
  $$(t, W) \circ (t, W') = (t+t', W\vert_{(-\infty, t]}
  \cup (W' + t\cdot e_1)\vert_{[t, \infty)} )$$ where $W' + t\cdot
  e_1$ denotes the manifold $W'$ translated by $t$ in the first
  coordinate. We topologise the total space of morphisms as a subspace
  of $(0, \infty) \times \psi_\theta(N+1,
  1)$.
  
  If $\epsilon < \epsilon'$ there is an inclusion $\mathcal{C}_\theta(\bR^N)_{\epsilon'} \subset \mathcal{C}_\theta(\bR^N)_\epsilon$, and we define $\mathcal{C}_\theta(\bR^N)$ to be the colimit  over all $\epsilon>0$.
\end{definition}

Note that a morphism $(t, (W, \ell))$ in this category is uniquely
determined by the restriction $(t, (W\vert_{[0,t]},
\ell\vert_{[0,t]}))$. We often think of morphisms in this
category as being given by such restricted manifolds, but the topology
on the space of morphisms is best described as we did above.

As explained in the introduction, we will also require a version of this category where the objects and
morphisms contain a fixed codimension zero submanifold. In order to
define this, we let
\begin{equation*}
  L \subset (-1/2, 0] \times (-1,1)^{N-1}
\end{equation*}
be a compact $(d-1)$-manifold which near $\{0\} \times \R^{N-1}$
agrees with $(-1,0]\times \partial L$. Furthermore, we let $\ell\vert_L : \epsilon^1 \oplus TL \to
\theta^*\gamma$ be a $\theta$-structure on $L$.  Near $\partial L$
we require that the structure is a product (i.e.\ that translation in
the collar direction preserves the structure).  Such an $\ell$ makes
$\R \times L$ into a $\theta$-manifold with boundary, and we make
the following definition.
\begin{definition}\label{defn:sub-cat-with-L}
  The topological subcategory $\mathcal{C}_{\theta, L}(\bR^N) \subset
  \mathcal{C}_{\theta}(\bR^N)$ has space of objects those $(M, \ell)$
  such that
  $$M \cap \big((-\infty, 0] \times \bR^{N-1}\big) = L$$
  as $\theta$-manifolds. It has space of morphisms from $(M_0,
  \ell_0)$ to $(M_1, \ell_1)$ given by those $(t, (W, \ell))$ such
  that
  $$W \cap \big(\bR \times (-\infty, 0] \times \bR^{N-1}\big) = \bR \times L$$
  as $\theta$-manifolds.
\end{definition}

\begin{remark}\label{rem:IndependenceL}
  The category $\mathcal{C}_{\theta, L}(\bR^N)$ does not really depend
  on $L$, but only on $\partial L$. It is sometimes convenient to
  think of the interior of $L$ as being cut out, so that objects in
  the category are manifolds with boundary $\partial L$ and morphisms
  are cobordisms between manifolds with boundary which are trivial
  along the boundary.

  If we take $L = D^{d-1}$ then the category $\mathcal{C}_{\theta,
    L}(\bR^N)$ is equivalent to the category of ``manifolds with
  basepoint'' defined in \cite[Definition 4.2]{GR-W}.  That case is
  sufficient for the proofs of Theorems~\ref{thmcor:rational-coho}
  and~\ref{thm:main-A}.
\end{remark}

The subject of our main technical theorem, from which we shall show
how to obtain results on diffeomorphism groups in Section \ref{sec:ApplyingGC}, is certain subcategories of
$\mathcal{C}_{\theta, L}(\bR^N)$ where we require the morphisms to
have certain connectivities relative to their outgoing boundaries, and
objects to be those $(M, \ell_M)$ whose Gauss map $M \to B$  (i.e.\ the map underlying $\ell_M : \epsilon^1 \oplus TM \to \theta^*\gamma$) has a certain injectivity range on homotopy groups.

\begin{definition}
  The topological subcategory $\mathcal{C}^{\kappa}_{\theta, L}(\bR^N)
  \subset \mathcal{C}_{\theta, L}(\bR^N)$ has the same space of
  objects. It has space of morphisms from $(M_0, \ell_0)$ to $(M_1,
  \ell_1)$ given by those $(t, (W, \ell))$ such that the pair
  $(W\vert_{[0,t]}, W \vert_{t})$ is $\kappa$-connected, i.e.\ such
  that $\pi_i(W\vert_{[0,t]}, W \vert_{t}) = 0$ for all basepoints and
  all $i \leq \kappa$. Thus this is the subcategory on those morphisms
  which are $\kappa$-connected relative to their outgoing
  boundary.
\end{definition}

The category $\mathcal{C}^0_{\theta}$ is the ``positive boundary
category" as in \cite{GMTW}, where each path component of a cobordism
is required to have non-empty outgoing boundary.

\begin{definition}
  The topological subcategory $\mathcal{C}^{\kappa, l}_{\theta,
    L}(\bR^N) \subset \mathcal{C}^{\kappa}_{\theta, L}(\bR^N)$ is the
  full subcategory on those objects $(M, \ell)$ such that the
  map
  \begin{equation*}
    \ell_* : \pi_i(M) \lra \pi_i(B)
  \end{equation*}
  is injective for all $i \leq l$ and all basepoints.  (In our main
  application in Section \ref{sec:ApplyingGC}, the map $L \to B$ will
  be $(l+1)$-connected.  In that case the requirement is equivalent to
  $(M,L)$ being $l$-connected.)
\end{definition}

For our final definition we specialise to even dimensions.
\begin{definition}\label{defn:CobCatA}
  Let $d=2n$ and let
  \begin{equation*}
    \mathcal{A} \subset \pi_0\left(\Ob(\mathcal{C}_{\theta, L}^{n-1,
        n-2}(\bR^N))\right)
  \end{equation*}
  be a collection of path components of the space of objects. The topological subcategory
  $\mathcal{C}^{n-1, \mathcal{A}}_{\theta, L}(\bR^N) \subset
  \mathcal{C}^{n-1, n-2}_{\theta, L}(\bR^N)$ is the full subcategory
  on the subspace of those objects in $\mathcal{A}$.
\end{definition}

For $N = \infty$, we shall often denote $\mathcal{C}_\theta(\bR^\infty) = \colim_N \mathcal{C}_\theta(\R^N)$ by $\mathcal{C}_\theta$, and similarly with any decorations.

\subsection{The homotopy type of the cobordism category}\label{sec:GMTW}

The main theorem of \cite{GMTW} identifies the homotopy type $\Omega B
\mathcal{C}_\theta$ in terms of the infinite loop space of a certain
Thom spectrum $MT\theta$.

Recall from the introduction that given a map $\theta : B \to BO(d) =
\mathrm{Gr}_d(\bR^\infty)$ we let $B(\bR^n) =
\theta^{-1}(\mathrm{Gr}_d(\bR^n))$ and define $\gamma_{n}^\perp \to
\mathrm{Gr}_d(\bR^n)$ to be the orthogonal complement of the
tautological bundle. The canonical map $B(\bR^n) \to B(\bR^{n+1})$
pulls back $\theta^*\gamma_{n+1}^\perp$ to $\theta^*\gamma_{n}^\perp
\oplus \epsilon^1$ and hence we obtain pointed maps
\begin{equation*}
  \left(B(\bR^n)^{\theta^*\gamma_{n}^\perp}\right)\wedge S^1 \lra
  B(\bR^{n+1})^{\theta^*\gamma_{n+1}^\perp}
\end{equation*}
of Thom spaces, which form a spectrum $MT\theta$. Its associated
infinite loop space is
\begin{equation*}
  \Omega^\infty MT\theta = \colim_{n \to \infty}
  \Omega^n\left(B(\bR^n)^{\theta^*\gamma_{n}^\perp}\right).
\end{equation*}

\begin{theorem}[Galatius--Madsen--Tillmann--Weiss \cite{GMTW}]
There is a canonical map
$$\Omega B\mathcal{C}_\theta \lra \Omega^\infty MT\theta$$
which is a weak homotopy equivalence.
\end{theorem}

We write $\Omega^\infty_\bullet MT\theta$ for the basepoint component
of $\Omega^\infty MT\theta$, and now describe the rational cohomology
of this space. The map $B \overset{\theta}\to BO(d) \overset{\det}\lra
BO(1)$ on fundamental groups defines a character $w_1 : \pi_1(B) \to
\bZ^{\times}$, and we write $H^*(B;\bQ^{w_1})$ for the rational
cohomology of $B$ with local coefficients given by this character. For
each $n$ there are evaluation maps
\begin{equation*}
  ev : \Sigma^n \Omega^n
  \left(B(\bR^n)^{\theta^*\gamma_{n}^\perp}\right) \lra
  B(\bR^n)^{\theta^*\gamma_{n}^\perp}
\end{equation*}
and so we can define the dotted map in the diagram
\begin{equation*}
	\xymatrix{
		H^{*+d}(B(\bR^n);\bQ^{w_1}) \ar@{.>}[r] \ar@{=}[d]_{\text{Thom iso.}} & H^*(\Omega^n (B(\bR^n)^{\theta^*\gamma_{n}^\perp});\bQ) \ar@{=}[d]^{\text{ Suspension iso.}}\\
		{\widetilde{H}^{*+n}(B(\bR^n)^{\theta^*\gamma_{n}^\perp};\bQ)} \ar[r]^-{ev^*} & {\widetilde{H}^{*+n}(\Sigma^n \Omega^n (B(\bR^n)^{\theta^*\gamma_{n}^\perp});\bQ)}
	}
\end{equation*}
by commutativity.  Taking limits and restricting to the basepoint
component, we obtain a map
\begin{equation*}
  \sigma : H^{*+d}(B;\bQ^{w_1}) \lra H^*(\Omega^\infty_\bullet MT\theta;\bQ)
\end{equation*}
and the right-hand side is a graded-commutative algebra, so $\sigma$
extends to the free graded-commutative algebra on the part of
$H^{*+d}(B;\bQ^{w_1})$ of degree $> 0$,
$$\Lambda(H^{*+d > 0}(B;\bQ^{w_1})) \lra H^*(\Omega^\infty_\bullet MT\theta;\bQ).$$
This is an isomorphism of graded-commutative algebras.

\subsection{Poset models}\label{sec:PosetModels}

A key step in the proofs of \cite{GMTW} and \cite{GR-W} identifying
the infinite loop space $B \mathcal{C}_\theta$ is to first identify
this classifying space with the classifying space of a certain
topological poset.  The result holds for all variations of the
cobordism category mentioned above; we prove the general result in
Proposition~\ref{prop:PosetModel} below.
\begin{definition}
  Let $\mathcal{C} \subset \mathcal{C}_\theta(\bR^N)$ be a
  subcategory. Let
$${D}^\mathcal{C}_\theta \subset \bR \times \bR_{>0} \times \psi_\theta(N+1, 1)$$
denote the subspace of triples $(t, \epsilon, (W, \ell))$ such that $[t-\epsilon, t+\epsilon]$ consists of regular values for $x_1 : W \to \bR$, and $W \vert_t \in \Ob(\mathcal{C})$. Define a partial order on $D_\theta^\mathcal{C}$ by
$$(t, \epsilon,  (W, \ell)) < (t', \epsilon', (W', \ell'))$$
if and only if $(W, \ell) = (W', \ell')$, $t+\epsilon < t' -\epsilon$ and $W \vert_{[t, t']} \in \Mor( \mathcal{C})$.
\end{definition}

\begin{proposition}\label{prop:PosetModel}
  Let $\mathcal{C} \subset \mathcal{C}_{\theta, L}(\bR^N) \subset
  \mathcal{C}_{\theta}(\bR^N)$ be a subcategory which consists of
  entire path components of the object and morphism spaces of
  $\mathcal{C}_{\theta, L}(\bR^N)$. Then there is a weak homotopy
  equivalence
$$B\mathcal{C} \simeq B{D}_\theta^\mathcal{C}.$$
\end{proposition}
\begin{proof}
We introduce an auxiliary topological
  poset $D^{\mathcal{C},\perp}_\theta$ which maps to both
  $D_\theta^\mathcal{C}$ and $\mathcal{C}$.  It is the subposet of
  $D_\theta^\mathcal{C}$ consisting of $(t,\epsilon,(W,\ell)$ such
  that $(W,\ell)$ is a product over $(t-\epsilon,t+\epsilon)$.  This
  conditions means that if we write $W \vert_{t} = \{t\} \times
  M$ and give $M$ the inherited $\theta$-structure, then
  \begin{equation*}
    W \vert_{(t-\epsilon,t+\epsilon)} = (t-\epsilon,
    t+\epsilon) \times M
  \end{equation*}
  as $\theta$-manifolds.  Then there is a zig-zag of functors
  \begin{equation*}
    D_\theta^\mathcal{C} \lla D_\theta^{\mathcal{C},\perp}
    \lra \mathcal{C},
  \end{equation*}
  where the first arrow is the inclusion of the subposet and the
  second is the functor that sends a morphism $(a < b, W,\ell)$ to the manifold $(W\vert_{[a,b]} -a\cdot e_1)$ extended cylindrically in $(-\infty,0] \times \R^N$ and $[b-a,\infty) \times \R^N$.  This induces a
  zig-zag diagram 
  \begin{equation*}
    N_kD_\theta^\mathcal{C} \lla N_kD_\theta^{\mathcal{C},\perp}
    \lra N_k\mathcal{C},
  \end{equation*}
  and we prove that both maps are weak equivalence for all $k$ in the
  same way as in \cite[Theorem 3.9]{GR-W}.
\end{proof}

Applying the above construction to the categories
$\mathcal{C}_{\theta, L}^{\kappa, l}(\bR^N)$ we obtain topological posets
${D}_{\theta, L}^{\kappa, l}(\bR^N)$ and weak homotopy equivalences
\begin{equation}\label{eq:PosetModel}
  B\mathcal{C}_{\theta, L}^{\kappa, l}(\bR^N) \simeq B{D}_{\theta, L}^{\kappa, l}(\bR^N).
\end{equation}
Similarly, when we specialise to the case $d=2n$ and let $\mathcal{A}
\subset \pi_0(\Ob(\mathcal{C}_{\theta, L}^{n-1,
  n-2}(\bR^N)))$ be a collection of path components of objects, we
obtain weak homotopy equivalences
\begin{equation}\label{eq:PosetModelA}
B\mathcal{C}_{\theta, L}^{n-1, \mathcal{A}}(\bR^N) \simeq B{D}_{\theta, L}^{n-1, \mathcal{A}}(\bR^N).
\end{equation}

\subsection{The homotopy type of $\mathcal{C}_{\theta, L}(\bR^N)$}
In \cite[Theorems 3.9 and 3.10]{GR-W} we proved that there is a weak homotopy equivalence $BD_\theta(\bR^N) \simeq \psi_\theta(N+1, 1)$, which combined with Proposition \ref{prop:PosetModel} gives
\begin{equation}\label{eq:SpaceModel}
B\mathcal{C}_\theta(\bR^N) \simeq B{D}_\theta(\bR^N) \simeq \psi_\theta(N+1, 1).
\end{equation}
(Strictly speaking, in that paper we worked with a version of
${D}_\theta(\bR^N)$ where $\epsilon = 0$, but the obvious map induces
a levelwise weak equivalence of nerves.) For the purposes of this paper we
require a slightly stronger version of this result, taking into
account the submanifold $L$.
\begin{proposition}\label{prop:SpaceModelL}
  There are weak homotopy equivalences
  \begin{equation}\nonumber
    B\mathcal{C}_{\theta, L}(\bR^N) \simeq B{D}_{\theta, L}(\bR^N)
    \simeq \psi_{\theta, L}(N+1, 1)
  \end{equation}
  where $\psi_{\theta, L}(N+1, 1) \subset \psi_{\theta}(N+1, 1)$ is
  the subspace consisting of those $(W, \ell)$ such that $W \cap (\bR
  \times (-\infty,0] \times \bR^{N-1}) = \bR \times L$ as
  $\theta$-manifolds.
\end{proposition}
\begin{proof}
  The proof of \cite[Theorem 3.10]{GR-W} applies verbatim.
\end{proof}
\begin{proposition}\label{prop:forget-L}
  The inclusion
  \begin{equation}\nonumber
    i : \psi_{\theta, L}(N+1, 1) \lra \psi_{\theta}(N+1, 1)
  \end{equation}
  is a weak homotopy equivalence.
\end{proposition}
\begin{proof}
  This is similar to \cite[Lemma 4.6]{GR-W}, which is essentially the
  case $L = D^{d-1}$. It requires careful analysis of
  $\theta$-structures, so let us, for this proof only, denote the
  $\theta_{d-1}$-structure on $L$ by $\ell_L : \epsilon^1 \oplus TL
  \to \theta^*\gamma$. We first want to construct the \emph{double}
  $D(L)$ of $L$ as a $\theta_{d-1}$-manifold, and a canonical
  $\theta$-null-bordism of it. Recall that $L$ is a submanifold of
  $(-1/2, 0] \times (-1,1)^{N-1}$ which we identify with $\{0\} \times
  (-1/2, 0] \times (-1,1)^{N-1} \subset (-1,0] \times (-1/2, 0] \times
  (-1,1)^{N-1}$.  Let $V \subset (-1,0] \times (-1/2,1/2) \times
  (-1,1)^{N-1}$ denote the subset swept out by rotating $L$ around
  $(0, 0)$ in the half-plane $(-1,0] \times (-1,1)$.  As $L$ was
  collared, this subset is a $d$-dimensional submanifold with
  boundary, and $L$ lies in its boundary. We define $D(L) = \partial
  V$, and $\overline{L} = D(L) - \Int(L)$. The inclusion $L
  \hookrightarrow V$ is a homotopy equivalence, so there is a unique
  extension up to homotopy
  \begin{equation*}
  \xymatrix{
  	\epsilon^1 \oplus TL \ar[r]^-{\ell_L} \ar[d]& \theta^*\gamma\\
  	TV, \ar@{.>}[ru]
  }
  \end{equation*}
  where the vertical map sends $\epsilon^1$ to the outwards pointing
  vector. This restricts to a $\theta$-structure on $D(L)$, and hence
  on $\overline{L}$, and $V$ gives a $\theta$-cobordism $V : \emptyset
  \leadsto D(L)$.

  Similarly, we can rotate $L$ in the half-plane $[0,1) \times (-1,1)$
  around the point $(0, -1/2)$ to obtain a submanifold of $[0,1]
  \times [-1,0] \times (-1,1)^{N-1}$, extending to a
  $\theta$-cobordism $U \subset [0,1] \times [-1,0] \times
  (-1,1)^{N-1}$, ending at $\{1\} \times [-1,0] \times \partial L$ and
  starting at $\{0\} \times (L \cup (\overline{L} - e_1))$, where
  $\overline{L} - e_1 \subset [-1,\frac12) \times (-1,1)^{N-1}$
  denotes the parallel translate of $\overline{L}$.

  The $\theta$-manifolds $U$ and $V$ give us the tools we need.
  $D(L)$ is a submanifold of $(-1/2, 1/2) \times (-1,1)^{N-1}$, so we
  have a $\theta$-manifold $\bR \times D(L) \subset \bR \times (-1/2,
  1/2) \times (-1,1)^{N-1}$. We define a map
  $$ r : \psi_{\theta}(N+1, 1) \lra \psi_{\theta, L}(N+1, 1)$$
  which given $(W, \ell) \subset \bR \times (-1,1) \times
  (-1,1)^{N-1}$ applies the affine diffeomorphism $(-1,1) \cong (1/2,
  1)$ to its second coordinate, and then takes the (disjoint) union
  with $\bR \times D(L)$.

  \begin{figure}[h]
    \includegraphics[bb=0 0 342 149]{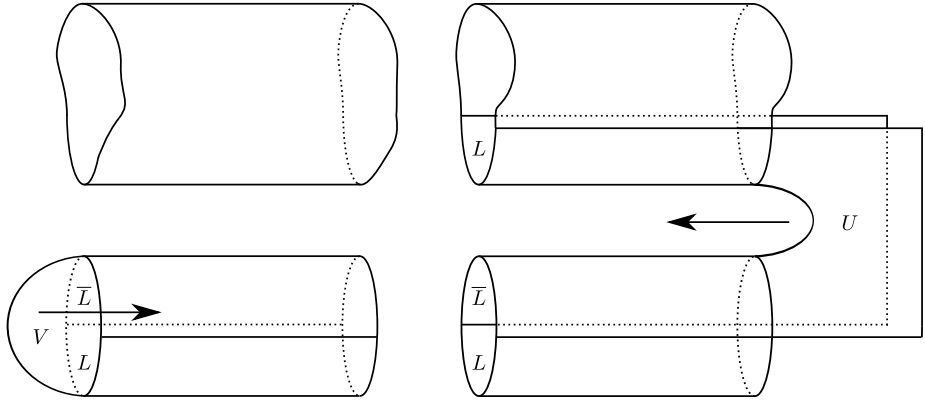}
    \caption{Adding and removing $L$.}\label{fig:AddingL}
  \end{figure}

  The composition $i \circ r$ is homotopic to the identity as the
  $\theta$-null-bordism $V$ of $D(L)$ may be used to push the cylinder
  $\bR \times D(L)$ off to the right.  A similar argument, pushing $U$
  to the left, proves that the composition $r \circ i$ is homotopic to
  the identity. Figure \ref{fig:AddingL} shows how.
\end{proof}

Combining this proposition with Proposition \ref{prop:SpaceModelL} and
the homotopy equivalence~(\ref{eq:SpaceModel}) gives the following corollary.

\begin{corollary}
  For any  pair $(L, \ell_L)$ as in Definition~\ref{defn:sub-cat-with-L}, the
  inclusion
  $$B\mathcal{C}_{\theta, L}(\bR^N) \lra B\mathcal{C}_{\theta}(\bR^N)$$
  is a weak homotopy equivalence.\qed
\end{corollary}

\subsection{A more flexible model}

From the poset models of Section \ref{sec:PosetModels} we construct the semi-simplicial spaces
$$D_{\theta, L}^{\kappa, l}(\bR^N)_\bullet = N_\bullet D_{\theta, L}^{\kappa, l}(\bR^N).$$
The remarks of Section \ref{sec:PosetModels} and Proposition \ref{prop:SpaceModelL} show that the geometric
realisations of these semi-simplicial spaces are models for the
classifying spaces of the categories $\mathcal{C}_{\theta, L}^{\kappa,
  l}(\bR^N)$ in which we are interested. The benefit of working with
these semi-simplicial spaces instead of the cobordism categories is
that we can often make constructions which are \emph{not} functorial, yet
give well-defined maps between geometric realisations of the
semi-simplicial spaces involved.

To make this technique easier to apply, we will define an auxiliary
semi-simplicial space $X_\bullet^{\kappa, l}$.  We will prove that its
geometric realisation is weakly equivalent to
$B\mathcal{C}_{\theta,L}^{\kappa,l}(\R^N)$, but it will be easier to
construct a simplicial map into $X_\bullet^{\kappa,l}$ than into
$N_\bullet\mathcal{C}_{\theta,L}^{\kappa,l}(\R^N)$ or $D_{\theta, L}^{\kappa, l}(\bR^N)_\bullet$.  The space $X_\bullet^{\kappa,l}$ also depends on $N$, but we omit that from the notation.

\begin{definition}\label{defn:X}
Let $\theta : B \to BO(d)$, $N$ and $L$ be as before. Let $-1 \leq \kappa \leq \tfrac{d-1}{2}$, $-1 \leq l \leq \kappa$ and $-1 \leq l \leq d-\kappa-2$. Define ${X}_\bullet^{\kappa, l}$ to be the semi-simplicial space with $p$-simplices consisting of certain tuples $(a,\epsilon, (W, \ell))$ such that $a = (a_0, \dots, a_p) \in
\R^{p+1}$, $\epsilon = (\epsilon_0, \dots, \epsilon_p) \in
(\R_{>0})^{p+1}$, and $(W, \ell) \in \Psi_\theta((a_0-\epsilon_0,
a_p+\epsilon_p) \times \bR^N)$, satisfying
\begin{enumerate}[(i)]
  \item \label{item:1}$W \subset (a_0 - \epsilon_0, a_p +
    \epsilon_p) \times (-1,1)^N$,
    
  \item \label{item:2}$W$ and $(a_0 - \epsilon_0,a_p + \epsilon_p)
    \times L$ agree as $\theta$-manifolds on the subspace
    $x_2^{-1}(-\infty,0]$,
    
  \item $a_{i-1} + \epsilon_{i-1} < a_i - \epsilon_i$ for all $i = 1,
    \dots, p$,\label{it:Disjoint}
    
  \item for each pair of regular values $t_0 < t_1 \in
    \cup_i(a_i-\epsilon_i, a_i+\epsilon_i)$, the cobordism $W \vert_{[t_0, t_1]}$ is $\kappa$-connected relative to its outgoing
    boundary,\label{it:MorConn}
    
  \item for each regular value $t \in (a_i-\epsilon_i, a_i+\epsilon_i)$, the map
    \begin{equation*}
      \pi_j(W \vert_{t}) \lra \pi_j(B),
    \end{equation*}
    induced by $\ell\vert_t$, is injective for all basepoints and all $j \leq
    l$.\label{it:ObConn}
  \end{enumerate}
  We topologise this set as a subspace of $\bR^{p+1} \times
  (\bR_{>0})^{p+1} \times \Psi_\theta((-1,1) \times \bR^N)$, where we
  use the standard affine diffeomorphism $(-1,1) \cong
  (a_0-\epsilon_0, a_p + \epsilon_p)$ to identify
  $\Psi_\theta((a_0-\epsilon_0, a_p+\epsilon_p) \times \bR^N)$ with
  $\Psi_\theta((-1,1) \times \bR^N)$. The $j$th face map is given by
  forgetting $a_j$ and $\epsilon_j$, and if $j=0$, composing with the
  restriction map $\Psi_\theta((a_0-\epsilon_0, a_p+\epsilon_p) \times
  \bR^N) \to \Psi_\theta((a_1-\epsilon_1, a_p+\epsilon_p) \times
  \bR^N)$, and similarly if $j=p$.

  There are semi-simplicial maps $D_{\theta, L}^{\kappa,
    l}(\bR^N)_\bullet \to X_\bullet^{\kappa,l}$, which on
  $p$-simplices are given by sending $(a,\epsilon,(W, \ell))$ with
  $(W, \ell) \in \Psi_\theta(\bR \times \bR^N)$ to the same thing
  restricted down to $\Psi_\theta((a_0-\epsilon_0, a_p+\epsilon_p)
  \times \bR^N)$.
\end{definition}

The semi-simplicial space ${X}_\bullet^{\kappa,l}$ is easier to map
into (by a semi-simplicial map) than $D_{\theta, L}^{\kappa,
  l}(\bR^N)_\bullet$ for two reasons. Firstly, we do not require that
the intervals $(a_i-\epsilon_i, a_i+\epsilon_i)$ consist entirely of
regular values: instead we allow critical values, and conditions
(\ref{it:MorConn})--(\ref{it:ObConn}) ensure that the critical values
do not affect the essential properties of the space. Secondly, we
discard those parts of the manifold outside of $(a_0-\epsilon_0,
a_p+\epsilon_p)$, and so do not need to worry about controlling parts
of the manifold outside of the region.

\begin{definition}\label{defn:XA}
  In the case $d=2n$, with $\mathcal{A} \subset \pi_0(\Ob(\mathcal{C}_{\theta, L}^{n-1, n-2}(\bR^N)))$ a collection of
  path components of objects, we make the entirely analogous
  definition of ${X}_\bullet^{n-1, \mathcal{A}}$. Precisely, in
  Definition \ref{defn:X} we replace condition (\ref{it:ObConn}) by
  \begin{enumerate}[(\ref{it:ObConn}$^\prime$)]
  \item for each regular value $t \in (a_i-\epsilon_i,
    a_i+\epsilon_i)$, the $\theta_{d-1}$-manifold $(W \vert_{t},
    \ell\vert_{t})$ lies in $\mathcal{A}$.
  \end{enumerate}
\end{definition}

The following is our main result concerning these models, and together
with (\ref{eq:PosetModel}) and (\ref{eq:PosetModelA}) provides weak
homotopy equivalences $B\mathcal{C}_{\theta, L}^{\kappa, l}(\bR^N)
\simeq |X_\bullet^{\kappa, l}|$ and, in the case $d=2n$,
$B\mathcal{C}_{\theta, L}^{n-1, \mathcal{A}}(\bR^N) \simeq
|X_\bullet^{n-1, \mathcal{A}}|$.

\begin{proposition}\label{prop:X}
  Let $\kappa$ and $l$ satisfy the inequalities in
  Definition~\ref{defn:X}.  The semi-simplicial map $D_{\theta,
    L}^{\kappa, l}(\bR^N)_\bullet \to X_\bullet^{\kappa,l}$, and in
  the case $d=2n$ also the map $D_{\theta, L}^{n-1,
    \mathcal{A}}(\bR^N)_\bullet \to X_\bullet^{n-1, \mathcal{A}}$,
  induce weak homotopy equivalences after geometric realisation.
\end{proposition}
\begin{proof}
  For the proof we introduce an auxiliary semi-simplicial space
  $\bar{X}_\bullet^{\kappa,l}$. Its $p$-simplices are those tuples
  $$(a,\epsilon,(W, \ell)) \in \bR^{p+1} \times (\bR_{>0})^{p+1} \times
  \psi_\theta(N+1, 1)$$
  satisfying the conditions of Definition~\ref{defn:X}, except that
  the interval $(a_0 - \epsilon_0, a_p + \epsilon_p)$ is replaced with
  $\R$ in (\ref{item:1}) and (\ref{item:2}).
  We can regard $D_{\theta, L}^{\kappa, l}(\bR^N)_\bullet$ as a
  subspace of $\bar{X}_\bullet^{\kappa,l}$, and we have a
  factorisation
  $$D_{\theta, L}^{\kappa, l}(\bR^N)_\bullet \overset{i}\lra
  \bar{X}_\bullet^{\kappa,l} \lra X_\bullet^{\kappa,l}.$$ 
  
  The map $\bar{X}_\bullet^{\kappa,l} \to X_\bullet^{\kappa,l}$ is a
  weak homotopy equivalence in each simplicial degree, by methods
  similar to \cite[Theorem 3.9]{GR-W}. Briefly, in simplicial degree
  $p$ choose---continuously in the data $(a_0, a_p, \epsilon_0,
  \epsilon_p)$---diffeomorphisms $(a_0-\epsilon_0, a_p+\epsilon_p)
  \cong \bR$ which are the identity on $[a_0, a_p]$. Using this family
  of diffeomorphisms to stretch gives a map ${X}_p^{\kappa,l} \to
  \bar{X}_p^{\kappa,l}$, which is homotopy inverse to the restriction
  map $\bar{X}_p^{\kappa,l} \to X_p^{\kappa,l}$.

  To show that the first map induces a weak homotopy equivalence on
  geometric realisation, we use a technique which we shall use many
  times in this paper. That is, we consider a map
  $$f : (D^n, \partial D^n) \lra (\vert \bar{X}_\bullet^{\kappa,l} \vert, \vert D_{\theta, L}^{\kappa, l}(\bR^N)_\bullet \vert)$$
  representing an element of the $n$th relative homotopy group, and
  show that it may be homotoped through maps of pairs to a map with
  image in $\vert D_{\theta, L}^{\kappa, l}(\bR^N)_\bullet \vert$.

  For each $x \in D^n$ the point $f(x)$ is a tuple $(t, a, \epsilon,
  (W(x), \ell))$, and we may choose a pair $(a^x, \epsilon^x)$ such
  that $[a^x-\epsilon^x, a^x + \epsilon^x] \subset \cup_i (
  (a_i-\epsilon_i, a_i+\epsilon_i) - \{a_i\})$ and that $[a^x -
  \epsilon^x, a^x + \epsilon^x]$ consists of regular values of $x_1:
  W(x) \to \R$.  By properness of $x_1: W(x) \to \R$, there is a
  neighbourhood $U_x \ni x$ for which $[a^x-\epsilon^x, a^x +
  \epsilon^x]$ still consists of regular values.  The $U_x$'s cover
  $D^n$ and we let $\{U_j\}_{j \in J}$ be a finite subcover. We may
  suppose that $a^j \neq a^k$, as otherwise we may change the cover by
  letting $U'_j = U_j \cup U_k$ with $(a^j)' = a^j=a^k$ and
  $(\epsilon^j)' = \min(\epsilon^j, \epsilon^k)$. Once the $a^j$ are
  distinct, we may shrink the $\epsilon^j$ so that the intervals
  $[a^j+\epsilon^j, a^j-\epsilon^j]$ are pairwise disjoint, and so
  that no $a_i$ lies in such an interval.

  As the intervals $[a^j+\epsilon^j, a^j-\epsilon^j]$ are chosen
  to consist of regular values, the data $\{(U_j, a^j,
  \epsilon^j)\}_{j \in J}$, together with a choice of partition of
  unity subordinate to the cover by the $U_j$'s, determine a map
  $\hat{f}: D^n \to \vert {D}^{-1, -1}_{\theta, L}(\bR^N)_\bullet
  \vert$ with the same underlying family of $\theta$-manifolds. As
  $[a^j-\epsilon^j, a^j+\epsilon^j] \subset \cup_i (a_i-\epsilon_i,
  a_i+\epsilon_i)$, this new family satisfies conditions
  (\ref{it:MorConn}) and (\ref{it:ObConn}) of Definition \ref{defn:X}
  (as the old family did) so $\hat{f}$ actually has image in the
  subspace $\vert {D}^{\kappa, l}_{\theta, L}(\bR^N)_\bullet \vert$.
  There is a homotopy $H$ of $p \circ \hat{f}$ to $f$ as follows: on
  underlying $\theta$-manifolds it is constant, but on the interval
  data we first use the straight-line homotopy from the data $\{(a^j,
  \epsilon^j)\}$ to the data $\{(a_i, \epsilon)\}$ where we choose
  $\epsilon \leq \min(\epsilon_i)$ small enough so that
  $[a_i-\epsilon, a_i+\epsilon]$ is disjoint from the
  $[a^j-\epsilon^j, a^j+\epsilon^j]$. This straight-line homotopy is
  in the barycentric coordinates: as the intervals are all disjoint,
  the join of the simplices they describe also lies in $\vert
  {D}^{\kappa, l}_{\theta, L}(\bR^N)_\bullet \vert$, and so there is a
  canonical straight line between them. Then we use the obvious
  homotopy from the data $\{(a_i, \epsilon)\}$ to the data $\{(a_i,
  \epsilon_i)\}$ that stretches the $\epsilon$'s. The restriction of
  $H$ to $\partial D^n$ remains in the subspace $\vert {D}^{\kappa,
    l}_{\theta, L}(\bR^N)_\bullet \vert$, and so $H$ gives a relative
  null-homotopy of $f$.
  
  The case when $d=2n$ and $\mathcal{A}$ is chosen is identical.
\end{proof}

%%% Local Variables: 
%%% mode: latex
%%% TeX-master: "Moduli"
%%% End: 

\section{Surgery on morphisms}
\label{sec:surgery-morphisms}

In this section we wish to study the filtration
$$\mathcal{C}_{\theta, L}^\kappa(\bR^N) \subset \cdots \subset
\mathcal{C}_{\theta, L}^1(\bR^N) \subset \mathcal{C}_{\theta,
  L}^0(\bR^N) \subset \mathcal{C}_{\theta, L}^{-1}(\bR^N) =
\mathcal{C}_{\theta,L}(\R^N)$$ and in particular establish the
following theorem.  The reader mainly interested in
Theorems~\ref{thmcor:rational-coho} and \ref{thm:main-A} can take
$d=2n$, $\theta = \theta^n: BO(2n)\langle n\rangle \to BO(2n)$, $L
\cong D^{2n-1}$, and $N = \infty$ (but the proof does not simplify
much in this special case).

\begin{theorem}\label{thm:kappafiltration}
Suppose that the following conditions are satisfied
\begin{enumerate}[(i)]
\item\label{item:10} $2\kappa \leq d-2$,

\item $\kappa+1+d < N$,

\item $L$ admits a handle decomposition only using handles of index $< d-\kappa-1$.
\end{enumerate}
Then the map
$$B\mathcal{C}_{\theta, L}^\kappa(\bR^N) \lra B\mathcal{C}_{\theta, L}^{\kappa-1}(\bR^N)$$
is a weak homotopy equivalence.
\end{theorem}

The proof of Theorem \ref{thm:kappafiltration} consists of performing
surgery on morphisms, in order to make them more highly connected
relative to their outgoing boundary.  Making this idea into a proof
has two main ingredients.  Firstly, we construct for each morphism in
$\mathcal{C}_{\theta,L}^{\kappa-1}$ a contractible space of surgery
data.  The space is defined in Definition~\ref{defn:ZComplex}, and the
precise statement is Theorem~\ref{thm:SurgeryComplexMor}.  Secondly,
we implement the surgery described by the surgery data, using a
standard one-parameter family of manifolds defined in
Section~\ref{sec:standard-family}.

In order to motivate some of the more technical constructions, let us
first give an informal account of this technique.  For simplicity, we
suppose that $N=\infty$, that we have no tangential structure, that $L
= \emptyset$, and that $\kappa=0$. We first apply the equivalence
(\ref{eq:PosetModel}) to reduce the problem to studying the map
$$BD^0 \lra BD^{-1}$$
of classifying spaces of posets. Let
$$\sigma = (t_0, t_1;a_0, a_1;\epsilon_0, \epsilon_1;W) \in BD^{-1}$$
be a point on a 1-simplex (for example), where $(t_0,t_1) \in \Delta^1$ are the barycentric coordinates. We will describe a way of producing a path from its image in $\vert X^{-1}_\bullet \vert$ into the subspace $\vert X_\bullet^{0} \vert$. The proof of Theorem \ref{thm:kappafiltration} will be a systematic, parametrised version of this construction.

If the cobordism $W\vert_{[a_0, a_1]}$ is already $0$-connected relative to its outgoing boundary, then the image of $\sigma$ in $\vert X^{-1}_\bullet \vert$ already lies in the subspace $\vert X_\bullet^{0} \vert$, and we are done. If not, we may choose distinct points
$$\{f_\alpha : * \to W \vert_{[a_0, a_1]}\}_{\alpha \in
  \Lambda}$$
such that the pair $(W \vert_{[a_0, a_1]}, (W \vert_{a_1}) \cup \bigcup_\alpha f_\alpha(*))$ is $0$-connected. We then choose tubular neighbourhoods of these points to obtain codimension 0 embeddings $\hat{f}_\alpha : D^d \to W \vert_{[a_0, a_1]}$, which we can extend to an embedding
$$e_\alpha : (S^0, \{+1\}) \times D^d \lra ([a_0, \infty) \times \bR^\infty, [a_1+\epsilon_1, \infty) \times \bR^\infty).$$
As the original points $f_\alpha(*)$ were distinct, we may suppose the embeddings $e_\alpha$ are disjoint. Now on each $e_\alpha(S^0 \times D^{d})$ we do the surgery move shown in Figure \ref{fig:0-surgeryMorphismsBasic}, a move similar in spirit, though much simpler, than that described in \cite[\S 6.2]{GMTW}. 
\begin{figure}[htb]
\includegraphics[bb=0 0 337 159]{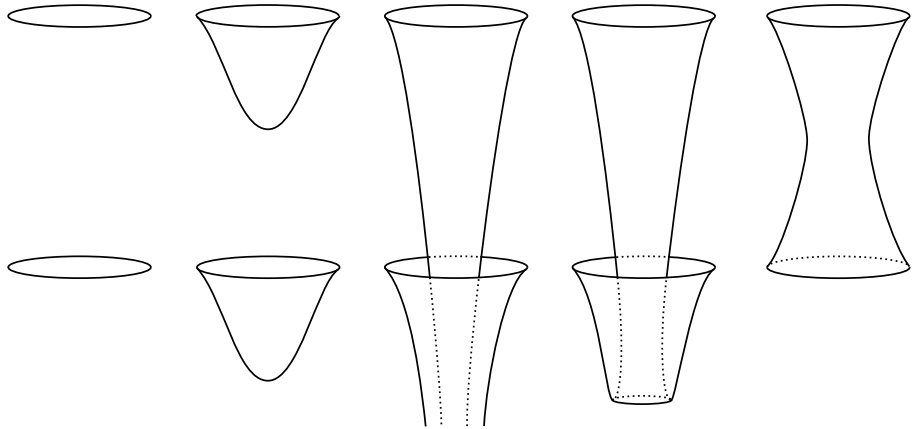}
\caption{The basic move for surgery on morphisms.}
\label{fig:0-surgeryMorphismsBasic}
\end{figure}

More precisely, Figure \ref{fig:0-surgeryMorphismsBasic} describes a
continuous 1-parameter family of $d$-manifolds $\mathcal{P}_t$, $t \in
[0,1]$, depicted (for $d=2$) by its values at times $t= 0, \frac14,
\frac24, \frac34, 1$.  The family comes equipped with functions to
$\R$, depicted in the figure as the height function.  The family
starts at the manifold $\mathcal{P}_0 = S^0 \times D^d$, and we may
cut out each $e_\alpha(S^0 \times D^{d})$ from $W$ and glue in
$\mathcal{P}_t$, to obtain a 1-parameter family of manifolds $W_t$,
each equipped with a height function $W_t \to \R$, with $W_0 = W$.
The values $\{a_0, a_1\}$ do not remain regular throughout this move,
so this does not describe a path in the space $BD^{-1}$. However, the
intervals $(a_i-\epsilon_i, a_i+\epsilon_i)$ do only contain isolated
critical values, so it does describe a path in the space $\vert
X_\bullet^{-1} \vert$. Furthermore, at the end of the move we obtain a
manifold $W_1 = \overline{W}$ such that $(\overline{W} \vert_{[a_0,
  a_1]}, \overline{W} \vert_{a_1})$ is $0$-connected, and hence a
point in $\vert X_\bullet^0 \vert$.  By Proposition~\ref{prop:X}, this
proves that $\pi_0(BD^0) \to \pi_0(BD^{-1})$ is surjective, as
required.

This surgery move generalises easily to the case when $N$ is finite
(but large enough), $L \neq \emptyset$, and $\kappa > 0$ (the analogue
of the surgery move will start with $S^\kappa \times
D^{d-\kappa}$). However, it does not generalise well to the case of
arbitrary tangential structures (to understand how it can fail, we
suggest that the reader attempt to impose a family of framings to the
family of 2-manifolds in Figure
\ref{fig:0-surgeryMorphismsBasic}). One way to fix this would be to
use the surgery move described in \cite[\S 6.2]{GMTW}, but that does
not seem to generalise to $\kappa > 0$. Instead we modify the surgery
move in Figure~\ref{fig:0-surgeryMorphismsBasic} as shown in Figure
\ref{fig:0-surgeryMorphisms}.
\begin{figure}[htb]
\includegraphics[bb=0 0 338 168]{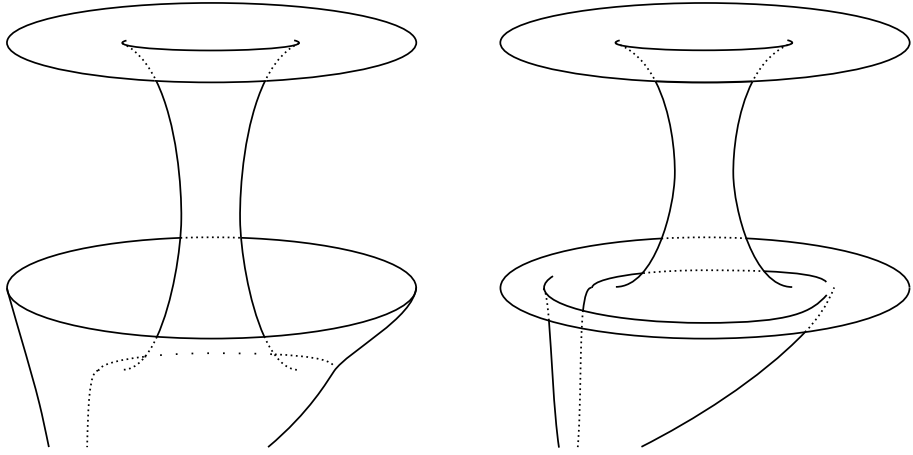}
\caption{In the refined move for surgery on morphisms, these pictures
  replace the last two frames in
  Figure~\ref{fig:0-surgeryMorphismsBasic}.}
\label{fig:0-surgeryMorphisms}
\end{figure}
As we shall see (in the proof of Proposition \ref{prop:StdFamilyMorphisms}, where we
also explain the analogous process for $\kappa > 0$) there is a
canonical way of extending any tangential structure on $\{-1\} \times
D^d$ to the resulting 1-parameter family of manifolds.

\subsection{Surgery data}\label{sec:SurgeryDataMor}

In order to implement the ideas discussed above, we will fatten the semi-simplicial space
$D^\kappa_{\theta,L}(\R^N)_\bullet$ up to a bi-semi-simplicial space
$D^\kappa_{\theta,L}(\R^N)_{\bullet, \bullet}$ which includes suitable
surgery data.  The
space $D^\kappa_{\theta,L}(\R^N)_{\bullet,\bullet}$ is described in
Definition \ref{defn:ZDComplex} below, using the following notation.
Let $V \subset \overline{V} \subset \bR^{\kappa+1} \times \bR^{d-\kappa}$ be the
subspaces
\begin{equation*}
  V = (-2,0) \times \R^d \quad\quad \overline{V} = [-2,0] \times \R^d
\end{equation*}
and let $h : \overline{V} \to [-2,0] \subset \bR$ denote projection to
the first coordinate, which we call the \emph{height
  function}. Let $\partial_- D^{\kappa+1} \subset \partial D^{\kappa+1}$ denote the
lower hemisphere (i.e.\ $\partial_- D^{\kappa+1} = \partial
D^{\kappa+1} \cap ([-1,0] \times \bR^{\kappa})$). We shall also use the notation
$[p]^\vee = \Delta([p],[1])$ when $[p] \in \Delta_\mathrm{inj}$.  The
elements of $[p]^\vee$ are in bijection with $\{0,\dots, p+1\}$, using
the convention that $\phi: [p] \to [1]$ corresponds to the number $i$
with $\phi^{-1}(1) = \{i,i+1, \dots, p\}$.
Finally, we fix once and for all an uncountable set $\Omega$.

\begin{definition}\label{defn:ZComplex}
  Let $x=(a, \epsilon, (W,\ell_W)) \in {D}^{\kappa-1}_{\theta,
    L}(\bR^N)_p$ and define $Z_q(x)$ to be the set of triples
  $(\Lambda, \delta, e)$, where $\Lambda \subset \Omega$ is a finite
  set, $\delta: \Lambda \to [p]^\vee \times [q]$ is a function, and
  \begin{equation*}
    e: \Lambda \times \overline{V} \hookrightarrow \R \times (0,1)
    \times (-1,1)^{N-1}
  \end{equation*}
  is an embedding, satisfying the conditions below.  We shall write
  $\Lambda_{i,j} = \delta^{-1}(i,j)$, $e_{i,j} = e\vert_{\Lambda_{i,j}
    \times \overline{V}}$ and $D_{i,j} = e_{i,j}\left(\Lambda_{i,j}
    \times \partial_-D^{\kappa+1} \times \{0\}\right)$ for
  $0 \leq i \leq p+1$ and $0 \leq j \leq q$.
  \begin{enumerate}[(i)]
  \item\label{it:ZHeightFn} On each subset $(x_1 \circ
    e\vert_{\{\lambda\} \times \overline{V}})^{-1}(a_k-\epsilon_k,
    a_k+\epsilon_k) \subset \{\lambda\} \times \overline{V}$, the
    height function $x_1 \circ e$ coincides with the height function
    $h$ up to an affine transformation.
    
  \item\label{it:ZInfty} $e$ sends $\Lambda \times h^{-1}(0)$ into
    $x_1^{-1}(a_p+\epsilon_p, \infty)$.
    
  \item For $i > 0$, $e$ sends $\Lambda_{i,j} \times h^{-1}(-3/2)$
    into $x_1^{-1}(a_{i-1}+\epsilon_{i-1}, \infty)$.
    
  \item\label{it:Z4} $e$ sends $\Lambda \times h^{-1}(-2)$ into $x_1^{-1}(-\infty, a_0-\epsilon_0)$.
  
  \item\label{it:ZIntersection} $e^{-1}(W) = \Lambda
    \times \partial_-D^{\kappa+1} \times \bR^{d-\kappa}$.

  \item For each $j$ and each $i \in \{1, \ldots, p\}$, the pair
    $$\left(W\vert_{[a_{i-1},a_i]}, W\vert_{a_i} \cup D_{i,j}\vert_{[a_{i-1}, a_i]}\right)$$
    is $\kappa$-connected.\label{it:EnoughSurgeryData}
  \end{enumerate}
  For each $x$, $Z_\bullet(x)$ is a semi-simplicial set: Given an
  injective map $k:[q] \to [q']$, we replace $\Lambda$ by the subset
  $\delta^{-1}([p]^\vee \times \IM(k))$, compose $\delta$ with
  $[p]^\vee \times k^{-1}$, and restrict $e$.  Explicitly, the face
  map $d_j$ forgets the embeddings $e_{*,j}$
\end{definition}

Note that the set $Z_q(x)$ consists of those $(q+1)$-tuples of
elements of $Z_0(x)$ which are disjoint.

\begin{definition}\label{defn:ZDComplex}
  We define a bi-semi-simplicial space ${D}^{\kappa}_{\theta,
    L}(\bR^N)_{\bullet,
    \bullet}$ as a set by
  \begin{equation*}
    {D}^{\kappa}_{\theta, L}(\bR^N)_{p,q} = \big\{(x,y)
    \,\big\vert\, x \in
    {D}_{\theta, L}^{\kappa-1}(\bR^N)_p, y \in Z_q(x)\big\}
  \end{equation*}
  topologised as a subspace of
  $${D}_{\theta, L}^{\kappa-1}(\bR^N)_p \times \left ( \coprod_{\Lambda
      \subset \Omega} C^\infty(\Lambda \times \overline{V}, \bR^{N+1})
  \right )^{(p+2)(q+1)}.$$ The space ${D}^{\kappa}_{\theta,
    L}(\bR^N)_{p,q}$ is functorial in $[p] \in \Delta_\mathrm{inj}$ by
  composing $\delta: \Lambda \to [p]^\vee \times [q]$ with the induced
  map $[p']^\vee \to [p]^\vee$ and functorial in $[q] \in
  \Delta_\mathrm{inj}$ in the same way as in
  Definition~\ref{defn:ZComplex}.  Explicitly, the face map $d_i$ in
  the $q$ direction forgets the embeddings $e_{*, i}$ and in the $p$
  direction takes the union of $e_{i,*}$ and $e_{i+1,*}$.  We shall
  write $D^\kappa_{\theta,L}(\R^N)_{p,-1} =
  D^{\kappa-1}_{\theta,L}(\R^N)_{p}$, and there is an augmentation map
  $D^\kappa_{\theta,L}(\R^N)_{p,q} \to
  D^\kappa_{\theta,L}(\R^N)_{p,-1}$ which forgets all surgery data.
\end{definition}

The main result concerning this bi-semi-simplicial space is the
following, whose proof we defer until Section \ref{sec:Connectivity}.

\begin{theorem}\label{thm:SurgeryComplexMor}
  Under the assumptions of Theorem \ref{thm:kappafiltration}, the
  augmentation map
  \begin{equation*}
    D^\kappa_{\theta, L}(\bR^N)_{\bullet, \bullet} \lra
    D^{\kappa-1}_{\theta, L}(\bR^N)_{\bullet}
  \end{equation*}
  induces a weak homotopy equivalence after geometric
  realisation.
\end{theorem}

In fact, we shall prove this theorem with condition (\ref{item:10}) of
Theorem~\ref{thm:kappafiltration} replaced by the weaker condition
$2\kappa \leq d-1$.  The stronger assumption $2\kappa \leq d-2$ will
be used in Lemma~\ref{lem:MorSurgeryDesiredEffect}.

\subsection{The standard family}
\label{sec:standard-family}

We will now construct a one-parameter family of submanifolds of $V =
(-2,0) \times \R^d$ which formalises the family of manifolds depicted
in Figures \ref{fig:0-surgeryMorphismsBasic} and \ref{fig:0-surgeryMorphisms}. Let us write coordinates in
$\bR^{\kappa+1} \times \bR^{d-\kappa}$ as $(u,v)$.  First define an
element $\widetilde{\mathcal{P}}_0 \in \Psi_d(\R \times \R^\kappa
\times \R^{d-\kappa})$ as
\begin{equation*}
  \widetilde{\mathcal{P}}_0 = \partial D^{\kappa+1} \times
  \R^{d-\kappa}.
\end{equation*}
Choose a function $\varphi : [0,\infty) \to [0,\infty)$ that is the
identity function on a neighbourhood of $[1/2, \infty)$, takes value
$1/4$ near 0, and has $\varphi'' \geq 0$. We then define an embedding
by
\begin{eqnarray*}
    g' : \bR^{\kappa+1} \times \partial D^{d-\kappa} & \lra &
    D^{\kappa+1} \times \bR^{d-\kappa} \\
(u,v) & \longmapsto & (u/\varphi(|u|), \varphi(|u|) \cdot v),
\end{eqnarray*}
and another embedding $g:
\R^{\kappa+1} \times \partial D^{d-\kappa} \to [-2,1] \times \R^\kappa
\times \R^{d-\kappa}$ by
\begin{equation*}
  g(u,v) = g'(u,v) + \tau(u)\left(\frac{v_1 - 1}2,0,0\right)
\end{equation*}
where $\tau: \R^{\kappa+1} \to [0,1]$ is a bump function supported in a
small neighbourhood of the point $u_0= (-1/2,0) \in \R \times
\R^\kappa$, having $\tau(u_0) = 1$, $\tau(u) < 1$ otherwise, and no
critical points in $\tau^{-1}((0,1))$.  We can arrange that the
support of $\tau$ be small enough that it is contained in the region
where $\varphi(|u|) = |u|$.  We let $\widetilde{\mathcal{P}}_1\subset
\R^{d+1}$ denote the image of $g$ and $\widetilde{\mathcal{P}}'_1$ be
the image of $g'$. We then define
$$\mathcal{P}_0 ,\mathcal{P}_1 \in \Psi_d(V)$$
by intersecting the manifolds $\widetilde{\mathcal{P}}_0,
\widetilde{\mathcal{P}}_1$ with the open set $V = (-2,0) \times \R^d$.

To construct $\mathcal{P}_t \in \Psi_d(V)$ for intermediate values of
$t \in [0,1]$, we first observe that $\widetilde{\mathcal{P}}_0$ and
$\widetilde{\mathcal{P}}'_1$ agree on the subset $\vert v\vert \geq
1/2$ and that $\widetilde{\mathcal{P}}_1$ agrees with them on the
smaller subset $\vert v\vert \geq 1$ (when the support of the bump
function $\tau$ is sufficiently small).  Starting with the two
submanifolds $\widetilde{\mathcal{P}}_0$ and
$\widetilde{\mathcal{P}}_1 \subset \R \times \R^\kappa \times
\R^{d-\kappa}$, we then pull the region $\{(u,v) \;\vert \;\vert v
\vert < 1\}$ downwards by decreasing the first coordinate
in $\R \times \R^d$, until the region where the submanifolds may
disagree is moved completely outside of $V$.  This gives two
one-parameter families of submanifolds which, upon restricting to $V$,
give two paths in $\Psi_d(V)$ starting at $\mathcal{P}_0$ and
$\mathcal{P}_1$ and ending at the same point in $\Psi_d(V)$.
Concatenating one path with the reverse of the other, we get the
desired path from $\mathcal{P}_0$ to $\mathcal{P}_1$.

Spelling this process out in a little more detail, we first choose a
function $\rho : [0,\infty) \to [0,\infty)$ taking the value $1$ near
$[0, 1]$, the value $0$ near $[2, \infty)$, and which is
strictly decreasing on $\rho^{-1}(0,1)$.  We then define embeddings
\begin{align*}
  H_t : \bR \times \bR^{\kappa} \times \bR^{d-\kappa} & \lra  \bR \times \bR^{\kappa} \times \bR^{d-\kappa}\\
  (s, x, y) & \longmapsto (s - t \cdot \rho(\vert y \vert), x,y)
\end{align*}
which for all $t$ restrict to the identity for $\vert y \vert \geq
2$.  Define one-parameter families of manifolds by
\begin{align*}
  \mathcal{P}^0_t &= V \cap H_t(\widetilde{\mathcal{P}}_0) =
  (H_{-t}\vert_V)^{-1}(\widetilde{\mathcal{P}}_0)\\
  \mathcal{P}^1_t &= V \cap H_t(\widetilde{\mathcal{P}}_1) =
  (H_{-t}\vert_V)^{-1}(\widetilde{\mathcal{P}}_1).
\end{align*}
The second description shows that these are closed subsets of $V$ and
describe continuous functions $\R \to \Psi_d(V)$.  It is easy to see
that we have $\mathcal{P}^0_t = \mathcal{P}^1_t \in \Psi_d(V)$ for $t
\geq 3$, and we then define the path $\mathcal{P}_t$ as the
concatenation
$$\mathcal{P}_0 = \mathcal{P}^0_0 \leadsto \mathcal{P}^0_{3} = \mathcal{P}^1_{3}
\leadsto \mathcal{P}^1_0 = \mathcal{P}_1$$ in $\Psi_d(V)$,
reparametrised so that the path has length 1.  We collect the most
important properties of this family in
Proposition~\ref{prop:StdFamilyMorphisms} below.  The following remark
partially explains how it relates to an ordinary $\kappa$-surgery.
\begin{remark}\label{remark:graph-of-morphisms-surgery}
  Let $Q(u,v) = - |u|^2 + |v|^2$, where as usual $(u,v) \in
  \R^{\kappa+1} \times \R^{d-\kappa}$.  For $t \in [0,3]$ the function
  $(u,v) \mapsto H_t({u}/{|u|},v)$ defines a diffeomorphism
  to $\mathcal{P}^0_t$ from an open subset of $Q^{-1}(t-3)$ (namely
  the inverse image of $V$ by that function) and similarly the
  function $(u,v) \mapsto H_t \circ g(u,{v}/{|v|})$ defines a
  diffeomorphism to $\mathcal{P}^1_t$ from an open subset of
  $Q^{-1}(3-t)$.  The inverses of these diffeomorphisms give smooth
  embeddings $\mathcal{P}^0_t \to Q^{-1}(t-3)$ and $\mathcal{P}^1_t
  \to Q^{-1}(3-t)$ and it is easy to verify that for $t=3$ the two
  resulting embedings $\mathcal{P}^0_3 = \mathcal{P}^1_3 \to
  Q^{-1}(0)$ agree, and so glue to a continuous family of embeddings
  $\mathcal{P}_t \to Q^{-1}(6t-3)$.

  The continuous map $t \mapsto \mathcal{P}_t$ has \emph{graph} given
  by $\mathcal{P} = \{(t,x) \in [0,1] \times V \,\,|\,\, x \in
  \mathcal{P}_t\}$.  The above remarks give an embedding $\mathcal{P}
  \to Q^{-1}([-3,3])$ and it is easy to verify that the image is
  disjoint from the straight lines from $0$ to $p_0 = (-1/2,0,-1/2,0)
  \in \R \times \R^\kappa \times \R \times \R^{d-\kappa-1}$ and from
  $p_0$ to $p_1 = (-1/2,0,-\sqrt{13}/2,0)$.  Thus we get a
  diffeomorphism from $\mathcal{P}$ to an open subset of the
  contractible set
  \begin{equation*}
    \mathcal{Q} = Q^{-1}([-3,3]) - \big([0,p_0] \cup [p_0,p_1]\big).
  \end{equation*}
\end{remark}

\begin{proposition}\label{prop:StdFamilyMorphisms}
  For $2\kappa \leq d-1$, the 1-parameter family $\mathcal{P}_t \in
  \Psi_d(V)$, defined for $t \in [0,1]$, has the following properties.
  \begin{enumerate}[(i)]
  \item The height function, i.e.\ the restriction of $h: V \to
    (-2,0)$ to $\mathcal{P}_t \subset V$, has isolated critical
    values\label{it:IsolatedCriticalValues}.
    
  \item $\mathcal{P}_0 = \Int(\partial_- D^{\kappa+1}) \times
    \bR^{d-\kappa}$, where $\partial_- D^{\kappa+1} = \partial
    D^{\kappa+1} \cap ([-1,0] \times \R^\kappa)$.\label{it:InitialValue}
    
  \item Independently of $t \in [0,1]$ we have
    \begin{equation*}
      \mathcal{P}_t - (\R^{\kappa+1} \times
      B^{d-\kappa}_{3}(0)) = \Int(\partial_- D^{\kappa+1}) \times
      (\R^{d-\kappa} - B^{d-\kappa}_{3}(0)).
    \end{equation*}
    For ease of notation we write $\mathcal{P}^\partial_t$ for this
    closed subset of $\mathcal{P}_t$.\label{it:ConstantOnBoundary}
    
  \item For all $t$ and each pair of regular values $-2 < a < b < 0$
    of the height function, the pair
    \begin{equation}
      \label{eq:17}
      (\mathcal{P}_t \vert_{[a,b]}, \mathcal{P}_t\vert_b \cup
      \mathcal{P}^\partial_t\vert_{[a,b]})
    \end{equation}
    is $\kappa$-connected.
    
    \label{it:kConnMoving}
  \item For each pair of regular values $-2 < a < b < 0$ of the height
    function, the pair
    $$(\mathcal{P}_1 \vert_{[a, b]}, \mathcal{P}_1 \vert_{b})$$
    is $\kappa$-connected.\label{it:kConn}
\end{enumerate}
Furthermore, if $\mathcal{P}_0$ is
equipped with a $\theta$-structure $\ell$ we can upgrade this,
continuously in $\ell$, to a 1-parameter family $\mathcal{P}_t(\ell)
\in \Psi_\theta(V)$ starting from $(\mathcal{P}_0, \ell)$ such that
\begin{enumerate}[(i$^\prime$)]
\setcounter{enumi}{2}
\item The path $\mathcal{P}_t(\ell)$ is constant as $\theta$-manifolds
  near $\mathcal{P}^\partial_t$.
\end{enumerate}
\end{proposition}
\begin{proof}
  We have seen properties
  (\ref{it:IsolatedCriticalValues})--(\ref{it:ConstantOnBoundary})
  during the construction (the statement
  in~(\ref{it:ConstantOnBoundary}) would still be true with $3$
  replaced by $2$, but we wish to emphasise the smaller
  set).  For property (\ref{it:kConnMoving}) we consider two
  cases depending on the value of $a$.  In the case $a > -1$, the
  pair~\eqref{eq:17} is homotopy equivalent to the pair
  \begin{equation*}
    (\mathcal{P}_t\vert_{[a,b]}, \mathcal{P}_t \vert_b),
  \end{equation*}
  using e.g.\ the gradient flow trajectories of $h$ to deform
  $\mathcal{P}^\partial_t\vert_{[a,b]}$ back to
  $\mathcal{P}^\partial_t\vert_{b}$.  In the case $a < -1$ we consider
  the \emph{modified height function}, defined using the coordinates
  $(u,v) \in \R^{\kappa + 1} \times \R^{d-\kappa}$ as
  $\overline{h}(u,v) = h(u,v) + \lambda(|v|)$, where $\lambda:
  [0,\infty) \to [0,\infty)$ is a smooth function which is 0 on
  $[0,4]$ and restricts to a diffeomorphism $(4,\infty) \to
  (0,\infty)$.  We claim that the inclusion of pairs
  \begin{equation}\label{eq:22}
    (\mathcal{P}_t \cap \overline{h}^{-1}([a,b]), \mathcal{P}_t
    \cap \overline{h}^{-1}(b)) \lra 
    (\mathcal{P}_t\vert_{[a,b]}, \mathcal{P}_t\vert_{b}
    \cup \mathcal{P}^\partial_t\vert_{[a,b]})
  \end{equation}
  is a homotopy equivalence.  To define a homotopy inverse, we first
  consider the continuous, piecewise smooth function $\rho_t:
  [0,\infty) \to (0,\infty)$ defined for $t \leq b$ by
  \begin{align*}
    \rho_t(s) &= 1 & &\text{for $s \in [0,2]$},\\
    \rho_t(s) &= \frac{\lambda^{-1}(b-t)}s& & \text{for $s \in [3,\infty)$},
  \end{align*}
  and by linear interpolation for $s \in [2,3]$.  Then the
  function $(u,v) \mapsto (u,v \cdot\rho_{u_1}(|v|))$ restricts to a
  homotopy inverse of~\eqref{eq:22}, where both homotopies are given
  by straight lines in $\R^{d+1}$.

  In either case, the connectivity question is reduced to studying the
  inverse image of an interval relative to its outgoing boundary and
  can be studied as in ordinary Morse theory one critical level at a
  time.  The proof of~(\ref{it:kConnMoving}) will be finished once we
  establish that for each critical value of $\overline{h}:
  \mathcal{P}_t \to \R$ in the interval $(a,b)$, the function can be
  perturbed in a neighbourhood of the critical set contained in
  $\overline{h}^{-1}((a,b))$ to a Morse function with at most a
  critical point of index $\leq d-\kappa -1$.  (In the case $a > -1$
  we have $h = \overline{h}$ near any critical point of $h$, so it
  suffices to consider $\overline{h}$.)  It is easy to verify that
  $\overline{h}:\mathcal{P}^0_t \to \R$ has at most two critical
  values in $(-2,0)$.  One critical value moves with $t$ and is
  homotopically Morse of index $0$ for $0 \leq t < 1$ and index
  $\kappa$ for $1 < t < 3$ (meaning that the function can be perturbed
  to a Morse function with one critical point of that index).  The
  other is at $-1$ and can be cancelled (meaning that the function can
  be perturbed to a non-singular function there).  Since $2\kappa \leq
  d-1$ and hence $\kappa \leq d-\kappa -1$, the index is at most
  $d-\kappa-1$ as claimed.  Similarly, one verifies that
  $\overline{h}: \mathcal{P}^1_t \to \R$ has at most two critical
  values in $(-2,0)$, one of which is $-1$ and can be cancelled, the
  other of which moves with $t$ and is homotopically Morse of index
  $d-\kappa-1$.

  Property (\ref{it:kConn}) can be proved in a similar way. In the case $a < -1 < b$ the pair is a relative $(d-1)$-cell, so it
  is $(d-2)$-connected and hence $\kappa$-connected (since $d \geq 2$
  and $2\kappa \leq d-1$).  In all other cases the inclusion
  $\mathcal{P}_t \vert_b \to \mathcal{P}_t\vert_{[a,b]}$ is a homotopy
  equivalence.

  To establish the extra properties which can be obtained given a
  $\theta$-structure $\ell$ on $\mathcal{P}_0 = \Int(\partial_-
  D^{\kappa+1}) \times \bR^{d-\kappa}$, we again use the graph
  $\mathcal{P} = \{(t,x) \in [0,1] \times \R^{d+1} \,\,|\,\, x \in
  \mathcal{P}_t\}$ and its identification with an open subset of the
  manifold $\mathcal{Q}$ from
  Remark~\ref{remark:graph-of-morphisms-surgery}.  The tangent bundles
  $T\mathcal{P}_t$ assemble to a $d$-dimensional vector bundle
  $T_v\mathcal{P} \to \mathcal{P}$ which then becomes identified with
  the restriction of the vector bundle $T_v \mathcal{Q} = \Ker(DQ:
  T\mathcal{Q} \to T[-3,3])$ and since both $\mathcal{P}_0$ and
  $\mathcal{Q}$ are contractible, there is no obstruction to picking a
  vector bundle map $r: T_v \mathcal{Q} \to T\mathcal{P}_0$ which is
  the identity (with respect to the identifications) over
  $\mathcal{P}_0$ and each $\mathcal{P}^\partial_t =
  \mathcal{P}^\partial_0 \subset \mathcal{P}_0$.  We can then restrict
  $r$ to $r_t: T \mathcal{P}_t \to T\mathcal{P}_0$ and let
  $\mathcal{P}_t(\ell)$ have the $\theta$-structure $\ell \circ r_t$.
\end{proof}

Let $(a, \epsilon, (W, \ell_W), e) \in D^\kappa_{\theta, L}(\bR^N)_{p,
  0}$, with $e = \{e_{i,0}\}_{i=0}^{p+1}$  (where we omit $\Lambda$ and
$\delta: \Lambda \to [p]^\vee$ from the notation).  We construct a
1-parameter family of $\theta$-manifolds
\begin{equation*}
  \mathcal{K}^t_{e}(W, \ell_W) \in
  \Psi_\theta((a_0-\epsilon_0, a_p+\epsilon_p) \times \R^N),
\end{equation*}
$t \in [0,1]$, by letting it be equal to $W\vert_{(a_0-\epsilon_0,
  a_p+\epsilon_p)}$ outside of the images of the
$e_{i,0}\vert_{\Lambda_{i,j} \times V}$, and on each
$e_{i,0}(\{\lambda\} \times V)$ we let it be given by
$e_{i,0}(\{\lambda \} \times \mathcal{P}_t(\ell_W \circ D e_{i,0}))$. This gives a
$\theta$-manifold, as by the properties established above,
$\mathcal{P}_t(\ell_W \circ D e_{i,0})$ and $\mathcal{P}_0(\ell_W \circ D
e_{i,0})$ agree as $\theta$-manifolds near the set $(-2,0) \times
\bR^\kappa \times (\bR^{d-\kappa} - B^{d-\kappa}_{3}(0))$.

\begin{lemma}\label{lem:MorSurgeryDesiredEffect}
  Let $2\kappa \leq d-2$. The tuple $(a, \epsilon,
  \mathcal{K}^t_{e}(W, \ell_W))$ is an element of
  $X^{\kappa-1}_p$.  If either $t=1$ or $(W, \ell_W) \in
  D^\kappa_{\theta, L}(\bR^N)_p$, then $(a,\epsilon,
  \mathcal{K}^t_{e}(W, \ell_W))$ lies in the subspace
  $X^\kappa_p \subset X^{\kappa-1}_p$.
\end{lemma}
\begin{proof}
  We must verify conditions (\ref{item:1})--(\ref{it:ObConn}) of
  Definition \ref{defn:X}. Condition (\ref{item:1}) is true by
  definition, and certainly (\ref{item:2}) is satisfied as the
  embeddings $e_{i,0}$ are disjoint from $\bR \times L$. For
  (\ref{it:Disjoint}) and (\ref{it:ObConn}) there is nothing to say.

  For (\ref{it:MorConn}), consider regular values $a < b \in
  \cup_i(a_i-\epsilon_i, a_i+\epsilon_i)$ of the height function $x_1
  : {W}_t = \mathcal{K}^t_{e}(W, \ell_W) \to \bR$. The cobordism
  ${W}_t \vert_{[a,b]}$ is obtained from $W \vert_{[a,b]}$ by cutting
  out embedded images of cobordisms $\mathcal{P}_0
  \vert_{[a_{\lambda},b_{\lambda}]}$ indexed by $\lambda \in \Lambda =
  \amalg_i \Lambda_{i,0}$ and gluing in $\mathcal{P}_t
  \vert_{[a_{\lambda},b_{\lambda}]}$, where $a_{\lambda} <
  b_{\lambda}$ are regular values of the height function on
  $\mathcal{P}_0$ and $\mathcal{P}_t$. If we denote by $X$ the
  complement of the embedded $e_{i,0}(\Int(\partial_- D^{\kappa+1})
  \times B^{d-\kappa}_{3}(0))$ in the manifold $W\vert_{[a,b]}$, there
  are homotopy push-out squares
  \begin{equation*}
    \xymatrix{
      X\vert_{b} \ar[r] \ar[d]& W_t\vert_{b} \ar[d]\\
      X \ar[r] & W_t\vert_{b} \cup X
    }
  \end{equation*}
  and
  \begin{equation*}
    \xymatrix{
      {\coprod_{\lambda \in \Lambda} \mathcal{P}_t\vert_{b_\lambda}
        \cup \left(\mathcal{P}^\partial_t \vert_{[a_\lambda,
            b_\lambda]}\right)} \ar[r] \ar[d] & W_t\vert_{b} \cup X
      \ar[d]\\
      {\coprod_{\lambda \in \Lambda} \mathcal{P}_t
        \vert_{[a_{\lambda},b_{\lambda}]}}  \ar[r] &
      W_t\vert_{[a,b]}.}
  \end{equation*}
  The left hand map of the second square is a disjoint union of the
  maps discussed in property (\ref{it:kConnMoving}) of
  Proposition~\ref{prop:StdFamilyMorphisms}, so is $\kappa$-connected.
  As this square is a homotopy push-out, the right hand map is also
  $\kappa$-connected.

The pair $(X, X\vert_b)$ is obtained from the manifold pair
$(W\vert_{[a,b]}, W\vert_b)$ by cutting out embedded copies of
$(D^\kappa, \partial D^\kappa)$. By transversality we see that this
does not change relative homotopy groups in dimensions $* \leq
d-\kappa-2$, which includes $* \leq \kappa$ by our assumption that
$2\kappa \leq d-2$. In particular, suppose the pair $(W\vert_{[a,b]},
W\vert_{b})$ is $k$-connected, with $k \leq \kappa$, then the pair
$(X, X\vert_b)$ is $k$-connected too. As the first square above is a
homotopy push-out square, the inclusion $W_t\vert_b \to W_t\vert_b
\cup X$ also has this connectivity.

Hence the composition $W_t\vert_b \to W_t\vert_b \cup X \to
W_t\vert_{[a,b]}$ has the same connectivity as $W\vert_b \to
W\vert_{[a,b]}$, up to a maximum of $\kappa$. This establishes that
the tuple $(a, \epsilon, \mathcal{K}^t_{e}(W, \ell_W))$ is an element of
$X^{\kappa-1}_p$, and also that it lies in $X^\kappa_p$ if $(W, \ell_W)$
lies in $D^\kappa_{\theta, L}(\bR^N)$. When $t=1$, there is a little
more to say.

\noindent\textit{Step 1.} Suppose $a < b \in (a_i-\epsilon_i,
a_i+\epsilon_i)$. Then $(W\vert_{[a,b]}, W\vert_b)$ is
$\infty$-connected and so $(W_1\vert_{[a,b]}, W_1\vert_b)$ is
$\kappa$-connected, by the discussion above.

\noindent\textit{Step 2.} Suppose $a \in (a_{i-1}-\epsilon_{i-1},
a_{i-1}+\epsilon_{i-1})$ and $b \in (a_i-\epsilon_i,
a_i+\epsilon_i)$. We now do the surgeries for $\Lambda_{i,0} \subset
\Lambda$ first, giving a family of manifolds $\widetilde{W}_t$.
We claim that the pair $(\widetilde{W}_1\vert_{[a,b]},
\widetilde{W}_1\vert_{b})$ is $\kappa$-connected. Once this is
established, doing the remaining surgeries to obtain $W_1$ does not
change this property, as we have seen above.

Recall from Definition~\ref{defn:ZComplex}(\ref{it:EnoughSurgeryData})
that the pair $(W_0\vert_{[a,b]}, (W_0\vert_{b}) \cup
(D_{i,0}\vert_{[a,b]}))$ is $\kappa$-connected, where $D_{i,0} =
e_{i,0}(\Lambda_{i,0} \times \partial_- D^{\kappa+1} \times \{0\})
\subset W = W_0$.  If we write $\tilde D_{i,0} =
e_{i,0}(\Lambda_{i,0} \times \partial_- D^{\kappa+1} \times \{v\}
\subset W = W_0$ for some $v \in \R^{d-\kappa} - B_4^{d-\kappa}(0)$,
then the pair $(W_0\vert_{[a,b]}, (W_0\vert_{b}) \cup (\tilde
D_{i,0}\vert_{[a,b]}))$ is also $\kappa$-connected.  Now the subset
$\tilde D_{i,0} \subset W$ is contained in $e_{i,0}(\Lambda_{i,0}
\times \mathcal{P}_0^\partial)$, so we can regard $\tilde D_{i,0}$ as
a subset of $\widetilde{W}_t$ for all $t \in [0,1]$.  The same transversality
argument as before now shows that $(X,X\vert_b \cup \tilde
D_{i,0}\vert_{[a,b]})$ is also $\kappa$-connected, and the same gluing
argument shows that $(\widetilde{W}_t\vert_{[a,b]}, \widetilde{W}_t\vert_b \cup \tilde
D_{i,0}\vert_{[a,b]})$ is $\kappa$-connected for all $t \in [0,1]$.
When $t=1$, Proposition~\ref{prop:StdFamilyMorphisms}(\ref{it:kConn})
shows that the inclusion $\tilde D_{i,0}\vert_{[a,b]} \to
\widetilde{W}_1\vert_{[a,b]}$ is homotopic relative to $\tilde D_{i,0}\vert_b$ to
a map into $\widetilde{W}_1\vert b$, and hence $(\widetilde{W}_1\vert_{[a,b]}, \widetilde{W}_1\vert_b)$ is
$\kappa$-connected.

\noindent\textit{Step 3.} For general $a < b \in \cup_i(a_i-\epsilon_i,
a_i+\epsilon_i)$, we may choose regular values in each intermediate
interval $(a_j-\epsilon_j, a_j+\epsilon_j)$. By the previous case,
this expresses ${W}_1 \vert_{[a, b]}$ as a composition of cobordisms which are each $\kappa$-connected relative to their outgoing
boundaries, and the hence the composition also has that property.
\end{proof}

\subsection{Proof of Theorem \ref{thm:kappafiltration}}\label{sec:proof-theor-refthm:k}
We begin with the composition
$$\vert D^\kappa_{\theta, L}(\bR^N)_{\bullet, \bullet} \vert \lra \vert D^{\kappa-1}_{\theta, L}(\bR^N)_{\bullet} \vert \lra \vert X_\bullet^{\kappa-1} \vert,$$
where the first map (induced by the augmentation) is a homotopy equivalence by Theorem \ref{thm:SurgeryComplexMor} and the second is a homotopy equivalence by Proposition \ref{prop:X}. We will define a homotopy
$$\surg : [0,1] \times \vert D^\kappa_{\theta, L}(\bR^N)_{\bullet, \bullet} \vert \lra \vert X_\bullet^{\kappa-1} \vert$$
starting from this map so that $\surg(1,-)$ factors through the
continuous injection $\vert X_\bullet^{\kappa} \vert \to \vert
X_\bullet^{\kappa-1}\vert$. Furthermore, there is an inclusion
$$\vert D^{\kappa}_{\theta, L}(\bR^N)_{\bullet} \vert \hookrightarrow \vert D^\kappa_{\theta, L}(\bR^N)_{\bullet, 0} \vert \hookrightarrow \vert D^\kappa_{\theta, L}(\bR^N)_{\bullet, \bullet} \vert$$
as manifolds equipped with no surgery data, and $\surg$ will be
constant on this subspace. The existence of a homotopy with these
properties establishes Theorem \ref{thm:kappafiltration}, as follows:
there is a diagram
\begin{equation*}
\xymatrix{
{\vert D_{\theta, L}^{\kappa}(\bR^N)_{\bullet} \vert} \ar[r] \ar[d] &
\vert X_\bullet^{\kappa} \vert \ar[d]\\
{\vert D_{\theta, L}^{\kappa}(\bR^N)_{\bullet, \bullet} \vert}
\ar[r]^-{\surg(0, -)} \ar[ru]^-{\surg(1,-)} & \vert X_\bullet^{\kappa-1}
\vert
}
\end{equation*}
where the square commutes, the horizontal maps are weak homotopy
equivalences, the top triangle commutes exactly and the bottom
triangle commutes up to the homotopy $\surg$. Taking homotopy groups
we see that the vertical maps are also weak equivalences. Under the
equivalence $B\mathcal{C}_{\theta, L}^{\kappa}(\bR^N) \simeq
|X_\bullet^{\kappa}|$, and similarly for $(\kappa-1)$, we obtain Theorem
\ref{thm:kappafiltration}.

To define the surgery map $\surg$ we will give a collection of
maps
$$\surg_{p,q} : [0,1] \times {D}^{\kappa}_{\theta, L}(\bR^N)_{p,q}
\times \Delta^q \lra X^{\kappa-1}_p$$ compatible on their faces. The
construction of the last section gives a 1-parameter family
\begin{eqnarray*}
  \mathcal{K}^r : D^\kappa_{\theta, L}(\bR^N)_{p, 0} &\lra& X^{\kappa-1}_p\\
  (a, \epsilon, W, e) &\longmapsto& (a, \epsilon,
  \mathcal{K}^{r}_{e}(W)),
\end{eqnarray*}
for $r \in [0,1]$, such that $\mathcal{K}^1$ lands in $X_p^{\kappa}$. When
$q=0$, we set
\begin{equation*}
  \surg_{p,0}(r, (a,\epsilon,W, e)) = (a,
  \epsilon,\mathcal{K}^{r}_{e}(W)) \in X_p^{\kappa-1},
\end{equation*}
More generally, for $q \geq 0$ we have $e = \{e_{i,j}\}$, and for each
$j$ we get an element $(a,\epsilon,W,e_{*,j}) \in D^\kappa_{\theta,
  L}(\bR^N)_{p, 0}$.  We then set
\begin{equation*}
  \surg_{p,q}(r, (a,\epsilon,W, e), s) = (a, \epsilon,
  \mathcal{K}^{\bar{s}_q \cdot r}_{e_{*,q}} \circ \cdots \circ
  \mathcal{K}^{\bar{s}_0 \cdot r}_{e_{*,0}}(W)),
\end{equation*}
where $\bar{s}_j = s_j / \max(s_k)$.  Note that some $\bar{s}_j$ is
always equal to 1, so when $r=1$, some $\mathcal{K}^1_{e_{*,j}}$ is
applied to $W$ making each morphism $\kappa$-connected relative to its
outgoing boundary. The remaining $\mathcal{K}^{\bar{s}_k}_{e_{*, k}}$
do not change this property, by Lemma
\ref{lem:MorSurgeryDesiredEffect}, and so the map $\surg_{p,q}(1,-)$
factors through the subspace $X_p^{\kappa}$.

The resulting map from $\amalg_q [0,1] \times
D_{\theta,L}^\kappa(\R^N)_{p,q} \times \Delta^q$ factors through a map
\begin{equation*}
  \surg_p: [0,1] \times |D_{\theta,L}^\kappa(\R^N)_{p,\bullet}| \to
  X_p^{\kappa-1},
\end{equation*}
which together form a map of semi-simplicial spaces
with geometric realisation $\surg : [0,1] \times \vert
D^\kappa_{\theta, L}(\bR^N)_{\bullet, \bullet} \vert \to \vert
X_\bullet^{\kappa-1} \vert$. On the subspace $\vert
D^{\kappa}_{\theta, L}(\bR^N)_{\bullet} \vert \hookrightarrow \vert
D^\kappa_{\theta, L}(\bR^N)_{\bullet, 0} \vert$, the homotopy is
constant as there is no surgery data.  At $r=1$ it factors through
$\vert X_\bullet^\kappa \vert$.

%%% Local Variables: 
%%% mode: latex
%%% TeX-master: "Moduli"
%%% End: 

\section{Surgery on objects below the middle dimension}
\label{sec:surg-objects-below}

In this section we wish to study the filtration
\begin{equation*}
  \mathcal{C}_{\theta, L}^{\kappa, l}(\bR^N) \subset \cdots \subset
  \mathcal{C}_{\theta, L}^{\kappa,1}(\bR^N) \subset
  \mathcal{C}_{\theta, L}^{\kappa,0}(\bR^N) \subset
  \mathcal{C}^{\kappa, -1}_{\theta, L}(\bR^N) =
  \mathcal{C}^\kappa_{\theta, L}(\bR^N)
\end{equation*}
and in particular establish the following theorem.  The reader mainly
interested in Theorems~\ref{thmcor:rational-coho} and \ref{thm:main-A}
can take $d=2n$, $\kappa = n-1$, $\theta = \theta^n: BO(2n)\langle
n\rangle \to BO(2n)$, $L \cong D^{2n-1}$, and $N = \infty$ (but the
proof does not simplify much in this special case).

\begin{theorem}\label{thm:lfiltration}
Suppose that the following conditions are satisfied
\begin{enumerate}[(i)]
\item $2(l+1) < d$,
\item $l \leq \kappa$,
\item\label{item:12} $l \leq d-\kappa-2$,
\item $l+2 + d < N$,
\item $L$ admits a handle decomposition only using handles of index $< d-l-1$,
\item\label{item:13} the map $\ell_L : L \to B$ is $(l+1)$-connected.
\end{enumerate}
Then the map
$$B\mathcal{C}_{\theta, L}^{\kappa, l}(\bR^N) \lra B\mathcal{C}_{\theta, L}^{\kappa, l-1}(\bR^N)$$
is a weak homotopy equivalence.
\end{theorem}

The proof will be similar in spirit to that of the last section, in so
far as we will define a contractible space of surgery data and
describe a surgery move which compresses $B\mathcal{C}_{\theta,
  L}^{\kappa, l-1}(\bR^N)$ into the subspace $B\mathcal{C}_{\theta,
  L}^{\kappa, l}(\bR^N)$. In the same way that the surgery move of the
last section was a refinement of that of \cite{GMTW}, the surgery move
we use in this and the next section is a refinement of that of
\cite{GR-W}. Let us first give an informal account of this move, and
for simplicity suppose that $N=\infty$, that we have no tangential
structure (i.e. we consider $\theta = \mathrm{id}: BO(d) \to BO(d)$),
that $L=\emptyset$, and that $d>2$, $l=0$ and $\kappa=0$. We first
apply the equivalence (\ref{eq:PosetModel}) to reduce the problem to
studying the map
$$BD^{0, 0} \lra BD^{0, -1}$$
of classifying spaces of posets. Let
$$\sigma = (t_0, t_1;a_0, a_1;\epsilon_0, \epsilon_1;W) \in BD^{0,-1}$$
be a point on a 1-simplex (for example), and let us suppose that $W
\vert_{a_1}$ is already connected (so $\pi_0(W\vert_{a_1})$ injects
into $\pi_0(BO(d))$). We will describe a way of producing a path from
the image of this point in $\vert X^{0, -1}_\bullet \vert$ into the
subspace $\vert X_\bullet^{0, 0} \vert$.

If $W \vert_{a_0}$ is already connected, then the point $\sigma$ already lies in $\vert X_\bullet^{0,0}\vert$ and there is nothing to prove. Otherwise, let us choose disjoint embeddings
$$\{f_\alpha : S^0 \hookrightarrow W \vert_{a_0}\}_{\alpha \in \Lambda}$$
such that if we perform 0-surgery along all of these embeddings, the resulting $(d-1)$-manifold is connected. As $\kappa=0$, the cobordism $W \vert_{[a_0, a_1]}$ is path connected relative to its top, and so we can extend the $f_\alpha$ to smooth maps
$$\hat{f}_\alpha : (a_0-\epsilon_0, a_1+\epsilon_1) \times S^0 \lra W$$
such that the standard height function (i.e.\ the projection to $(a_0
- \epsilon_0, a_1 + \epsilon_1)$) and $x_1 \circ \hat{f}_\alpha$ agree
inside $(x_1 \circ \hat{f}_\alpha)^{-1}(\cup (a_i-\epsilon_i,
a_i+\epsilon_i))$. As we have supposed that $d>2$, we may assume that
these $\hat{f}_\alpha$ are mutually disjoint embeddings. By taking a
tubular neighbourhood, we extend the $\hat{f}_\alpha$ to embeddings
$$\hat{e}_\alpha : (a_0-\epsilon_0, a_1+\epsilon_1) \times \R^{d-1} \times S^0 \hookrightarrow W$$
which are still mutually disjoint, and extend these further to
disjoint embeddings
$$e_\alpha : (a_0-\epsilon_0, a_1+\epsilon_1) \times \R^{d-1} \times D^1 \hookrightarrow \bR \times \bR^\infty$$
such that $e_\alpha^{-1}(W) = (a_0-\epsilon_0, a_1+\epsilon_1) \times
\R^{d-1} \times S^0$. It is clear that we can arrange the same
relationship between the standard height function on $(a_0-\epsilon_0,
a_1+\epsilon_1) \times \R^{d-1} \times D^1$ and the function $x_1 \circ
e_\alpha$ as we have over $(a_0-\epsilon_0, a_1+\epsilon_1) \times
\R^{d-1} \times S^0$.

The surgery move is then given by gluing the trace of a 0-surgery on
$\R^{d-1} \times S^0$ inside of $(a_0-\epsilon_0, a_1+\epsilon_1)
\times \R^{d-1} \times D^1$, into $\bR \times \bR^\infty$ using each of
the embeddings $e_\alpha$, as shown in
Figure~\ref{fig:Half0-surgeryObjects}.
\begin{figure}[h]
\includegraphics[bb=0 0 336 132]{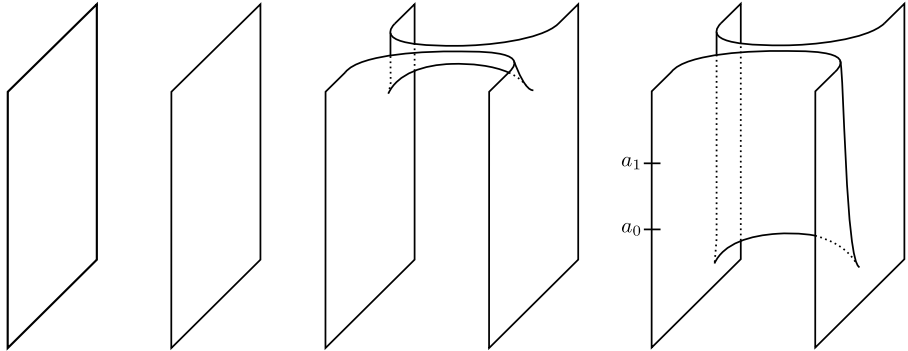}
\caption{The surgery move for surgery on objects below the middle dimension.}\label{fig:Half0-surgeryObjects}
\end{figure}
This does not define a path in $B D^{0, -1}$, as $(a_1-\epsilon_1,
a_1+\epsilon_1)$ will contain a critical value at some points during
the path. However, it does define a path in $\vert X_\bullet^{0,-1}
\vert$. Furthermore, if we let $\overline{W}$ be the manifold obtained
at the end of the path, then $\overline{W} \vert_{a_0}$ is obtained
from $W \vert_{a_0}$ by doing 0-surgery along the data
$\{\hat{e}_\alpha \vert_{\{a_0\} \times \R^{d-1} \times S^0}\}_{\alpha
  \in \Lambda}$ and so is connected. Also, $\overline{W} \vert_{a_1}$
is obtained from $W \vert_{a_1}$ by doing 0-surgery along the data
$\{\hat{e}_\alpha \vert_{\{a_1\} \times \R^{d-1}\times S^0}\}_{\alpha
  \in \Lambda}$, and as it was connected to start with (and $d>2$), it
remains connected. Hence $(t_0, t_1;a_0, a_1;\epsilon_0,
\epsilon_1;\overline{W}) \in \vert X_\bullet^{0,0} \vert$, as
required.

This surgery move generalises well to $l >0$, to finite (but large
enough) $N$, and to non-empty $L$, but to make it work with general
tangential structures $\theta$ we must equip the surgery data
$\{e_\alpha\}_{\alpha \in \Lambda}$ with extra data describing how to
induce a $\theta$-structure on the surgered manifold. We will first
give a definition of $\theta$-surgery, then describe the standard
family, and finally go on to describe the semi-simplicial space of
surgery data analogous to that of Section \ref{sec:SurgeryDataMor}.

\subsection{$\theta$-surgery}\label{sec:ThetaSurgery}

Consider a $\theta$-manifold $(M, \ell_M)$ and an embedding $e :
\bR^{d-l-1} \times S^l \hookrightarrow M$, and let $C$ be the
$d$-dimensional cobordism obtained as the trace of the surgery along
$e$. Thus $\partial_{in} C = M$ and $\partial_{out} C = \overline{M}$
is the result of the surgery.

The data of a \emph{$\theta$-surgery} on $M$ is an embedding $e$ as
above along with a $\theta$-structure $\ell$ on $C$ which agrees with
$\ell_M$ on $M$. This induces a $\theta$-structure on
$\overline{M}$. We will typically give the data of a $\theta$-surgery
extending an embedding $e$ by giving an extension of the
$\theta$-structure $\ell_M \circ De$ on $\bR^{d-l-1} \times S^l$ to
$\bR^{d-l-1} \times D^{l+1}$.  Up to homotopy, this is the same as
specifying a null-homotopy of the map $S^l \to B$ underlying the
$\theta$-structure on $\R^{d-l-1} \times S^l$.

\subsection{The standard family}\label{sec:ObSurgStdFam}

Let us construct the one-parameter family of manifolds depicted in
Figure~\ref{fig:Half0-surgeryObjects}. Choose a function $\rho : \bR \to \bR$ which is the identity on
$(-\infty, 1/2)$, has nowhere negative derivative, and has $\rho(t)=1$
for all $t \geq 1$. We define
$$K = \{(x, y) \in \bR^{d-l}\times \bR^{l+1}  \vert \,\, \vert y
\vert^2 = \rho(\vert x \vert ^2-1)\},$$ a smooth $d$-dimensional
submanifold, contained in $\bR^{d-l} \times D^{l+1}$, which outside of
the set $B^{d-l}_{\sqrt{2}}(0) \times D^{l+1}$ is identically equal to
$\bR^{d-l}\times S^l$. Let
$$h = x_1 : K \lra \bR$$
denote the first of the $x$ coordinates, which is the height function
we will consider on $K$. This function is Morse with precisely two
critical points: $(-1,0,\ldots,0)$ of index $l+1$ and $(1,0,\ldots,0)$ of
index $d-l-1$.

We now define a 1-parameter family of
$d$-dimensional submanifolds $\mathcal{P}_t$ inside $(-6,-2) \times
\R^{d-l-1} \times D^{l+1}$ in the following way.  Pick a smooth
one-parameter family of embeddings $\lambda_s: (-6,-2) \to (-6,0)$, such that $\lambda_0 =
\mathrm{Id}$, that $\lambda_s\vert_{(-6,-4)} = \mathrm{id}$ for all
$s$, and that $\lambda_1(-3) = -1$.  Then we get an embedding
$\lambda_t \times \mathrm{Id}_{\bR^d}: (-6,-2)\times \R^d \to (-6,0) \times \R^d$ and
define
\begin{equation*}
  \mathcal{P}_t = (\lambda_t \times \mathrm{Id}_{\bR^d})^{-1}(K) \in
  \Psi_d((-6,-2)\times \R^{d-l-1} \times \R^{l+1}).
\end{equation*}
It is easy to verify that $\mathcal{P}_t$ agrees with $(-6,-2) \times
\R^{d-l-1} \times S^l$ outside $(-4,-2) \times B_{\sqrt{2}}^{d-l-1}(0)\times D^{l+1}$, independently
of $t$.

We shall also need a tangentially structured version of this
construction, given a structure $\ell: TK\vert_{(-6,0)} \to
\theta^* \gamma$.  For this purpose, let $\omega : [0, \infty) \to
[0,1]$ be a smooth function such that $\omega(r) = 0$ for $r \geq 2$
and $\omega(r) = 1$ for $r \leq \sqrt{2}$.  We define a 1-parameter
family of embeddings by
\begin{eqnarray*}
  \psi_t : (-6,-2) \times \bR^{d-l-1} \times \R^{l+1} &\lra&
  (-6,0) \times \bR^{d-l-1} \times \R^{l+1}\\
  (s, x, y) &\longmapsto& (\lambda_{t\omega(\vert x \vert)}(s),x,  y), 
\end{eqnarray*}
It is easy to see that we also have $\psi_t^{-1}(K) = (\lambda_t
\times \mathrm{Id}_{\bR^d})^{-1}(K) = \mathcal{P}_t$, and we define a $\theta$-structure
on $\mathcal{P}_t$ by pullback along $\psi_t$.  This gives a family
\begin{equation*}
  \mathcal{P}_t(\ell) \in \Psi_\theta((-6,-2) \times \R^{d-l-1} \times
  \R^{l+1}),
\end{equation*}
and we record some important properties in the following
proposition. We will omit $\ell$ from the notation when it is
unimportant.
\begin{proposition}\label{prop:StdFamilyObjectsBelowMid}
  The elements $\mathcal{P}_t(\ell) \in \Psi_\theta((-6,-2) \times
  \R^{d-l-1} \times \R^{l+1})$ are $\theta$-submanifolds of $(-6,-2)
  \times \R^{d-l-1} \times D^{l+1}$ satisfying
  \begin{enumerate}[(i)]
  \item $\mathcal{P}_0(\ell) = K\vert_{(-6,-2)} = (-6,-2) \times \bR^{d-l-1} \times
    S^l$ as $\theta$-manifolds.\label{it:StdFamilyObjectsBelowMid.InitialValue}
    
  \item For all $t$, $\mathcal{P}_t(\ell)$ agrees with
    $K\vert_{(-6,-2)}$ as $\theta$-manifolds outside of $(-4,-2)
    \times B_{2}^{d-l-1}(0) \times D^{l+1}$.
    \label{it:StdFamilyObjectsBelowMid.ConstantOnBoundary}
    
  \item For all $t$ and each pair of regular values $-6 <a < b < -2$ of
    the height function $h : \mathcal{P}_t \to \bR$, the pair
    \begin{equation*}
      (\mathcal{P}_t\vert_{[a,b]}, \mathcal{P}_t\vert_b)
    \end{equation*}
    is
    $(d-l-2)$-connected.\label{it:StdFamilyObjectsBelowMid.kConnMoving}
    
  \item For each regular value $a$ of $h : \mathcal{P}_t \to (-6,-2)$,
    the manifold $\mathcal{P}_t \vert_{a}$ is either isomorphic to
    $\mathcal{P}_0 \vert_{a}$ or is obtained from it by
    $l$-surgery.\label{it:StdFamilyObjectsBelowMid.lConnMoving}
    
  \item The only critical value of $h: \mathcal{P}_1 \to (-6,-2)$
    is $-3$, and for $a \in (-3,-2)$, $\mathcal{P}_1 \vert_{a}$ is
    obtained by $l$-surgery from $\mathcal{P}_0 \vert_{a} =
    \bR^{d-l-1} \times S^l$ along the standard
    embedding.\label{it:StdFamilyObjectsBelowMid.lConn}
  \end{enumerate}
  In~(\ref{it:StdFamilyObjectsBelowMid.lConnMoving})
  and~(\ref{it:StdFamilyObjectsBelowMid.lConn}), the
  $\theta$-structure on the surgered manifold is determined (up to
  homotopy, cf.\ Section \ref{sec:ThetaSurgery}) by the
  $\theta$-structure on $K\vert_{(-6,0)}$.
\end{proposition}

The precise meaning of the word isomorphic
in~(\ref{it:StdFamilyObjectsBelowMid.lConnMoving}) above is the
following: By (\ref{it:StdFamilyObjectsBelowMid.ConstantOnBoundary})
we know that the manifolds are \emph{equal}\ outside $(-4,-2) \times
B^{d-l-1}_{2}(0) \times D^{l+1}$. Being isomorphic means that the
identity extends to a diffeomorphism which preserves
$\theta$-structures up to a homotopy of bundle maps which is constant
outside $(-4,-2) \times B^{d-l-1}_{2}(0) \times D^{l+1}$.
\begin{proof}
  (\ref{it:StdFamilyObjectsBelowMid.InitialValue}) and
  (\ref{it:StdFamilyObjectsBelowMid.ConstantOnBoundary}) follow easily
  from the properties of $\lambda_t$ and $\psi_t$, and the fact that
  $K$ agrees with $\R^{d-l} \times S^l$ outside $B_{\sqrt{2}}^{d-l}
  \times \R^{l+1}$.  It follows from the properties of $\omega$ that
  the $\theta$-structures agree outside $B_2^{d-l} \times \R^{l+1}$.
  For (\ref{it:StdFamilyObjectsBelowMid.kConnMoving}), the Morse
  function $\mathcal{P}_t \to (-6,-2)$ has at most one critical point,
  and that has index $l+1$.  If the critical value is in $(a,b)$, then
  the pair is $(d-l-2)$-connected, otherwise
  $\mathcal{P}\vert_{[a,b]}$ deformation retracts to
  $\mathcal{P}\vert_b$.  The fact that the Morse function has at most
  one critical point, of index $l+1$, also implies
  (\ref{it:StdFamilyObjectsBelowMid.lConnMoving}) by definition of
  surgery (and $\theta$-surgery, cf.\ Section~\ref{sec:ThetaSurgery}).
  Finally, the property that $\lambda_1(-3) = -1$ and $\lambda_1(-4) =
  -4$ implies that $h: \mathcal{P}_1 \to (-6,-2)$ does have a critical
  point of index $l+1$, with critical value $-3$, which
  proves~(\ref{it:StdFamilyObjectsBelowMid.lConn}).
\end{proof}

\subsection{Surgery data}

We can now describe the semi-simplicial space of surgery data out of which we will construct a ``perform surgery" map. In the following section we will describe how to construct this map.

Before doing so, we choose once and for all, smoothly in the
data $(a_i, \epsilon_i, a_p, \epsilon_p)$, increasing diffeomorphisms
\begin{equation}\label{eq:reparam}
\varphi = \varphi(a_i, \epsilon_i, a_p, \epsilon_p) : (-6,-2) \cong
(a_i-\epsilon_i, a_p+\epsilon_p)
\end{equation}
sending $-3$ to $a_i
-\tfrac{1}{2}\epsilon_i$ and $-4$ to $a_i - \tfrac34
\epsilon_i$.

\begin{definition}\label{defn:YComplex}
  Let $x = (a, \epsilon, (W, \ell_W)) \in D^{\kappa, l-1}_{\theta,
    L}(\bR^N)_p$, and write $M_i = W\vert_{a_i}$. Define the set
  $Y_q(x)$ to consist of tuples $(\Lambda, \delta, e,\ell)$, where
  $\Lambda \subset \Omega$ is a finite set, $\delta: \Lambda \to
  [p]\times [q]$ is a function,
  \begin{equation*}
    e: \Lambda \times (-6,-2) \times \R^{d-l-1} \times D^{l+1} \hookrightarrow \R
    \times (0,1) \times (-1,1)^{N-1}
  \end{equation*}
  is an embedding, and $\ell: T(\Lambda \times K\vert_{(-6,0)}) \to
  \theta^* \gamma$ is a bundle map, satisfying the conditions below.
  We shall write $\Lambda_{i,j} = \delta^{-1}(i,j)$ and
  \begin{equation*}
    e_{i,j}: \Lambda_{i,j} \times (a_i -
    \epsilon_i, a_p + \epsilon_p) \times \R^{d-l-1} \times D^{l+1} \to \R
    \times (0,1) \times (-1,1)^{N-1}
  \end{equation*}
  for the embedding obtained by restricting $e$ and reparametrising
  using~\eqref{eq:reparam}.
  \begin{enumerate}[(i)]
  \item $e^{-1}(W) = \Lambda \times (-6,-2) \times \bR^{d-l-1} \times
    S^l$. We let
    $$\partial e : \Lambda \times (-6,-2) \times \bR^{d-l-1} \times
    S^l \hookrightarrow W$$ denote the embedding restricted to the
    boundary.
    
  \item\label{it:YCxHeightFun} For any $i=0,\dots, p$ and $t \in
    (a_i-\epsilon_i, a_i+\epsilon_i)$, we have $(x_1 \circ
    e_{i,j})^{-1}(t) = \Lambda_{i,j} \times \{t\} \times \bR^{d-l-1}
    \times D^{l+1}$.
    
  \item\label{item:11} The composition $\ell_W \circ D \partial
    e:T(\Lambda \times K\vert_{(-6,-2)}) \to \theta^* \gamma$ agrees
    with the restriction of $\ell$.
  \end{enumerate} 
  If we let $\ell_{i,j}$ denote the restriction of $\ell$ to
  $T(\Lambda_{i,j} \times K\vert_{(-6,0)})$, the data $(e_{i,j},
  \ell_{i,j})$ is enough to perform $\theta$-surgery on $M_i$ (as
  $K\vert_{(-6,0)}$ is the trace of an $l$-surgery), and we further
  insist that
  \begin{enumerate}[(i)]
    \setcounter{enumi}{3}
  \item For each $j = 0,\dots, q$ and $i = 0, \dots, p$, the resulting
    $\theta_{d-1}$-manifold $\overline{M}_i$ has the property that
    $\pi_k(\overline{M}_i) \to \pi_k(B)$ is injective for $k \leq
    l$.\label{it:EnoughSurgeryDataOb}
  \end{enumerate}
  For each $x$, $Y_\bullet(x)$ is a semi-simplicial set in the same
  way as in Definition~\ref{defn:ZComplex}.
\end{definition}

Note that the set $Y_q(x)$ consists of those $(q+1)$-tuples of
elements of $Y_0(x)$ which are disjoint.
\begin{definition}\label{defn:DoubleCxOb}
  We define a bi-semi-simplicial space $D^{\kappa, l}_{\theta,
    L}(\bR^N)_{\bullet, \bullet}$ (augmented in the second
  semi-simplicial direction) as a set by
  \begin{equation*}
    D^{\kappa, l}_{\theta, L}(\bR^N)_{p, q} = \{(x,y)\,\vert\, x \in
    D^{\kappa, l-1}_{\theta, L}(\bR^N)_p, y \in Y_q(x)\},
  \end{equation*}
  and topologise it as a subspace of
  $${D}_{\theta, L}^{\kappa, l-1}(\bR^N)_p \times \left( \coprod_{\Lambda
      \subset \Omega}C^\infty(\Lambda \times V, \bR^{N+1}) \times \Bun
    \left(T(\Lambda \times K\vert_{(-6,0)}), {\theta}^*\gamma \right)
  \right)^{(p+1)(q+1)}$$ where $V$ denotes the manifold $(-6,-2)
  \times \bR^{d-l-1} \times D^{l+1}$.  Explicitly, the face map $d_k$
  in the $q$ direction forgets the surgery data $(e_{i,j},
  \ell_{i,j})$ with $j=k$, and the face map $d_k$ in the $p$ direction
  forgets both the surgery data $(e_{i,j}, \ell_{i,j})$ with $i=k$ and
  the $k$th regular value.
\end{definition}

The main result about this
bi-semi-simplicial space of manifolds equipped with surgery data is
the following, whose proof we defer until Section
\ref{sec:Connectivity}.

\begin{theorem}\label{thm:SurgeryComplexOb}
  Under the assumptions of Theorem \ref{thm:lfiltration}, the maps
  \begin{equation*}
    \vert D^{\kappa, l}_{\theta, L}(\bR^N)_{\bullet, 0} \vert \lra
    \vert D^{\kappa, l}_{\theta, L}(\bR^N)_{\bullet, \bullet} \vert
    \lra \vert D^{\kappa, l-1}_{\theta, L}(\bR^N)_{\bullet}\vert
  \end{equation*}
  are weak homotopy equivalences, where first map is the inclusion of
  0-simplices and the second is the augmentation, in the second
  simplicial direction.
\end{theorem}

In fact, we shall prove this theorem assuming the conditions of
Theorem~\ref{thm:lfiltration} except (\ref{item:12}).  That condition
will be used in the proof of Lemma~\ref{lem:ObSurgeryDesiredEffect}.

\subsection{Proof of Theorem \ref{thm:lfiltration}}\label{sec:proof-theor-refthm:l}

We now go on to prove Theorem \ref{thm:lfiltration}, so suppose that
the inequalities in the statement of that theorem are satisfied:
$2(l+1) < d$, $l \leq \kappa$, $l \leq d-\kappa-2$, $l+2 + d < N$, $L$
admits a handle decomposition using only handles of index $< d-l-1$,
and and the map $\ell_L: L \to B$ is $(l+1)$-connected.

Let $(a, \epsilon, (W, \ell_W), e,\ell) \in D_{\theta, L}^{\kappa,
  l}(\bR^N)_{p,0}$.  For each $i = 0,\dots, p$, we have an embedding
$e_i = e_{i,0}$ and a bundle map $\ell_i = \ell_{i,0}$, from which we
shall construct a 1-parameter family of elements
$\mathcal{K}^t_{e_{i}, \ell_{i}}(W, \ell_W) \in \Psi_\theta((a_0 -
\epsilon_0, a_p + \epsilon_p) \times \R^N)$, $t \in [0,1]$ as
follows. Changing the first coordinate of the manifolds
$\mathcal{P}_t(\ell_i)$ by composing with the reparametrisation
functions of (\ref{eq:reparam}), we get a family of manifolds
\begin{equation*}
  \overline{\mathcal{P}}_t(\ell_i) \in \Psi_\theta((a_i-\epsilon_i, a_p+\epsilon_p)
  \times \bR^{d-l-1} \times \R^{l+1}) 
\end{equation*}
having all the properties of Proposition~\ref{sec:ObSurgStdFam}, where property
(\ref{it:StdFamilyObjectsBelowMid.lConn}) now holds for all regular
values in $(a_i-\tfrac{1}{2}\epsilon_i, a_p+\epsilon_p)$.  Then for $t
\in [0,1]$, let
\begin{equation*}
  \mathcal{K}^t_{e_{i}, \ell_{i}}(W, \ell_W) \in
  \Psi_\theta((a_0-\epsilon_0, a_p+\epsilon_p)\times \bR^N)
\end{equation*}
be equal to $W \vert_{(a_0-\epsilon_0, a_p+\epsilon_p)}$ outside the
image of $e_{i}$, and on $e_{i}(\Lambda_i \times (a_i-\epsilon_i,
a_p+\epsilon_p) \times \bR^{d-l-1} \times D^{l+1})$ be given by
$e_{i}(\Lambda_i \times \overline{\mathcal{P}}_t(\ell_i))$.  This
gives a $\theta$-manifold, because $\Lambda_i \times
\overline{\mathcal{P}}_t(\ell_i)$ and $\Lambda_i \times
\overline{\mathcal{P}}_0(\ell_i)$ agree as $\theta$-manifolds outside
of $(a_i-\tfrac34 \epsilon_i, a_p+\epsilon_p) \times
B^{d-l-1}_{2}(0) \times D^{l+1}$.

As the embeddings $e_{i}$ are all disjoint, this procedure can be
iterated, and for a tuple $t = (t_0, \ldots, t_p) \in [0,1]^{p+1}$ we let
$$
\mathcal{K}^{t}_{e, \ell}(W, \ell_W) = \mathcal{K}^{t_p}_{e_{p},
  \ell_{p}} \circ \cdots \circ \mathcal{K}^{t_0}_{e_{0}, \ell_{0}}(W,
\ell_W) \in \Psi_\theta((a_0-\epsilon_0, a_p+\epsilon_p)\times \bR^N).
$$

\begin{lemma}\label{lem:ObSurgeryDesiredEffect}
  Firstly, the tuple $(a, \tfrac{1}{2}\epsilon, \mathcal{K}^{t}_{e,
    \ell}(W, \ell_W))$ is an element of $X_p^{\kappa, l-1}$. Secondly,
  if $t_i=1$---so the surgery for the regular value $a_i$ is fully
  done---then for any regular value $b$ of $x_1 : \mathcal{K}^{t}_{e,
    \ell}(W, \ell_W) \to \bR$ in the interval
  $(a_i-\tfrac{1}{2}\epsilon_i, a_i+\tfrac{1}{2}\epsilon_i)$ we have
  that
  $$\pi_j(\mathcal{K}^{t}_{e, \ell}(W, \ell_W) \vert_{b}) \lra
  \pi_j(B)$$ is injective for $j \leq l$.
\end{lemma}
\begin{proof}
  For the first part we must verify the conditions of Definition
  \ref{defn:X}. Conditions (\ref{item:1})--(\ref{it:Disjoint}) are
  immediate from the properties of $(a, \epsilon)$ that we start with,
  the disjointness of the surgery data from $\bR \times L$, and the
  fact that the standard family $\mathcal{P}_t$ has isolated critical
  values.

  For condition (\ref{it:MorConn}) we proceed exactly as in the proof
  of Lemma \ref{lem:MorSurgeryDesiredEffect}, using property
  (\ref{it:StdFamilyObjectsBelowMid.kConnMoving}) of the standard
  family, that the pair $(\mathcal{P}_t\vert_{[a,b]},
  \mathcal{P}_t\vert_{b})$, and hence the homotopy equivalent pair
  \begin{equation*}
    \big(\mathcal{P}_t\vert_{[a,b]}, (\mathcal{P}_t\vert_{b}) \cup
    ([a,b] \times (\R^{d-l-1} - B_2^{d-l-1}(0)) \times S^l) \big)
  \end{equation*}
  is $(d-l-2)$-connected, and so in particular $\kappa$-connected as
  we have supposed that $l \leq d-\kappa-2$.

  For condition (\ref{it:ObConn}), let $b \in
  (a_i-\tfrac{1}{2}\epsilon_i, a_i+\tfrac{1}{2}\epsilon_i)$ be a
  regular value of the height function on $\mathcal{K}^{t}_{e,
    \ell}(W, \ell_W)$, and define $\theta_{d-1}$-manifolds
  \begin{eqnarray*}
    \overline{M} & = & \mathcal{K}^{t}_{e,\ell}(W, \ell_W) \vert_{b}\\
    M & = & W \vert_{b}.
\end{eqnarray*}
By
Proposition~\ref{prop:StdFamilyObjectsBelowMid}~(\ref{it:StdFamilyObjectsBelowMid.lConnMoving}),
the $\theta_{d-1}$-manifold $\overline{M}$ is obtained from $M$ by
performing $\theta$-$l$-surgeries.  Let $C : M \leadsto \overline{M}$
be the $\theta$-cobordism given by the trace of these surgeries. We
have the commutative diagram
\begin{equation}\label{fig:TS}
\begin{aligned}
\xymatrix{
\pi_j(M) \ar[r]^-i \ar[rd] & \pi_j(C) \ar[d] & \ar[l]_-{\overline{\imath}} \ar[ld] \pi_j(\overline{M})\\
& \pi_j(B)
}
\end{aligned}
\end{equation}
and $C$ is obtained by attaching $(l+1)$-cells to $M$ or by attaching
$(d-l-1)$-cells to $\overline{M}$.  Hence $i$ is surjective for $j
\leq l$ and $\overline{\imath}$ is bijective for $j \leq d-l-3$.  Since
our assumption $2(l+1) < d$ implies $l \leq d-l-3$, as long as $j \leq
l$, the right hand diagonal map is injective whenever the left hand
one is, and in particular for all $j \leq l-1$.

We now prove the second part, so suppose $t_i=1$. We construct the
manifold $\mathcal{K}^{t}_{e, \ell}(W, \ell_W)$ by first taking
$\mathcal{K}^{1}_{e_{i}, \ell_{i}}(W, \ell_W)$ and then performing the
remaining surgeries to it. Let $\widetilde{M} =
\mathcal{K}^{1}_{e_{i}, \ell_{i}}(W, \ell_W) \vert_{b}$, so that
$\overline{M}$ is obtained from $\widetilde{M}$ by $l$-surgery.

We first show that $\pi_j(\widetilde{M}) \to \pi_j(B)$ is injective
for $j \leq l$. By property (\ref{it:EnoughSurgeryDataOb}) of the
complex of surgery data, $(e_i, \ell_i)$ is enough surgery data on $M
= W\vert_b$ to make the map on $\pi_l$ be injective after performing
it. By property (\ref{it:StdFamilyObjectsBelowMid.lConn}) of the
standard family, as $b >a_i-\tfrac{1}{2}\epsilon_i$ the manifold
$\widetilde{M}$ has all of this surgery done, and so
$\pi_j(\widetilde{M}) \to \pi_j(B)$ is injective for $j \leq l$.

By the previous argument, with $M$ replaced with $\widetilde{M}$
in~\eqref{fig:TS}, the remaining surgeries do not change this
injectivity property.
\end{proof}

In the composition
$$\vert D^{\kappa, l}_{\theta, L}(\bR^N)_{\bullet, \bullet} \vert \lra \vert D^{\kappa, l-1}_{\theta, L}(\bR^N)_{\bullet} \vert \lra \vert X_\bullet^{\kappa, l-1}\vert$$
both maps are homotopy equivalences by Theorem \ref{thm:SurgeryComplexOb} and Proposition \ref{prop:X} respectively. There is also an inclusion
$$\vert D^{\kappa, l}_{\theta, L}(\bR^N)_{\bullet} \vert \lra \vert D^{\kappa, l}_{\theta, L}(\bR^N)_{\bullet, 0} \vert \lra \vert D^{\kappa, l}_{\theta, L}(\bR^N)_{\bullet, \bullet} \vert$$
as the subspace of manifolds equipped with no surgery data, and the second map is a weak homotopy equivalence by Theorem \ref{thm:SurgeryComplexOb}.

We define a map
\begin{eqnarray*}
  \surg_p : [0,1]^{p+1} \times D^{\kappa, l}_{\theta, L}(\bR^N)_{p,0} &
  \lra & X_p^{\kappa, l-1}\\
  (t, (a, \epsilon, (W, \ell_W), e,
  \ell)) & \longmapsto & (a,
  \tfrac{1}{2}\epsilon, \mathcal{K}^{t}_{e,
    \ell}(W, \ell_W)), 
\end{eqnarray*}
which has the desired range by the first part of Lemma
\ref{lem:ObSurgeryDesiredEffect}, and furthermore sends $(1,\ldots,1)
\times D^{\kappa, l}_{\theta, L}(\bR^N)_{p,0}$ into $X_p^{\kappa,
  l}$. On the boundary of the cube this map has further distinguished
properties: one is given by the second part of Lemma
\ref{lem:ObSurgeryDesiredEffect}. The second is that, by Proposition
\ref{prop:StdFamilyObjectsBelowMid}
(\ref{it:StdFamilyObjectsBelowMid.InitialValue}), we have an equality
$\mathcal{K}^0_{e_i, \ell_i}(W') = W'$ of $\theta$-submanifolds of
$(a_0-\epsilon_0, a_p+\epsilon_p) \times \bR^N$. Thus we obtain the
formula
\begin{equation}\label{eq:FaceMapsMatch}
d_i \surg_p(d^i t, x) = \surg_{p-1}(t, d_ix)
\end{equation}
where $d^i : [0,1]^p \to [0,1]^{p+1}$ adds a zero in the $i$th
position, and the $d_i$ are the face maps of the semi-simplicial
spaces $D^{\kappa, l}_{\theta, L}(\bR^N)_{\bullet,0}$ and
$X_\bullet^{\kappa,l-1}$.

We wish to assemble the maps $\surg_p$ to a homotopy $\surg :
[0,1]\times \vert {D}^{\kappa, l}_{\theta, L}(\bR^N)_{\bullet,0} \vert
\to \vert X_\bullet^{\kappa, l-1} \vert$. Hence we define $\lambda,
\psi : \Delta^p \to [0,1]^{p+1}$ by the formul\ae
\begin{eqnarray*}
  \lambda_i(t) &=& \min\left (1, 2\overline{t}_i \right)\\
  \psi_i(t) &=& \max\left(0, 2\overline{t}_i -1 \right),
\end{eqnarray*}
where again $\overline{t}_i = t_i/\max(t_j)$, and a map $H : [0,1]
\times \Delta^p \to [0,1]^{p+1}\times \Delta^p$ by
$$H(s,t) = \left(s\cdot\lambda(t) , \frac{\psi(t)}{\sum_j
    \psi_j(t)}\right).$$ These may be used to form the composition
$$F_p : [0,1] \times {D}^{\kappa, l}_{\theta, L}(\bR^N)_{p,0} \times \Delta^p \overset{H}\lra {D}^{\kappa, l}_{\theta, L}(\bR^N)_{p,0} \times [0,1]^{p+1} \times \Delta^p \overset{\surg_p \times \Delta^p}\lra X_p^{\kappa, l-1} \times \Delta^p.$$

\begin{lemma}\label{lem:glue}
  These maps glue to a homotopy $\surg : [0,1]\times \vert
  {D}^{\kappa, l}_{\theta, L}(\bR^N)_{\bullet,0} \vert \to \vert
  X_\bullet^{\kappa, l-1} \vert$.
\end{lemma}
\begin{proof}
  The points $F_p(s, x, d^i t)$ and $F_{p-1}(s, d_ix, t)$ are identified
  under the usual face maps among the $X_p^{\kappa, l-1} \times
  \Delta^p$.  This follows immediately from the formula
  (\ref{eq:FaceMapsMatch}) and the observation that $\lambda(d^it) =
  d^i(\lambda(t))$, $\psi(d^it) = d^i(\psi(t))$ and $\sum_j
  \psi_j(d^it) = \sum_j \psi_j(t)$.
\end{proof}

By construction, the map $\surg(1, -) : \vert {D}^{\kappa, l}_{\theta,
  L}(\bR^N)_{\bullet,0} \vert \to \vert X_\bullet^{\kappa, l-1} \vert$
factors through the continuous injection $\vert X_\bullet^{\kappa, l}
\vert \to \vert X_\bullet^{\kappa,l-1}\vert$.  This may be seen at the
level of the maps $F_p$, since the domain of $F_p$ is covered by the
$2^{p+1}$ closed sets obtained by requiring for each $i$ either
$\lambda_i(t) = 1$ or $\psi_i(t) = 0$, on each of which the map
$F_p(1,-)$ composed with $X_p^{\kappa,l-1} \times \Delta^p \to \vert
X_\bullet^{\kappa,l-1} \vert$ factors through $\vert
X_\bullet^{\kappa,l}\vert$: If $\lambda_i(t) = 1$, the surgery near
the regular value $a_i$ is completely done (and so by the second part
of Lemma \ref{lem:ObSurgeryDesiredEffect} it does not matter what the
remaining surgeries do near the regular value $a_i$); if not, then
$\psi_i(t) = 0$ and by the face identifications, we can forget the
regular value $a_i$.

The homotopy $\surg$ is constant on the subspace $\vert D^{\kappa,
  l}_{\theta, L}(\bR^N)_{\bullet} \vert \hookrightarrow \vert
D^{\kappa, l}_{\theta, L}(\bR^N)_{\bullet, 0} \vert$, and by the
argument in Section~\ref{sec:proof-theor-refthm:k} we deduce the
weak equivalence in Theorem~\ref{thm:lfiltration}.

\begin{remark}
  It is possible to weaken condition~(\ref{item:13}) of
  Theorem~\ref{thm:lfiltration} to the map $\ell_L: L \to B$ being
  $l$-connected, in which case the method of
  Section~\ref{sec:surg-objects-middle} below can be used to prove
  that the inclusion gives a weak equivalence from
  $B\mathcal{C}_{\theta, L}^{\kappa, l}(\bR^N)$ to a union of path
  components of $B\mathcal{C}_{\theta, L}^{\kappa, l-1}(\bR^N)$.
\end{remark}

%%% Local Variables: 
%%% mode: latex
%%% TeX-master: "Moduli"
%%% End: 

\section{Surgery on objects in the middle dimension}
\label{sec:surg-objects-middle}

We now restrict our attention to even dimensions, and write
$d=2n$. Given a collection of path components $\mathcal{A} \subset
\pi_0(\Ob(\mathcal{C}_{\theta, L}^{n-1, n-2}(\bR^N)))$, in Definition \ref{defn:CobCatA} we defined
$$\mathcal{C}_{\theta,
  L}^{n-1, \mathcal{A}}(\bR^N) \subset \mathcal{C}_{\theta, L}^{n-1,
  n-2}(\bR^N)$$
  to be the full subcategory on this collection of
objects. To state our main theorem concerning these subcategories, we
first need a definition.

\begin{definition}\label{defn:Reversible}
  We say a tangential structure $\theta$ is \emph{reversible} if
  whenever there is a morphism $C : M \leadsto N$ in
  $\mathcal{C}_{\theta, L}$, there also exists a morphism
  $\overline{C} : N \leadsto M$ in this category, whose underlying
  manifold is the reflection of $C$.
\end{definition}

In Proposition \ref{prop:Reversible}, we prove that this property is
equivalent to $\theta$ being \emph{spherical}, as defined in
Section~\ref{sec:intr-stat-results} (i.e.\ the $2n$-sphere admits a
$\theta$-structure extending any given structure on a disc).  We can
now state our main theorem concerning these subcategories, analogous
to Theorem \ref{thm:lfiltration} but in the middle dimension.  The
reader mainly interested in Theorems~\ref{thmcor:rational-coho} and
\ref{thm:main-A} can take $\theta = \theta^n: BO(2n)\langle n\rangle
\to BO(2n)$, $L \cong D^{2n-1}$, $N = \infty$, and $\mathcal{A}$ the
class of objects which are either diffeomorphic to $S^{2n-1}$ with its
standard smooth structure and $\theta$-structure or are \emph{not}
$\theta$-bordant to $S^{2n-1}$.  (Again, the proof does not simplify
much in this special case).

\begin{theorem}\label{thm:MidSurgery}
  Suppose that
  \begin{enumerate}[(i)]
  \item $2n \geq 6$,
  
  \item $3n+1 < N$,
  
  \item $\theta$ is reversible,
  
  \item $L$ admits a handle decomposition only using handles of index less than
    $n$,
    
  \item the map $\ell_L : L \to B$ is $(n-1)$-connected,
  
  \item\label{it:Afull} the natural map $\mathcal{A} \to \pi_0(B\mathcal{C}_{\theta, L}^{n-1, n-2}(\bR^N))$ is surjective.
  \end{enumerate}
  Then
  $$B\mathcal{C}_{\theta, L}^{n-1, \mathcal{A}}(\bR^N) \lra
  B\mathcal{C}_{\theta, L}^{n-1, n-2}(\bR^N)$$ is a weak homotopy
  equivalence.
\end{theorem}

The surgery move that we will employ is similar to that of the last
section, but has a crucial difference. In the last section, when we
performed the surgery move to make $a_i$ be a good regular value, we
glued a family of manifolds having the effect of performing
$l$-surgery on the level sets $W\vert_{a_i}$, but at the same time
performing $l$-surgery on all higher level sets.  In
Section~\ref{sec:surg-objects-below}, $l < (d-2)/2 = n-1$, and
therefore performing $l$-surgery on a $(2n-1)$-manifold which is
$l$-connected preserves its $l$-connectedness.  In this section, we
will need to change level sets by doing $(n-1)$-surgery on
$(2n-1)$-manifolds, and this is much more delicate.  For example, any
1-manifold can be made connected by performing 0-surgeries, but
performing further 0-surgeries will disconnect it again.

Instead we use a modified surgery move, which will let us perform
$(n-1)$-surgery on a level set $W\vert_{a}$ and leave all other level
sets $W\vert_{b}$ unchanged, except when $b$ is very close to $a$.
For $n=1$, this was done in \cite{GR-W}, and the construction there
generalises to higher $n$.  Let us briefly recall and depict the case
$n=1$.  We start with the same surgery data as in
Section~\ref{sec:surg-objects-below}, a collection of embeddings
$$\{e_\alpha : (a_0-\epsilon_0, a_1+\epsilon_1) \times D^{1} \times
D^1 \hookrightarrow \bR \times \bR^\infty\}_{\alpha \in \Lambda},$$
but glue in to the image of each $e_\alpha$ the path of manifolds
shown in Figure \ref{fig:Full0-surgeryObjects}.
\begin{figure}[h]
  \includegraphics[bb=0 0 336 132]{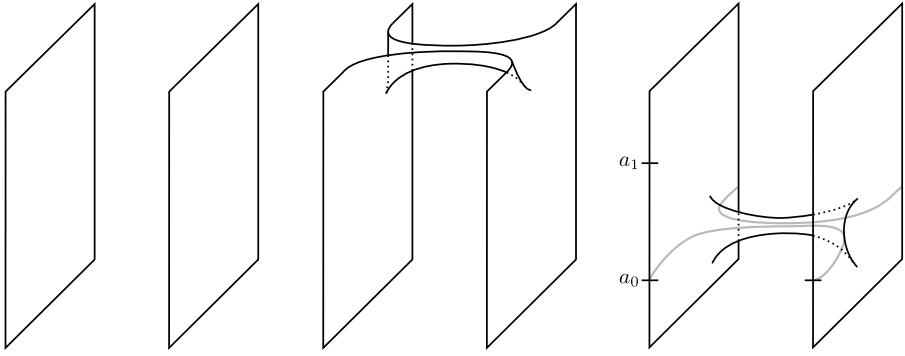}
  \caption{The surgery move on objects in the middle dimension. In the
    last frame we have indicated the level set at $a_0$ in light grey,
    to emphasise that it \emph{is} modified by the
    surgery.}\label{fig:Full0-surgeryObjects}
\end{figure}
This defines a path in the space $\vert X_\bullet^{0, -1} \vert$, and
if the handle in Figure \ref{fig:Full0-surgeryObjects} which we have
moved into the manifold is ``thin" enough (with respect to the height
function) then the manifold $\overline{W}$ obtained at the end of the
path has $\overline{W} \vert_{a_0}$ and $\overline{W} \vert_{a_1}$
both connected, and so lies in $\vert X_\bullet^{0,0}\vert$.

In order to make sense of this surgery move in the presence of
$\theta$-structures, we must equip the 1-parameter family of manifolds
shown in Figure \ref{fig:Full0-surgeryObjects} with
$\theta$-structures which start at a given structure, are constant
near the vertical boundaries, and at the end of the path the level
sets above and below the handle should be isomorphic as
$\theta$-manifolds to the level sets before the handle was added.
This last property does not hold in general: for example, if we equip
the original manifold in Figure \ref{fig:Full0-surgeryObjects} with a
framing, one may easily see (using the Poincar{\'e}--Hopf theorem)
that there is no framing on the final manifold consistent with these
requirements.  As we will see, this problem goes away when $\theta$ is
assumed to be reversible.  Let us first discuss the reversibility condition
in more detail.

\subsection{Reversibility}

Recall that a tangential structure $\theta: B \to BO(d)$ is called
\emph{spherical} if any structure on a disc $D \subset S^d$ extends to
one on $S^d$.  (When $B$ is path connected, this is equivalent to the
$d$-sphere admitting any structure at all.)  Let us first discuss some
related conditions on tangential structures $\theta: B \to BO(d)$.
\begin{definition}
  \label{def:once-stable}
  A tangential structure $\theta: B \to BO(d)$ is \emph{once-stable}
  if there exists a map $\bar{\theta}: \bar{B} \to BO(d+1)$ and a
  commutative diagram
  \begin{equation*}
  \xymatrix{
  B \ar[r] \ar[d]_{\theta} & \bar{B} \ar[d]_{\bar{\theta}} \\
    BO(d) \ar[r] & BO(d+1)
  }
  \end{equation*}
  which is homotopy pullback.

  A tangential structure $\theta$ is \emph{weakly once-stable} if
  there exists such a diagram for which $\pi_i(BO(d),B) \to
  \pi_i(BO(d+1),\bar B)$ is surjective for $i = d+1$ and bijective for
  $i \leq d$, for all basepoints.
\end{definition}

From the commutative diagram in the definition, there is a bundle map
$\epsilon^1 \oplus \theta^* \gamma \to \bar{\theta}^* \gamma$. Hence a
$\theta$-structure $TW \to \theta^*\gamma$ on a $d$-manifold $W$
induces a bundle map $\epsilon^1 \oplus TW \to \bar \theta^* \gamma$.
If $\theta$ is weakly once-stable we may deduce the converse, that a bundle map $\epsilon^1 \oplus TW \to \bar \theta^*\gamma$ is
homotopic to one that arises from a $\theta$-structure.  More
precisely, we have the following useful lemma.

\begin{lemma}\label{lemma:stable-extension}
  Let $\theta: B \to BO(d)$ be weakly once-stable.  Let $W$ be a
  $d$-manifold and let $\ell: TW\vert_{A} \to \theta^* \gamma$ be a
  $\theta$-structure defined on a closed submanifold $A \subset W$.
  Then $\ell$ extends to a $\theta$-structure $TW \to \theta^*\gamma$
  if and only if the stabilised bundle map $\epsilon^1 \oplus \ell:
  \epsilon^1 \oplus TW\vert_{A} \to \epsilon^1 \oplus \theta^* \gamma$
  extends to a bundle map over all of $W$.
\end{lemma}
\begin{proof}
Without loss of generality, we may assume that $\theta$ and $\bar
  \theta$ are Serre fibrations.  Let us write $s: BO(d) \to BO(d+1)$
  for the stabilisation map, and let us pick a classifying map $t: W
  \to BO(d)$ for the tangent bundle. Tangential structures on $TW$ then correspond to lifts of $t$ along some fibration, and tangential structures on $\epsilon^1 \oplus TW$ correspond to lifts of $s \circ t$ along some fibration.
  
  We write $\widetilde{\theta} : \widetilde{B} \to BO(d)$ for the
  pullback of $\bar{\theta}$, so the commutative diagram in Definition
  \ref{def:once-stable} gives a map $i : B \to \widetilde{B}$ over
  $BO(d)$. A $\bar{\theta}$-structure on $\epsilon^1 \oplus TW$ is
  then nothing but a $\widetilde{\theta}$-structure on $TW$.
  
The long exact sequence on homotopy for the various fibrations combine to give 
$$\cdots \lra \pi_i(\widetilde{B}, B) \lra \pi_{i}(BO(d), B) \lra \pi_{i}(BO(d+1), \bar{B}) \lra \pi_{i-1}(\widetilde{B}, B) \lra \cdots$$
from which we deduce that $(\widetilde{B}, B)$ is $d$-connected. Now,
the situation described in the statement is a lifting problem
\begin{equation*}
  \xymatrix{
    A\ar[r]\ar[d] & B\ar[d]\\
    W \ar[r]\ar@{..>}[ur]& {\tilde B,}
  }
\end{equation*}
which has a solution as $(W, A)$ has
cells of dimension at most $d$, and $(\widetilde{B}, B)$ is
$d$-connected.
\end{proof}

\begin{lemma}
  \label{lem:weakly-once-stable}
  The tangential structure $\theta: B \to BO(d)$ is weakly once-stable
  if and only if it is spherical.
\end{lemma}
\begin{proof}
  Given any bundle map $\ell: T S^d\vert_{D} \to \theta^* \gamma$ we
  can of course extend the stabilised map to $\epsilon^1 \oplus TS^d
  \to \epsilon^1 \oplus \theta^* \gamma$, and if $\theta$ is weakly
  once stable, the above lemma implies that the $\theta$ structure
  extends to all of $S^d$.

  Conversely, given a spherical structure $\theta: B \to BO(d)$ we may
  pick $\theta$-structures $\ell_i: TS^d \to \theta^* \gamma$, one for
  each path component of $B$, and form $\bar{B}$ by attaching an
  $(n+1)$-cell to $B$ along each map.  The compositions $S^d \to B \to
  BO(d) \to BO(d+1)$ are null-homotopic, so we obtain an extension
  $\bar{\theta} : \bar{B} \to BO(d+1)$.
  
  It follows that $H_i(\bar B, B)
  \to H_i(BO(d+1),BO(d))$ is surjective for $i = d+1$ and an
  isomorphism for $i \leq d$, even with local coefficients.  By the
  Hurewicz theorem, $\pi_i(\bar B, B) \to \pi_i(BO(d+1),BO(d))$ is
  surjective for $i=d+1$ and bijective for $i \leq d$, for all
  basepoints.  It follows that $\pi_i(BO(d),B) \to \pi_i(BO(d+1),\bar
  B)$ is surjective for $i = d+1$ and bijective for $i \leq d$.
\end{proof}

We now show that these conditions on $\theta$ are also equivalent to
reversibility.

\begin{proposition}\label{prop:Reversible}
  The tangential structure $\theta$ is reversible if and only if it is
  spherical.
\end{proposition}
\begin{proof}
  If $\theta$ is reversible and a structure on $D^d$ is given, we
  think of $D^d$ as a morphism from the empty set to $S^{d-1}$.  By
  assumption, a compatible structure exists on the disc, thought of as
  a morphism from $S^{d-1}$ to the empty set.

  For the reverse direction we use Lemma~\ref{lemma:stable-extension}.
  Suppose given a cobordism $C: M \leadsto N$ with $\theta$-structure
  $\ell: TC \to \theta^* \gamma$.  Let $\overline{C} : N \leadsto M$
  be the cobordism whose underlying manifold is $C$, but regarded as a
  morphism in the other direction.  Since $TC\vert_{\partial C} =
  \epsilon^1 \oplus T(\partial C)$, we may reflect in the
  $\epsilon^1$-direction to get a reversed $\theta$-structure near
  $\partial \bar{C} = N \amalg M$, and our task is to extend the
  reversed structure to $\bar{C}$.  By the lemma, it suffices to
  extend the stabilised bundle map, but that is easy: Pick a non-zero
  section of the vector bundle $\epsilon^1 \oplus TC$ which over
  $\partial C$ is the inwards pointing normal to $T(\partial C)
  \subset TC \vert_{\partial C}$, and reflect the stabilised bundle
  map in that field.
\end{proof}

One key property of reversible tangential structures is that they
allow us to connect-sum $\theta$-manifolds, which of course is not
possible in general: the connect-sum of framed manifolds is not
typically framable. In fact, more is true. We can perform arbitrary
surgeries on a $\theta$-manifold and find a $\theta$-structure on the
new manifold.

\begin{proposition}\label{prop:ConnectSum}
  Let $(M, \ell_M)$ be a $d$-dimensional $\theta$-manifold, and suppose that
  $$e : S^{n-1} \times D^{d-n+1} \hookrightarrow M$$
  is a piece of surgery data such that the map $S^{n-1} \to B$ induced
  by $e \circ \ell_M$ is null-homotopic.  Then if $\theta$ is
  reversible, the surgered manifold
  $$\overline{M} = (M - \Int(S^{n-1} \times D^{d-n+1})) \cup_{S^{n-1} \times S^{d-n}} (D^n \times S^{d-n})$$
  admits a $\theta$-structure which agrees with $\ell_M$ on $(M -
  \Int(S^{n-1} \times D^{d-n+1}))$.
\end{proposition}
\begin{proof}
  If we let $V$
  denote the trace of the surgery, then the $\theta$-structure on $M$
  and a choice of null-homotopy of $e \circ \ell_M$ induces a bundle map $TV \to
  \epsilon^1 \oplus \theta^* \gamma$, and by restriction a bundle map
  $\epsilon^1 \oplus T\overline{M} \to \epsilon^1 \oplus \theta^*
  \gamma$, which we can assume agrees with the stabilisation of
  $\ell_M$ on $(M - \Int(S^{n-1} \times D^{d-n+1})) \subset \overline{M}$.  But when $\theta$
  is weakly once-stable, Lemma~\ref{lemma:stable-extension} says that this bundle map can be replaced with one induced from a $\theta$-structure.
\end{proof}

For tangential structures that are once-stable (not just weakly), we
can say that for a $d$-manifold $W$ with a fixed $\theta$-structure $\ell_0 : TW\vert_{\partial W} \to \theta^*\gamma$, the stabilisation map
$$\Bun^\partial(TW, \theta^*\gamma ; \ell_0) \lra \Bun^\partial(\epsilon^1 \oplus TW, \bar{\theta}^*\gamma; \epsilon^1 \oplus \ell_0)$$
is a weak homotopy equivalence. (Weakly once-stable only implies
that this map is 0-connected.) We shall not make explicit use of the
stronger condition in this paper, but point out that most of the
naturally occuring tangential structures \emph{are} once-stable.  In
particular, the following construction will be our main source of
once-stable tangential structures.  Let $W$ be a connected
$d$-dimensional manifold with basepoint, and $\tau : W \to BO(d)$ be
its Gauss map, which we may assume to be pointed. For each $k$ there
are Moore--Postnikov factorisations of $\tau$
$$W \overset{j_k}\lra B_W(k) \overset{p_k}\lra BO(d)$$
where $\pi_*(j_k)$ is an isomorphism for $* < k$ and an epimorphism
for $* = k$, and $\pi_*(p_k)$ is an isomorphism for $* > k$ and a
monomorphism for $*=k$. These connectivity properties characterise
$B_W(k)$, by obstruction theory. Then $\theta_W(k) = p_k$ is a
tangential structure.

\begin{lemma}
  The tangential structure $\theta_W(k) : B_W(k) \to BO(d)$ is
  once-stable for any $k \leq d$.
\end{lemma}
\begin{proof}
  We let $\bar{B}_W(k)$ denote the same Moore--Postnikov construction
  applied to the composition $W \to BO(d) \to BO(d+1)$. The claim then
  follows as $BO(d) \to BO(d+1)$ is
  $d$-connected.
\end{proof}

\begin{remark}
  There do exist tangential structures which are reversible but not
  once-stable, which justifies our emphasis on reversibility.  An
  interesting example is $BU(3) \to BO(6)$, which is reversible as
  $S^6$ admits an almost complex structure, but is not once-stable: if
  it were pulled back from a fibration $f: \bar{B} \to BO(7)$, one can
  easily use the Serre spectral sequence to check that the kernel of
  the map $f^*$ on $\bF_2$-cohomology would be the ideal $ I = (w_1,
  w_3, w_5) \subset H^*(BO(7);\bF_2)$, but this is not closed under
  the action of the Steenrod algebra as $\mathrm{Sq}^4(w_5) = w_4 \cdot w_5 +
  w_3 \cdot w_6 + w_2 \cdot w_7 \nin
  I$.
\end{remark}

\subsection{The standard family}

We will prove Theorem~\ref{thm:MidSurgery} by performing
$(n-1)$-surgery on objects until we reach an object in $\mathcal{A}$,
just as in Section~\ref{sec:surg-objects-below} we performed
$l$-surgery on objects to make them $l$-connected (relative to
$L$).  As in that section, the surgery shall be
performed by gluing in a suitable family of manifolds along certain
families of embeddings, whose existence we shall prove in
Section~\ref{sec:Connectivity}.  The standard family to be glued in is
very similar to that in Section~\ref{sec:surg-objects-below}, where we
started with a certain submanifold $K \subset \R^{d-l} \times
\R^{l+1}$.  In this section, $l = n-1$, so we have a submanifold $K
\subset \R^{n+1} \times D^n$ defined as follows.  We first chose a
smooth function $\rho : \bR \to \bR$ which is the identity on
$(-\infty, \tfrac{1}{2})$, has nowhere negative derivative, and has
$\rho(t)=1$ for all $t \geq 1$, and we let
$$K = \{(x, y) \in \bR^{n+1} \times \bR^n \,\vert \,\,
\vert y \vert^2 = \rho(\vert x \vert ^2-1)\}.$$ The first coordinate
restricts to a Morse function $h = x_1: K \to \R$ with exactly two critical points: $(-1,0,\ldots,0;0)$ and $(+1,0,\ldots,0;0)$ both of index $n$.

In Section~\ref{sec:surg-objects-below}, we constructed from $K$ a
one-parameter family of manifolds $\mathcal{P}_t \subset (-6,-2)
\times \R^{d-l-1} \times \bR^{l+1}$, obtained from $K$ by moving the lowest critical
point down as $t \in [0,1]$ increases, as in
Figure~\ref{fig:Half0-surgeryObjects}.  In this section we shall need
a two-parameter family $\mathcal{P}_{t,w}\subset \R \times (-6,-2)
\times \R^{n} \times \R^n$ which is constructed from $\{0\} \times K$
by moving \emph{both} critical points down as $t \in [0,1]$ increases,
as in Figure~\ref{fig:Full0-surgeryObjects}.  As $w \in [0,1]$
decreases, we shrink the width of the handle so that the distance
between the two critical values is $2w$.  In order for the manifold to
stay embedded in the limit $w = 0$, we need an extra ambient dimension.

Let us first construct a 1-parameter family of submanifolds $K_w
\subset \R \times \R^{n+1} \times D^n$ such that $K_1 = \{0\} \times
K$.  Let $\mu : \bR \to [0,1]$ be a smooth function which is zero on $(2,\infty)$ and
identically 1 on $(-\infty,\sqrt{2})$, and define a 1-parameter family
of embeddings
\begin{eqnarray*}
  \varphi_w : \bR^{n+1} \times D^n & \lra & \bR \times \bR^{n+1} \times D^n\\
  (x,y) & \longmapsto & (x_1(1-w)\mu(\vert x \vert), x_1 (1- (1-w)\mu(\vert x \vert)), x_2, \ldots, x_{n+1},y).
\end{eqnarray*}
We now let
$$K_w = \varphi_w(K) \subset \bR \times \bR^{n+1} \times D^n$$
for $w \in [0,1]$.  A calculation shows that (for $w > 0$) the
critical points of the height function $h : K_w \to \bR$, which is now projection to the second coordinate, are $\varphi_{w}(\pm 1, 0, \ldots, 0)$ and so lie at
heights $\pm w$. They remain Morse of index $n$.

We now define a 2-parameter family of $d$-dimensional submanifolds
$\mathcal{P}_{t,w}$ inside $\bR \times (-6,-2) \times \R^{n} \times D^{n}$ in
much the same way as $\mathcal{P}_t$ was constructed from $K$ in
Section~\ref{sec:ThetaSurgery}.  Apart from the extra width parameter,
the main difference is that in this section we will use a larger part
of $K$, including \emph{both} critical points.  Pick a smooth
one-parameter family of embeddings $\lambda_s: (-6,-2) \to (-6,2)$,
such that $\lambda_0 = \mathrm{id}$, that $\lambda_s\vert_{(-6,-5)} =
\mathrm{Id}$ for all $s$, and that $\lambda_1(-4) = -1$ and
$\lambda_1(-3) = 1$.  Then we get embeddings $\mathrm{Id}_{\bR} \times \lambda_t \times  \mathrm{Id}_{\bR^{2n}}: \bR \times (-6,-2)\times \R^{2n} \to \bR \times (-6,2) \times \R^{2n}$ and define
\begin{equation*}
  \mathcal{P}_{t,w} = (\mathrm{Id}_{\bR} \times \lambda_t \times  \mathrm{Id}_{\bR^{2n}})^{-1}(K_w) \in
  \Psi_d(\bR \times (-6,-2)\times \R^{n}\times \R^n).
\end{equation*}
It is easy to verify that $\mathcal{P}_{t,w}$ agrees with $\{0\} \times (-6,-2) \times \R^{n} \times S^{n-1}$ outside $(-2,2) \times (-5,-2) \times B_{2}^{n}(0)\times D^n$, independently of $t$ and
$w$.

We shall also need a tangentially structured version of this
construction, given a structure $\ell: TK\vert_{(-6,2)} \to \theta^*
\gamma$.  For this purpose, let $\omega=\mu: \R \to [0,1]$ be the function defined above and
define a 1-parameter family of embeddings by
\begin{eqnarray*}
  \psi_t : \bR \times (-6,-2) \times \bR^{n} \times \R^{n} &\lra& \bR \times
  (-6,2) \times \bR^{n} \times \R^{n}\\
  (s; x_1, \ldots, x_{n+1}; y) &\longmapsto& (s;\lambda_{t\omega(\vert x \vert)}(x_1), x_2, \ldots, x_{n+1};  y), 
\end{eqnarray*}
It is easy to see that we also have
$\psi_t^{-1}(K_w) = (\mathrm{Id}_{\bR} \times \lambda_t \times  \mathrm{Id}_{\bR^{2n}})^{-1}(K_w) =
\mathcal{P}_{t,w}$, and we define a $\theta$-structure on
$\mathcal{P}_{t,w}$ by pullback along $\psi_t$.  This gives a
two-parameter family
\begin{equation*}
  \mathcal{P}_{t,w}(\ell) \in \Psi_\theta(\bR \times (-6,-2) \times
  \R^{n}\times \R^n).
\end{equation*}
Let $P : [0,1] \to [0,1]^2$ be
the piecewise linear path with $P(0) = (0,0)$, $P(\frac12) = (1,0)$
and $P(1) = (1,1)$, and define a one-parameter family
\begin{equation*}
  \mathcal{P}_{t}(\ell) = \mathcal{P}_{P(t)}(\ell) \in \Psi_\theta(\bR \times (-6,-2) \times
  \R^{n}\times \R^n).
\end{equation*}
We will omit $\ell$ from the notation when it is unimportant. We record some important properties of this family in
Proposition~\ref{prop:StdFamilyObjectsInMid} below, using the
following definition.
\begin{definition}\label{defn:Extendible} Let $\ell
  : TK \to \theta^*\gamma$ be a $\theta$-structure on $K$.  Recall
  that outside of $\R \times B_2^{n}(0) \times D^n$ the manifold $K$
  agrees with $\R \times \bR^n \times S^{n-1}$.  We say that $\ell$ is
  \emph{extendible} if the $\theta$-structure $\ell \vert_{\R \times
    (\bR^n - B_2^n(0)) \times S^{n-1}}$ extends to a
  $\theta$-structure on the whole of $\R \times \bR^n \times S^{n-1}$.
\end{definition}
\begin{proposition}\label{prop:StdFamilyObjectsInMid}
Suppose $\ell$ is extendible.  The elements $\mathcal{P}_{t}(\ell) \in \Psi_\theta(\bR \times (-6,-2) \times
  \R^{n}\times \R^n)$ are $\theta$-submanifolds of $\bR \times (-6,-2) \times \R^{n}
  \times D^{n}$ satisfying
  \begin{enumerate}[(i)]
  \item $\mathcal{P}_{0}(\ell) = K_1\vert_{(-6,-2)} = \{0\} \times (-6,-2) \times \bR^{n} \times S^{n-1}$ as $\theta$-manifolds.\label{it:StdFamilyObjectsInMid.InitialValue}
  
  \item For all $t$, $\mathcal{P}_{t}(\ell)$ agrees with
    $K_{1}\vert_{(-6,-2)}$ as a $\theta$-manifold, outside of $(-2,2) \times (-5,-2)
    \times B_{2}^{n}(0) \times D^{n}$
    \label{it:StdFamilyObjectsInMid.ConstantOnBoundary}
    
  \item For all $t$ and each pair of regular values $-6 <a < b
    < -2$ of the height function $h : \mathcal{P}_{t} \to \bR$, the
    pair
    \begin{equation*}
      (\mathcal{P}_{t}\vert_{[a,b]}, \mathcal{P}_{t}\vert_b)
    \end{equation*}
    is $(n-1)$-connected.\label{it:StdFamilyObjectsInMid.kConnMoving}
    
  \item Let $a$ be a regular value of $h : \mathcal{P}_t(\ell) \to (-6,-2)$. If $a$ is outside of $(-4,-3)$ then the manifold $\mathcal{P}_t(\ell)\vert_a$ is isomorphic to $\mathcal{P}_0(\ell)\vert_a = \{a\} \times \{0\} \times \R^n \times S^{n-1}$ as a $\theta$-manifold. If $a$ is inside of $(-4,-3)$ then the manifold $\mathcal{P}_t(\ell)\vert_a$ is either isomorphic to $\mathcal{P}_0(\ell)\vert_a$ as a $\theta$-manifold, or is obtained from it by $(n-1)$-surgery along the standard embedding.\label{it:StdFamilyObjectsInMid.lConnMoving}
  
  \item  The critical values of $h:
    \mathcal{P}_{1}(\ell) \to (-6,-2)$ are $-4$ and $-3$.  For $a \in (-4,-3)$, $\mathcal{P}_{1}(\ell)
    \vert_{a}$ is obtained by $(n-1)$-surgery from $\mathcal{P}_{0}(\ell)
    \vert_{a} = \{0\} \times \bR^{n} \times S^{n-1}$ along the
    standard embedding.
    \label{it:StdFamilyObjectsInMid.lConn}
  \end{enumerate}
  In~(\ref{it:StdFamilyObjectsInMid.lConnMoving}) and (\ref{it:StdFamilyObjectsInMid.lConn}), the
  $\theta$-structure on the surgered manifold is determined (up to
  homotopy) by the $\theta$-structure $\ell$ on $K\vert_{(-6,2)}$.
\end{proposition}

The precise meaning of the word isomorphic
in~(\ref{it:StdFamilyObjectsInMid.lConnMoving}) above is the following: By
(\ref{it:StdFamilyObjectsInMid.ConstantOnBoundary}) we know that the
manifolds are \emph{equal} outside $(-2,2) \times (-5,-2) \times
B^{n}_{2}(0) \times D^n$.  Being isomorphic means that the
identity extends to a diffeomorphism which preserves
$\theta$-structures up to a homotopy of bundle maps which is constant
outside $(-2,2) \times (-5,-2) \times B^{n}_{2}(0) \times D^n$.

\begin{proof}
  (\ref{it:StdFamilyObjectsInMid.InitialValue}) and
  (\ref{it:StdFamilyObjectsInMid.ConstantOnBoundary}) follow easily
  from the properties of $\lambda_t$ and $\psi_t$, and the fact that
  $K$ agrees with $\R^{n+1} \times S^{n-1}$ outside $B_{2}^{n+1}(0)
  \times \R^{n}$.  For (\ref{it:StdFamilyObjectsInMid.kConnMoving}),
  the Morse function $\mathcal{P}_{t,w} \to (-6,-2)$ has at most two
  critical point, both of index $n$.  If a critical value is in
  $(a,b)$, then the pair is $(n-1)$-connected, otherwise it is $\infty$-connected as
  $\mathcal{P}_{t,w}\vert_{[a,b]}$ deformation retracts to
  $\mathcal{P}_{t,w}\vert_b$.  
  
  To see (\ref{it:StdFamilyObjectsInMid.lConnMoving}) and (\ref{it:StdFamilyObjectsInMid.lConn}), first suppose that $t \in [0, \frac12]$. Then $P(t) = (?,0)$ so $\mathcal{P}_t(\ell) = \mathcal{P}_{?,0}(\ell)$. The function $K_0 \to \R$ has exactly one critical value, $0$, so
  $K_0\vert_{a}$ is diffeomorphic to $\{a\} \times \R^n \times
  S^{n-1}$ for all regular values $a$, and extendibility
  implies that they are also isomorphic as $\theta$-manifolds. If instead $t \in [\frac12,1]$ then $P(t) = (1,?)$, and so $\mathcal{P}_t(\ell) = \mathcal{P}_{1,?}(\ell)$. The fact that the function $K_w \to \R$
  has exactly two critical points with value $\pm w$ implies that
  $K_w\vert_{a}$ is diffeomorphic to $\{a\} \times \R^n \times
  S^{n-1}$ for regular values $a \in \R- (-1,1) \subset \R- (-w,w)$ and extendibility
  implies that they are also isomorphic as $\theta$-manifolds. When $w=1$, for regular values $a \in (-1,1)$ we have that $K_1\vert_a$ is obtained from $\{a\} \times \R^n \times
  S^{n-1}$ by $(n-1)$-surgery along the standard embedding.
\end{proof}

\subsection{Surgery data}\label{sec:MidDimSurgeryData}

We can now describe the semi-simplicial space of surgery data in the
middle dimension. It is similar to the space of surgery data below the
middle dimension, but taking into account the slightly different range
of definition of the standard family in this case.

Before doing so, we choose once and for all, smoothly in the data
$(a_i, \epsilon_i, a_p, \epsilon_p)$ increasing diffeomorphisms
\begin{equation}
  \label{eq:16}
  \psi = \psi(a_i, \epsilon_i, a_p, \epsilon_p) : (-6,-2) \cong
  (a_i-\epsilon_i, a_p+\epsilon_p)
\end{equation}
sending $[-4, -3]$ linearly onto $[a_i -\tfrac{1}{2}\epsilon_i, a_i +\tfrac{1}{2}\epsilon_i]$.

\begin{definition}\label{defn:YComplexMid}
  Let $x= (a, \epsilon, (W, \ell_W)) \in D^{n-1, n-2}_{\theta,
    L}(\bR^N)_p$, and write $M_i = W\vert_{a_i}$. Define the set
  $Y_q(x)$ to consist of tuples $(\Lambda,\delta,e,\ell)$, where
  $\Lambda$ and $\delta$ are as in
  Definition~\ref{defn:YComplex},
\begin{equation*}
e: \Lambda \times \bR \times (-6,-2) \times \R^{n} \times D^{n} \hookrightarrow \R
    \times (0,1) \times (-1,1)^{N-1}
\end{equation*}
is an embedding, and $\ell$ is a bundle map
  $T(\Lambda \times K) \to \theta^* \gamma$.  (In
  Definition~\ref{defn:YComplex}, it was only defined on $T(\Lambda
  \times K\vert_{(-6,0)}$.)  Define $\Lambda_{i,j}$, $e_{i,j}$ and
  $\ell_{i,j}$ in the same way as in Definition~\ref{defn:YComplex}.
  This data is required to satisfy the following conditions.
  \begin{enumerate}[(i)]
  \item $e^{-1}(W) = \Lambda \times \{0\} \times (-6,-2) \times
    \bR^{n} \times \partial D^{n}$.  We let
    $$\partial e : \Lambda \times \{0\} \times (-6,-2) \times \bR^{n}
    \times \partial D^n \hookrightarrow W$$ denote the embedding
    restricted to the boundary.
    
  \item\label{it:YComplexMidHeight} For $t \in \cup_i (a_i-\epsilon_i,
    a_i+\epsilon_i)$, we have $(x_1 \circ e_{i,j})^{-1}(t) =
    \Lambda_{i,j} \times \bR \times\{t\} \times \bR^{n} \times D^{n}$.
    
  \item\label{it:YComplexMidTheta} The composition $\ell_W \circ D(\partial e): T(\Lambda \times
    K\vert_{(-6,-2)}) \to \theta^* \gamma$ agrees with the restriction
    of $\ell$.
    
  \item For each $\lambda \in \Lambda$, the restriction of $\ell$
    to $T(\{\lambda\} \times K)$ is extendible.
  \end{enumerate}
For each $j$, the data $(e_{i,j}, \ell_{i,j})$ is enough to perform $\theta$-surgery on $M_i$ (as $K\vert_{(-6, 0)}$ is the trace of an $(n-1)$-surgery), and we further insist that
  \begin{enumerate}[(i)]
    \setcounter{enumi}{4}
  \item The resulting $\theta$-manifold $\overline{M}_i$ lies in
    $\mathcal{A}$.\label{it:EnoughSurgeryDataObMid}
  \end{enumerate}
  For each $x$, $Y_\bullet(x)$ is a semi-simplicial set.
\end{definition}

Define a bi-semi-simplicial space $D^{n-1, \mathcal{A}}_{\theta,
  L}(\bR^N)_{\bullet, \bullet}$ (augmented in the second
semi-simplicial direction) from this, as in Definition
\ref{defn:DoubleCxOb}. The main result about this bi-semi-simplicial
space of manifolds equipped with surgery data is the following, whose
proof we defer until Section \ref{sec:Connectivity}.

\begin{theorem}\label{thm:SurgeryComplexObMid}
  Under the assumptions of Theorem~\ref{thm:MidSurgery}, the maps
  \begin{equation*}
    \vert D^{n-1, \mathcal{A}}_{\theta,
      L}(\bR^N)_{\bullet, 0} \vert \lra \vert D^{n-1,
      \mathcal{A}}_{\theta, L}(\bR^N)_{\bullet, \bullet} \vert
     \lra
    \vert D^{n-1, n-2}_{\theta, L}(\bR^N)_{\bullet}\vert
  \end{equation*}
  are weak homotopy equivalences, where the first map is the inclusion
  of 0-simplices and the second is the augmentation, in the second
  simplicial direction.
\end{theorem}

\subsection{Proof of Theorem \ref{thm:MidSurgery}}

The proof of this theorem will be almost identical with that of
Theorem \ref{thm:lfiltration}. Thus, suppose that the
conditions in the statement of Theorem \ref{thm:MidSurgery} are
satisfied, and let $(a, \epsilon, (W, \ell_W), e,\ell) \in D^{n-1,
  \mathcal{A}}_{\theta, L}(\bR^N)_{p,0}$.  For each $i = 0,\dots, p$,
we have an embedding $e_i = e_{i,0}$ and a bundle map $\ell_i =
\ell_{i,0}$, and precisely as in Section \ref{sec:proof-theor-refthm:l} we may construct a one-parameter family of elements $\mathcal{K}^{t}_{e_{i}, \ell_{i}}(W, \ell_W) \in
\Psi_\theta((a_0-\epsilon_0, a_p+\epsilon_p)\times \bR^N)$ for $t \in [0,1]$. From this, for each tuple $t = (t_0, \dots, t_p) \in [0,1]^{p+1}$ we may form the element
$$
\mathcal{K}^{t}_{e, \ell}(W, \ell_W) = \mathcal{K}^{t_p}_{e_{p}, \ell_{p}} \circ \cdots \circ \mathcal{K}^{t_0}_{e_{0}, \ell_{0}}(W, \ell_W) \in \Psi_\theta((a_0-\epsilon_0,
a_p+\epsilon_p)\times \bR^N).
$$
To apply the same proof as that of Theorem \ref{thm:lfiltration}, we
need an analogue of Lemma \ref{lem:ObSurgeryDesiredEffect} to tell us
how the manifold improves when we apply the various surgery
operations.

\begin{lemma}\label{lem:ObSurgeryDesiredEffectMid}
  Firstly, the tuple $(a, \tfrac{1}{2}\epsilon, \mathcal{K}^{t}_{e,
    \ell}(W, \ell_W))$ is an element of $X_p^{n-1, n-2}$. Secondly, if
  $t_i$ is $1$---so the surgery for the regular value $a_i$
  is fully done---then for each regular value $b \in
  (a_i-\tfrac{1}{2}\epsilon_i, a_i+\tfrac{1}{2}\epsilon_i)$ of $x_1 :
  \mathcal{K}^{t}_{e, \ell}(W, \ell_W) \to \bR$, the
  $\theta$-manifold $\mathcal{K}^{t}_{e, \ell}(W, \ell_W)
  \vert_{b}$ lies in $\mathcal{A}$.
\end{lemma}
\begin{proof}
  For the first part we must verify the conditions of Definition
  \ref{defn:X}. This part of the argument of Lemma
  \ref{lem:ObSurgeryDesiredEffect} applies equally well when
  $\kappa=n-1$, $l=n-1$.

  For the second part, we suppose $t_i=1$. Let $b \in
  (a_i-\tfrac{1}{2}\epsilon_i, a_i+\tfrac{1}{2}\epsilon_i)$ be a
  regular value of the height function on $\mathcal{K}^{t}_{e,
    \ell}(W, \ell_W)$ and define $\theta$-manifolds
  \begin{eqnarray*}
    \overline{M} & = & \big(\mathcal{K}^{t}_{e, \ell}(W,
    \ell_W)\big)\vert_b\\
    \widetilde{M} & = & \big(\mathcal{K}^{t_i}_{e_{i}, \ell_{i}}(W,
    \ell_W)\big)\vert_b\\ 
    M & = & W\vert_b.
  \end{eqnarray*}
  By Definition \ref{defn:YComplexMid} (\ref{it:EnoughSurgeryDataObMid}), performing surgery on $M$ using the data
  $(e_{i}, \ell_{i})$ gives a $\theta$-manifold in
  $\mathcal{A}$. By Proposition \ref{prop:StdFamilyObjectsInMid} (\ref{it:StdFamilyObjectsInMid.lConn}), $\mathcal{K}^{t_i}_{e_{i}, \ell_{i}}(W,
  \ell_W)$ has this surgery done, so $\widetilde{M}$ lies in
  $\mathcal{A}$. Now $\overline{M}$ is obtained from $\widetilde{M}$
  by applying the remaining operations $\mathcal{K}^{t_j}_{e_{j}, \ell_{j}}$
  for $j \neq i$, but by Proposition \ref{prop:StdFamilyObjectsInMid} (\ref{it:StdFamilyObjectsInMid.lConnMoving}), applying each of these only changes $\widetilde{M}$ up to isomorphism (because $b \in
  (a_i-\tfrac{1}{2}\epsilon_i, a_i+\tfrac{1}{2}\epsilon_i)$, so it is not in $(a_j-\tfrac{1}{2}\epsilon_j, a_j+\tfrac{1}{2}\epsilon_j)$), so $\overline{M}$ lies in $\mathcal{A}$.
\end{proof}

As in Section \ref{sec:proof-theor-refthm:l} we define a map
\begin{eqnarray*}
  \surg_p : [0,1]^{p+1} \times D^{n-1,
    \mathcal{A}}_{\theta, L}(\bR^N)_{p,0} & \lra & X_p^{n-1, n-2}\\
  \big(t, (a, \epsilon, (W, \ell_W),e,\ell)\big) & \longmapsto & (a,
  \tfrac{1}{2}\epsilon, \mathcal{K}^{t}_{e, \ell}(W, \ell_W)),
\end{eqnarray*}
which has the desired range by the first part of Lemma
\ref{lem:ObSurgeryDesiredEffectMid}. The argument of Section \ref{sec:proof-theor-refthm:l} gives maps
\begin{equation*}
F_p : [0,1] \times D_{\theta, L}^{n-1, \mathcal{A}}(\bR^N)_{p,0} \times \Delta^p \lra X_p^{n-1, n-2} \times \Delta^p
\end{equation*}
which glue to a homotopy $\surg : [0,1]\times \vert
  {D}^{n-1, \mathcal{A}}_{\theta, L}(\bR^N)_{\bullet,0} \vert \to
  \vert X_\bullet^{n-1, n-2} \vert$ which is constant on the subspace $\vert D^{n-1,
  \mathcal{A}}_{\theta, L}(\bR^N)_{\bullet} \vert \hookrightarrow
\vert D^{n-1, \mathcal{A}}_{\theta, L}(\bR^N)_{\bullet,0} \vert$ of manifolds equipped with no surgery data. It also provides a factorisation of the map $\surg(1,-)$ through the continuous injection $\vert X_\bullet^{n-1, \mathcal{A}} \vert \to  \vert X_\bullet^{n-1, n-2} \vert$. The
argument in Section~\ref{sec:proof-theor-refthm:k} then gives the
weak equivalence in Theorem~\ref{thm:MidSurgery}.

%%% Local Variables: 
%%% mode: latex
%%% TeX-master: "Moduli"
%%% End: 

\section{Contractibility of spaces of surgery data}\label{sec:Connectivity}

In order to finish the proofs of the results of the last three
sections, we must supply proofs of Theorems
\ref{thm:SurgeryComplexMor}, \ref{thm:SurgeryComplexOb} and
\ref{thm:SurgeryComplexObMid} concerning the bi-semi-simplicial spaces
of manifolds equipped with surgery data.  For convenience we assume
that the domain $B$ of the map $\theta$ defining the tangential
structure is path connected.  In the category
$\mathcal{C}_\theta^{\kappa,l}$, this implies that objects are path
connected as long as $l > -1$, and morphisms are path connected if in
addition $\kappa > -1$.

\subsection{The first part of Theorems \ref{thm:SurgeryComplexOb} and
  \ref{thm:SurgeryComplexObMid}}\label{sec:ProofSecondPart}

Theorems
\ref{thm:SurgeryComplexOb} and \ref{thm:SurgeryComplexObMid} both
assert that two maps are weak equivalences.  In either theorem, the
proof for the first map will use that the second map is a weak
equivalence, but is otherwise simpler, so we first consider the maps
$$\vert D^{\kappa, l}_{\theta, L}(\bR^N)_{\bullet, 0} \vert \lra \vert
D^{\kappa, l}_{\theta, L}(\bR^N)_{\bullet, \bullet} \vert
\quad\quad\quad \vert D^{n-1, \mathcal{A}}_{\theta,
  L}(\bR^N)_{\bullet, 0} \vert \lra \vert D^{n-1,
  \mathcal{A}}_{\theta, L}(\bR^N)_{\bullet, \bullet} \vert.$$

\begin{proof}[Proof (assuming the second part)]
  The proof in both cases is the same, so let us write $D_{\bullet,
    \bullet}$ for either $D^{\kappa, l}_{\theta, L}(\bR^N)_{\bullet,
    \bullet}$ or $D^{n-1, \mathcal{A}}_{\theta, L}(\bR^N)_{\bullet,
    \bullet}$. We define, for this proof only, a bi-semi-simplicial
  space $D'_{\bullet, \bullet}$ in the same way as $D_{\bullet,
    \bullet}$ except that the usual inequalities $a_i + \epsilon_i <
  a_{i+1}-\epsilon_{i+1}$ and $\epsilon_i > 0$ are replaced by $a_i
  \leq a_{i+1}$ and $\epsilon_i \geq 0$ (so the intervals
  $[a_i-\epsilon_i, a_i+\epsilon_i]$ are allowed to overlap).

  The inclusion $D_{\bullet, \bullet} \hookrightarrow D'_{\bullet,
    \bullet}$ is easily seen to be a levelwise weak homotopy
  equivalence, by spreading the $a_i$ out and making the $\epsilon_i$
  positive but small, so it is enough to work with $D'_{\bullet,
    \bullet}$ throughout and show that $\vert D'_{\bullet, 0} \vert
  \to \vert D'_{\bullet, \bullet} \vert$ is a weak homotopy
  equivalence.

To do so, we describe a retraction $r : \vert D'_{\bullet, \bullet}
\vert \to \vert D'_{\bullet, 0} \vert$ which will be a weak homotopy
inverse to the inclusion. The map $r$ does not change the underlying
manifold $W \in \psi_\theta(N+1, 1)$, but only modifies the $a_i$ and
barycentric coordinates. There is a map
$$D'_{p, q} \lra D'_{(p+1)(q+1)-1, 0}$$
given by considering $(p+1)$ regular values, each equipped with
$(q+1)$ pieces of surgery data, as $(p+1)(q+1)$ not-necessarily
distinct regular values, each with a \emph{single} piece of surgery
data. There is also a map $\Delta^p \times \Delta^q \to
\Delta^{(p+1)(q+1)-1} \subset \bR^{(p+1)(q+1)}$ with $(j + (q+1)i)$th
coordinate given by $(t, s) \mapsto s_i t_j$. Taking
the product of these maps gives
$$r_{p,q} : D'_{p, q} \times \Delta^p \times \Delta^q \lra D'_{(p+1)(q+1)-1, 0} \times \Delta^{(p+1)(q+1)-1}$$
which glue together to give the map $r: |D'_{\bullet, \bullet}| \to
|D'_{\bullet,0}|$.  It is clear that $r$ is a retraction (i.e.\ left
inverse to the inclusion), so the induced map on homotopy groups is
surjective.  To see that it is injective, we use the map
$|D'_{\bullet,\bullet}| \to |D'_{\bullet}|$ induced by the
augmentation in the second bi-semi-simplicial direction (by forgetting
all surgery data).  This is a weak equivalence by the second part of
Theorem \ref{thm:SurgeryComplexOb} or \ref{thm:SurgeryComplexObMid}
respectively, but it clearly factors as
$$\vert D'_{\bullet, \bullet} \vert \overset{r}\lra \vert D'_{\bullet,
  0} \vert \lra \vert D'_{\bullet} \vert,$$
where the second map is
again induced by the augmentation in the second bi-semi-simplicial
direction.  Therefore $r$ is also injective on homotopy groups, and
hence a weak homotopy equivalence.
\end{proof}

\subsection{A simplicial technique}

In order to give the proofs of Theorems \ref{thm:SurgeryComplexMor}, \ref{thm:SurgeryComplexOb} and \ref{thm:SurgeryComplexObMid}, we need a technique for showing that for certain augmented semi-simplicial spaces $X_\bullet \to X_{-1}$, the map $\vert X_\bullet \vert \to X_{-1}$ is a weak homotopy equivalence. The semi-simplicial spaces occurring in those theorems are all of the following special type.

\begin{definition}
  Let $X_\bullet \to X_{-1}$ be an augmented semi-simplicial space. We
  say it is an \emph{augmented topological flag complex} if
  \begin{enumerate}[(i)]
  \item The map $X_n \to X_0 \times_{X_{-1}} X_0 \times_{X_{-1}}
    \cdots \times_{X_{-1}} X_0$ to the $(n+1)$-fold product---which
    takes an $n$-simplex to its $(n+1)$ vertices---is a homeomorphism
    onto its image, which is an open subset.
    
  \item A tuple $(v_0, \ldots, v_n) \in X_0 \times_{X_{-1}} X_0
    \times_{X_{-1}} \cdots \times_{X_{-1}} X_0$ lies in $X_n$ if and
    only if $(v_i, v_j) \in X_1$ for all $i < j$.
  \end{enumerate}
  If elements $v$, $w \in X_0$ lie in the same fibre over $X_{-1}$ and
  $(v,w) \in X_1$, we say $w$ is \emph{orthogonal} to $v$.  (We do not
  require the relation to be symmetric, although in our
  applications it will be.)  If $X_{-1}=*$ we omit
  the adjective augmented.
\end{definition}

The semi-simplicial space $Z_\bullet(a, \epsilon, (W,\ell_W)) \to *$
from Definition \ref{defn:ZComplex} and the semi-simplicial spaces
$Y_\bullet(a, \epsilon, (W, \ell_W)) \to *$ from Definitions
\ref{defn:YComplex} and \ref{defn:YComplexMid} are topological flag
complexes. Furthermore, ${D}^{\kappa}_{\theta, L}(\bR^N)_{p,\bullet}
\to {D}^{\kappa-1}_{\theta, L}(\bR^N)_{p}$ from Definition
\ref{defn:ZDComplex}, $D^{\kappa, l}_{\theta, L}(\bR^N)_{p, \bullet}
\to D^{\kappa, l-1}_{\theta, L}(\bR^N)_{p}$ from Definition
\ref{defn:DoubleCxOb} and $D^{n-1, \mathcal{A}}_{\theta, L}(\bR^N)_{p,
  \bullet} \to D^{n-1, n-2}_{\theta, L}(\bR^N)_{p}$ from Section
\ref{sec:MidDimSurgeryData} are all augmented topological flag
complexes. In all cases this is immediate from the definition:
firstly, a $p$-simplex of these semi-simplicial spaces consists of
$(p+1)$-tuples of surgery data, which are each 0-simplices; secondly,
the pieces of surgery data are subject to the requirement that they
are all disjoint, but disjointness is a property that can be verified
pairwise.

\begin{theorem}\label{thm:SimplicialTech}
  Let $X_\bullet \to X_{-1}$ be an augmented topological flag
  complex. Suppose that
  \begin{enumerate}[(i)]
  \item\label{item:4} The map $\epsilon : X_0 \to X_{-1}$ has local
    sections (in the strong sense that given any $x \in X_0$, there is
    a neighbourhood $U \subset X_{-1}$ of $\epsilon(x)$ and a
    section $s: U \to X_0$ with $s(\epsilon(x)) = x$).
  \item \label{it:2a} $\epsilon: X_0 \to X_{-1}$ is
    surjective.
  \item \label{it:2b} For any $p \in X_{-1}$ and any (non-empty)
    finite set $\{v_1, \ldots, v_n\} \subset \epsilon^{-1}(p)$ there
    exists a $v \in \epsilon^{-1}(p)$ with $(v_i,v) \in X_1$ for all
    $i$.
  \end{enumerate}
  Then $\vert X_\bullet \vert \to X_{-1}$ is a weak homotopy equivalence.
\end{theorem}
Condition (\ref{it:2a}) can be viewed as the special case $n=0$ of
condition~(\ref{it:2b}), but we prefer to keep the cases $n=0$ and $n
> 0$ separate.

\begin{remark}
  \label{rem:flag-complex-easy-case-of-theorem}
  To motivate the proof of this theorem, let us first consider the
  case where $X_{-1}=*$ and that each $X_i$ is discrete, so $\vert
  X_\bullet \vert$ has the structure of a $\Delta$-complex. Then any
  map $f: S^n \to \vert X_\bullet \vert$ may be homotoped to be
  simplicial, for some triangulation of $S^n$, and so hits finitely
  many vertices $v_1, \ldots, v_k$. By (\ref{it:2b}) there exists a $v
  \in X_0$ such that $(v_i,v)$ is a 1-simplex for all $i$. But then
  the map $f$ extends to the join
  \begin{equation*}
    f * \{v\} :S^n * \{v\} \lra \vert X_\bullet \vert
  \end{equation*}
  and so $f$ is null-homotopic.
\end{remark}
The proof we give below follows this in spirit, although is
necessarily more complicated when the $X_i$ carry a topology.  To deal
with the topology, we require the following technical result.

\begin{proposition}\label{prop:microfib}
  Let $Y_\bullet$ be a semi-simplicial set, and $X$ be a Hausdorff
  space. Let $Z_\bullet \subset Y_\bullet \times X$ be a
  sub-semi-simplicial set which in each degree is an open subset.  For
  $x \in X$, let $Z_\bullet(x) \subset Y_\bullet$ be the
  sub-semi-simplicial set defined by $Z_\bullet \cap (Y_\bullet \times
  \{x\}) = Z_\bullet(x) \times \{x\}$ and suppose that
  $|Z_\bullet(x)|$ is contractible for all $x \in X$.  Then the map
  $\pi: |Z_\bullet| \to X$ is a Serre fibration with contractible
  fibres.
\end{proposition}
\begin{proof}
  This follows from \cite[Proposition 2.7]{Stability} and \cite[Lemma
  2.2]{WeissCatClass}.
\end{proof}
\begin{corollary}\label{corlem:AdmitsSection}
  Let $\Omega$ be a set and $X$ a Hausdorff space and let
  \begin{equation*}
    P \subset \bN \times \Omega \times X
  \end{equation*}
  be a subset which is open (when $\bN$ and $\Omega$ is given the
  discrete topology) and such that the projection $P \to \bN \times X$
  is surjective.  We give $\bN \times \Omega \times X$ the partial
  order defined by
  \begin{equation*}
    \text{$(n,\alpha,x) < (m,\beta,y)$ iff $n < m$ and $x = y$}
  \end{equation*}
  and give the subspace $P$ the induced order.  Then the natural map
  $\pi: \vert N_\bullet P \vert \to X$ is a Serre fibration with
  contractible fibres.
\end{corollary}
\begin{proof}
  We apply Proposition~\ref{prop:microfib} with $Y_\bullet = N_\bullet
  (\bN \times \Omega)$ and $Z_\bullet = N_\bullet P$.  For $x \in X$,
  the semi-simplicial subset $Z_\bullet(x) \subset N_\bullet(\bN \times
  \Omega)$ is contractible by the argument in
  Remark~\ref{rem:flag-complex-easy-case-of-theorem}.
\end{proof}

\begin{proof}[Proof of Theorem \ref{thm:SimplicialTech}]
  We begin with an element of the relative homotopy group of the pair
  of spaces $(X_{-1}, \vert X_\bullet \vert)$,
  \begin{equation*}
  \xymatrix{
  \partial D^k \ar[r]^{\hat{f}} \ar[d] & \vert X_\bullet \vert \ar[d]^-{\epsilon}\\
    D^k \ar[r]^{f} & X_{-1}
  }
  \end{equation*}
  and we will show that after changing $\hat{f}$ by a fibrewise
  homotopy there is a diagonal map $D^k \to \vert X_\bullet \vert$
  making the diagram commutative, so the homotopy class is
  trivial.
  
  We first explain how to construct a continuous map $D^k \to
  |X_\bullet|$, making the lower triangle commute, ignoring the upper
  triangle for a moment.  To do this we first pick an infinite set
  $\Omega$ (topologised discretely) and note that it suffices to find
  open sets $P_n \subset \Omega \times D^k$ together with maps $g_n:
  P_n \to X_0$ with the properties that the projection $\pi_n: P_n \to
  D^k$ is surjective, that $\epsilon \circ g_n = f \circ \pi_n$, and
  that for all $x \in D^k$ and $n < m$, any $p \in \pi_n^{-1}(x)$ and
  $q \in \pi_m^{-1}(x)$ have $(g_n(p), g_m(q)) \in X_1$.  Namely,
  given such $(P_n,g_n)$ we can let $P = \cup \{n\} \times P_n
  \subset \bN \times \Omega \times D^k$ and assemble the $g_n$ to a
  simplicial map $g: N_\bullet P \to X_\bullet$.  By
  Corollary~\ref{corlem:AdmitsSection}, the map $\pi: |N_\bullet P|
  \to D^k$ is a Serre fibration with contractible fibres, so we may
  pick a section $s: D^k \to |N_\bullet P|$.  Then the composition
  $|g| \circ s: D^k \to |X_\bullet|$ gives a diagonal map in the
  diagram, making the lower triangle commute.

  The $(P_n,g_n)$ will be constructed by an inductive procedure, for
  which it is useful to construct a slightly stricter structure.
  If we write $\overline{P}_n \subset \Omega \times D^k$ for the
  closure, we will demand an extension $\overline{g}_n: \overline{P}_n \to X_0$
  satisfying
  \begin{enumerate}[(i)]
  \item\label{item:17} the projection $\overline{\pi}_n:
    \overline{P}_n \to D^k$ is proper, and the restriction $\pi_n: P_n
    \to D^k$ is surjective,
  \item\label{item:18} $\epsilon \circ \overline{g}_n = f \circ
    \overline{\pi}_n$,
  \item\label{item:19} for all $x \in D^k$ and $n < m$, any $p \in
    \overline{\pi}_n^{-1}(x)$ and $q \in \overline{\pi}_m^{-1}(x)$
    have $(\overline{g}_n(p), \overline{g}_m(q)) \in X_1$.
  \end{enumerate}
  The properness of $\overline{\pi}_n$ is equivalent to the
  compactness of $\overline{P}_n$, which in turn is equivalent to the
  image of $\overline{P}_n$ in $\Omega$ being finite.  For the
  construction, we first pick for each $x \in D^k$ an element $g_x(x)
  \in \epsilon^{-1}(f(x)) $ which is orthogonal to each element of the
  finite set $\cup_{i < n} \overline{g}_i(\overline{\pi}_i^{-1}(x))$,
  as is possible by assumption.  Then, since $\epsilon$ has local
  sections, we can extend to a map $g_x: V_x \to X_0$ which is a lift
  of $D^k \to X_{-1}$, defined on a neighbourhood $V_x$ of $x$. The
  maps
  \begin{equation*}
    \overline{g}_i \times g_x: \overline{P}_i \times_{D^k} V_x \lra
    X_0 \times_{X_{-1}} X_0
  \end{equation*}
  for $i < n$ all send $\overline{P}_i \times_{D^k} \{x\}$ into the
  open subset $X_1$, so by properness of $\overline{\pi}_i$ we can
  ensure that all these maps have image in $X_1$, after perhaps
  shrinking the open set $V_x$.  If we let $U_x \subset V_x$ be a
  smaller neighbourhood of $x$ with $\overline{U}_x \subset V_x$, then
  $g_x$ restricts to a continuous map $\overline{U}_x \to X_0$.  The
  sets $U_x$ give an open cover of $D^k$, and we let $U_{x_1}, ...,
  U_{x_m}$ be a finite subcover.  Finally, we pick distinct $\omega_1,
  \dots, \omega_m \in \Omega$, disjoint from the image of $\cup_{i <
    n} P_i \to \Omega$ and let
  \begin{equation*}
    P_n = \bigcup_{i=1}^m \{\omega_i\} \times U_{x_i}
    \subset \Omega \times  D^k
  \end{equation*}
  and define the map $\overline{g}_n: \overline{P}_n \to X_0$ by
  $\overline{g}_n(\omega_i,y) = g_{x_i}(y)$.  The sequence of
  $(P_n,\overline{g}_n)$ thus constructed will satisfy the
  properties~(\ref{item:17}), (\ref{item:18}) and (\ref{item:19})
  above, and hence gives a lift $D^k \to |X_\bullet|$.

  The construction of the lift $D^k \to |X_\bullet|$ so far has not
  used the given $\hat{f}$ in any way, so we should not expect the
  upper triangle to commute.  We shall add an extra step preceding the
  above inductive construction of $(P_n,\overline{g}_n)$, in order to
  fix this.  Namely, we shall construct a compact subset
  $\overline{P}_{-1} \subset \Omega \times \partial D^k$ and a
  continuous $\overline{g}_{-1}: \overline{P}_{-1} \to X_0$ such that
  after changing $\hat{f}$ by a fibrewise homotopy, all vertices of
  $\hat{f}(x)$ are contained in
  $\overline{g}_{-1}(\overline{\pi}_{-1}^{-1}(x))$, where
  $\overline{\pi}_{-1}: \overline{P}_{-1} \to \partial D^k$ again
  denotes the projection.  In the inductive construction of $(P_n,
  \overline{g}_n)$, we can then ensure that condition~(\ref{item:19})
  above is satisfied also for $n=-1$ when $x \in \partial D^k$.  Then
  all vertices of $\hat{f}(x)$ will be orthogonal to all vertices of
  $|g| \circ s(x)$, so these two points are connected by the straight
  line inside the join of the simplices that contain them.  These
  straight lines then assemble to a fibrewise homotopy between $\hat
  f$ and $|g| \circ s\vert_{\partial D^k}$.

  To construct $P_{-1}$ and $\overline{g}_{-1}$ we shall, for the rest
  of this proof, replace the usual coordinates $(t_0, \dots, t_p) \in
  \Delta^k$ (which are non-negative numbers with $\sum t_i = 1$) by
  the coordinates $s_i = t_i/\max(t_i)$ (which are non-negative
  numbers with $\max(s_i) = 1$).  Points in $|X_\bullet|$ are then
  written as $(y,s_0, \dots, s_q)$ where $y \in X_q$, $s_i \geq 0$ and
  $\max{s_i} = 1$.  For $t \in (0,1)$ we shall write $U_{t} \subset
  |X_\bullet|$ for the subset where no $s_i$ is equal to $t$.  There
  is a function $U_t \to \amalg_{p} X_p \times \Delta^p$ which
  to $(y,s_0, \dots, s_q) \in |X_\bullet|$ associates
  $(\theta^*(y),s_{\theta(0)}, \dots, s_{\theta(p)})$, where $\theta:
  [p] \to [q]$ is the order-preserving monomorphism defined as the
  composition of the unique order-preserving bijection $[p] \cong \{i
  \in [q] | s_i > t\}$ and the inclusion to $[q]$, and $\theta^*:X_q
  \to X_p$ is the corresponding face map.  It is easy to verify that
  $U_t \subset |X_\bullet|$ is open (in the usual quotient topology
  from $\amalg_{p} X_p \times \Delta^p \to |X_\bullet|$) and that the
  function $U_{t} \to \amalg_{p} X_{p}\times \Delta^p$ is
  continuous (when $U_{t} \subset |X_\bullet|$ is given the subspace
  topology).  Furthermore it is clear that any infinite collection of
  numbers $t \in (0,1)$ will give a cover of $|X_\bullet|$ by the
  corresponding $U_t$'s.

  Proceeding with the construction of $P_{-1}$ and
  $\overline{g}_{-1}$, we may cover $\partial D^k$ by the open sets
  $\hat{f}^{-1}(U_t)$ with $t \in (0,\frac 12)$, and hence by
  compactness we can find a finite cover of $\partial D^k$ by open
  sets $U_i \subset \partial D^k$ such that each closure
  $\overline{U}_i$ is contained in some $\hat{f}^{-1}(U_t)$, and hence
  we get a continuous map $\overline{U}_i \to \amalg X_p$.  Writing
  $\overline{U}_{i,p}$ for the subspace mapping into $X_p \subset
  X_0^{p+1}$, we get a continuous adjoint $\overline{g}_{i,p}: [p]
  \times \overline{U}_{i,p} \to X_0$ such that $\overline{g}_{i,p}([p]
  \times \{x\}) \subset X_0$ consists of the vertices of
  $\hat{f}(x)$ with simplicial coordinate greater than $t_i$.  In
  particular, this set contains all vertices of $\hat{f}(x)$ with
  simplicial coordinate $\geq \frac12$.  We can then pick an injection
  $\amalg_{i,p} [p] \to \Omega$ and let $\overline{P}_{-1}$ be the
  image of the resulting embedding
  \begin{equation*}
    \coprod_{i,p} [p] \times \overline{U}_{i,p} \lra \Omega
    \times \partial D^k,
  \end{equation*}
  and assemble the $\overline{g}_{i,p}$ to a map $\overline{g}_{-1}:
  \overline{P}_{-1} \to X_0$.  After changing $\hat{f}$ by composing
  with the self-map of $|X_\bullet|$ which replaces all simplicial
  coordinates $s_i$ by $\max(0,2s_i-1)$ (which is obviously continuous
  and fibrewise homotopic to the identity), the finite set
  $\overline{g}_{-1} (\overline{\pi}_{-1}^{-1}(x)) \subset X_0$
  contains all vertices of $\hat{f}(x)$, as required.

  After replacing the simplicial coordinates of $\hat{f}$ as
  described, the restriction to $U_{i,p} \subset \partial D^k$
  factors through $X_p \times \Delta^p \to |X_\bullet|$.  This implies
  that the linear homotopy from $\hat{f}$ to $|g| \circ
  s\vert_{\partial D^k}$ is continuous on each $U_{i,p}$.
\end{proof}

\subsection{Proof of Theorem \ref{thm:SurgeryComplexMor}}
\label{sec:proof-theor-refthm:s}

Recall that this theorem states that the augmentation
$$D^\kappa_{\theta, L}(\bR^N)_{\bullet, \bullet} \lra
D^{\kappa-1}_{\theta, L}(\bR^N)_{\bullet}$$ induces a weak homotopy
equivalence after geometric realisation, as long as the conditions of Theorem \ref{thm:kappafiltration} are satisfied. In fact, we only require the following weaker set of conditions:
\begin{enumerate}[(i)]
\item $2\kappa \leq d-1$,

\item $\kappa+1+d < N$,

\item $L$ admits a handle decomposition only using handles of index $< d-\kappa-1$.
\end{enumerate}

We will use Theorem~\ref{thm:SimplicialTech} to
prove that for each $p$ the augmentation map induces a weak equivalence
\begin{equation*}
  |D^\kappa_{\theta,L}(\bR^N)_{p,\bullet}| \lra D^{\kappa-1}_{\theta,L}(\bR^N)_p.
\end{equation*}
Theorem~\ref{thm:SimplicialTech} does not apply directly to the
augmentation $D^\kappa_{\theta,L}(\bR^N)_{p,\bullet} \to
D^{\kappa-1}_{\theta,L}(\bR^N)_{p}$, but we will show that it does apply
after replacing with weakly equivalent spaces.

Recall that an element of $D^\kappa_{\theta,L}(\bR^N)_{p,q}$ consists
of an element $(a,\epsilon, (W, \ell_W)) \in D_{\theta,L}^{\kappa-1}(\R^N)_p$,
together with an element $(\Lambda, \delta, e) \in Z_q(a,\epsilon,(W,\ell_W))$,
where $\Lambda \subset \Omega$ is a finite set equipped with a map
$\delta: \Lambda \to [p]^\vee \times [q] = \{0,\dots, p+1\} \times
\{0, \dots, q\}$ and $e$  is an embedding $e: \Lambda \times \overline{V} \hookrightarrow \R \times (0,1) \times (-1,1)^{N-1}$.

\begin{definition}
  The \emph{core} of $\overline{V}$ is the submanifold $C = [-2,0] \times
  D^\kappa \times \{0\} \subset \overline{V} = [-2,0] \times
  \R^\kappa \times \R^{d-\kappa}$.  Let $\widetilde{Z}_\bullet(a,
  \epsilon, (W,\ell_W))$ be the semi-simplicial space defined as in Definition
  \ref{defn:ZComplex} except that instead of demanding that $e:
  \Lambda \times \overline{V} \to \R \times (0,1) \times (-1,1)^{N-1}$ be an embedding,
  we demand only it be a smooth map which restricts to an embedding of
  a neighbourhood of $\Lambda \times C$.  We still require that $e$
  satisfy the numbered conditions listed in
  Definition~\ref{defn:ZComplex}.  Let $\widetilde{D}^\kappa_{\theta,
    L}(\bR^N)_{\bullet, \bullet} \to D^{\kappa-1}_{\theta,
    L}(\bR^N)_{\bullet}$ be the augmented bi-semi-simplicial space
  defined as in Definition \ref{defn:ZDComplex}, but using
  $\widetilde{Z}_\bullet(x)$ instead of ${Z}_\bullet(x)$.
\end{definition}

\begin{proposition}\label{prop:TildeIncIsEq}
The inclusion ${D}^\kappa_{\theta, L}(\bR^N)_{\bullet, \bullet} \hookrightarrow \widetilde{D}^\kappa_{\theta, L}(\bR^N)_{\bullet, \bullet}$ induces a weak homotopy equivalence in each bidegree, and so on geometric realisation.
\end{proposition}
\begin{proof}
  It is easy to see that there is an isotopy of embeddings $j_t: \overline{V} \to
  \overline{V}$, $t \in [1,\infty)$, such that $j_1 = \mathrm{Id}$, $j_t\vert_C =
  \mathrm{Id}$ for all $t$ and $j_t(\overline{V})$ is contained in the
  $(1/t)$-neighbourhood of $C$ for large $t$, and also such that every
  $j_t$ preserves the submanifold $\Int(\partial_- D^{\kappa+1}) \times
  \bR^{d-\kappa}$ and preserves the height function $h: \overline{V} \to [-2,0]$.

  Precomposing the embedding $e: \Lambda \times \overline{V} \to
  \R\times (0,1) \times (-1,1)^{N-1}$ with the maps $\mathrm{Id}_{\Lambda} \times j_t$ induces a
  deformation $[1,\infty) \times \widetilde{Z}_q(a,\epsilon,(W,\ell_W)) \to
  \widetilde{Z}_q(a,\epsilon,(W,\ell_W))$ and in turn $[1,\infty) \times
  \widetilde{D}^\kappa_{\theta,L}(\R^N)_{p, q} \to
  \widetilde{D}^\kappa_{\theta,L}(\R^N)_{p, q}$.  Elements of
  $\widetilde{Z}_q(a,\epsilon,(W,\ell_W))$ have disjoint cores, so in a compact
  family $K \to \widetilde{D}^\kappa_{\theta,L}(\R^N)_{p, q}$, there
  exists an $\epsilon > 0$ such that the $\epsilon$-neighbourhoods of
  all cores are also disjoint.  Composing with the deformation of
  $\widetilde{D}^\kappa_{\theta,L}(\R^N)_{p, q}$, the map from $K$
  will eventually deform into ${D}^\kappa_{\theta,L}(\R^N)_{p, q}$.
  It follows easily from this that the relative homotopy groups
  vanish.
\end{proof}

In order to prove Theorem \ref{thm:SurgeryComplexMor}, we will show that for each $p$ the map
$$\widetilde{D}^\kappa_{\theta, L}(\bR^N)_{p, \bullet} \lra D^{\kappa-1}_{\theta, L}(\bR^N)_{p}$$
is a weak homotopy equivalence after geometric realisation, by
applying Theorem \ref{thm:SimplicialTech}. Hence we must verify the
conditions of that theorem.  First we establish
condition~(\ref{item:4}).

\begin{proposition}\label{prop:LocalSections}
  The map $\widetilde{D}^\kappa_{\theta, L}(\bR^N)_{p, 0} \to
  D^{\kappa-1}_{\theta, L}(\bR^N)_{p}$ has local sections.
\end{proposition}
\begin{proof}
  Let's consider a point $x \in
  \widetilde{D}^\kappa_{\theta,L}(\R^N)_{p,0}$, given by elements $(a,
  \epsilon, (W,\ell_W)) \in D^{\kappa-1}_{\theta, L}(\bR^N)_{p}$ and
  $(\Lambda, \delta, e) \in \widetilde{Z}_0(a, \epsilon, (W,\ell_W))$.  Choose $t_0 <
  a_0-\epsilon_0$ and $t_1 > a_p+\epsilon_p$ which are regular values
  for $x_1 : W \to \bR$, and such that $(x_1 \circ
  e)(\Lambda \times \overline{V}) \subset (t_0, t_1)$. There is a
  (connected) open neighbourhood $U \subset D^{\kappa-1}_{\theta,
    L}(\bR^N)_{p}$ of $(a, \epsilon, (W,\ell_W))$ on which the $t_i$ remain
  regular values.

  The inclusion $U \hookrightarrow  D^{\kappa-1}_{\theta,
    L}(\bR^N)_{p}$ has graph $\Gamma \subset U \times \R^N$.  All
  fibres of the projection $\pi: \Gamma \vert_{[t_0, t_1]} \to U$ are
  diffeomorphic to the same manifold $M = W\vert_{[t_0,t_1]}$.
  Sending a point in $U$ to its fibre defines a function
  \begin{eqnarray*}
    F: U & \lra & \mathrm{Emb}_\partial(M, [t_0, t_1] \times (-1,1)^N) / \Diff(M)\\
    u & \longmapsto & \pi^{-1}(u)
  \end{eqnarray*}
  where $\mathrm{Emb}_\partial$ denotes embeddings which send the
  boundary to the boundary, and the definition of the topology on
  $\Psi_\theta(\bR \times \bR^N)$ makes this continuous (manifolds
  near to a point $W \in \psi_\theta(N+1, 1) \subset \Psi_\theta(\bR
  \times \bR^N)$ look like a section of the normal bundle of $W$
  inside a compact set, e.g.\ inside $[t_0, t_1] \times [-1,1]^N$).

  We now require two results on spaces of embeddings. Firstly, the map
  $$\mathrm{Emb}_\partial(M, [t_0, t_1] \times (-1,1)^N) \lra
  \mathrm{Emb}_\partial(M, [t_0, t_1] \times (-1,1)^N) / \Diff(M)$$ is
  well-known to be a principal $\Diff(M)$-bundle, and has local
  sections (see e.g.\ \cite{MR613004}).  Thus, after perhaps passing
  to a smaller open neighbourhood, which we will still call $U$, $F$
  has a lift $\widetilde{F} : U \to \mathrm{Emb}_\partial(M, [t_0,
  t_1] \times (-1,1)^N)$, and we call $f = \widetilde{F}(a, \epsilon,
  (W,\ell_W))$.

Secondly, we need the following generalisation of a
  technical theorem of Cerf \cite[2.2.1 Th{\'e}or{\`e}me 5]{Cerf} (the
  ``first isotopy and extension theorem''), an especially elementary
  proof of which was given by Lima \cite{Lima}. We follow Lima's proof.

  \begin{lemma}
    Let $C \subset [t_0, t_1]$ be a closed subset and let $S
    \subset \Emb_\partial(M, [t_0, t_1] \times (-1,1)^N)$ be the open subset
    of those embeddings $e$ for which $\pi_1 \circ e : M \to [t_0,
    t_1]$ has no critical values inside $C$.

    Given an $f \in S$, there is a neighbourhood $U$ of $f$ in $S$ and
    a continuous map $\varphi : U \to \Diff([t_0, t_1] \times
    (-1,1)^N)$ such that $\varphi(g) \circ f$ and $g$ have the same
    image, and $\varphi(g)$ is height-preserving over $C$.
  \end{lemma}
\begin{proof}
  Consider $M$ to be a submanifold of $[t_0, t_1] \times (-1,1)^N$ via
  $f$. We choose a tubular neighbourhood $\pi: T \to M$ of radius
  $\epsilon$ which over the boundary and $x_1^{-1}(C)$ has fibres
  contained in level sets of $x_1$ (this is possible as $C$ is closed
  and consists of regular values).  If $g \in S$ is sufficiently close
  to $f$, it will have image in $T$ and we may define an element
  $\bar{\varphi}(g) \in C^\infty(M,M)$ by
  $$\bar{\varphi}(g)(x) = \pi(g(x)).$$
  This is a diffeomorphism for $g = f$, and so there is a
  neighbourhood $U'$ of $f$ in $S$ where this remains true.  We get a
  function $\bar{\varphi}: U' \to \Diff(M)$ and for each $g \in U'$ we
  define a new embedding $G = G(g): M \to [t_0, t_1] \times (-1,1)^N$
  by $G = g\circ (\bar{\varphi}(g)^{-1})$.  It has the same image as
  $g$ and has $\pi(G(x)) = x$.  Therefore $x$ and $G(x)$ have the same
  height when $x \in x_1^{-1}(C)$.

  Let $\lambda$ be a bump function which is 1 on $[0,\epsilon/4)$ and
  0 on $[\epsilon/2,\infty)$.  Now let
  $$\varphi(g)(x) = x + \lambda(\vert x - \pi(x)\vert)\cdot(G(\pi(x)) - \pi(x))$$
  define a compactly-supported smooth map $\varphi(g) \in
  \mathcal{C}^\infty_c([t_0, t_1] \times (-1,1)^N)$. For $g=f$ it is a
  diffeomorphism, and so there is a smaller neighbourhood $U$ of $f$
  in $S$ where this remains true.  We get a function $\varphi: U \to
  \Diff_c([t_0, t_1] \times (-1,1)^N)$.

  By construction $\varphi(g) \circ f(x) = \varphi(g)(x) = x + (G(x) -
  x) = G(x)$, so $\varphi(g) \circ f$ has the same image as $g$. Also,
  if $x \in x_1^{-1}(C)$ then the vector $G(\pi(x)) - \pi(x)$ has no
  component in the $x_1$ direction, so $x_1(\varphi(g)(x)) = x_1(x)$
  and $\varphi(g)$ is height function preserving over $C$.
\end{proof}

  It now follows that $\pi$ also has this local structure: after
  posibly shrinking $U$, there is a map
  $$\varphi : U \lra \Diff([t_0, t_1] \times (-1,1)^N)$$
  with the properties described in the lemma, such that
  $\Gamma\vert_{[t_0,t_1]} \subset U \times [t_0, t_1] \times
  (-1,1)^N$ 
  is obtained from $W\vert_{[t_0, t_1]}$ by applying the family of
  diffeomorphisms $\varphi$.

  The element $(\Lambda, \delta, e) \in \widetilde{Z}_0(a, \epsilon, (W,\ell_W))$ has
  surgery data
  $$e : \Lambda \times \overline{V} \hookrightarrow [t_0, t_1] \times (0,1)
  \times (-1,1)^{N-1},$$ so we attempt to define a section $U \to
  \widetilde{D}^\kappa_{\theta, L}(\bR^N)_{p, 0}$ by sending $u$ to
  the point $(\Lambda, \delta, \varphi(u) \circ e)$.  We must verify that this is
  indeed an element of $\widetilde{Z}_0(u)$ by checking the conditions
  of Definition \ref{defn:ZComplex}. Conditions
  (\ref{it:ZHeightFn})--(\ref{it:Z4}) hold as $\varphi$ is height
  preserving over each $[a_i-\epsilon_i, a_i+\epsilon_i]$. Condition
  (\ref{it:ZIntersection}) holds by construction, as $\varphi(x)$ is a
  diffeomorphism which carries $W$ into $\pi^{-1}(x)$. Condition
  (\ref{it:EnoughSurgeryData}) need not hold in general, but it does
  hold at the point $(a, \epsilon, (W,\ell_W))$, and is an open condition.
  Thus, after possibly replacing $U$ with a smaller open set, this
  does define a continuous section as required.
\end{proof}

Next, we establish condition~(\ref{it:2b}) in
Theorem~\ref{thm:SimplicialTech}.
\begin{proposition}\label{prop:MorSurgA}
  Fix a point $(a, \epsilon, (W,\ell_W)) \in D^{\kappa-1}_{\theta,
    L}(\bR^N)_p$, and let $v_1$, \dots, $v_k \in \widetilde{Z}_0(a,
  \epsilon, (W,\ell_W))$ be a non-empty collection of pieces of surgery data
  (not necessarily forming a $(k-1)$-simplex).  Then, if $2\kappa < d$
  and $\kappa+1+d < N$, there exists a piece of surgery data $v \in
  \widetilde{Z}_0(a, \epsilon, (W,\ell_W))$ such that each $(v_i, v)$ is a
  1-simplex.
\end{proposition}
\begin{proof}
  Each $v_j$ is given by a set $\Lambda^j$ (which is a subset of the
  uncountable set $\Omega$), a function $\delta^j : \Lambda^j \to [p]^\vee$ and a map $e^j: \Lambda^j \times \overline{V} \to \R
  \times (0,1) \times (-1,1)^{N-1}$, satisfying certain properties.  We first pick a
  set $\Lambda$ which is disjoint from all $\Lambda^j$ and a
  bijection $\phi: \Lambda \to \Lambda^1$, let $\delta =\delta^1 \circ \varphi : \Lambda \to [p]^\vee$, and then set
  \begin{equation*}
    \tilde{e} = e^1 \circ (\phi \times \mathrm{Id}_{\overline{V}}): \Lambda \times \overline{V}
    \lra \R \times (0,1) \times (-1,1)^{N-1}.
  \end{equation*}
  This gives a new element of $\widetilde{Z}_0(a,\epsilon,(W,\ell_W))$, but it is
  of course not orthogonal to $v_1$ (and not necessarily orthogonal to
  the other $v_j$).  We then perturb $\tilde e$ inside the class of
  functions satisfying the requirements of
  Definition~\ref{defn:ZComplex}, to a new function $e: \Lambda \times
  \overline{V} \to \R \times (0,1) \times (-1,1)^{N-1}$ whose core is in general position with
  respect to the cores of the $v_j$.  More explicitly, $\tilde e$
  restricts to a map
  \begin{equation*}
    \Lambda \times \partial_- D^{\kappa+1} \times \bR^{d-\kappa} \lra W,
  \end{equation*}
  and we first perturb this so that $\Lambda \times \partial_-
  D^{\kappa+1} \times \{0\}$ is transverse in $W$ to the corresponding
  part of the other embeddings, and remains disjoint from $L$, then we extend this perturbation to a
  map $e: \Lambda \times \overline{V} \to \R \times (0,1) \times (-1,1)^{N-1}$ whose restriction
  to the interior of $C$ is transverse to the corresponding part of
  the other embeddings.  In the first step we make
  $\kappa$-dimensional manifolds transverse in a $d$-dimensional
  manifold, and in the second we make $(\kappa+1)$-dimensional
  manifolds disjoint in an $(N+1)$-dimensional manifold.  As $2\kappa
  < d$ and $2(\kappa+1) \leq \kappa + d + 2 < N+1$, the new core will
  actually be disjoint from all other cores, producing the required
  element $v \in \widetilde{Z}_0(a,\epsilon,(W,\ell_W))$.
\end{proof}

Finally, we establish condition~(\ref{it:2a}) of
Theorem~\ref{thm:SimplicialTech}.
\begin{proposition}\label{prop:MorSurgB}
  $\widetilde{Z}_0(a, \epsilon, (W,\ell_W))$ is non-empty as long as $2\kappa <
  d$, $\kappa+1+d < N$, and $L$ admits a handle decomposition only using handles of index $< d-\kappa-1$.
\end{proposition}
\begin{proof}
  For each $i = 1, \ldots, p$ we consider the pair $(W \vert_{[a_{i-1},
    a_{i}]}, W \vert_{a_{i}})$.  Since it is $(\kappa-1)$-connected,
  it is homotopy equivalent to a finite relative CW complex $(X,W
  \vert_{a_i})$ with cells of dimension $\geq \kappa$ only.  Since
  $2\kappa < d$, the homotopy equivalence $(X,W \vert_{a_i}) \to (W
  \vert _{[a_{i-1},a_i]}, W\vert_{a_i})$ may be assumed to restrict to
  a smooth embedding of the relative $\kappa$-cells.  If we pick a subset
  $\Lambda_{i,0} \subset \Omega$ with one element for each relative $\kappa$-cell (choosing disjoint sets for each $i$), we
  may therefore pick an embedding
  $$\hat{e}_{i,0} : \Lambda_{i,0} \times (D^{\kappa}, \partial
  D^\kappa) \lra (W \vert_{[a_{i-1} + \epsilon_{i+1}, a_{i} +
    \epsilon_i]}, W \vert_{a_{i} + \epsilon_i}),$$ which we may assume
  collared on $[a_i-\epsilon_i, a_i + \epsilon_i]$, such that the pair
  \begin{equation*}
    (W \vert_{[a_{i-1}, a_{i}]}, W \vert_{a_{i}} \cup
    \IM(\hat{e}_{i,0})\vert_{[a_{i-1},a_i]})
  \end{equation*}
  is $\kappa$-connected. Furthermore, $\R \times L \subset W$ has a
  core of dimension $< d-\kappa$, by our assumption on the indices of
  handles of $L$, and so we may suppose that the embedding
  $\hat{e}_{i,0}$ is disjoint from $\R \times L$. As $2\kappa < d$ we may also suppose that the images of the $\hat{e}_{i,0}$ are mutually disjoint.

  The embedding
  $$\hat{e}_{i,0} \vert_{\Lambda_{i,0} \times \partial D^\kappa} :
  \Lambda_{i,0} \times \partial D^\kappa \times \{0\} \lra W
  \vert_{a_i + \epsilon_i} \subset W \vert_{[a_i + \epsilon_i, a_{i+1}
    + \epsilon_{i+1}]}$$ extends to an embedding of $\Lambda_{i,0}
  \times \partial D^\kappa \times [0,1]$, where $\Lambda_{i,0}
  \times \partial D^\kappa \times \{1\}$ is sent into $W
  \vert_{a_{i+1} + \epsilon_{i+1}}$ and is collared on the
  $\epsilon$-neighbourhoods of both boundaries. This may be seen as
  follows: to extend $\hat{e}_{i,0} \vert_{\Lambda_{i,0}
    \times \partial D^\kappa}$ to a continuous map having this
  property is possible as $\pi_{\kappa-1}(W \vert_{[a_{i-1}, a_{i}]},
  W \vert_{a_{i}})=0$, but this may then be perturbed to be an
  embedding as $2\kappa < d$. As above, this may be made disjoint from
  $\R \times L$, and they can be made mutually disjoint.

  We may glue the two embeddings together.  Using a suitable
  diffeomorphism $D^\kappa \approx D^\kappa \cup \big(\partial
  D^\kappa \times [0,1]\big)$, this gives a new embedding of
  $\Lambda_{i,0} \times D^\kappa$.  Continuing in this way, we obtain
  an extension of $\hat{e}_{i,0}$ to an embedding
  $$\tilde{e}_{i,0} : \Lambda_{i,0} \times (D^{\kappa}, \partial
  D^\kappa) \lra (W \vert_{[a_{i-1} + \epsilon_{i-1}, a_{p} +
    \epsilon_p]}, W \vert_{a_{p} + \epsilon_p})$$ which is disjoint
  from $\R \times L$, and which are mutually disjoint.  Identifying $D^\kappa$ with the disc
  $\partial_- D^{\kappa+1} \subset [-1,0] \times \R^{\kappa+1}$
  gives a height function $D^\kappa \to [-1,0]$ and if we pick the
  diffeomorphisms $D^\kappa \approx D^\kappa \cup \big(\partial
  D^\kappa \times [0,1]\big)$ carefully, we can arrange that on each
  $\tilde{e}_{i,0}^{-1} (W\vert_{(a_k-\epsilon_k, a_k+\epsilon_k)})$,
  the embedding $\tilde e_{i,0}$ is height function preserving up to
  an affine transformation.

  We now want to extend the $\tilde{e}_{i,0}$ from $\Lambda_{i,0}
  \times (\partial_- D^{\kappa+1} \times \{0\}) \subset
  \Lambda_{i,0} \times \overline{V}$ to the whole of $\Lambda_{i,0}
  \times \overline{V}$ so that it satisfies the conditions of
  Definition \ref{defn:ZComplex}. As $\kappa+1+d < N$, there is no
  trouble with extending the maps $\tilde{e}_{i,0}$ to disjoint maps $e_{i,0}$
  from $\Lambda_{i,0} \times \overline{V}$ to $\bR \times (0,1) \times
  (-1,1)^{N-1}$ satisfying conditions
  (\ref{it:ZHeightFn})--(\ref{it:ZIntersection}) of Definition
  \ref{defn:ZComplex}: we first extend each $\tilde{e}_{i,0}$ to an
  embedding of $[-2,0] \times \R^\kappa \times \{0\}$ (which is
  possible as $2(\kappa+1) \leq d+\kappa+1 < N$), then make this
  intersect $W$ only in $\partial_-D^{\kappa+1}$ (which is possible as
  $\kappa+1+d < N$), and finally thicken it up by $\bR^{d-\kappa}$.
  Property (\ref{it:EnoughSurgeryData}) is ensured by the way we chose
  $\hat{e}_{i,0}$.
  
  Finally, we let $\Lambda = \coprod_{i=1}^p \Lambda_{i,0}$, $\delta : \Lambda \to [p]^\vee$ be given by $\delta(\Lambda_{i,0}) = i \in [p]^\vee$, and $e = \coprod_{i=1}^p e_{i,0}$. The data $(\Lambda, \delta, e)$ thus lies in $\widetilde{Z}_0(a, \epsilon, (W,\ell_W))$.
\end{proof}

\subsection{Proof of Theorem \ref{thm:SurgeryComplexOb}}
\label{sec:proof-theor-refthm:s-1}

We have already proved the first part of this theorem in Section \ref{sec:ProofSecondPart}. Recall that the second part states that the augmentation map
$$D^{\kappa, l}_{\theta, L}(\bR^N)_{\bullet, \bullet} \lra D^{\kappa, l-1}_{\theta, L}(\bR^N)_{\bullet}$$
induces a weak homotopy equivalence after geometric realisation, as long as the conditions of Theorem \ref{thm:lfiltration} are satisfied. In fact, we only require the following weaker set of conditions:
\begin{enumerate}[(i)]
\item $2(l+1) < d$,

\item $l \leq \kappa$,

\item $l+2+d < N$,

\item $L$ admits a handle decomposition only using handles of index $< d-l-1$,

\item the map $\ell_L : L \to B$ is $(l+1)$-connected.
\end{enumerate}

We will proceed much as in the previous section. Recall that each
point of $D^{\kappa, l}_{\theta, L}(\bR^N)_{p, 0}$ lying over $(a,
\epsilon, (W, \ell_W)) \in D^{\kappa, l-1}_{\theta, L}(\bR^N)_{p}$ is a tuple $(\Lambda, \delta, e, \ell)$ where $\Lambda \subset \Omega$ is a subset, $\delta : \Lambda \to [p] \times [0]$ is a function, 
$$e : \Lambda \times (-6,-2) \times
\bR^{d-l-1} \times D^{l+1} \hookrightarrow \bR \times (0,1) \times
(-1,1)^{N-1}$$
is an embedding, and $\ell : T(\Lambda \times K\vert_{(-6,0)}) \to \theta^*\gamma$ is a bundle map (where $K$ is defined in Section \ref{sec:ObSurgStdFam}). Let us define
$$C = (-6,-2) \times \{0\} \times D^{l+1}
\subset (-6,-2) \times \bR^{d-l-1} \times
D^{l+1}$$ and call it the \emph{core}. Shrinking in the
$\bR^{d-l-1}$-direction gives an isotopy from the identity map of
$(-6,-2) \times \bR^{d-l-1} \times D^{l+1}$
into any neighbourhood of its core.

\begin{definition}\label{defn:yComplexExtended}
  Let $\widetilde{Y}_\bullet(a, \epsilon, (W, \ell_W))$ be the
  semi-simplicial space defined as in Definition \ref{defn:YComplex}, expect we only ask for $e$ to be a smooth map which restricts to an embedding on a neighbourhood of $\Lambda \times C \subset \Lambda \times (-6,-2) \times \bR^{d-l-1} \times
D^{l+1}$. Note that condition
  (\ref{it:EnoughSurgeryDataOb}) still makes sense: although the
  surgery data is no longer disjoint, it is still disjoint when
  restricted to a small enough neighbourhood of each core.

  Let $\widetilde{D}^{\kappa, l}_{\theta, L}(\bR^N)_{\bullet, \bullet}
  \to D^{\kappa, l-1}_{\theta, L}(\bR^N)_{\bullet}$ be the augmented
  bi-semi-simplicial space defined as in Definition
  \ref{defn:DoubleCxOb}, but using $\widetilde{Y}_\bullet(a, \epsilon,
  (W, \ell_W))$ instead of ${Y}_\bullet(a, \epsilon, (W, \ell_W))$.
\end{definition}

We have the following analogue of Proposition \ref{prop:TildeIncIsEq},
although the proof is slightly more complicated in this case, due to
the tangential structures on the surgery data.

\begin{proposition}\label{prop:TildeIncIsEqOb}
  The inclusion ${D}^{\kappa, l}_{\theta, L}(\bR^N)_{\bullet, \bullet}
  \hookrightarrow \widetilde{D}^{\kappa, l}_{\theta,
    L}(\bR^N)_{\bullet, \bullet}$ induces a weak homotopy equivalence
  in each bidegree, and so on geometric realisation.
\end{proposition}
\begin{proof}
  This is very similar to Proposition \ref{prop:TildeIncIsEq}.  We
  pick an isotopy of maps $\psi_t: \R^{d-l-1} \to \R^{d-l-1}$, $t \in
  [0,\infty)$ which starts at the identity, has $\psi_t(0) = 0$ for
  all $t$, and has image in the ball of radius $1/t$ for all $t$.
  Applying $\psi_t$ in the $\R^{d-l-1}$ direction gives an isotopy of
  self-embeddings of $\Lambda \times (-6,-2) \times \R^{d-l-1} \times D^{l+1}$.  Similarly, we can
  get an isotopy of self-embeddings of the manifold $K\vert_{(-6,0)}$
  from Section~\ref{sec:ObSurgStdFam}, which applies $\psi_t$ in the
  $\R^{d-l-1}$ direction on $h^{-1}((-6,-2])$, is the identity on
  $h^{-1}((-\sqrt{2},0))$, and interpolates inbetween.  Precomposing
  with these isotopies gives a homotopy of self-maps of
  $\widetilde{D}^{\kappa, l}_{\theta, L}(\bR^N)_{\bullet, \bullet}$,
  which eventually deforms any compact space into ${D}^{\kappa,
    l}_{\theta, L}(\bR^N)_{\bullet, \bullet}$.
\end{proof}

Therefore it is enough to show that for each $p$, the augmentation map
$$\widetilde{D}^{\kappa, l}_{\theta, L}(\bR^N)_{p, \bullet} \lra
{D}^{\kappa, l-1}_{\theta, L}(\bR^N)_{p}$$ which forgets all surgery
data induces a weak homotopy equivalence after geometric realisation,
which we do by establishing the conditions of Theorem
\ref{thm:SimplicialTech}.  The proofs that conditions (\ref{item:4}) and
(\ref{it:2b}) hold are very similar to the analogous case in
Section~\ref{sec:proof-theor-refthm:s}, so we consider those first.

\begin{proposition}
  The map $\widetilde{D}^{\kappa, l}_{\theta, L}(\bR^N)_{p, 0} \to
  {D}^{\kappa, l-1}_{\theta, L}(\bR^N)_{p}$ has local sections.
\end{proposition}
\begin{proof}
  Exactly as in the proof of Proposition \ref{prop:LocalSections}.
\end{proof}

\begin{proposition}
  Fix a point $(a, \epsilon, (W, \ell_W)) \in D^{\kappa, l-1}_{\theta,
    L}(\bR^N)_p$, and let $v_1$, \ldots, $v_k \in \widetilde{Y}_0(a,
  \epsilon, (W, \ell_W))$ be a non-empty collection of pieces of surgery data. Then, if $2(l+1) < d$ and $l+2+d < N$, there exists a piece
  of surgery data $v \in \widetilde{Y}_0(a, \epsilon, (W, \ell_W))$ such that
  each $(v_i, v)$ is a 1-simplex.
\end{proposition}
\begin{proof}
  This is essentially the same as Proposition \ref{prop:MorSurgA}:
  first we let $v = v_1$, then we perturb it to have its cores
  transverse to the cores of all the $v_j$. We first do the
  perturbation on the part of the cores inside $W$. On the boundary
  the cores are $(l+1)$-dimensional, so disjoint when they are
  transverse as $2(l+1) < d$. We now make sure the cores intersect $W$
  only on their boundary, which is possible as $l+2 + d < N$. We
  finally make sure that the cores are also disjoint on their
  interiors, which is possible as $(l+2) + (l+2) \leq (l+2) + d < N$.
\end{proof}

Finally, we establish condition~(\ref{it:2a}).

\begin{proposition}\label{prop:YNonEmpty}
  $\widetilde{Y}_0(a, \epsilon, (W, \ell_W))$ is non-empty as long as
  $2(l+1) < d$, $l \leq \kappa$, $l+2+d < N$,  $L$ admits a handle decomposition only using handles of index $< d-l-1$, and the map $\ell_L : L \to B$ is $(l+1)$-connected.
\end{proposition}
\begin{proof}
  For each $i$, we consider the map $\pi_*(W\vert_{a_i}) \to
  \pi_*(B)$, induced by the tangential structure.  By assumption, this
  map is injective for $* \leq
  l-1$.  Since $\{a_i\} \times L \subset W\vert_{a_i}$, and $L \to B$
  is assumed $(l+1)$-connected, we
  deduce that the map $L \to W\vert_{a_i}$ is $(l-1)$-connected,
  $W\vert_{a_i} \to B$ is $l$-connected, that
  \begin{equation}
    \label{eq:14}
    \pi_l(W\vert_{a_i}) \lra \pi_l(B) \approx \pi_l(L)
  \end{equation}
  is split surjective, and that $\pi_l(L) \to \pi_l(W\vert_{a_i})$ is split
  injective.  We first claim that the kernel of~\eqref{eq:14} is
  finitely generated as a module over $\pi_1(L)$ (interpreted
  appropriately when $l=0$ and $l=1$; we shall leave the necessary
  modifications of the following argument in those two cases to the
  reader).  Since the kernel is isomorphic to the cokernel of the
  splitting, we deduce the exact sequence
  \begin{equation}
    \label{eq:15}
    \pi_l(W\vert_{a_i},L) \lra \pi_l(W\vert_{a_i}) \lra \pi_l(B) \lra 0.
  \end{equation}
As $(W\vert_{a_i},L)$ is $(l-1)$-connected, we can find a relative CW-complex $(K,L)$, where $K$ is built from $L$ by attaching only cells of dimension $\geq l$, and a weak homotopy equivalence $p: (K, L) \to (W\vert_{a_i},L)$. Since $(W\vert_{a_i}, L)$ has the homotopy type of a CW pair, this map has a homotopy inverse $q: (W\vert_{a_i}, L) \to (K,L)$, and since $W\vert_{a_i}$
  is compact, its image in $K$ is contained in a finite subcomplex $K'
  \subset K$. Then $\pi_l(K') \to \pi_l(W\vert_{a_i})$ is
  surjective.  Since $\pi_l(W\vert_{a_i}) \to \pi_l(W\vert_{a_i},L)$
  is also surjective (as $\pi_{l-1}(L) \to \pi_{l-1}(W\vert_{a_i})$ is
  split injective), we conclude that $\pi_l(K', L) \to
  \pi_l(W\vert_{a_i},L)$ is surjective, and hence that
  $\pi_l(W\vert_{a_i},L)$ is a finitely generated module over
  $\pi_1(L)$, as claimed.
  
  Let $\{\hat{f}_\alpha : S^l \to W\vert_{a_i}\}_{\alpha \in
    \Lambda_i}$ be a finite collection of elements which generate the
  kernel of $\pi_l(W\vert_{a_i}) \to \pi_l(B)$, where $\Lambda_i \subset \Omega$ are disjoint subsets.  As the vector bundle
  $\epsilon^1 \oplus TW\vert_{a_i}$ is pulled back from $B$, it
  becomes trivial when pulled back via $\hat{f}_\alpha$ so we can pick
  an isomorphism $\epsilon^1 \oplus \hat{f}_\alpha^*(TW\vert_{a_i})
  \cong \epsilon^{d}$.  As $l+1 < d$ this isomorphism can be
  destabilised to an isomorphism $\hat{f}_\alpha^*(TW\vert_{a_i})
  \cong \epsilon^{d-1} \cong \epsilon^{d-l-1} \oplus TS^l$ and by
  Smale--Hirsch theory $\hat{f}_\alpha$ is then homotopic to an
  immersion with trivial normal bundle.  We can make this immersion
  self-transverse, and as $2l < d-1$ it is then an embedding with
  trivial normal bundle.  Thus each $\hat{f}_\alpha$ gives rise to an
  embedding ${f}_\alpha : \bR^{d-l-1} \times S^l \hookrightarrow
  W\vert_{a_i}$ representing the same homotopy class.  As $2l < d-1$
  we may also assume that the ${f}_\alpha$ are disjoint, so we obtain
  an embedding
  $${f}_i\vert_{a_i} : \Lambda_i \times \{a_i\} \times \bR^{d-l-1}
  \times S^l \hookrightarrow W\vert_{a_i},$$
  and as $L$ only has handles of index $< d-l-1$, we may suppose this embedding is disjoint from $L$.
  As $l+1 + d < N$, this
  extends to an embedding
  $${e}_i \vert_{a_i} : \Lambda_i \times \{a_i\} \times \bR^{d-l-1}
  \times D^{l+1} \hookrightarrow \{a_i\} \times (0,1) \times(-1,1)^{N-1}$$ which
  intersects $W\vert_{a_i}$ precisely on the boundary.  Furthermore,
  as each $S^l \overset{\hat{f}_\alpha}\to W\vert_{a_i} \to B$ is
  null-homotopic, the $\theta$-structure $\ell\vert_{a_i} \circ
  D{f}_\alpha$ extends to a $\theta$-structure on $\bR^{d-l-1} \times
  D^{l+1}$ and so gives ${f}_i\vert_{a_i}$ the data of a
  $\theta$-surgery, cf.\ Section \ref{sec:ThetaSurgery}.

  We can extend the map ${e}_i \vert_{a_i}$ to an embedding $\Lambda_i
  \times (a_i-\epsilon_i, a_i+\epsilon_i) \times \bR^{d-l-1} \times
  D^{l+1} \hookrightarrow (a_i-\epsilon_i, a_i+\epsilon_i) \times (0,1) \times
  (-1,1)^{N-1}$ using just the cylindrical structure of $W$ over
  $(a_i-\epsilon_i, a_i+\epsilon_i)$, but we wish to extend it to an
  embedding of $\Lambda_i \times (a_i-\epsilon_i, a_p+\epsilon_p)
  \times \bR^{d-l-1} \times D^{l+1}$, which is cylindrical over each
  $(a_j-\epsilon_j, a_j+\epsilon_j)$ and intersects $W$ precisely on
  the boundary. We will do this by extending it step-by-step over each
  interval $[a_j, a_{j+1}]$: if it is defined up to $a_j$ we have an
  embedding
  $${e}_i \vert_{a_j} : \Lambda_i \times \{a_j\} \times \bR^{d-l-1}
  \times D^{l+1} \hookrightarrow \{a_j\} \times (0,1) \times (-1,1)^{N-1},$$ and as the
  pair $(W \vert_{[a_j, a_{j+1}]}, W\vert_{a_{j+1}})$ is
  $\kappa$-connected, and $l\leq \kappa$, on the boundary this extends
  to a continuous map
  $${f}_i \vert_{[a_j, a_{j+1}]} : \Lambda_i \times [a_j, a_{j+1}]
  \times \bR^{d-l-1} \times S^l \lra W \vert_{[a_j, a_{j+1}]}.$$ By
  the Smale--Hirsch argument above, we may perturb this to be a
  self-transverse immersion of the core, and hence an embedding of the
  core as $2(l+1) < d$, while keeping it as it was near
  $a_j$. Shrinking in the $\bR^{d-l-1}$-direction, we can ensure that
  it is an embedding of the whole manifold, and then make the
  embedding cylindrical over the necessary $\epsilon$-neighbourhood of
  the ends and disjoint from $[a_j, a_{j+1}] \times L$. Finally, as $l+2 + d < N$ we may extend this to an
  embedding
  $${e}_i \vert_{[a_j, a_{j+1}]} : \Lambda_i \times [a_j, a_{j+1}]
  \times \bR^{d-l-1} \times D^{l+1} \hookrightarrow [a_j, a_{j+1}]
  \times (0,1) \times (-1,1)^{N-1}$$ which is cylindrical over each $(a_j-\epsilon_j,
  a_j+\epsilon_j)$ and intersects $W$ precisely on the boundary. In
  total we obtain an embedding
  $${e}_i : \Lambda_i \times (a_i-\epsilon_i, a_p+\epsilon_p) \times
  \bR^{d-l-1} \times D^{l+1} \hookrightarrow (a_i-\epsilon_i,
  a_p+\epsilon_p) \times (0,1) \times (-1,1)^{N-1}$$ which is cylindrical over each
  $(a_j-\epsilon_j, a_j+\epsilon_j)$ and intersects $W$ precisely on
  the boundary. Furthermore, by doing the above in increasing order of
  $i$, we can ensure that the different $e_i$ have disjoint
  cores: while constructing $e_i$ make sure that its core stays
  disjoint from those of the $e_j$ for all $j < i$, which is possible
  as $2(l+1) < d$ and $2(l+2) < N$.
  
  We let $\Lambda = \coprod_{i=0}^p \Lambda_{i}$,  $\delta : \Lambda \to [p] \times [0]$ be given by $\delta(\Lambda_i) = (i,0)$, and $e$ be given by $\coprod_{i=0}^p e_i$, reparametrised using the $\varphi(a_i, \epsilon_i, a_p, \epsilon_p)$. Then the data $(\Lambda, \delta, e)$ gives the embedding part of the data of an element of
  $\widetilde{Y}_0(a, \epsilon, (W, \ell_W))$, and we must now provide
  the bundle part. Under the chosen diffeomorphism
  $$K \vert_{(-6,-2)} = (-6, -2) \times \bR^{d-l-1} \times S^l
  \cong_{\varphi(a_i, \epsilon_i, a_p, \epsilon_p)} (a_i -\epsilon_i, a_p+\epsilon_p) \times \bR^{d-l-1}
  \times S^l,$$ the embedding $f_i : \Lambda_i \times (a_i-\epsilon_i,
  a_p+\epsilon_p) \times \bR^{d-l-1} \times S^l \hookrightarrow W$
  gives a $\theta$-structure $\ell_i \vert_{(-6,-2)} = \ell_W \circ D
  f_i$ on $\Lambda_i \times K \vert_{(-6,-2)}$. For $\ell_i$ we may
  take any extension of this $\theta$-structure to $\Lambda_i \times
  K$, and so only need to know that such an extension exists. This is
  a purely homotopical problem, and homotopically $K\vert_{(-6,0)}$ is
  obtained from $K \vert_{(-6,-2)}$ by attaching a $D^{l+1}$, so the
  extension problem can be solved if and only if $\ell_W \circ D
  f_i\vert_{a_i} : T(\Lambda_i \times \{a_i\} \times \bR^{d-l-1}
  \times S^l) \to \theta^*\gamma$ extends over $\Lambda_i \times
  \{a_i\} \times \bR^{d-l-1} \times D^{l+1}$, but we have seen above
  that it does, because $\Lambda_i \times \{a_i\} \times \bR^{d-l-1}
  \times S^l \to W\vert_{a_i} \to B$ is null-homotopic.
\end{proof}

\subsection{Proof of Theorem \ref{thm:SurgeryComplexObMid}}

Recall that the statement of the theorem is as follows. We work in
dimension $2n$, and fix a tangential structure $\theta$ which is
reversible (cf.\ Definition \ref{defn:Reversible}), a
$(2n-1)$-manifold with boundary $L$ equipped with $\theta$-structure,
and a collection $\mathcal{A} \subset \pi_0(\Ob(\mathcal{C}_{\theta,
  L}^{n-1, n-2}(\bR^N)))$ of objects. This allows
us to define the augmented bi-semi-simplicial space
$$D_{\theta, L}^{n-1, \mathcal{A}}(\bR^N)_{\bullet, \bullet} \lra
D_{\theta, L}^{n-1, n-2}(\bR^N)_\bullet$$ of surgery data, and the second part of Theorem
\ref{thm:SurgeryComplexObMid} states that if the conditions of Theorem \ref{thm:MidSurgery} are satisfied, then the induced map on geometric realisation is a weak homotopy
equivalence. (We have already proved the first part of Theorem
\ref{thm:SurgeryComplexObMid} in Section \ref{sec:ProofSecondPart}.) We recall that these conditions are:
\begin{enumerate}[(i)]
\item $2n \geq 6$,

\item $3n+1 < N$,

\item $\theta$ is reversible,

\item $L$ admits a handle decomposition only using handles of index $< n$,

\item $\ell_L : L \to B$ is $(n-1)$-connected,

\item the natural map $\mathcal{A} \to \pi_0(B\mathcal{C}_{\theta,
    L}^{n-1, n-2}(\bR^N))$ is surjective.
\end{enumerate}

\noindent Note that the penultimate condition implies
that for any object, the map $M \to B$ induced by the tangential
structure induces a surjection on $\pi_*$ for $* < n$.

In many respects the proof of this theorem is very similar to what we
did in Section~\ref{sec:proof-theor-refthm:s-1}, but in that section
we often used the inequality $2(l+1) < d$ so that pairs of transverse
$(l+1)$-dimensional submanifolds of a $d$-manifold are automatically
disjoint. In Theorem \ref{thm:SurgeryComplexObMid}, $d=2n$ and the
analogue of $l$ is $(n-1)$ so this observation fails. Instead, we will
use a version of the Whitney trick to separate $n$-dimensional
submanifolds of our $2n$-manifolds; this accounts for the restriction
$2n \geq 6$ in the statement of the theorem.

We proceed precisely as in Definition \ref{defn:yComplexExtended} by
for $(a, \epsilon, (W,\ell_W)) \in D_{\theta, L}^{n-1, n-2}(\bR^N)_p$
letting $\widetilde{Y}_\bullet(a, \epsilon, (W,\ell_W))$ be the
analogue of $Y_\bullet(a, \epsilon, (W,\ell_W))$ from Definition
\ref{defn:YComplexMid}, where instead of asking that $e$ be an embedding, we only ask for it to be a smooth map which restricts to an embedding on a neighbourhood of $\Lambda \times C$. We use this to define the bi-semi-simplicial space
$\widetilde{D}_{\theta, L}^{n-1, \mathcal{A}}(\bR^N)_{\bullet,
  \bullet}$, and by the same argument as Proposition
\ref{prop:TildeIncIsEqOb}, the inclusion
$$D_{\theta, L}^{n-1, \mathcal{A}}(\bR^N)_{\bullet, \bullet}
\hookrightarrow \widetilde{D}_{\theta, L}^{n-1,
  \mathcal{A}}(\bR^N)_{\bullet, \bullet}$$ is a weak homotopy
equivalence in each bidegree. We are now left to verify the conditions
of Theorem \ref{thm:SimplicialTech} for the augmented semi-simplicial
spaces
$$\widetilde{D}_{\theta, L}^{n-1, \mathcal{A}}(\bR^N)_{p, \bullet}
\lra D_{\theta, L}^{n-1, n-2}(\bR^N)_p.$$ That the map on 0-simplices
has local sections is proved as in the previous two sections.

\begin{proposition}\label{prop:RunOutOfNames}
  Fix a point $(a, \epsilon, (W, \ell_W)) \in D^{n-1, n-2}_{\theta,
    L}(\bR^N)_p$, and let $v_1$, \ldots, $v_k \in \widetilde{Y}_0(a,
  \epsilon, (W, \ell_W))$ be a non-empty collection of pieces of surgery data.  Then if $2n \geq 6$ and $3n+1 < N$ there exists a piece of
  surgery data $v \in \widetilde{Y}_0(a, \epsilon, (W, \ell_W))$ such that each
  $(v_i, v)$ is a 1-simplex.
\end{proposition}
\begin{proof}
  Let us
  write $v_j = (\Lambda^j, \delta^j, e^j, \ell^j)$. First we let $v = v_1$, then we
  perturb it to have its cores transverse to the cores of all the
  $v_j$. We first do the perturbation on the part of the cores inside
  $W$. On the boundary the cores are $n$-dimensional, so when they are
  transverse, they intersect in a finite set of points. We now make
  sure the cores intersect $W$ only on their boundary, which is
  possible as $(n+1) + 2n < N$. We finally make sure that the cores
  are also disjoint on their interiors, which is possible as $2(n+1) <
  N$.

  We are left with surgery data $v$ whose core is disjoint from the
  cores of $v_j$ away from $W$, and on $W$ intersects the other cores
  transversely. It has a finite number of transverse intersections
  with all the other cores in $W$, so it is enough to give a procedure
  which reduces the number of intersections by 1.  Let $x$ be such an
  intersection point, between $v$ and some $v_j$.  More precisely,
  suppose it is a point of intersection of the cylinders
  $$e_i(\Lambda_i \times (-6,-2) \times \{0\}
  \times S^{n-1}) \quad\quad e_k^j(\Lambda_k^j \times (-6,-2) \times \{0\} \times S^{n-1}).$$

  \begin{claim}
    Let $T \subset \R^2$ denote the triangle $\{(x,y) \,|\, y \leq 0, y
    +1 \geq |x|\}$ and $U$ a small open neighbourhood of it, e.g.\
    defined by $y < \epsilon$, $y+1 < |x| + \epsilon$.  There is a
    \emph{Whitney disc} $w : U \hookrightarrow W$ such that
    \begin{enumerate}[(i)]
    \item $w$ is disjoint from $\bR \times L$.
    
    \item\label{it:WhitneyDiscs1} $w\vert_{[-1,1] \times \{0\}}$ is a
      path in $W \vert_{a_p}$ which on its interior is disjoint from
      all the cores.
      
    \item\label{it:WhitneyDiscs2} The inverse image of the first
      cylinder is the line on $\partial T$ from $(0,-1)$ to $(-1,0)$.
      The inverse image of the second cylinder is the line from
      $(0,-1)$ to $(1,0)$.
      
    \item\label{it:WhitneyDiscHeight} The height functions $x_1 \circ 
      w$ and $y: T \to \R$ agree up to an affine transformation inside
      each $(x_1 \circ w)^{-1}(a_j-\epsilon_j, a_j+\epsilon_j)$.
    \end{enumerate}
  \end{claim}
  Given such a disc, we can extend it to a standard neighbourhood
  $w(U) \times \bR^{n-1} \times \bR^{n-1} \subset W$ as in the proof
  of \cite[Theorem 6.6]{MilnorHCob}. Note the argument is easier in
  this case as we are canceling intersection points against the
  boundary instead of against each other, and so no framing problems
  arise. We can further extend this to a neighbourhood
  $$w(U) \times \bR^{n-1} \times \bR^{n-1} \times \bR^{N+1-2n} \subset
  \bR \times (0, 1) \times \bR^{N-1}.$$ There is a compactly supported vector field on $U$ which is
  $\partial / \partial x$ on $D^2_-$, and we extend it using bump
  functions in the euclidean directions to this open subset of
  $\bR \times (0, 1) \times \bR^{N-1}$. The flow associated to this vector field gives a
  1-parameter family of diffeomorphisms $\phi_t$, and flowing $e_i$
  along using $\phi_t$ will eventually lead to a new $e_i$ whose core
  has one fewer intersection point with other cores (at least inside
  $W\vert_{(a_0-\epsilon_0, a_p]}$, but we can then use the
  cylindrical structure of $W\vert_{(a_p-\epsilon_p, a_p+\epsilon_p)}$
  to remove any intersections above $a_p$). It will still satisfy
  condition (\ref{it:YComplexMidHeight}) of Definition
  \ref{defn:YComplexMid} by property (\ref{it:WhitneyDiscHeight})
  above, and the other conditions are clear.

  It remains to prove the claim. If it were not for property
  (\ref{it:WhitneyDiscHeight}) the argument is clear: choose an
  embedded path from $x$ in each cylinder up to $W
  \vert_{a_p}$. Together these give an element of $\pi_1(W\vert_{[a_i,
    a_p]}, W\vert_{a_p})$ which is $0$ as $\kappa = n-1 \geq 2$, and
  so this extends to a continuous map $w\vert_{T} : T \to W$ which
  gives these two paths along the lower part of its boundary and lies
  in $W\vert_{a_p}$ in the top part of its boundary. As $2\cdot 2 <
  2n$, this map may be perturbed to be an embedding into $W$, still
  enjoying these two properties. Finally, as $2 + n < 2n$,
  $w\vert_{D^2_-}$ can be made disjoint from the other cores on its
  interior. This may now be extended to a map on $U$, enjoying
  properties (\ref{it:WhitneyDiscs1}) and (\ref{it:WhitneyDiscs2}).

  To obtain property (\ref{it:WhitneyDiscHeight}) as well, we instead
  build up the embedding $w\vert_T$ in pieces inside each $W
  \vert_{[a_j, a_{j+1}]}$, which is possible as each
  $\pi_1(W\vert_{[a_j, a_{j+1}]}, W\vert_{a_{j+1}})$ is 0.
\end{proof}

In the proof of Proposition \ref{prop:YNonEmpty}, it was easy to see
that for an object $M \in \mathcal{C}_{\theta, L}^{\kappa,
  l-1}(\bR^N)$ there exists a piece of $\theta$-surgery data $e : \Lambda
\times \bR^{d-l-1} \times S^l \hookrightarrow M$ such that the
resulting manifold $\overline{M}$ has $\pi_l(\overline{M}) \to
\pi_l(B)$ injective, so satisfies condition
(\ref{it:EnoughSurgeryDataOb}) of Definition \ref{defn:YComplex}. In
the present situation we have $M \in \mathcal{C}_{\theta, L}^{n-1,
  n-2}(\bR^N)$ and require surgery data so that $\overline{M} \in
\mathcal{A}$, to satisfy condition (\ref{it:EnoughSurgeryDataObMid})
of Definition \ref{defn:YComplexMid}. This is rather more difficult,
and we first describe how to accomplish this step.

\begin{lemma}\label{lem:SurgeryDataExists}
  Let $M \in \mathcal{C}_{\theta, L}^{n-1, n-2}(\bR^N)$ be an object,
  and suppose that $\theta$ is reversible, $2n \geq 6$, $L$ has a
  handle structure with only handles of index $< n$, $\ell_L : L \to
  B$ is $(n-1)$-connected, and $\mathcal{A}$ contains an object in the same path component of
  $B\mathcal{C}_{\theta, L}^{n-1, n-2}(\bR^N)$ as $M$. Then there is a
  piece of $\theta$-surgery data, given by an embedding $e : \Lambda
  \times \bR^{d-l-1} \times S^l \hookrightarrow M$ disjoint from $L$
  and a compatible bundle map $T(\Lambda \times \R^{d-l-1} \times
  D^{l+1}) \to \theta^* \gamma$, such that the resulting surgered
  manifold $\overline{M}$ lies in $\mathcal{A}$.
\end{lemma}
\begin{proof}
  Part of this proof is very similar to \cite[pp.\ 722--724]{Kreck}.

  We first claim that if there is a morphism $W : M_0 \leadsto M_1 \in
  \mathcal{C}_{\theta, L}^{n-1, n-2}(\bR^N)$, then there is another,
  $W'$ say, with the property that $(W',M_0)$ is also
  $(n-1)$-connected.  (By definition, $(W',M_1)$ is
  $(n-1)$-connected.)  In fact, we claim that it is possible to do
  surgery along a finite set of embeddings of $S^{n-1} \times D^{n+1}$
  into the interior of $W$ (and disjoint from $L$), such that the resulting cobordism $W'$ is
  $(n-1)$-connected with respect to both boundaries.  Let us first
  point out that doing any such $(n-1)$-surgery does not change the property that
  $\pi_{k}(W,M_1) = 0$ for $k \leq (n-1)$: up to homotopy it amounts
  to cutting out a manifold of codimension $(n+1)$ and then attaching
  a cells of dimension $n$ and $2n$.  We have assumed that $L \to B$
  is $(n-1)$-connected so $\pi_k(M_0) \to \pi_k(B)$ is surjective for
  $k \leq n-1$.  Since $M_0 \in \mathcal{C}_{\theta,L}^{n-1,n-2}$, it
  is an isomorphism for $k \leq n-2$ and similarly for $M_1$.  Since
  $\pi_k(M_1) \to \pi_k(W)$ is an isomorphism for $k \leq n-2$, we
  conclude that $\pi_k(M_0) \to \pi_k(W)$ is an isomorphism for $k \leq
  n-2$, but it need not be surjective for $k=n-1$.  In fact, the long
  exact sequence in homotopy groups identifies the cokernel with
  $\pi_{n-1}(W,M_0) \cong H_{n-1}(\widetilde{W}, \widetilde{M}_0)$.  The fact that
  $\pi_{n-1}(M_0) \to \pi_{n-1}(B)$ is surjective implies that the composition
  \begin{equation*}
    \Ker\bigl(\pi_{n-1}(W) \to \pi_{n-1}(B)\bigr) \lra \pi_{n-1}(W) \lra
    \pi_{n-1}(W,M_0)
  \end{equation*}
  is still surjective.  By the Hurewicz theorem, $\pi_{n-1}(W,M_0) \cong
  H_{n-1}(\widetilde{W}, \widetilde{M}_0)$ is finitely generated as a module
  over $\pi_1$, and we have proved that there exist finitely many
  elements $\alpha_i \in \Ker(\pi_{n-1}(W) \to \pi_{n-1}(B))$ which
  generate the cokernel of $\pi_{n-1}(M_0) \to \pi_{n-1}(W)$.  These
  elements may be represented by disjoint embedded framed spheres in
  the interior of $W$, and as $L$ has a handle structure with only handles of index $< n$ they can be made disjoint from $L$, and we let $W'$ denote the result of performing
  surgery. Both pairs $(W',M_0)$ and $(W',M_1)$ are now
  $(n-1)$-connected, and by Proposition \ref{prop:ConnectSum}, $W'$
  again admits a $\theta$-structure.

  We now return to the proof of the lemma. There is a zig-zag of
  morphisms in the category $\mathcal{C}_{\theta,L}^{n-1,n-2}(\bR^N)$
  from $M$ to an object of $\mathcal{A}$, as $\mathcal{A}$ was assumed
  to hit the path component of $M$. By the above discussion we can
  suppose that it is a zig-zag of $\theta$-cobordisms which are
  $(n-1)$-connected relative to both ends. Then, by reversibility, we
  can reverse the backwards-pointing arrows and obtain a single
  morphism
  $$(C, \ell_C) : (M, \ell_{M}) \leadsto (A, \ell_A) \in
  \mathcal{C}_{\theta,L}^{n-1,n-2}(\bR^N),$$which is $(n-1)$-connected relative to both ends, so $\pi_*(C, A) =
  \pi_*(C, M) = 0$ for $* \leq n-1$.

  If such a cobordism $C$ admits a Morse function with only critical
  points of index $n$, then the descending manifolds of the critical
  points, and $\ell_C$ restricted to them, gives the required
  $\theta$-surgery data. It remains to produce such a Morse function.

  If $\pi_1(L)=0$ then all of the manifolds appearing above are also
  simply-connected, and we deduce by Poincar{\'e} duality and the
  Universal coefficient theorem that $H_*(C, M)$ is concentrated in
  degree $n$ and is free abelian. We can choose a self-indexing Morse
  function on $C$ and as in the proof of the $h$-cobordism theorem we
  can first modify it to have no critical points of index 0 or 1
  \cite[Theorem 8.1]{MilnorHCob}, do the same to the negative of the
  Morse function to remove critical points of index $2n$ and $(2n-1)$,
  and finally by the Basis Theorem \cite[Theorem 7.6]{MilnorHCob} we
  can diagonalise the differentials in the Morse homology complex, and
  so modify the Morse function to only have critical points of index
  $n$.

  When $\pi_1(L) \neq 0$ we must go to a little more trouble, and use
  techniques from the proof of the $s$-cobordism theorem. As these are
  less well known, we go into more detail, but recommend
  \cite{LuckSCobordism} and \cite{KervaireSCobordism} for details of
  that argument. As above, pick a self-indexing Morse function on $C$
  and let us write
  $$\pi = \pi_1(L) = \pi_1(M) = \pi_1(C) = \pi_1(A)$$
  for the common fundamental group, and $\bZ[\pi]$ for its integral
  group ring.

  When $M \hookrightarrow C$ is 1-connected, \cite[Theorem
  8.1]{MilnorHCob} is still true: we may modify the Morse function to
  have no critical points of index 0 or 1, and as above do the same on
  the opposite Morse function to eliminate critical points of index
  $2n$ and $(2n-1)$. The cores of the handles given by this Morse
  function on the universal cover give a cell complex with cellular
  chain complex $C_*(\widetilde{C}, \widetilde{M})$, and
  $C_*(\widetilde{C}, \widetilde{A})$ for the opposite Morse
  function. These are chain complexes of based free
  $\bZ[\pi]$-modules, and geometric Poincar{\'e} duality gives an
  isomorphism
  $$C_*(\widetilde{C}, \widetilde{M}) \cong \mathrm{Hom}_{\bZ[\pi]}(C_{2n-*}(\widetilde{C}, \widetilde{A}), \bZ[\pi])$$
  of chain complexes, by sending basis elements to their ``dual" basis
  elements (we use the convention of \cite[Ch.\ 2]{Wall} to interchange right and left $\bZ[\pi]$-module structures).
  
  The chain complex $C_{2n-*}(\widetilde{C}, \widetilde{A})$ is one of
  free $\bZ[\pi]$-modules and $0 = \pi_*(C, A) = \pi_*(\widetilde{C},
  \widetilde{A}) = H_*(\widetilde{C}, \widetilde{A} ; \bZ)$ for $*
  \leq n-1$, so it is acyclic in degrees $2n-* \leq n-1$. By the
  Universal coefficient spectral sequence, the same is true for its
  $\bZ[\pi]$-dual and so $C_*(\widetilde{C}, \widetilde{M})$ is
  acyclic for $* \geq n+1$. Furthermore $0 = \pi_*(C, M) =
  \pi_*(\widetilde{C}, \widetilde{M}) = H_*(\widetilde{C},
  \widetilde{M} ; \bZ)$ for $* \leq n-1$, so the homology of
  $C_*(\widetilde{C}, \widetilde{M})$ is concentrated in degree
  $n$. By the usual modification technique, we can use handle
  exchanges to modify the Morse function to only have critical points
  of index $n$ and $(n-1)$.  We are left with a short exact sequence
  of $\bZ[\pi]$-modules
  $$0 \lra H_n(\widetilde{C}, \widetilde{M} ; \bZ) \lra
  C_n(\widetilde{C}, \widetilde{M}) \overset{\partial_n}\lra
  C_{n-1}(\widetilde{C}, \widetilde{M}) \lra 0.$$

  The rightmost term is a free $\bZ[\pi]$-module and so this sequence
  is split: in particular, $H_n(\widetilde{C}, \widetilde{M} ; \bZ)$
  is stably free as a $\bZ[\pi]$-module. If $H_n(\widetilde{C},
  \widetilde{M} ; \bZ)$ is not actually free as a $\bZ[\pi]$-module,
  there cannot exist a Morse function on $C$ with only critical points
  of index $n$. In this case we replace $C$ by $C \#^g S^n\times S^n$
  for $g$ sufficiently large (and this manifold admits a
  $\theta$-structure by Proposition \ref{prop:ConnectSum}).  This has the effect
  of adding on a large free $\bZ[\pi]$-module to $H_n(\widetilde{C},
  \widetilde{M} ; \bZ)$, so we may assume that this homology group is
  now free, and pick a basis of it.

  Choosing a splitting of the short exact sequence above, we obtain an
  isomorphism
  \begin{equation}
    \label{eq:7}
    C_n(\widetilde{C}, \widetilde{M}) \cong H_n(\widetilde{C},
    \widetilde{M} ; \bZ) \oplus C_{n-1}(\widetilde{C}, \widetilde{M})
  \end{equation}
  of based free $\bZ[\pi]$-modules, and so an element of
  $K_1(\bZ[\pi])$. However, the basis we chose for $H_n(\widetilde{C},
  \widetilde{M} ; \bZ)$ was not geometrically meaningful and we are
  free to change it.  After possibly stabilising $C$ further, it is
  possible to choose a basis for which \eqref{eq:7} represents the
  zero class in $K_1(\bZ[\pi])$, and hence in the Whitehead group
  $\mathrm{Wh}(\pi)$ too.  We may then use the modification lemma to
  rearrange the index $n$ critical points of the Morse function so
  that $\partial_n : C_n(\widetilde{C}, \widetilde{M}) \to
  C_{n-1}(\widetilde{C}, \widetilde{M})$ is simply projection onto the
  first few basis elements: this allows us to cancel all the critical
  points of index $(n-1)$.
\end{proof}

\begin{proposition}
  $\widetilde{Y}_0(a, \epsilon, (W, \ell_W))$ is non-empty as long as
  $3n+1 < N$, $2n \geq 6$, $\theta$ is reversible, $L$ admits a handle
  structure with only handles of index $< n$, $\ell_L : L \to B$ is
  $(n-1)$-connected, and the natural map $\mathcal{A} \to
  \pi_0(B\mathcal{C}_{\theta, L}^{n-1, n-2}(\bR^N))$ is surjective.
\end{proposition}
\begin{proof}
  We follow the proof of Proposition \ref{prop:YNonEmpty}, with a few
  changes. Let $d=2n$ and $l=n-1$. The first step of Proposition
  \ref{prop:YNonEmpty} is to produce for each $W\vert_{a_i}$ the
  $\theta$-surgery data $f_i\vert_{a_i}$. The method described in that
  proposition no longer works, and we use Lemma
  \ref{lem:SurgeryDataExists} to produce the necessary data
  instead. From this point up to constructing the maps $e_i
  \vert_{(a_i-\epsilon_i, a_i+\epsilon_i)}$ there is no difference,
  and the argument given in Proposition \ref{prop:YNonEmpty} goes through.

  It remains to explain how given an embedding $e_i\vert_{a_j}$ we can
  extend it to $e_i\vert_{[a_j, a_{j+1}]}$. We proceed in the same
  way: we have the embedding
  $$f_i \vert_{a_j} : \Lambda_i \times \{a_j\} \times \bR^{n} \times S^{n-1} \hookrightarrow W\vert_{a_j}$$
disjoint from $L$, which extends to a continuous map
  $$f_i \vert_{[a_j, a_{j+1}]} : \Lambda_i \times [a_j, a_{j+1}] \times \bR^{n} \times S^{n-1} \lra W\vert_{[a_j, a_{j+1}]}$$
  as $(W\vert_{[a_j, a_{j+1}]}, W\vert_{a_{j+1}})$ is
  $(n-1)$-connected by assumption. We can again make this be a
  self-transverse immersion of the core, but this no longer implies
  that the core is embedded: it will have isolated points of
  self-intersection. As $2n \geq 6$ we can remove these using the
  Whitney trick, as in the proof of Proposition
  \ref{prop:RunOutOfNames}. The core may still intersect the core of $[a_j , a_{j+1}] \times L$, as
  they are both of dimension $n$ inside a $2n$-manifold, but we can
  again use the Whitney trick to separate them. Given $f_i
  \vert_{[a_j, a_{j+1}]}$ which is an embedding of the core and whose
  core is disjoint from that of $[a_j, a_{j+1}] \times L$, we
  can shrink in the $\bR^n$ direction and isotope it to get an
  embedding disjoint from $[a_j, a_{j+1}] \times L$, and then extend this to $e_i
  \vert_{[a_j, a_{j+1}]}$ as in Proposition \ref{prop:YNonEmpty}.

  This gives the required embeddings $e_i$, which are then combined as in Proposition \ref{prop:YNonEmpty} to get $(\Lambda, \delta, e)$, the embedding part of the data of an element of $\widetilde{Y}_0(a, \epsilon, (W, \ell_W))$. The remaining bundle part of the
  data consists of an extendible (cf.\ Definition
  \ref{defn:Extendible}) $\theta$-structure $\ell$ on $\Lambda
  \times K$which agrees with $\ell_W \circ D(\partial e)$ on $\Lambda \times
  K \vert_{(-6, -2)}$, and such that the effect of the
  $\theta$-surgery described by this data (i.e.\ the restriction of
  $\ell$ to $K\vert_{(-6,0]}$) lies in $\mathcal{A}$.  We will
  describe a construction which for each $\lambda \in \Lambda$
  produces a $\theta$-structure $\ell_{\lambda}$ on $K
  \subset \R^{n+1} \times \R^n$; these are then combined in the obvious way.

  Firstly, there is a unique $\theta$-structure on the subspace
  \begin{equation}\label{eq:8}
    K\vert_{(-6,-2)} = \big((-6,-2) \times \R^n\big) \times S^{n-1},
  \end{equation}
  such that the embedding $\partial e$ preserves $\theta$-structures (i.e.\
  satisfies requirement (\ref{it:YComplexMidTheta}) of Definition~\ref{defn:YComplexMid}).
  Secondly, the manifold
  \begin{equation}
    \label{eq:9}
    K\vert_{(-6,0]} \subset \R^{n+1} \times \R^n
  \end{equation}
  is obtained from~\eqref{eq:8} by attaching an $n$-handle.  To extend
  the $\theta$-structure requires a null-homotopy of the structure map
  $S^{n-1} \to B$ from the $\theta$-structure on~\eqref{eq:8}, and
  this is provided as part of the $\theta$-surgery data in
  Lemma \ref{lem:SurgeryDataExists}.  Finally, we need to prove that this
  structure extends to a $\theta$-structure over all of $K$, which is furthermore extendible.  To
  see this, we observe that the restriction of this structure to the
  subspace
  \begin{equation}\label{eq:11}
    K\vert_{(-6,0]} - (B_2^{n+1}(0) \times \R^n) = \big((-6,0] \times
    \R^n - B_2^{n+1}(0) \big) \times S^{n-1}
  \end{equation}
  admits a (homotopically unique) extension to the manifold $\big(\R
  \times \R^n\big) \times S^{n-1}$, since this deformation retracts
  to~\eqref{eq:11}.  Restrict this extension to the manifold
  $\big(\R\times \R^n - B_2^{n+1}(0)\big) \times S^{n-1} = K -
  (B_2^{n+1}(0) \times \R^n)$ and glue it with the structure
  on~\eqref{eq:9}, to get a $\theta$-structure on
  \begin{equation}
    \label{eq:12}
    K\vert_{(-\infty,0]} \cup (K - B_2^{n+1}(0)\times \R^n).
  \end{equation}
  It remains to prove that this structure can be extended to all of
  $K$.  It is easy to see that the stabilised structure extends to a
  bundle map $\epsilon^1 \oplus TK \to \theta^* \gamma$, but then
  Lemma~\ref{lemma:stable-extension} implies that the unstabilised
  bundle map also extends.
\end{proof}

%%% Local Variables: 
%%% mode: latex
%%% TeX-master: "Moduli"
%%% End: 

\section{Proofs of the main theorems}
\label{sec:ApplyingGC}

In this section, we use the results of
Sections~\ref{sec:surgery-morphisms}--\ref{sec:Connectivity} to prove
the theorems stated in Section~\ref{sec:intr-stat-results}.  As
explained in Remark~\ref{rem:limit-of-B-diff-theta},
Theorem~\ref{thm:main-A} follows from Theorem~\ref{thm:main-C-new},
which we prove in full detail in Sections~\ref{sec:category-mathcalc},
\ref{sec:UnivThetaEnds}, \ref{sec:group-completion},
and~\ref{sec:proof-theor-refthm:m} below.  Nevertheless, we shall
first outline in some detail how to deduce Theorem~\ref{thm:main-A}
directly from the results in
Sections~\ref{sec:surgery-morphisms}--\ref{sec:Connectivity}, in the
hope of putting the general case in a useful perspective.

In the following we shall work entirely in even dimension $d = 2n > 4$
and always set $N = \infty$.  (Suitably interpreted, all results hold
for sufficiently large finite $N$, but we shall not pursue this here.)

\subsection{Outline of proof of Theorem~\ref{thm:main-A}}
\label{sec:outl-proof-theor}

To apply the theorems in Sections~\ref{sec:surgery-morphisms}
and~\ref{sec:surg-objects-below}, we must specify a structure $\theta:
B \to BO(2n)$ and a $(2n-1)$-dimensional manifold $L$ with
$\theta$-structure $\ell_L: \epsilon^1 \oplus TL \to \theta^* \gamma$.
For the purpose of deducing Theorem~\ref{thm:main-A}, we let $\theta =
\theta^n: BO(2n)\langle n\rangle \to BO(2n)$ be the $n$-connective
cover, and let $L \subset (-1,0] \times \R^N$ be a $(2n-1)$-manifold with
collared boundary, diffeomorphic to $D^{2n-1}$.  Now, the inclusion
functors induce weak equivalences
\begin{equation*}
  B\mathcal{C}_{\theta^n,L}^{n-1,n-2} \simeq
  B\mathcal{C}_{\theta^n,L}^{n-1} \simeq
  B\mathcal{C}_{\theta^n,L} \simeq \psi_{\theta^n,L}(\infty,1) \simeq
  \psi_{\theta^n}(\infty,1)\simeq \Omega^{\infty-1} MT\theta^n,
\end{equation*}
obtained by applying Theorem~\ref{thm:lfiltration} $(n-1)$ times,
Theorem~\ref{thm:kappafiltration} $n$ times,
Proposition~\ref{prop:SpaceModelL} and
Proposition~\ref{prop:forget-L}, respectively, composed with the weak
equivalence $\psi_{\theta^n}(\infty,1) \simeq \Omega^{\infty-1}
MT\theta^n$ from \cite[Theorem 3.12]{GR-W}.

To apply the result of Section~\ref{sec:surg-objects-middle}, we must
specify a subset $\mathcal{A} \subset
\pi_0(\mathrm{Ob}(\mathcal{C}^{n-1,n-2}_{\theta^n,L}))$.  There is a
unique path component of
$\mathrm{Ob}(\mathcal{C}^{n-1,n-2}_{\theta^n,L})$ consisting of
manifolds diffeomorphic to $S^{2n-1}$ (with its standard smooth
structure).  Letting $\mathcal{A}$ consist of this path component,
$\mathcal{C}^{n-1,\mathcal{A}}_{\theta^n,L}$ is the full subcategory
of $\mathcal{C}^{n-1,n-2}_{\theta^n,L}$ on the objects in
$\mathcal{A}$.  It is clear that $\theta^n$ is spherical and hence
reversible (cf.\ Proposition~\ref{prop:Reversible}), that $L \cong
D^{2n-1}$ admits a handle decomposition using only handles of index
less than $n$ (since a single 0-handle suffices), and that the map
$\ell_L: D^{2n-1} \to BO(2n)\langle n\rangle$ is $(n-1)$-connected (it
is even $n$-connected).  Theorem~\ref{thm:MidSurgery} would give the
weak equivalence $B\mathcal{C}^{n-1,\mathcal{A}}_{\theta^n,L} \simeq
B\mathcal{C}^{n-1,n-2}_{\theta^n,L}$, except that that theorem
requires $\mathcal{A}$ to contain at least one object from each path
component of $B\mathcal{C}^{n-1,n-2}_{\theta^n,L}$, which may not hold
here.  Therefore we let $\overline{\mathcal{A}} \subset \pi_0(
\mathrm{Ob}(\mathcal{C}^{n-1,n-2}_{\theta^n,L}))$ be the union of
$\mathcal{A}$ and the set of path components of objects which map to a
path component of $B\mathcal{C}^{n-1,n-2}_{\theta^n,L}$ disjoint from
that of $\mathcal{A}$.  Theorem~\ref{thm:MidSurgery} does apply to
$\overline{\mathcal{A}}$ and gives the weak equivalence
\begin{equation*}
  B\mathcal{C}^{n-1,\overline{\mathcal{A}}}_{\theta^n,L} \simeq
  B\mathcal{C}^{n-1,n-2}_{\theta^n,L} \simeq \Omega^{\infty-1} MT\theta^n.
\end{equation*}
By definition, the inclusion
$B\mathcal{C}^{n-1,{\mathcal{A}}}_{\theta^n,L} \subset
B\mathcal{C}^{n-1,\overline{\mathcal{A}}}_{\theta^n,L}$ is just the
inclusion of a path component, and hence becomes a homeomorphism after
taking based loop space, so we get the weak equivalence
\begin{equation*}
  \Omega B \mathcal{C}^{n-1,\mathcal{A}}_{\theta^n,L} \simeq
  \Omega^\infty MT\theta^n.
\end{equation*}
The category $\mathcal{C}^{n-1,\mathcal{A}}_{\theta^n,L}$ is not quite
a monoid, since it contains multiple objects (namely all those
manifolds diffeomorphic to $S^{2n-1}$), but the space of objects is
path connected, and we let $\mathcal{M}$ be the endomorphism monoid of
some chosen object. Then the nerve $N_p\mathcal{M}$ is the fibre of the fibration
$N_p\mathcal{C}^{n-1,\mathcal{A}}_{\theta^n,L} \to
(N_0\mathcal{C}^{n-1,\mathcal{A}}_{\theta^n,L})^{p+1}$, and the
Bousfield--Friedlander theorem (\cite[Theorem B.4]{MR513569} or the
earlier special case \cite[Theorem 12.7]{IteratedLoopSpaces}) implies
that the inclusion $B\mathcal{M} \to
B\mathcal{C}^{n-1,\mathcal{A}}_{\theta^n,L}$ is a weak equivalence
(this can also be seen more geometrically as in \cite[Proposition
4.26]{GR-W}).  Altogether, we obtain a weak equivalence $\Omega B
\mathcal{M} \simeq \Omega^\infty MT\theta^n$.

The monoid $\mathcal{M}$ is described up to homotopy as
\begin{equation*}
  \mathcal{M} \simeq \coprod_{W} B\Diff(W,D),
\end{equation*}
where $W$ ranges over $(n-1)$-connected closed $2n$-manifolds
admitting a $\theta^n$-structure, and $D \subset W$ is a submanifold
equipped with a diffeomorphism $D \cong D^{2n}$.  (Admitting a
$\theta^n$-structure is equivalent to being parallelisable over the
$n$-skeleton.  Since the pair $(W,D)$ is $(n-1)$-connected, the space
of $\theta^n$-structures is contractible when it is non-empty.)  In
this description, the monoid structure corresponds to connected sum
and therefore $\mathcal{M}$ is homotopy commutative.  The classical
``group completion'' theorem (cf.\ \cite{McDuff-Segal}) then gives an
isomorphism in homology
\begin{equation*}
  H_*(\mathcal{M})[\pi_0 \mathcal{M}^{-1}] \stackrel{\cong}{\lra}
  H_*(\Omega^\infty MT\theta^n),
\end{equation*}
where the left hand side denotes the ring $H_*(\mathcal{M})$ localised
by inverting the multiplicative subset $\pi_0\mathcal{M}$.  Finally,
we claim that the localisation on the left hand side may be calculated
by inverting only the element of $\pi_0\mathcal{M}$ corresponding to
$T = S^n \times S^n$.  To see this, we use that if $W$ is an element
of $\mathcal{M}$, then there is another element $\overline{W}$ with
the same underlying manifold, but where the identification $D^{2n}
\cong D \subset W$ is changed by an orientation-reversing
diffeomorphism.  Then the connected sum $W \# \overline{W}$ may be identified with $\partial ((W - \Int(D))
\times [0,1])$, which we claim is diffeomorphic to the connected sum
of $b$ copies of $T$, where $b = \mathrm{rank}(H_n(W))$. This can be
seen by picking a minimal Morse function on the bounding manifold $(W
- \Int(D)) \times [0,1]$.  (It has homology $\bZ$ in degree 0 and
$\bZ^b$ in degree $n$ and is parallelisable; cancelling critical
points in a Morse function as in \cite{MilnorHCob} proves that $(W -
\Int(D)) \times [0,1]$ is diffeomorphic to the boundary connected sum
of $b$ copies of $S^{n} \times D^{n+1}$.)  Therefore the element $[W]
\in \pi_0\mathcal{M}$ is invertible in the ring $H_*(\mathcal{M})
[T^{-1}]$, with inverse $[T]^{-b}[\overline{W}]$.  The localisation by
inverting the element $[T]$ may be calculated as a direct limit, and
hence we have the homology equivalence
\begin{equation*}
  \hocolim(\mathcal{M} \stackrel{\cdot T}{\to} \mathcal{M}
  \stackrel{\cdot T}{\to} \cdots) \lra \Omega^\infty MT\theta^n,
\end{equation*}
which upon restricting to the appropriate path component gives
Theorem~\ref{thm:main-A}.\qed\medskip

We now embark on the detailed proof of Theorem~\ref{thm:main-C-new},
which will occupy
Sections~\ref{sec:category-mathcalc},~\ref{sec:UnivThetaEnds},
\ref{sec:group-completion}, and~\ref{sec:proof-theor-refthm:m}, as
follows.  Suppose given a spherical tangential structure $\theta: B
\to BO(2n)$, a $(2n-1)$-manifold $L$ which admits a handle structure
with handles of index less than $n$, and a $\theta$-structure $\ell_L:
\epsilon^1 \oplus TL \to \theta^* \gamma$ such that the underlying map
$L \to B$ is $(n-1)$-connected.  In this situation the results of
Section~\ref{sec:surgery-morphisms}--\ref{sec:Connectivity} apply, and
will be summarised in Section~\ref{sec:category-mathcalc} below as a
weak equivalence between $\Omega^\infty MT\theta$ and the loop space
of the classifying space of a category $\mathcal{C}$.  Then in
Section~\ref{sec:group-completion} we apply a version of the ``group
completion'' theorem to relate the homology of $\Omega B\mathcal{C}$
to the homology of morphism spaces of $\mathcal{C}$, suitably
localised using the theory of universal $\theta$-ends developed in
Section~\ref{sec:UnivThetaEnds}.  In
Section~\ref{sec:proof-theor-refthm:m} we explain how to apply these
results to prove Theorem~\ref{thm:main-C-new}. Finally, in
Section~\ref{sec:proof-lemma-refl} we explain how to deduce the
results about algebraic localisation from
Theorem~\ref{thm:main-C-new}.

\subsection{The category $\mathcal{C}$}\label{sec:category-mathcalc}

Suppose that $2n > 4$, let $\theta: B \to BO(2n)$ be a spherical
tangential structure, and $L$ be a $(2n-1)$-dimensional manifold with
boundary which admits a handle structure using handles of index at
most $(n-1)$. Let $\ell_L$ be a $\theta$-structure on $L$, \emph{and
  suppose that the underlying map $L \to B$ is $(n-1)$-connected}.

Picking a collared embedding $L \hookrightarrow (-1/2,0] \times
(-1,1)^{\infty-1}$, we have defined a category
$\mathcal{C}_{\theta,L}^{n-1,n-2}$.  Finally, let $\mathcal{A}\subset
\pi_0(\Ob(\mathcal{C}_{\theta,L}^{n-1,n-2}))$ be the set of objects
$(M,\ell)$ for which $M - \Int(L)$ is diffeomorphic to a handlebody
with handles of index at most $(n-1)$. In Definition \ref{defn:CobCatA} we have defined
\begin{equation*}
  \mathcal{C}_{\theta,L}^{n-1,{\mathcal{A}}} \subset
  \mathcal{C}_{\theta,L}^{n-1,n-2}
\end{equation*}
as the full subcategory on those objects contained in ${\mathcal{A}}$.

\begin{definition}\label{defn:setup} 
A morphism in
  $\mathcal{C}_{\theta,L}^{n-1,\mathcal{A}}$ is a manifold $W \subset
  [0,t] \times (-1,1)^\infty$ with $W \cap x_2^{-1}((-\infty,0]) =
  [0,t] \times L$ (equality as $\theta$-manifolds).  Write
  \begin{equation*}
    W^\circ = W - ([0,t] \times \Int(L))
  \end{equation*}
  for morphisms and similarly
  \begin{equation*}
    M^\circ = M - \Int(L)
  \end{equation*}
  for objects. Morphisms or objects $X \in
  \mathcal{C}_{\theta,L}^{n-1,\mathcal{A}}$ are completely determined
  by $X^\circ$ and we denote by $\mathcal{C}$ the category with
  \begin{align*}
    \Ob(\mathcal{C}) &= \{M^\circ \;\vert\; M \in \Ob(\mathcal{C}_{\theta,L}^{n-1,\mathcal{A}})\}\\
    \Mor(\mathcal{C}) &= \{W^\circ \;\vert\; W \in \Mor(\mathcal{C}_{\theta,L}^{n-1,\mathcal{A}})\}
  \end{align*}
  made into a topological category by insisting that the functor
  $\mathcal{C}_{\theta,L}^{n-1,\mathcal{A}} \to \mathcal{C}$ given by
  $X \mapsto X^\circ$ is an isomorphism of topological categories.
\end{definition}

Our work in Sections~\ref{sec:surgery-morphisms},
\ref{sec:surg-objects-below}, \ref{sec:surg-objects-middle}, and
\ref{sec:Connectivity} determines the homotopy type of the space
$\Omega B \mathcal{C}$, as follows.  (We emphasise that in this
section $L \to B$ is assumed $(n-1)$-connected and $\theta: B \to
BO(2n)$ is assumed spherical.)

\begin{theorem}\label{thm:gp-compl}
  There is a weak equivalence
  \begin{equation*}
    \Omega B\mathcal{C} \simeq \Omega^\infty MT\theta,
  \end{equation*}
  where loops are based at any object $P^\circ \in \Ob(\mathcal{C})$, and
  $MT\theta$ is the Thom spectrum associated to $\theta : B \to
  BO(2n)$.
\end{theorem}
\begin{proof}
  This is identical with the argument given in Section
  \ref{sec:outl-proof-theor}. Briefly, we define the set $\overline{\mathcal{A}}$ to be the union
  of $\mathcal{A}$ and all objects not in a path component of
  $B\mathcal{C}_{\theta,L}^{n-1,n-2}$ containing an element of
  $\mathcal{A}$, and use the string of weak equivalences
  \begin{equation*}
  B\mathcal{C}_{\theta,L}^{n-1,\overline{\mathcal{A}}} \simeq
    B\mathcal{C}_{\theta,L}^{n-1,n-2} \simeq
    B\mathcal{C}_{\theta,L}^{n-1} \simeq
    B\mathcal{C}_{\theta,L} \simeq \psi_{\theta,L}(\infty,1) \simeq
    \psi_\theta(\infty,1) \simeq
    \Omega^{\infty-1} MT\theta
  \end{equation*}
as well as the homeomorphism $B\mathcal{C} \cong B\mathcal{C}^{n-1,
  \mathcal{A}}_{\theta, L}$, and the fact that the inclusion
  $B\mathcal{C}^{n-1, \mathcal{A}}_{\theta, L} \to B\mathcal{C}^{n-1,
    \overline{\mathcal{A}}}_{\theta, L}$ is a homeomorphism onto the
  path components it hits.
\end{proof}

From now on we will work with the category $\mathcal{C}$, and we need
a lemma to translate what the connectivity conditions in
$\mathcal{C}_{\theta,L}^{n-1,\mathcal{A}}$ mean after cutting $\Int(L)$ out.

\begin{lemma}\label{lem:ConnectivityInCPrelim}\mbox{}
\begin{enumerate}[(i)]
\item\label{it:1:ConnectivityInCPrelim} Let $N$ be an object in $\mathcal{C}_{\theta, L}^{n-1, \mathcal{A}}$ and $W : M \leadsto N$ be a morphism in the larger category $\mathcal{C}_{\theta, L}^{n-1, n-2}$. Then the pair $(W, M)$ is $(n-1)$-connected.

\item \label{it:2:ConnectivityInCPrelim}Let $W^\circ : M^\circ \leadsto N^\circ$ be a morphism in $\mathcal{C}$. Then the pairs $(W^\circ, M^\circ)$ and $(W^\circ, N^\circ)$ are $(n-1)$-connected.
\end{enumerate}
\end{lemma}
\begin{proof}
By definition, $\mathcal{A}$ consists of manifolds $M$ such that $M
  - \Int(L)$ admits a handle structure with handles of index at most
  $(n-1)$, and reversing such a handle structure we see that this is
  equivalent to $M - \Int(L)$ being obtained from $\partial L$ by
  attaching handles of index at least $n$.
    
  The pairs $(N, \{1\} \times L)$ and $(W, N)$ are both $(n-1)$-connected, from which we deduce that $(W, \{1\} \times L)$, and so $(W, \{0\} \times L)$, is also $(n-1)$-connected. As $(M, \{0\} \times L)$ is $(n-2)$-connected, we deduce that $(W, M)$ is $(n-1)$-connected, which establishes the first part.
  
  For the second part, observe that
  $W$ deformation retracts to $W^\circ \cup N$.  Therefore all pairs
  $(N,L)$, $(M,L)$, $(W,N)$ and $(W,W^\circ)$ are homotopy equivalent
  to relative CW complexes with relative cells of dimension at least
  $n$.  As $n \geq 3$, it follows that all the inclusions between $L$,
  $\partial L$, $N$, $N^\circ, M$, $M^\circ, W$, $W^\circ,$ and
  $\partial W^\circ$ induce isomorphisms on fundamental groups, and we write $\pi$ for
  the common fundamental group. There are isomorphisms
  $$H_*(W, N ;\bZ[\pi]) \cong H_*(W^\circ \cup N, N ; \bZ[\pi]) \cong H_*(W^\circ, N^\circ ;\bZ[\pi])$$
  given by the homotopy equivalence $W^\circ \cup N \simeq W$ and
  excision of $\Int(L)$ respectively, and so $H_*(W^\circ,
  N^\circ;\bZ[\pi])=0$ for $* \leq n-1$. Hence $(W^\circ,N^\circ)$ is
  $(n-1)$-connected, and the same argument applies to the pair
  $(W^\circ, M^\circ)$.
\end{proof}

\begin{definition}\label{defn:setup2}
  Recall from the proof of Proposition \ref{prop:forget-L} that we
  constructed a $\theta$-manifold $D(L)$ which is diffeomorphic to the
  double of $L$.  This contains $L \subset D(L)$ with its standard
  $\theta$-structure, and we write $\overline{L}$ for the
  $\theta$-manifold $D(L)- \Int(L)$. As $L$ has a handle structure
  with handles of index at most $(n-1)$, $D(L)$ can be obtained from
  $L$ by attaching handles of index at least $n$.  We extend the
  embedding of $L$ to an embedding $D(L)\to (-1,1)^\infty$ to get
  objects $D(L) \in \mathcal{C}_{\theta,L}^{n-1,\mathcal{A}}$ and
  $D(L)^\circ = \overline{L} \in \mathcal{C}$.
\end{definition}

The relevance of the category $\mathcal{C}$ to
Theorem~\ref{thm:main-C-new} is evident from the following proposition.

\begin{proposition}\label{prop:cut-L}
  For any object $P \in \mathcal{C}_{\theta,L}^{n-1,\mathcal{A}}$,
  there is a weak equivalence $\varphi_P: \mathcal{C}(\overline{L},
  P^\circ) \to \mathcal{N}^\theta(P,\ell_P)$ such that if $K : P
  \leadsto P'$ is a morphism in
  $\mathcal{C}_{\theta,L}^{n-1,\mathcal{A}}$, then the diagram
  \begin{equation*}
    \xymatrix{
      {\mathcal{C}(\overline{L}, P^\circ)} \ar[r]^{-\circ K^\circ}
      \ar[d]_{\varphi_P} & {\mathcal{C}(\overline{L}, (P')^\circ)}
      \ar[d]_{\varphi_{P'}} \\
      {\mathcal{N}^\theta(P, \ell_P)} \ar[r]^{-\circ K} &
      {\mathcal{N}^\theta(P', \ell_{P'})}
    }
  \end{equation*}
  commutes, i.e.\ $\varphi_P$ is a natural transformation of functors
  $\mathcal{C}_{\theta,L}^{n-1,\mathcal{A}} \to \mathbf{Top}$.
\end{proposition}
\begin{proof}
  $\varphi_P$ is defined as the composition
  \begin{equation*}
    \varphi_P : \mathcal{C}(\overline{L}, P^\circ) \cong
    \mathcal{C}_{\theta,L}^{n-1,\mathcal{A}}(D(L), P) \overset{V \circ
      -}\lra \mathcal{C}_{\theta}^{n-1}(\emptyset, P) \overset{\simeq}{\lra}
    \mathcal{N}^\theta(P, \ell_P),
  \end{equation*}
  where $V : \emptyset \leadsto D(L)$ is the $\theta$-cobordism
  constructed in the proof of Proposition \ref{prop:forget-L} and the
  last map is $(t,W) \mapsto W - t\cdot e_1$.  It is clear that the
  square commutes, so it remains to show that $\varphi_P$ is a
  homotopy equivalence. To do this, consider first the trivial
  cobordism $P \times [0,1]$. This contains $L \times [1/4, 3/4]$,
  which is diffeomorphic to $V$ and has a homotopic
  $\theta$-structure. Cutting this out gives a $\theta$-cobordism $P
  \amalg D(L) \leadsto P$ containing $L \times I$. Composition along
  the incoming $P$ of this $\theta$-cobordism defines a continuous map
  $$\mathcal{N}^\theta(P, \ell_P) \overset{\simeq}\lra \mathcal{C}_\theta^{n-1}(\emptyset, P) \lra \mathcal{C}_{\theta,L}^{n-1}(D(L), P) = \mathcal{C}_{\theta,L}^{n-1,\mathcal{A}}(D(L), P) \cong \mathcal{C}(\overline{L}, P^\circ)$$
which is homotopy inverse to $\varphi_P$.
\end{proof}

\subsection{Universal $\theta$-ends and the proof of Addendum~\ref{add:1}}\label{sec:UnivThetaEnds}

Let $\theta: B \to BO(2n)$ be spherical.  Recall from Definition
\ref{defn:universalthetaend} that a \emph{universal $\theta$-end} is a
submanifold $K \subset [0,\infty) \times \bR^\infty$ with
$\theta$-structure $\ell_K$ such that $x_1 : K \to [0,\infty)$ has the
natural numbers as regular values. We insist that
\begin{enumerate}[(i)]
\item Each $K\vert_{[i,i+1]}$ is a highly connected cobordism, i.e.\ is $(n-1)$-connected relative to either end,\label{it:UnivThetaEnd:2}

\item For each highly connected $\theta$-cobordism $W : K\vert_i
  \leadsto P$, there is an embedding $j : W \hookrightarrow
  K\vert_{[i,\infty)}$, and a homotopy $\ell_K \circ Dj \simeq
  \ell_W$, both relative to $K\vert_i$.\label{it:UnivThetaEnd:3}
\end{enumerate}

We wish to have the notion of universal $\theta$-end available to us in the cobordism category $\mathcal{C}$. Let $K\vert_0, K\vert_1, \ldots$ be a sequence of objects in $\mathcal{C}$, and $K\vert_{[i-1,i]} : K\vert_{i-1} \to K\vert_i$ be a sequence of morphisms in $\mathcal{C}$. For integers
$0 \leq a < b$, let us write
$$K\vert_{[a,b]} = K\vert_{[a, a+1]} \circ K\vert_{[a+1, a+2]} \circ \cdots \circ K\vert_{[b-1, b]}$$
for the composition of the morphisms from $K\vert_a$ to
$K\vert_b$. There are natural inclusions $K\vert_{[0,a]} \subset
K\vert_{[0, a+1]} \subset \cdots$ and we let $K$ denote the union: a
non-compact smooth manifold with $\theta$-structure. The symbol
$K\vert_{[a,b]}$ is not ambiguous, and we can also make sense of
$K\vert_{[a, \infty)} = \cup_{b > a} K\vert_{[a,b]}$. 

\begin{definition}
Say that $K$ is a \emph{universal $\theta$-end in $\mathcal{C}$} if, in the notation just introduced, properties (\ref{it:UnivThetaEnd:2}) and (\ref{it:UnivThetaEnd:3}) above hold, where in (\ref{it:UnivThetaEnd:3}) we require $W$ to be a morphism in $\mathcal{C}$.
\end{definition}

Proposition~\ref{prop:addendum-plus-version-in-C} below proves
Addendum~\ref{add:1}, together with a version for universal
$\theta$-ends in $\mathcal{C}$.  Before giving the proof, we make some
preparations.

\begin{lemma}\label{lem:7.8}
  Let $W : N \leadsto M$ be a highly connected cobordism.  There exist
  cobordisms $F : M \leadsto M$ and $G : N \leadsto N$ such that
  $W\circ F$ and $G \circ W$ both admit handle structures using only
  handles of index $n$. Similarly, if $W$ is a morphism in the
  category $\mathcal{C}$, then $F$ and $G$ can be taken to be
  morphisms in this category, with the same conclusion (in this case,
  attaching handles along embeddings $S^{n-1} \times D^n \to
  \Int(N)$).
\end{lemma}
\begin{proof}
  The pairs $(W, M)$ and $(W, N)$ are both $(n-1)$-connected, so if we
  let $F$ and $G$ be sufficiently large multiples of $([0,1] \times M)
  \# (S^n \times S^n)$ and $([0,1] \times N) \# (S^n \times S^n)$
  respectively then, by the method used in the proof of Lemma
  \ref{lem:SurgeryDataExists}, both $W\circ F$ and $G
  \circ W$ admit the required handle decompositions.
\end{proof}

\begin{proposition}\label{prop:addendum-plus-version-in-C}
  Let $K\vert_{[i,i+1]}$ be a
  sequence of composable morphisms in $\mathcal{C}$ and let $K = \cup
  K\vert_{[0,i]}$ be the infinite composition.  Then $(K,\ell_K)$ is a
  universal $\theta$-end in $\mathcal{C}$ if and only if the following
  conditions hold.
  \begin{enumerate}[(i)]
  \item\label{it:add-C:1} For each integer $i$, the map
    $\pi_n(K\vert_{[i, \infty)}) \to \pi_n(B)$ is surjective, for all
    basepoints in $K$.
  \item\label{it:add-C:2} For each integer $i$, the map
    $\pi_{n-1}(K\vert_{[i, \infty)}) \to \pi_{n-1}(B)$ is injective,
    for all basepoints in $K$.
  \item\label{it:add-C:3} For each integer $i$, each path component of
    $K\vert_{[i, \infty)}$ contains a submanifold diffeomorphic to
    $S^n \times S^n - \Int(D^{2n})$, which in addition has
    null-homotopic structure map to $B$.
  \end{enumerate}
  Similarly, if $K\vert_{[i,i+1]}$ is a sequence of composable, highly
  connected cobordisms in $\mathcal{C}_\theta$, then $K$ is a universal
  $\theta$-end if and only if conditions~(\ref{it:add-C:1}),
  (\ref{it:add-C:2}) and (\ref{it:add-C:3}) hold. (I.e.,
  Addendum~\ref{add:1} is true).
\end{proposition}
\begin{proof}
  We shall only give the proof in $\mathcal{C}$, the other case being
  completely analogous.  To prove the ``if'' direction, we must show
  that for each integer $i$ and each highly connected cobordism $W :
  K\vert_i \leadsto P$ with $\theta$-structure $\ell_W$, there is an
  embedding $j: W \hookrightarrow K\vert_{[i,\infty)}$ and a homotopy
  $\ell_K \circ Dj \simeq \ell_W$, all relative to $K\vert_i$.

  By Lemma \ref{lem:7.8}, for any such $W$ there is a cobordism $F : P
  \leadsto P$ so that $W \circ F$ admits a handle structure with
  handles of index $n$ only, so it suffices to consider the case where
  $W$ consists of a single $n$-handle relative to $K\vert_i$, attached
  along an embedding $S^{n-1} \times D^n \hookrightarrow
  {K}\vert_i$. We need to find an extension of this embedding into
  ${K}\vert_{[i, \infty)}$ (with the correct homotopy class of
  $\theta$-structure).  The map $S^{n-1} \times D^n \to {K}\vert_i \to
  {K}\vert_{[i,\infty)}$ is null-homotopic by assumption
  (\ref{it:add-C:2}): it is certainly null-homotopic when composed
  with ${K}\vert_{[i, \infty)} \to B$, because that composition is
  equal to the composition $S^{n-1} \times D^n \to {K}_i \to {W} \to
  B$.  Thus there is a continuous map $f : {W} \to {K}\vert_{[i,
    \infty)}$ relative to ${K}\vert_i$. Furthermore, as
  $\pi_n({K}\vert_{[i, \infty)}) \to \pi_n(B)$ is surjective by
  assumption (\ref{it:add-C:1}), we can change $f$ by adding on
  elements of $\pi_n({K}\vert_{[i, \infty)})$ so that
  \begin{equation*}
    {W} \overset{f}\lra {K}\vert_{[i, \infty)}
    \xrightarrow{\ell_{{K}}} B
  \end{equation*}
  is homotopic relative to ${K}\vert_i$ to $\ell_{{W}}$.  The
  $\theta$-structures on $W$ and $K$ now give bundle isomorphisms
  \begin{equation*}
    T{W} \cong \ell_{{W}}^*\theta^*\gamma \cong
    f^*\ell_{{K}}^*\theta^*\gamma \quad \text{and}\quad
    T{K}\vert_{[i,\infty)} \cong \ell_{{K}}^*\theta^*\gamma,
  \end{equation*}
  and hence an isomorphism $T{W} \cong f^*T{K}\vert_{[i, \infty)}$
  relative to ${K}\vert_i$, i.e.\ $f: W \to K$ is covered by a bundle
  map $TW \to TK$, which near $K\vert_i$ is the derivative of the
  embedding.  By Smale--Hirsch theory, we may therefore homotope $f: W
  \to K\vert_{[i,\infty)}$ to an immersion, without changing it near
  $K\vert_i$.

  Finally, we explain how to replace the immersion $f: W \to
  K\vert_{[i,\infty)}$ by an embedding.  It suffices to make $f$ an
  embedding near a core $(D^n, \partial D^n) \subset (W,K\vert_i)$
  of the $n$-handle, and we shall write $\hat f: D^n \to
  K\vert_{[i,\infty)}$ for the restriction of $f$.  After changing $f$
  by a small isotopy, we may assume that all self-intersections of
  $\hat{f}$ are transverse.  We shall explain how to remove one
  self-intersection point of $\hat{f}$, changing the homotopy class of
  $f$ in the process.  Around a self-intersection point, choose a
  coordinate $\bR^n \times \bR^n \hookrightarrow {K}\vert_{[i,
    \infty)}$ so that $\bR^n \times \{0\}$ and $\{0\} \times \bR^n$
  give local coordinates around the two preimages of the double
  point. By assumption (\ref{it:add-C:3}) we can find an embedded $S^n
  \times S^n - \Int(D^{2n}) \subset {K}\vert_{[i, \infty)}$ with
  null-homotopic map to $B$.  We can also assume it is disjoint from
  the image of $f$, since $W$ is compact.  Then we choose an embedded
  path from this $S^n \times S^n - \Int(D^{2n})$ to the patch $\bR^n
  \times \bR^n$, and thicken it up: inside this we have a subset
  diffeomorphic to the boundary connect sum
  \begin{equation*}
    (D^n \times D^n) \nat (S^n \times S^n - \Int(D^{2n})),
  \end{equation*}
  which the image of $\hat{f}$ intersects in $D^n \times \{0\} \cup
  \{0\} \times D^n$. Inside this subset there are embedded disjoint
  discs which give the same embedding on the boundary, and we can
  modify $\hat{f}$ by redefining it to have these discs as image
  instead. This reduces by 1 the number of geometric
  self-intersections of $\hat{f}$, and up to homotopy we have added an
  element of $\pi_n(S^n \times S^n - \Int(D^{2n}))$ to the homotopy
  class of $\hat{f}$. As $S^n \times S^n - \Int(D^{2n}) \to
  {K}\vert_{[i, \infty)} \to B$ was null-homotopic, we have not
  changed the homotopy class of $\hat{f}$ in $B$.

  After finitely many steps, we have changed $\hat{f}$ to an
  embedding. The corresponding embedding $f : {W} \to {K}\vert_{[i,
    \infty)}$ (obtained by thickening $\hat{f}$ up again) is homotopic
  to the original one after composing with $\ell_{{K}} : {K}\vert_{[i,
    \infty)} \to B$, so $\ell_{{K}} \circ f \simeq \ell_{{W}}$
  relative to ${K}\vert_i$. Hence the induced $\theta$-structure on
  ${W}$ is homotopic to the given one relative to ${K}\vert_i$.

  To prove the ``only if''
  direction, we must prove that any universal $\theta$-end
  $(K,\ell_K)$ satisfies the three conditions.  It is clear
  that~(\ref{it:add-C:3}) is necessary: For any $i$ we can let $W$ be
  the boundary connected sum of the cylinder $K\vert_i \times [i,i+1]$
  and the (parallelisable) manifold $S^n \times S^n - \Int(D^{2n})$
  equipped with a trivial $\theta$-structure.  Universality implies
  that this admits an embedding into $K\vert_{[i,\infty)}$, and hence
  $S^n \times S^n - \Int(D^{2n})$ does too.

  For property~(\ref{it:add-C:1}), it suffices to prove that for any
  $i$ and any $\alpha \in \pi_n(B)$, there exists a morphism $W_\alpha
  \in \mathcal{C}(K\vert_i, P)$ for some $P$, with $\alpha \in
  \IM(\pi_n(W_\alpha) \to \pi_n(B))$.  To construct such a manifold,
  we may represent $\alpha$ by a map $S^n \to B$ and lift the
  composition $\theta \circ \alpha: S^n \to B \to BO(2n)$ to a map $f:
  S^n \to BO(n)$.  If we let $D \to S^n$ be the disc bundle of the
  vector bundle classified by $f$, the tangent bundle of $D$ is
  classified by $\theta \circ \alpha$, and therefore admits a
  $\theta$-structure whose underlying map $S^n \simeq D \to B$
  represents $\alpha$.  We can then let $W_\alpha$ be the boundary
  connected sum of $K\vert_i \times [i,i+1]$ and $D$.

  Finally, for property~(\ref{it:add-C:2}), we use that each
  $K\vert_{[j,j+1]}$ is a highly connected cobordism to see that
  $\pi_{n-1}(K\vert_i) \to \pi_{n-1}(K\vert_{[i,\infty)})$ is
  surjective.  It therefore suffices to prove that for any $\alpha \in
  \Ker(\pi_{n-1}(K\vert_i) \to \pi_{n-1}(B))$, there exists a morphism
  $W_\alpha \in \mathcal{C}(K\vert_i,P)$ for some $P$, with $\alpha
  \in \Ker(\pi_{n-1}(K\vert_i) \to \pi_{n-1}(W_\alpha))$.  We may
  represent $\alpha$ by an embedding $S^{n-1} \to K\vert_i$.  Since
  the composition $S^{n-1} \to K\vert_i \to B \to BO(2n)$ is trivial,
  the normal bundle of the embedding is stably trivial and hence
  trivial, so we may extend to an embedding $f: S^{n-1} \times D^n \to
  K\vert_i$.  The underlying manifold of the morphism $W_\alpha$ is
  then defined as the trace of surgery along $f$, and a
  $\theta$-structure is constructed from a choice of null-homotopy of
  $S^{n-1} \to K\vert_i \to B$.
\end{proof}

The following three propositions establish further useful properties
of universal $\theta$-ends. The first proposition gives a refinement
of property (\ref{it:UnivThetaEnd:3}), which lets us exert more
control on the behaviour of the embedding $j$ which is provided by
(\ref{it:UnivThetaEnd:3}).  Propositions~\ref{prop:UniquenessThetaEnd}
and~\ref{prop:existence-theta-ends} give strong existence and
uniqueness properties for universal $\theta$-ends (and universal
$\theta$-ends in $\mathcal{C}$), which essentially say that a
universal $\theta$-end $(K, \ell_K)$ is determined up to
diffeomorphism (respecting $\theta$-structures) by $(K\vert_0,
\ell_K\vert_0)$.

\begin{proposition}\label{prop:HighConnComplement}
  If $(K, \ell_K)$ is a universal $\theta$-end (or a universal
  $\theta$-end in $\mathcal{C}$) then it also satisfies
\begin{enumerate}[(i$^\prime$)]
\setcounter{enumi}{1}
\item For each highly connected $\theta$-cobordism $W : K\vert_i
  \leadsto P$, there is a $k \gg i$, an embedding $j : W
  \hookrightarrow K\vert_{[i,k]}$, and a homotopy $\ell_K \circ Dj
  \simeq \ell_W$, both relative to $K\vert_i$, such that the complement
  of $j(W)$ is a cobordism $Z : P \leadsto K\vert_k$ which is highly
  connected.
\end{enumerate}
\end{proposition}
\begin{proof}
  Let us treat the case of a universal $\theta$-end; working in
  $\mathcal{C}$ can be done in the same way. As $W$ is
  $(n-1)$-connected relative to either end, Lemma \ref{lem:7.8}
  applies, and for sufficiently large $g$, the manifold $W' = W \# (g
  S^n \times S^n) = W \circ (([0,1] \times P) \#(g S^n \times S^n))$
  admits a handle structure relative to $K\vert_i$ using handles of
  index $n$ only.

  By universality, there is an embedding of $\theta$-manifolds $j' :
  W' \hookrightarrow K\vert_{[i,k']}$ relative to $K\vert_i$. We wish
  to modify this embedding, and increase $k'$, so that if
  \begin{equation*}
    \{e_\alpha : (D^n \times D^n, D^n \times S^{n-1}) \hookrightarrow
    (W', K\vert_i)\}_{\alpha \in I}
  \end{equation*}
  denotes the collection of relative $n$-handles of $W'$, there exist
  embedded spheres $\{f_\beta : S^n \hookrightarrow
  K\vert_{[i,k']}\}_{\beta \in I}$ so that
  \begin{equation*}
    e_\alpha(\{0\} \times D^n) \cap f_\beta(S^n) =
    \begin{cases}
      \emptyset & \alpha \neq \beta\\
      \{*\} & \alpha = \beta.
    \end{cases} 
  \end{equation*}
  We can ensure this property as follows: by property (\ref{it:add:3})
  of Proposition \ref{prop:addendum-plus-version-in-C}, we may find an
  embedded $S^n \times S^n - \Int(D^{2n})$ in $K\vert_{[k',\infty)}$
  with null-homotopic structure map to $B$. We may form the
  connect-sum of $S^n \times D^n$ with the handle $e_\alpha$ away from
  the other handles, and let $f_\alpha$ be the embedding of $\{*\}
  \times S^n$. Repeating this for each handle, we can ensure the
  required property, and because the $S^n \times S^n - \Int(D^{2n})$
  we used had trivial structure map to $B$, the new embedding of $W'$
  we obtain still has the correct homotopy class of
  $\theta$-structure.

  We denote by $j' : W' \hookrightarrow K\vert_{[i,k]}$ this improved
  embedding, and $Z'$ the complement of the image of $j'$. The
  cobordism $W'$ only has relative $n$-handles, and $n \geq 3$, so
  $K\vert_i$, $W'$, $P$, $Z'$, $K\vert_{[i,k]}$ and $K\vert_{k}$ all
  have the same fundamental group, which we denote by $\pi$. To
  understand the connectivity of the pair $(Z', K\vert_{k})$, we look
  at the long exact sequence for $\bZ[\pi]$-homology of the triple
  $(K\vert_{[i,k]}, Z', K\vert_{k})$ and use excision
  $(K\vert_{[i,k]}, Z') \sim (W', P)$ to obtain the exact sequence
  \begin{equation*}
    H_n(K\vert_{[i,k]}, K\vert_{k};\bZ[\pi]) \overset{\varphi}\lra
    H_n(W', P;\bZ[\pi]) \lra H_{n-1}(Z', K\vert_{k};\bZ[\pi]) \lra 0.
  \end{equation*}
By Poincar{\'e} duality $H_n(W', P;\bZ[\pi]) \cong \Hom_{\bZ[\pi]}(H_n(W', K\vert_i;\bZ[\pi]), \bZ[\pi])$ we see that $\varphi$ is surjective, as $\varphi([f_\beta]) : H_n(W', K\vert_i;\bZ[\pi]) = \bZ[\pi]\langle e_\alpha \, \vert \, \alpha \in I\rangle \to \bZ[\pi]$ is the dual basis element to $e_\beta$, so $(Z', K\vert_k)$ is $(n-1)$-connected, as required. That the pair $(Z',P)$ is $(n-1)$-connected follows by the long exact sequence for $\bZ[\pi]$-homology of the triple $(K\vert_{[i,k]}, W', K\vert_{i})$ and excision $(Z',P) \sim (K\vert_{[i,k]}, W')$.

Finally, we note that if $Z$ denotes the complement of the image of $j = j'\vert_{W}$, then we have $Z = ([0,1] \times P) \#(g S^n \times S^n) \circ Z'$ so it is also a highly connected cobordism.
\end{proof}

\begin{proposition}\label{prop:UniquenessThetaEnd}
  Let $(K, \ell_K)$ and $(K', \ell_{K'})$ be universal $\theta$-ends,
  and suppose we are given a highly connected $\theta$-cobordism $W :
  K\vert_0 \leadsto K'\vert_0$. Then there is a diffeomorphism
  $\varphi : W \cup_{K'\vert_0} K' \cong K$, and a homotopy $\ell_K
  \circ D\varphi \simeq \ell_{W \cup K'}$, both relative to
  $K\vert_0$. Furthermore, there is a weak homotopy equivalence
  \begin{equation*}
    \hocolim_{i \to \infty} \mathcal{N}^\theta(K\vert_i,
    \ell_K\vert_i) \simeq \hocolim_{i \to \infty}
    \mathcal{N}^\theta(K'\vert_i, \ell_{K'}\vert_i). 
  \end{equation*}
\end{proposition}
\begin{proof}
  By replacing $K'$ with $W \cup_{K'\vert_0} K'$, we may as well
  assume that $K\vert_0 = K'\vert_0$ as $\theta$-manifolds, and that
  $W$ is the trivial cobordism. As $K'$ is a universal $\theta$-end,
  we may find an embedding of $\theta$-manifolds $j_1:K\vert_{[0,1]}
  \hookrightarrow K'\vert_{[0,k'_1]}$ relative to $K\vert_0$, and by
  Proposition \ref{prop:HighConnComplement} we may suppose its
  complement $Z_1 : K\vert_1 \leadsto K'\vert_{k'_1}$ is highly
  connected. Now, as $K$ is a universal $\theta$-end, we may find an
  embedding of $\theta$-manifolds $j'_1:Z_1 \hookrightarrow
  K\vert_{[1,k_{1}]}$ relative to $K\vert_1$, again with
  highly-connected complement $Z_2 : K'\vert_{k'_1} \leadsto
  K\vert_{k_1}$. Together, $j_1^{-1}$ and $j'_1$ give an embedding of
  $\theta$-manifolds $K'\vert_{[0,k'_1]} \hookrightarrow K\vert_{[0,
    k_1]}$. Continuing in this way, we produce the required
  diffeomorphism $\varphi$ and homotopy.

For the second part, note that we have constructed a direct system
$$\mathcal{N}^\theta(K\vert_0) \overset{K\vert_{[0,1]}}\lra \mathcal{N}^\theta(K\vert_1) \overset{Z_1}\lra \mathcal{N}^\theta(K'\vert_{k'_1}) \overset{Z_2}\lra \mathcal{N}^\theta(K\vert_{k_1}) \lra \cdots$$
which contains cofinal subsystems which are also cofinal in either of the direct systems used to form the homotopy colimits in the statement.
\end{proof}

\begin{proposition}\label{prop:existence-theta-ends}
  Let $\pi_n(B)$ be countable.  Then for any object $(M,\ell_M) \in
  \mathcal{C}$ there exists a universal $\theta$-end $(K,\ell_K)$ in
  $\mathcal{C}$ with $(K\vert_0,\ell_K\vert_0) = (M,\ell_M)$.
  Furthermore, $K \cup ([0,\infty) \times L)$ is then a universal $\theta$-end.
\end{proposition}
\begin{proof}
  In the proof of Proposition~\ref{prop:addendum-plus-version-in-C} we saw that
  for each $\alpha \in \Ker(\pi_{n-1}(M) \to \pi_{n-1}(B))$,
  there exists a morphism $W_\alpha \in \mathcal{C}(M,P)$ with $\alpha
  \in \Ker(\pi_{n-1}(M) \to \pi_{n-1}(W_\alpha))$, and for each
  element $\alpha \in \pi_{n}(B)$, there exists a morphism $W_\alpha
  \in \mathcal{C}(M,P)$ with $\alpha \in \IM(\pi_n(W_\alpha) \to
  \pi_n(B))$.  A priori, the target $P$ depends on $\alpha$, but as $\theta$ has been assumed to be spherical, it is reversible (by Proposition~\ref{prop:Reversible}), and we may find
  another morphism $P \leadsto M$; after composing, we may assume that
  $M = P$ so we have endomorphisms $W_\alpha \in \mathcal{C}(M,M)$.
  We then construct a universal $\theta$-end in $\mathcal{C}$ by
  letting $K\vert_i = M$ for each integer $i\geq 0$ and letting each
  $K\vert_{[i,i+1]}$ be of the form $W_{\alpha_i}\# (S^n \times S^n)$,
  where the $\alpha_i$ form a sequence of elements of $\pi_n(B) \cup
  \Ker(\pi_{n-1}(M) \to \pi_{n-1}(B))$ in which each element occurs
  infinitely often.  (This is possible because $\pi_n(B)$ is assumed
  countable and $\pi_{n-1}(M)$ is automatically countable.) It then
  follows from Proposition~\ref{prop:addendum-plus-version-in-C} that
  $K$ is a universal $\theta$-end in
  $\mathcal{C}$.

  It is obvious that gluing $[0,\infty) \times L$ to a universal
  $\theta$-end in $\mathcal{C}$ gives a universal $\theta$-end, since
  the homotopical properties in Proposition~\ref{prop:addendum-plus-version-in-C}
  are clearly preserved.
\end{proof}

\begin{corollary}\label{cor:replace-with-end-in-C}
  Let $(K,\ell_K)$ be a universal $\theta$-end for which $P =
  (K\vert_0, \ell_K\vert_0)$ is an object of
  $\mathcal{C}_{\theta,L}^{n-1,\mathcal{A}}$.  Then we may isotope the
  proper embedding $K \to [0,\infty) \times (-1,1)^\infty$ and
  homotope the bundle map $\ell_K: TK \to \theta^*\gamma$, both
  relative to $K\vert_0$, after which $K$ is of the form $K^\circ \cup
  ([0,\infty) \times L)$ where $K^\circ$ is a universal $\theta$-end in
  $\mathcal{C}$.
\end{corollary}
\begin{proof}
  By Proposition~\ref{prop:addendum-plus-version-in-C}, the structure
  map $\ell_K: K \to B$ induces a surjection on $\pi_n$.  Since $K$ is
  a manifold, $\pi_n(K)$, and hence $\pi_n(B)$, is countable, so
  there exists a universal $\theta$-end in $\mathcal{C}$, by
  Proposition~\ref{prop:existence-theta-ends}.  Denoting this by
  $K^\circ$, the $\theta$-manifold $K^\circ \cup ([0,\infty) \times
  L)$ is a universal $\theta$-end, and hence by Proposition \ref{prop:UniquenessThetaEnd} is isomorphic to the original $K$.
\end{proof}

\subsection{Group completion}
\label{sec:group-completion}

Let us return to the category $\mathcal{C}$ of Section
\ref{sec:category-mathcalc}. Assigning to a morphism $W \in
\mathcal{C}(P_0, P_1)$ the corresponding 1-simplex in the nerve of
$\mathcal{C}$ gives a continuous map $\mathcal{C}(P_0, P_1) \to
\Omega_{P_0, P_1} B\mathcal{C}$, analogous to the map $\mathcal{M} \to
\Omega B\mathcal{M}$ in the outline in
Section~\ref{sec:outl-proof-theor}.  As in that section, the effect in
homology can be studied by a version of the ``group completion''
theorem.  The classical group completion theorem concerns a
topological \emph{monoid} $M$, and says that the map $H_*(M) \to
H_*(\Omega BM)$ is an algebraic localisation at the multiplicative
subset $\pi_0(M) \subset H_*(M)$.  The group completion theorem holds
under the assumption that this localisation \emph{admits a calculus of
  right fractions}, cf.\ \cite{McDuff-Segal}.  A similar result holds
for topological \emph{categories}, and here implies that $H_*(\Omega
B\mathcal{C})$ is a suitable direct limit of
$H_*(\mathcal{C}(P_0,P_1))$, generalising the localisation in the
monoid case.  As in the monoid case, some assumption is needed in
order to apply the group completion theorem:
Lemma~\ref{lem:GCbijection} below can be seen as a multi-object
version of admitting a calculus of right fractions.

\begin{theorem}\label{thm:GC}
  Let
  \begin{equation*}
    K\vert_0 \overset{K\vert_{[0,1]}}\lra K\vert_1
    \overset{K\vert_{[1,2]}}\lra K\vert_2 \overset{K\vert_{[2,3]}}\lra
    K\vert_3 \overset{K\vert_{[3,4]}}\lra \cdots
  \end{equation*}
  be a sequence of composable morphisms in $\mathcal{C}$ such that $K$
  is a universal $\theta$-end in the category $\mathcal{C}$ and
  $\mathcal{C}(\overline{L}, K\vert_0) \neq \emptyset$. Then there is
  a map
  \begin{equation*}
    \hocolim_{i \to \infty} \mathcal{C}(\overline{L}, K\vert_i) \lra
    \Omega B \mathcal{C}
  \end{equation*}
  which is a homology equivalence.
\end{theorem}

The proof will be based on
Proposition~\ref{prop:FinftyHomologyEq} below.  Let $F_i :
\mathcal{C}^\mathrm{op} \to \mathbf{Top}$ denote the representable
functor $\mathcal{C}(-, K\vert_i)$ and let $F_\infty :
\mathcal{C}^\mathrm{op} \to \mathbf{Top}$ denote the (objectwise)
homotopy colimit of the natural transformations $F_i \to F_{i+1}$
given by  right composition with
$K\vert_{[i,i+1]}$.

\begin{proposition}\label{prop:FinftyHomologyEq}
  The functor
  $F_\infty$ sends each morphism in $\mathcal{C}$ to a homology
  equivalence.
\end{proposition}

Given this proposition, by \cite[Theorem 7.1]{GMTW} and the discussion
following it, the pull-back square
\begin{equation*}
\xymatrix{
F_\infty(\overline{L}) \ar[r] \ar[d] & B( \mathcal{C} \wr F_\infty )
\ar[d] & {\!\!\!\!\!\!\!\!\!\!\!\!\!\!\!\!\!\simeq \hocolim_i
B(\mathcal{C}\wr F(i)) \simeq *} \\
{\{\overline{L}\}} \ar[r] & {B \mathcal{C}}
}
\end{equation*}
is homology cartesian and so $F_\infty(\overline{L}) \to \Omega
B\mathcal{C}$ is a homology equivalence, which establishes Theorem
\ref{thm:GC}.

\begin{proof}[Proof of Proposition~\ref{prop:FinftyHomologyEq}]
  By Lemma~\ref{lem:7.8}, it suffices to prove that $F_\infty$ sends
  any cobordism admitting a handle structure with a single $n$-handle
  to a homology isomorphism: indeed, in the notation of that lemma,
  for any cobordism $W$, the functor $F_\infty$ sends both $W \circ F$
  and $G \circ W$ to homology isomorphisms, but then it must send $W$
  to one as well.  We therefore consider a cobordism $W \in
  \mathcal{C}(N,M)$ admitting a handle structure with a single
  $n$-handle.  The cobordism $W$ gives a map of direct systems
  \begin{equation*}
    \xymatrix{
      {\mathcal{C}(M, K\vert_0)} \ar[r] \ar[d]_{W \circ -} & {\mathcal{C}(M,
        K\vert_1)} \ar[r] \ar[d]_{W \circ -} & {\mathcal{C}(M, K\vert_2)} \ar[r]
      \ar[d]_{W \circ -} & {\mathcal{C}(M, K\vert_3)} \ar[r] \ar[d]_{W \circ
        -} & \cdots\\
      {\mathcal{C}(N, K\vert_0)} \ar[r] & {\mathcal{C}(N, K\vert_1)} \ar[r] & {\mathcal{C}(N, K\vert_2)} \ar[r] & {\mathcal{C}(N, K\vert_3)} \ar[r] & {\cdots}.
    }
  \end{equation*}
  Taking homotopy colimits of the rows gives a map $F_\infty(M) \to
  F_\infty(N)$, and Lemma \ref{lem:GCbijection} below
  implies that the induced map on homology is a bijection, finishing the proof of
  Proposition~\ref{prop:FinftyHomologyEq}.
\end{proof}

\begin{lemma}\label{lem:GCbijection}
  Let $W : N \leadsto M$ be a cobordism which is obtained by attaching
  a single $n$-handle to $N$. For each $i$ there is a $k \geq i$ such
  that the commutative square
  \begin{equation}\label{eq:20}
    \begin{gathered}
      \xymatrix{
        {\mathcal{C}(M, K\vert_i)} \ar[rr]^-{-\circ K\vert_{[i,k]}}
        \ar[d]_{W \circ -} && {\mathcal{C}(M, K\vert_k)} \ar[d]^{W \circ -}\\ 
        {\mathcal{C}(N, K\vert_i)} \ar@{.>}[rru] \ar[rr]^-{-\circ K\vert_{[i,k]}} && {\mathcal{C}(N, K\vert_k)}
      }
    \end{gathered}
  \end{equation}
  admits a dotted map making the top square commute up to homotopy,
  and a (possibly different) dotted map making the bottom square
  commute up to homotopy.
\end{lemma}
\begin{proof}
  The objects $M$ and $N$ in $\mathcal{C}$ are $(2n-1)$-dimensional
  submanifolds of $[0,1) \times \R^\infty$ (with $\theta$-structure),
  and the morphism $W \in \mathcal{C}(N,M)$ is a submanifold of $[0,t]
  \times [0,1) \times \R^\infty$.  Rotating $W$ in the first two
  coordinate directions gives a submanifold $\overline{W} \subset
  [0,t] \times (-1,0] \times \R^\infty$ with incoming boundary $\{0\}
  \times \overline{M}$ and outgoing boundary $\{t\} \times
  \overline{N}$.  As in the proof of Proposition~\ref{prop:forget-L},
  the $\theta$-structure on $M$ extends to a $\theta$-structure on the
  closed manifold $\overline{M} \cup M \subset (-1,1) \times
  \R^\infty$, giving an object $\langle M,M\rangle \in
  \mathcal{C}_\theta^{n-1}$ with a canonical null-bordism $V \in
  \mathcal{C}_\theta^{n-1}(\emptyset,\langle M,M\rangle)$.  Similarly,
  we have objects $\langle M,K\vert_i\rangle = \overline{M} \cup
  K\vert_i$ and $\langle N,K\vert_i\rangle = \overline{N} \cup
  K\vert_i$, and the submanifold $\overline{W} \cup ([0,t] \times
  K\vert_i) \subset [0,t] \times (-1,1) \times \R^\infty$ inherits a
  $\theta$-structure from $W$, giving an element of
  $\mathcal{C}_\theta^{n-1}(\langle M, K\vert_i\rangle, \langle N,
  K\vert_i\rangle)$ which we shall denote $\langle W, K\vert_i
  \rangle$. The resulting diagram
  \begin{equation*}
    \xymatrix{
      {\mathcal{C}(M,K\vert_i)} \ar[rrr]^-{W \circ -}\ar[d]_\simeq &&&
      {\mathcal{C}(N,K\vert_i)} \ar[d]^\simeq\\
      {\mathcal{N}^\theta(\langle M, K\vert_i\rangle)}
      \ar[rrr]^-{-\circ\langle W, K\vert_i \rangle}&&&
      {\mathcal{N}^\theta(\langle N,  K\vert_i\rangle)}
    }
  \end{equation*}
  homotopy commutes, where the vertical equivalences are as in
  Proposition~\ref{prop:cut-L}.  The diagram of solid arrows
  in~\eqref{eq:20} may now be replaced with
  \begin{equation*}
    \xymatrix{
      {\mathcal{N}^\theta(\langle M, K\vert_i\rangle)}
      \ar[rrr]^-{-\circ\langle M, {K}\vert_{[i,k]}\rangle}
      \ar[d]_{-\circ \langle W, K\vert_i \rangle} &&&
      {\mathcal{N}^\theta(\langle M, K\vert_k\rangle)}
      \ar[d]^{-\circ\langle W, K\vert_k \rangle}\\ 
      {\mathcal{N}^\theta(\langle N, K\vert_i\rangle)} \ar[rrr]^-{-\circ
        \langle  N, {K}\vert_{[i,k]}\rangle} &&&
      {\mathcal{N}^\theta(\langle N,  K\vert_k\rangle)}
    }
  \end{equation*}
  where $\langle M, {K}\vert_{[i,k]}\rangle = ([i,k] \times
  \overline{M}) \cup K\vert_{[i,k]} \subset [i,k] \times (-1,1) \times
  \R^\infty$, and similarly for $\langle N,
  {K}\vert_{[i,k]}\rangle$.
  
  Let us first show that there is a dotted map making the top triangle
  commute up to homotopy, for some $k \gg i$. We wish to find an
  embedding (of $\theta$-manifolds) of $\langle W, K\vert_i \rangle$
  into $\langle M, K\vert_{[i,k]}\rangle$ relative to $\langle M,
  K\vert_i \rangle$, with complement a $\theta$-cobordism $Z : \langle
  N, K\vert_i \rangle \leadsto \langle M, K\vert_k\rangle$. If we can
  ensure that $(Z,\langle M, K\vert_k\rangle)$ is $(n-1)$-connected,
  then gluing on $Z$ gives a map
  $$-\circ Z : {\mathcal{N}^\theta(\langle N, K\vert_i\rangle)} \lra {\mathcal{N}^\theta(\langle M, K\vert_k\rangle)}$$
  making the top triangle commute (as $\langle W, K\vert_i \rangle \circ Z \cong \langle M, K\vert_{[i,k]}\rangle$ as $\theta$-manifolds), as required.
  
  By definition of the category $\mathcal{C}$, $\overline{M}$ is
  obtained from its boundary, $\partial L$, by attaching handles of
  index $n$ and above. Thus, by transversality, the attaching map for
  the $n$-handle of $\overline{W}$ relative to $\overline{M}$ may be
  assumed to have image in a collar neighbourhood $[-\epsilon,0]
  \times \partial L \subset \overline{M}$. Thus $\langle W, K\vert_i
  \rangle$ may be obtained from $\langle M, {K}\vert_i\rangle$ by
  attaching a single $n$-handle along $f: S^{n-1} \times D^n
  \hookrightarrow [0,\epsilon] \times \partial L \subset K\vert_i$, so
  up to diffeomorphism (relative to its incoming boundary) the
  cobordism $\langle W, K\vert_i \rangle$ is of the form $\langle M,
  W'\rangle$ for some cobordism $W' : K\vert_i \leadsto X$ in
  $\mathcal{C}$. As $K\vert_{[i,\infty)}$ is a universal $\theta$-end
  in the category $\mathcal{C}$, there exists an embedding of
  $\theta$-manifolds $j' : W' \hookrightarrow K\vert_{[i,k]}$ relative
  to $K\vert_i$, for some $k \gg i$, and by Proposition
  \ref{prop:HighConnComplement} we may assume that its complement $Z'$
  is highly connected. Gluing $M$ back in, we obtain an embedding $j :
  \langle W, K\vert_i\rangle \hookrightarrow \langle M,
  K\vert_{[i,k]}\rangle$ relative to $\langle M, K\vert_{k} \rangle$
  whose complement $Z \cong \langle M, Z'\rangle$ is highly connected,
  as required.
    
  To produce the dotted map making the bottom triangle commute up to
  homotopy, we must produce an embedding relative to $\langle N,
  K\vert_k\rangle$ of $\langle W, K\vert_k \rangle$ into $\langle N,
  K\vert_{[i,k]}\rangle$, for some suitably large $k$, with an
  appropriate connectivity condition on its complement.  As we shall
  explain, this reduces to the same embedding problem as for the upper
  triangle.  We have collar neighbourhoods $[-\epsilon,0]
  \times \partial L \subset \overline{N}$ and $[0,\epsilon]
  \times \partial L \subset K\vert_i$, and as above, we can suppose
  $\overline{W}$ is obtained from $\overline{N}$ by attaching a single
  $n$-handle along a map $f : S^{n-1} \times D^n \hookrightarrow
  [-\epsilon,0] \times \partial L \subset \overline{N}$. We now
  consider $\overline{N}$ to lie inside
  \begin{equation*}
    ([i-\epsilon,i] \times (\overline{N} \cup K\vert_i)) \cup K\vert_{[i,\infty)},
  \end{equation*}
  where we may extend $f$ inside $[i-\epsilon,i] \times
  [-\epsilon,\epsilon] \times \partial L$ to an embedding of $[0,1]
  \times S^{n-1} \times D^n$ so that $\{1\} \times S^{n-1} \times D^n$
  is embedded into $\{i\} \times [0,\epsilon] \times \partial L
  \subset K\vert_i$.  Since $K\vert_{[i,\infty)}$ is a universal
  $\theta$-end in $\mathcal{C}$, we may extend the embedding of $\{1\}
  \times S^{n-1} \times D^n$ to an embedding of the handle $\{1\}
  \times D^n \times D^n$ into $K\vert_{[i,\infty)}$ having highly
  connected complement, and such that the $\theta$-structure is
  homotopic to the one given on the $n$-handle of $W$.  By
  compactness, the handle has image in $K\vert_{[i,k]}$ for some $k
  \gg i$, and we extend $f$ cylindrically to an embedding
  \begin{equation*}
    ([-1,1] \times S^{n-1}
    \cup \{1\} \times D^n) \times D^n \stackrel{f}\lra ([i-\epsilon,i]
    \times (\overline{N} \cup K\vert_i)) \cup (K\vert_{[i,k]} \cup
    [i,k] \times \overline{N})
  \end{equation*}
  which sends $\{-1\} \times S^{n-1} \times D^n$ to $\{k\} \times
  \overline{N}$.  The source of this map is diffeomorphic to a tubular
  neighbourhood of the $n$-handle in $W$, and the target is
  diffeomorphic relative to $K\vert_k \cup \overline{N} = \langle N, K
  \vert_k\rangle$ to $K\vert_{[i,k]} \cup ([i,k] \times \overline{N})
  = \langle N, K\vert_{[i,k]}\rangle$.
\end{proof}

The argument above can \emph{not} be improved to show that
$F_\infty$ sends each morphism in $\mathcal{C}$ to a weak homotopy
equivalence, since the dotted maps we constructed in no sense preserve
basepoints.  The case $n=0$ gives rise to the following example from
\cite{McDuff-Segal}: we have $F_\infty(\emptyset) \simeq \bZ \times
B\Sigma_\infty$ and the morphism $1 : \emptyset \leadsto \emptyset$
given by a single point induces the shift map on $\Sigma_\infty$, that
is, the map induced by the self-embedding given by $\{1, 2, \ldots\}
\cong \{2, 3, \ldots\} \hookrightarrow \{1, 2, \ldots \}$. This is not
surjective, so the map is not a homotopy equivalence; it is however a
homology equivalence, by the argument we have presented.

\subsection{Proof of Theorem \ref{thm:main-C-new}}
\label{sec:proof-theor-refthm:m}

In the situation of Theorem \ref{thm:main-C-new} we have a
$\theta$-manifold $(K, \ell_K)$ and a proper map $x_1 : K \to [0,
\infty)$ with the integers as regular values, satisfying the property
of being a universal $\theta$-end (cf.\ Definition
\ref{defn:universalthetaend}).  Let $\theta' : B' \to B
\overset{\theta}\to BO(2n)$ be obtained as the $n$th stage of the
Moore--Postnikov tower of $\ell_K : K \to B$, and $\ell'_K$ be the
$\theta'$-structure on $K$ given by the Moore--Postnikov
factorisation.  By Proposition~\ref{prop:addendum-plus-version-in-C},
the map $K \to B$ induces an injection in $\pi_{n-1}$ and a surjection
in $\pi_n$, so the $n$th and $(n-1)$st stages of the Moore--Postnikov
tower actually agree, and in particular the homotopy fibres of $B'\to
B$ are $(n-2)$-types.

The following two lemmas
allow us to work with $\theta'$-manifolds instead of
$\theta$-manifolds for many purposes.

\begin{lemma}\label{lem:ChangeOfThetaManifold}
  Let $W$ be a manifold with boundary $\partial W$, and suppose that
  $(W, \partial_0 W)$ is $(n-1)$-connected, for $\partial_0 W
  \subset \partial W$ a collection of boundary components. If
  $\ell'_{\partial_0 W}$ is a $\theta'$-structure with underlying
  $\theta$-structure $\ell_{\partial_0 W}$, then
  \begin{equation*}
    \Bun_\partial(TW, (\theta')^*\gamma;\ell'_{\partial_0 W}) \lra
    \Bun_\partial(TW, \theta^*\gamma;\ell_{\partial_0 W})
  \end{equation*}
  is a weak homotopy equivalence.  Consequently, the natural map
  induces a weak equivalence
  \begin{equation*}
    \mathcal{N}^{\theta'}(K\vert_i, \ell_K' \vert_i) \overset\simeq\lra
    \mathcal{N}^{\theta}(K\vert_i, \ell_K \vert_i).
  \end{equation*}
\end{lemma}
\begin{proof}
  As the homotopy fibres of $B' \to B$ are $(n-2)$-types and
  $(W, \partial_0 W)$ is $(n-1)$-connected,  the space of lifts
  \begin{equation*}
    \xymatrix{
      \partial_0 W \ar[r]^-{\ell'_{\partial_0 W}}\ar@{_(->}[d] & B'\ar[d]\\
      W \ar[r]^-{\ell_W}\ar@{..>}[ur] & B
    }
  \end{equation*}
  is contractible, for each $\theta$-structure $\ell_W$ on $W$
  restricting to $\ell_{\partial_0 W}$ on the boundary.  But this
  space of lifts is easily identified with the homotopy fibre of the
  map $\Bun_\partial(TW, (\theta')^*\gamma;\ell'_{\partial_0 W}) \to
  \Bun_\partial(TW, \theta^*\gamma;\ell_{\partial_0 W})$ over the
  point $\ell_W$.

  The last claim follows from the case $\partial_0 W = \partial W$ by
  forming the homotopy orbit space by the action of $\Diff(W,\partial
  W)$ and taking disjoint union over all $W$ with $\partial W =
  K\vert_i$ for which $(W,\partial W)$ is $(n-1)$-connected.
\end{proof}

\begin{lemma}\label{lem:ChangeOfThetaEnd}
  The $\theta'$-manifold $(K, \ell'_K)$ is a universal $\theta'$-end.
\end{lemma}
\begin{proof}
  We verify the conditions of Definition
  \ref{defn:universalthetaend}. The cobordisms $K\vert_{[i-1,i]}$ are
  highly connected, as we have assumed that $K$ is a universal
  $\theta$-end.  If $(W : K\vert_i \leadsto P, \ell'_W)$ is a highly
  connected $\theta'$-cobordism, with underlying $\theta$-structure
  $\ell_W$, then by assumption there is an embedding $j : W \to
  K\vert_{[i,\infty)}$ and a homotopy $\ell_K \circ D j \simeq
  \ell_W$, all relative to $K\vert_i$, but then by Lemma
  \ref{lem:ChangeOfThetaManifold} there is also a homotopy $\ell'_K
  \circ D j \simeq \ell'_W$ relative to $K\vert_i$.
\end{proof}

We can now give the proof of
Theorem~\ref{thm:main-C-new}. Recall that the theorem asserts a
homology equivalence between the homotopy colimit of the direct system
\begin{equation}\label{eq:23}
  \mathcal{N}^{\theta}(K\vert_{0}, \ell_{K}\vert_{0})
  \xrightarrow{K\vert_{[0,1]}} \mathcal{N}^{\theta}(K\vert_{1},
  \ell_{K}\vert_{1}) \xrightarrow{K\vert_{[1,2]}}
  \mathcal{N}^{\theta}(K\vert_{2}, \ell_{K}\vert_{2})
  \xrightarrow{K\vert_{[2,3]}} \cdots
\end{equation}
and the infinite loop space $\Omega^\infty MT\theta'$.  By
Lemmas~\ref{lem:ChangeOfThetaManifold} and \ref{lem:ChangeOfThetaEnd},
it suffices to prove the theorem in the case $\theta = \theta'$, i.e.\
when $\ell_K: K \to B$ is $n$-connected. In order to apply Theorem~\ref{thm:GC}, we first
need to define a $\theta$-manifold $L$ (in order to have the category
$\mathcal{C}$ defined).  To do so, we pick a self-indexing Morse
function $f: K\vert_0 \to [0,2n-1]$ and let $L =
f^{-1}([0,n-\frac12])$.  Then the inclusions $L \to K\vert_0$ and
$K\vert_0 \to K$ are both $(n-1)$-connected, so the structure map $L
\to B$ is $(n-1)$-connected and we have defined the category
$\mathcal{C}$, satisfying Theorem~\ref{thm:GC}.  By
Proposition~\ref{prop:UniquenessThetaEnd} we may replace $(K,\ell_K)$
with any other universal $\theta$-end without changing the homotopy
type of the homotopy colimit~\eqref{eq:23}, as long as $K\vert_0$ is
unchanged, and by Corollary~\ref{cor:replace-with-end-in-C} there
exists a universal $\theta$-end of the form $K^\circ \cup ([0,\infty)
\times L)$, where $K^\circ$ is a universal $\theta$-end in
$\mathcal{C}$.  Now, by Proposition \ref{prop:cut-L} the direct
system~\eqref{eq:23} is homotopy equivalent to
\begin{equation*}
  \mathcal{C}(\overline{L}, {K}\vert_{0}^\circ)
  \xrightarrow{{K}\vert_{[0,1]}^\circ} \mathcal{C}(\overline{L},
  {K}\vert_{1}^\circ) \xrightarrow{{K}\vert_{[1,2]}^\circ}
  \mathcal{C}({L}, {K}\vert_{2}^\circ)
  \xrightarrow{{K}\vert_{[2,3]}^\circ} \cdots.
\end{equation*}
By Theorem~\ref{thm:GC}, the homotopy colimit is homology equivalent
to $\Omega B\mathcal{C}$, which in turn is weakly equivalent to
$\Omega^\infty MT\theta = \Omega^\infty MT\theta'$, by
Theorem~\ref{thm:gp-compl}.  \qed

\subsection{Proof of Lemma \ref{lem:Kcomm} and Theorem
  \ref{thm:main-D}}\label{sec:proof-lemma-refl}

Let us first show that $\mathcal{K} \subset \mathcal{K}_0$ is a
submonoid, and that it is commutative.  Recall that $\mathcal{K}_0$
was the set of isomorphism classes of highly connected cobordisms $K
\subset [0,1] \times \R^\infty$ with $\theta$-structure, starting and
ending at $(P,\ell_P)$ and that $\mathcal{K}$ is the subset admitting
representatives containing $[0,1] \times (P-A)$ with product
$\theta$-structure, where $A \subset P$ is a closed regular
neighbourhood of a simplicial complex of dimension at most $(n-1)$
inside $P$.  Let $K_0, K_1 : P \leadsto P$ be two such cobordisms and
let $K_i$ have support in $A_i$, a regular neighbourhood of a
simplicial complex $X_i$ of dimension at most $(n-1)$.  As $P$ is
$(2n-1)$-dimensional, we can perturb the $X_i$ to be disjoint and then
shrink the $A_i$ so they are disjoint.  But if $W_0$ and $W_1$ have
support in the disjoint sets $A_0$ and $A_1$, then $W_0 \circ W_1$ has
support in $A_0 \amalg A_1$ which is a regular neighbourhood of $X_0
\amalg X_1$ which is again a simplicial complex of dimension at most
$(n-1)$.  Furthermore $K_0 \circ K_1$ is isomorphic to the
$\theta$-bordism $K_{01}$ which is supported in $A_0 \amalg A_1$ and
agrees with $K_i$ on $[0,1] \times A_i$, and this in turn is
isomorphic to $K_1 \circ K_0$, so $\mathcal{K}$ is commutative.

Recall that we have a monoid map $\mathcal{K}' \to \mathcal{K}$, where
$\mathcal{K}'$ is defined like $\mathcal{K}$, but with $\theta'$
instead of $\theta$.  We saw in Lemma~\ref{lem:ChangeOfThetaManifold}
that the map $\mathcal{N}^{\theta'}(P,\ell'_P) \to
\mathcal{N}^\theta(P,\ell_P)$ is a weak equivalence, and we claim that a
similar obstruction theoretic argument shows that $\mathcal{K}' \to
\mathcal{K}$ is an isomorphism.  Explicitly, $[0,1] \times P$ has a
canonical lift of its $\theta$-structure to a $\theta'$-structure.  If
an element of $\mathcal{K}$ is represented by a cobordism $K$
supported in $A \subset P$, it contains the subset $(\{0\} \times P)
\cup ([0,1] \times (P - A)) \cup (\{1\} \times P)$ which has a
canonical $\theta'$-structure.  Because $A$ is a regular neighbourhood
of a simplicial complex of dimension at most $(n-1)$, the manifold $K$
is obtained up to homotopy from this subset by attaching cells of
dimension at least $n$, so up to homotopy there is a unique extension
of the lift. This shows that $\mathcal{K}' \to \mathcal{K}$ is a
bijection.

Before embarking on the proof of Theorem~\ref{thm:main-D}, we
establish the following useful strengthening of
assumption~(\ref{item:21}) of that theorem.

\begin{lemma}\label{lem:3KContains}
  Let $[W] \in \mathcal{K}$ be such that each path component of
  $W$ contains a submanifold diffeomorphic to $S^n \times S^n
  -\Int(D^{2n})$. Then each path component of $3W = W \circ W \circ W$ contains such a
  submanifold which in addition has null-homotopic structure map to
  $B$.
\end{lemma}
\begin{proof}
  Let us suppose that $W$ is path connected: otherwise we repeat the
  argument below for each path component. Finding an embedded $S^n
  \times S^n -\Int(D^{2n})$ is equivalent to finding two embedded
  $n$-spheres with trivial normal bundles, which intersect at a single
  point.  By assumption, this holds for $W$ so we have
  \begin{equation*}
    S^n \times S^n -\Int(D^{2n}) \hookrightarrow W
    \overset{\ell_W}\lra B,
  \end{equation*}
  which in $\pi_n$ induces a homomorphism $\bZ \oplus \bZ = \pi_n(S^n
  \times S^n -\Int(D^{2n})) \to \pi_n(B)$, sending the basis elements
  to $x, y \in \pi_n(B)$.

  In a separate copy of $W$ we have a framed embedding
  \begin{equation*}
    S^n \times \{*\} \overset{\text{reflection}}\cong S^n \times \{*\}
    \hookrightarrow S^n \times S^n -\Int(D^{2n}) \hookrightarrow W
  \end{equation*}
  which in $\pi_n(B)$ gives the element $-x$. Thus in $2W$, the
  connect-sum of this embedded framed sphere and the original one
  gives an embedded framed sphere with null-homotopic map to
  $B$. Using the third copy of $W$ we can fix the remaining sphere,
  without changing the property that the two spheres intersect
  transversely in one point.
\end{proof}

We shall first prove Theorem~\ref{thm:main-D} under an additional
countability hypothesis, namely we prove the following.

\begin{proposition}\label{prop:countable-L}
  Let $\theta$, $(P,\ell_P)$ and $\mathcal{K}$ be as in
  Theorem~\ref{thm:main-D}, and let $\mathcal{L} \subset
  \mathcal{K}$ be a submonoid satisfying conditions~(\ref{item:16}),
  (\ref{item:20}) and (\ref{item:21}) of that theorem.  Assume in
  addition that $\mathcal{L}$ is countable.  Then the induced morphism
  \begin{equation*}
    H_*(\mathcal{N}^{\theta}(P, \ell_P))[\mathcal{L}^{-1}] \lra
  H_*(\Omega^\infty MT\theta')
  \end{equation*}
  is an isomorphism.
\end{proposition}
\begin{proof}
  By countability of $\mathcal{L}$, we may pick a sequence of
  $\theta$-manifolds $(K\vert_{[i,i+1]}, \ell_i)$ which are
  self-bordisms of $(P,\ell_P)$ representing elements of
  $\mathcal{L}$, in a way that each element of $\mathcal{L}$ is
  represented infinitely often.  We then let $K$ be the infinite
  composition of the $K\vert_{[i,i+1]}$, and deduce from
  Addendum~\ref{add:1} that $(K,\ell_K)$ is a universal $\theta$-end.
  (That property~(\ref{it:add:3}) of the Addendum is satisfied follows
  from assumption~(\ref{item:21}) and Lemma~\ref{lem:3KContains}.)
  Then Theorem~\ref{thm:main-C-new} gives a homology equivalence
  \begin{equation*}
    \hocolim \mathcal{N}^\theta(P,\ell_P) \lra \Omega^\infty MT\theta',
  \end{equation*}
  where the homotopy colimit is over composition with the
  $K\vert_{[i,i+1]}$.  Taking homology turns the homotopy colimit into
  a colimit of the $\bZ[\mathcal{L}]$-module
  $H_*(\mathcal{N}^\theta(P,\ell_P))$ over multiplying with elements
  of $\mathcal{L}$, each element occuring infinitely many times.  But
  that precisely calculates the localisation at $\mathcal{L}$.
\end{proof}

The proposition above proves Theorem~\ref{thm:main-C-new} in the case
where $\mathcal{K}$ is countable.  (To apply
Proposition~\ref{prop:countable-L} with $\mathcal{L} = \mathcal{K}$,
we need to check that conditions~(\ref{item:16}), (\ref{item:20}) and
(\ref{item:21}) hold.  This is proved using the manifolds $W_\alpha$
from the proof of Proposition~\ref{prop:addendum-plus-version-in-C}.)
We will deduce the general case by a colimit argument, based on the
following result.

\begin{corollary}\label{cor:L-countable}
  Let $\theta$, $(P,\ell_P)$ and $\mathcal{K}$ be as in
  Theorem~\ref{thm:main-D}, and let $\mathcal{L} \subset
  \mathcal{K}$ be a submonoid satisfying conditions~(\ref{item:20}) and
  (\ref{item:21}) of that theorem, but not
  necessarily~(\ref{item:16}).  Assume in addition that $\mathcal{L}$
  is countable.  Then the induced morphism
  \begin{equation}\label{eq:25}
    H_*(\mathcal{N}^{\theta_\mathcal{L}}(P, \ell_P))[\mathcal{L}^{-1}] \lra
    H_*(\Omega^\infty MT\theta_\mathcal{L})
  \end{equation}
  is an isomorphism, where $\theta_\mathcal{L}: B_\mathcal{L} \to
  BO(2n)$ is obtained as the $n$th Moore-Postnikov factorisation of a
  certain map $\ell: X_\mathcal{L} \to B$, defined as follows.  Each
  self-bordism $(K,\ell_K)$ representing an element of $\mathcal{L}$
  has incoming boundary $P \subset K$, and we let $X_\mathcal{L}$ be
  obtained by gluing every such $K$ along their common incoming
  boundary; the structure maps $\ell_K$ then glue to the map $\ell:
  X_\mathcal{L} \to B$.
\end{corollary}
\begin{proof}
  The structure $\ell_P$ lifts canonically to a
  $\theta_\mathcal{L}$-structure $\ell_P^{\mathcal{L}}$, and we claim
  that every representative of an element of $\mathcal{L}$ admits a
  homotopically unique lift to a $\theta_\mathcal{L}$-manifold which
  is an endomorphism of $(P,
  \ell_P^{\mathcal{L}})$. % \mnote{sg: explain
    % how to obtain endo\\orw: How about this?\\sg: added uniqueness
    % claim (otherwise, would need to pick all lifts in order to have
    % monoid, and worry about countability again)}
  Granted this claim,
  $\mathcal{L}$ satisfies the conditions of
  Proposition~\ref{prop:countable-L} with respect to
  $\theta_\mathcal{L}$, and the result follows.
  
  To prove the claim, let $(K, \ell_K)$ be a self-cobordism of $(P,
  \ell_P)$ which is supported inside $A \subset P$, a regular
  neighbourhood of a simplicial complex of dimension $(n-1)$, and
  consider the lifting problem
  \begin{equation*}
    \xymatrix{
      \{0\} \times P \ar@{^(->}[r]\ar[d]& (\{0,1\} \times P) \cup
      ([0,1] \times P - A) \ar[r]^-{\ell_P^{\mathcal{L}}}\ar@{_(->}[d]
      & B_{\mathcal{L}}\ar[d]\\
     K \ar[rru] \ar@{=}[r]& K \ar[r]^-{\ell_K}\ar@{..>}[ur] & B.
    }
  \end{equation*}
  As the fibre of $B_{\mathcal{L}} \to B$ is an $(n-1)$-type, and the
  pair $(K,(\{0,1\} \times P) \cup ([0,1] \times P - A))$ is
  $(n-1)$-connected, there is a unique obstruction
  \begin{equation*}
    \omega_n \in H^n(K,(\{0,1\} \times P) \cup ([0,1] \times P -
    A);\pi_n(B, B_{\mathcal{L}}))
  \end{equation*}
  to finding the desired lift. As $A$ is a regular neighbourhood of a
  simplicial complex of dimension $(n-1)$, the group $H^{n-1}((\{0,1\}
  \times P) \cup ([0,1] \times P - A), \{0\} \times P; \pi_n(B,
  B_{\mathcal{L}}))$ vanishes, so $\omega_n$ is zero if and only if it
  is zero when restricted to the group $H^n(K, \{0\} \times P;
  \pi_n(B, B_{\mathcal{L}}))$. But $K$ has a canonical lift relative
  to $\{0\} \times P$, so this last obstruction is zero.

  Uniqueness of (homotopy classes of) lifts is similar, but easier.
\end{proof}

\begin{proof}[Proof of Theorem~\ref{thm:main-D}]
Let $\mathcal{L} \subset \mathcal{K}$ satisfy the conditions of Theorem~\ref{thm:main-D}. We may replace $\theta: B \to BO(2n)$ by $\theta': B' \to BO(2n)$,
  the $(n-1)$st Moore--Postnikov stage of $\ell_P: P \to B$, in the statement of Theorem~\ref{thm:main-D}. Then for each countable submonoid $\mathcal{L}' \subset \mathcal{L}$, the maps $B_{\mathcal{L}'} \to B$ from
  Corollary~\ref{cor:L-countable} factor canonically as $B_{\mathcal{L}'}
  \to B' \to BO(2n)$, and if $\mathcal{L}'' \subset \mathcal{L}'$ is a submonoid we also have a factorisation $B_{\mathcal{L}''} \to
  B_{\mathcal{L}'} \to B'$, as our description of
  $B_\mathcal{L}$ is strictly functorial in the monoid $\mathcal{L}$
  (using a functorial model for Moore--Postnikov factorisation). Therefore we may form the colimit of the isomorphisms~\eqref{eq:25}
  over the poset of countable submonoids $\mathcal{L}' \subset
  \mathcal{L}$.
  
  Using the manifolds $W_\alpha$ constructed in the proof of
  Proposition \ref{prop:addendum-plus-version-in-C}, it is easy to see
  that the homotopy colimit of the $B_{\mathcal{L}'}$ is (weakly
  equivalent to) $B'$, and hence the homotopy colimit of the
  $\Omega^\infty MT\theta_{\mathcal{L}'}$ is $\Omega^\infty
  MT\theta'$.  If we can prove that the homotopy colimit of the spaces
  $\mathcal{N}^{\theta_{\mathcal{L}'}}(P,\ell^{\mathcal{L}'}_P)$ is
  $\mathcal{N}^{\theta'}(P,\ell_P') \simeq
  \mathcal{N}^\theta(P,\ell_P)$, Theorem~\ref{thm:main-D} will
  therefore follow as the direct limit of the
  isomorphism~\eqref{eq:25}.

  We saw in Lemma~\ref{lem:ChangeOfThetaManifold} that the map
  $\mathcal{N}^{\theta'}(P,\ell'_P) \to \mathcal{N}^\theta(P,\ell_P)$
  is a weak equivalence.  A similar obstruction-theoretic argument as
  in that lemma shows that the map
  \begin{equation*}
    \mathcal{N}^{\theta_{\mathcal{L}'}}(P,\ell^{\mathcal{L}'}_P) \to \mathcal{N}^\theta(P,\ell_P)
  \end{equation*}
  induces an injection on $\pi_0$ and a weak equivalence of each path
  component.  Viz., $\mathcal{N}^{\theta_{\mathcal{L}'}}(P,\ell^{\mathcal{L}'}_P)$ is up
  to homotopy a disjoint union of path components of
  $\mathcal{N}^{\theta}(P,\ell_P)$; the component containing
  $(W,\ell_W)$ is included precisely when $\ell_W$ admits a lift to a
  $\theta_{\mathcal{L}'}$-structure.  Up to homotopy, the system
  of spaces $\mathcal{N}^{\theta_{\mathcal{L}'}}(P,\ell^{\mathcal{L}'}_P)$ therefore
  just consists of including more and more components of
  $\mathcal{N}^{\theta}(P,\ell_P)$, including all of them in the
  colimit.
\end{proof}

%%% Local Variables: 
%%% mode: latex
%%% TeX-master: "Moduli"
%%% End: 

\bibliographystyle{amsalpha}

\bibliography{biblio}

\end{document}